\documentclass{amsart}	
    \usepackage{amsmath,amssymb,amsfonts,mathrsfs, amsthm}
	\usepackage{amsaddr}
	\usepackage{dsfont}
	\usepackage{calc}
	\usepackage{bm} 
	\usepackage{algorithm}
	\usepackage{csquotes} 
\usepackage[noend]{algpseudocode}
	\usepackage{eufrak} 
	\usepackage{enumitem}
	\usepackage{subfiles} 
	\usepackage{theoremref}
	\usepackage{verbatim}
	\usepackage{graphicx}
	\usepackage{float} 
	\usepackage{hyperref} 
	\usepackage{hypcap} 
	\usepackage{fancyhdr} 
	\newlength\mylen
	\newlist{mycases}{enumerate}{1}
	\setlist[mycases,1]{label=\textbf{Case~\arabic*.}, 
		labelwidth=\dimexpr-\mylen-\labelsep\relax,leftmargin=0pt,align=right}
	\makeindex

\begin{document}
\pagestyle{plain}
\fancyhead{}

		\bibliographystyle{alpha}
		\newtheorem{theorem}{Theorem}[section]
		\newtheorem{lemma}{Lemma}[section]
		\newtheorem{definition}{Definition}[section]
		\newtheorem{proposition}{Proposition}[section]
		\newtheorem{remark}{Remark}[section]
		\newtheorem{property}{Property}[section]
		\newtheorem{corollary}{Corollary}[section]
		\newtheorem{convention}{Convention}[section]
		\newcounter{casenum}
		\newenvironment{caseof}{\setcounter{casenum}{1}}{\vskip.5\baselineskip}
		\newcommand{\case}[2]{\vskip.5\baselineskip\par\noindent {\bfseries Case \arabic{casenum}:} #1\\#2\addtocounter{casenum}{1}}
		\newcommand{\bigO}{\ensuremath{\mathop{}\mathopen{}\mathcal{O}\mathopen{}}}
		
		\newcommand\smallO[1]{
			\mathchoice
			{
				\scriptstyle\mathcal{O}
			}
			{
				\scriptstyle\mathcal{O}
			}
			{
				\scriptscriptstyle\mathcal{O}
			}
			{
				\scalebox{0.8}{$\scriptscriptstyle\mathcal{O}$}
			}
		}
		
		\newlength{\dhatheight}
      \newcommand{\doublehat}[1]{
    \settoheight{\dhatheight}{\ensuremath{\hat{#1}}}%
    \addtolength{\dhatheight}{-0.35ex}%
    \hat{\vphantom{\rule{1pt}{\dhatheight}}%
    \smash{\hat{#1}}}}

		\title{The Word and Conjugacy Problems in Lacunary Hyperbolic Groups}
		\author{Arman Darbinyan}
		
		\maketitle

	\begin{abstract}
	
We study the word and conjugacy problems in  lacunary hyperbolic groups (briefly, LHG). In particular, we  describe a necessary and sufficient condition for decidability of the word problem in LHG. Then, based on the graded small-cancellation theory of Olshanskii, we develop a  general framework which allows us to construct lacunary hyperbolic groups with word and conjugacy problems highly controllable and flexible both in terms of computability and computational complexity.

As an application, we show that for any recursively enumerable subset $\mathcal{L} \subseteq \mathcal{A}^*$, where $\mathcal{A}^*$ is the set of words over arbitrarily chosen non-empty finite alphabet $\mathcal{A}$, there exists a lacunary hyperbolic group $G_{\mathcal{L}}$ such that the membership problem for $ \mathcal{L}$ is `almost' linear time equivalent to the conjugacy problem in $G_{\mathcal{L}}$. Moreover, for the mentioned group the word and individual conjugacy problems are decidable in `almost' linear time. 

Another application is the construction of a lacunary hyperbolic group with `almost' linear time word problem and with all the individual conjugacy problems being undecidable except the word problem.

As yet another  application of the developed framework, we construct infinite verbally complete groups and torsion free Tarski monsters, i.e. infinite torsion-free groups all of whose proper subgroups are cyclic, with `almost' linear time word and polynomial time conjugacy problems.  
 These groups are constructed as quotients of arbitrarily given non-elementary torsion-free hyperbolic groups and are lacunary hyperbolic.

Finally, as a consequence of the main results, we answer a few open questions. 
 

\end{abstract}

\tableofcontents  

	\section{Introduction}
	Traditionally, computability questions of \textit{word and conjugacy problems} in groups, along with the groups isomorphism problem, are considered as some of the most important properties and questions in combinatorial and geometric group theory. For a given finitely generated group $G = \langle X \rangle$, $|X|<\infty$, the word problem is an algorithmic problem of deciding whether any arbitrarily given word $W\in X^*$ represents  the trivial element in $G$ or not. Here and later, whenever a set, say $X$, is a set of group generators, by $X^*$ we denote the set of all words in the alphabet $X \cup X^{-1}$.\index{word problem} \index{conjugacy problem} Otherwise, if $X$ is merely a (finite) set, then $X^*$ means the set of all finite words composed by letters from $X$. 
	
	The conjugacy problem considers on input an arbitrary pair $(U, V)\in X^* \times X^*$ and decides whether $U$ is conjugate to $V$ in $G$ or not. If for the word problem in $G$ such a decision algorithm exists, then it is said that the word problem (briefly, WP) is decidable in $G$. Analogously, if there is an decision algorithm for the conjugacy problem (briefly, CP) in $G$, then it is said that the conjugacy problem is decidable in $G$. 
	
	Observe that, since the triviality of an element of $G$ is equivalent to the fact that it is conjugate to the trivial element of $G$,  decidability of the conjugacy problem in $G$ implies   decidability of the word problem. Another obvious observation is that the decidability of WP and CP do not depend on the choice of the finite generating set. 
	
	Word and conjugacy problems in groups first were introduced by Max Dehn in 1911. A bit later, in 1912, Max Dehn described algorithms for word and conjugacy problems for surface groups (i.e. fundamental groups of two dimensional manifolds) for surfaces of genus $g \geq 2$. The algorithm described by him for the word problem is one of the most important word problem solving algorithms. It is one of the most important word problem solving algorithms not only because of its simplicity and good time complexity behavior or, say, because of its historical importance, but also because, based on generalizations of underlying properties of surface groups, this algorithm was generalized to a much broader class of groups, called hyperbolic groups (or, word hyperbolic groups). The notion of hyperbolic groups was first introduced by Gromov in his seminal paper \cite{gromov-hyperbolic}. In fact,  it is well-known that hyperbolic groups are essentially  the finitely presented groups on which one can extend  Dehn's original algorithm for the word problem in surface groups. See, for example, \cite{gromov-hyperbolic, lysenok}. 
	
	To describe Dehn's algorithm\index{Dehn's algorithm}, let us consider any finitely presented group $G$ with its finite presentation
	     \begin{equation}
	     \label{dehn's presentation}
	     G = \langle X \mid \mathcal{R} \rangle.	
	     \end{equation}
	     Then the presentation \eqref{dehn's presentation} is said to be \textit{Dehn's presentation}\index{Dehn's presentation of groups} if the following property holds: $\mathcal{R}$ is a finite symmetric set of words (i.e. it is closed under operations of taking cyclic shifts and inverses of words); for any freely cyclically reduced word $W \in X^*$, if $W =_G 1$ (i.e. $W$ represents the trivial element in $G$), then there exists a word $R=R_1R_2 \in \mathcal{R}$ such that $\|R_1\|>\|R_2\|$, and a cyclic shift $W'$ of $W$ such that $W'= W_1R_1W_2$. (Throughout this text, by the symbol $\|\cdot \|$ we denote lengths of words in a given alphabet. Another notation which we use in this work extensively is the following: For $G=\langle X \rangle$ suppose $U, V \in X^*$, then $U=_GV$ means that the words $U$ and $V$ represent the same element from $G$.)
	     
	     Note that if \eqref{dehn's presentation} is a Dehn's presentation, then to check whether or not a cyclically reduced word $W \in X^*$ is trivial in $G$, one can simply consider all cyclic shifts of $W$ and all relator words from $\mathcal{R}$ in order to find the above mentioned cyclic shift $W'$ and relator word $R=R_1R_2$. Then the key observation is that $W=_G 1$ if and only if $W_1R_2^{-1} W_2=_G 1$. But $\|W_1R_2^{-1} W_2\| < \|W\|$. Thus the word problem for $W$ is reduced to the word problem for a strictly shorter word $W_1R_2^{-1} W_2$. Next, in order to check whether or not $W_1R_2^{-1} W_2 =_G 1$, in a similar way as for $W$, we can try to reduce this question to the word problem for a shorter word. If at some point this shortening procedure cannot be applied anymore, then it means that either we obtained an empty word, hence we conclude $W=_G 1$ or, otherwise, we conclude $W \neq_G 1$. Also it is clear that this procedure of shortening can be applied only finitely many times (bounded from above by $\|W\|$), hence the process will eventually halt, giving us the wanted answer about triviality of $W$ in $G$. Since this procedure is based on the original algorithm of Dehn, following the established tradition, we call it \textit{Dehn's algorithm.} \index{Dehn's algorithm} 
	     
	   Note that there exist finitely presented groups with undecidable word problem. In fact, the first examples of finitely presented groups with algorithmically undecidable word problem were given by Novikov in 1955, see \cite{novikov} and independently by Boone in 1958, see \cite{boone}. These results of Novikov and Boone are considered as one of the most important and classical results in the algorithmic theory of groups. Another famous example is a construction by Kharlampovich (see \cite{kharlampovich}), where the first example of finitely presented solvable group with undecidable word problem was constructed, answering a long standing open problem by Adian. 
	   
	   Speaking about  word and conjugacy problems in finitely generated groups, there are several key aspects one might consider. Below we mention some of them.
	   	   \begin{enumerate}
	   	\item[(a).] Whether or not the WP (resp. CP) is decidable?
	   	\item[(b).] If it is undecidable, what is the Turing degree of undecidability of the WP (resp. CP)?
	   	\item[(c).] If it is decidable, what computational complexity classes does it belong to?
	   \end{enumerate}
	   	   
	   Note that for a given group the answer to (a) reveals not only computational properties of the group, but also its algebraic properties. This follows, for example, from a classical theorem of Boone and Higman, \cite{boone-higman, lyndon schupp}, which says that a finitely generated group $G$ has decidable word problem if and only
if $G$ can be embedded in a simple subgroup of a finitely presented group. Moreover, after the works of Gromov \cite{gromov-asymptotic}, Sapir, Birget, Rips \cite{sapir-birget-rips},  Birget, Olshanskii,  Rips, Sapir \cite{birget-olshanskii-rips-sapir}, Olshanskii \cite{olshanskii_space_functions}, Grigorchuk, Ivanov \cite{grigorchuk-ivanov}, Bridson \cite{bridson-geometry_of_word_problem} and others, it becomes apparent that the answer to the question 3 may reveal information not only about the computational properties of the group, but also about its topological and geometric properties. Therefore, in the light of  modern developments in the theory of groups, investigation of these questions is important from the perspective of computational, algebraic, topological and geometric points of view. Note that since for any two finite sets of generators $X$ and $Y$ of a given group, the words in $X^*$ can be in linear time translated into corresponding words in $Y^*$, the answer to the above formulated questions (a), (b) and (c) is independent of finite sets of group generators.
	   
	   We would like to mention that even the question of existence of a lacunary hyperbolic group with decidable word problem and undecidable conjugacy problem was still open. This question was asked by Olshanskii, Osin and Sapir as Problem 7.5 in \cite{olshanskii-osin-sapir}. A positive answer to this question follows from Theorems \ref{theorem_about_connecton_of_word_and_conjugacy_problems} and \ref{theorem-schupp-miasnikov-question} of the current paper.\\

	  In this paper we systematically study all the above mentioned aspects of word and conjugacy problems in the class of so called lacunary hyperbolic groups, with a special emphasize on the ones obtained via small cancellation techniques.
	   
	   The formal definition of the class of lacunary hyperbolic groups (more briefly, LHG) was first introduced by Olshanskii, Osin and Sapir in \cite{olshanskii-osin-sapir}. Intuitively, lacunary hyperbolic groups can be thought of as the finitely generated but not necessarily finitely presented versions of word hyperbolic groups. In the next sections we will recall the mathematically rigorous definitions of both hyperbolic and lacunary hyperbolic groups. But for this introductory part let us just add to the already mentioned that all lacunary hyperbolic groups are inductive limits of hyperbolic groups as it is established in \cite{olshanskii-osin-sapir} and recalled in Lemma \ref{lem lacunary hyp gps} of the current work.
	   
	   Speaking about inductive limits of hyperbolic groups, here we would like to mention that many such groups were constructed by using various generalized small cancellation techniques and many of them possess various exotic group theoretical properties. For example, this way Olshanskii constructed Burnside  groups of large exponents \cite{ol'shanskii - burnside problem} and \cite{ivanov_burnside} and Tarski Monsters \cite{ol'shanskii - torsion-free Tarski monsters, ol'shanskii - Tarski monsters}. 
	   For a more complete exposition of these constructions see also \cite{olshanskii_the_book}.
	   
	   Following an already established tradition, we call the groups which possess  exotic properties and are obtained as inductive limits of hyperbolic groups via small cancellation techniques, \textit{monster groups}. \index{monster groups}
	   
	   For the monster groups appearing, for example, in \cite{olshanskii_the_book}, in the currently existing literature there are no known time complexity effective algorithms for the basic decision problems such as the word and conjugacy problems. The methods developed in this work help us to construct monster groups with effective word and conjugacy problems, see Theorems \ref{theorem_verbally_complete} and \ref{theorem_tarskii_monsters}.
	    
~\\
\textbf{Acknowledgements.} I am grateful to my advisor Alexander Olshanskii for his encouragement to work on this subject and for his comments and suggestions which were invaluable. Also I would like to thank Goulnara Arzhantseva for pointing out several misprints and inaccuracies.
	\section{Main results}
	
	The main objective of this paper is twofold.
	
 First, based on the small cancellation theory of Olshanskii (see \cite{Olsh G-groups}), we describe general constructions of lacunary hyperbolic groups under which the word and conjugacy problems can be effectively reduced to much simpler problems. 

	Even more, we develop a general framework in sections \ref{section Small cancellation conditions}--\ref{section-about-general-scheme} which provides  with necessary tools to understand the rich nature of word and conjugacy problems in the class of LHG. In fact, this framework will allow us to shed light on the rich nature of word and conjugacy problems in LHG from several perspectives. More specifically:  
	\begin{enumerate}
		\item From the perspective of computability, e.g. in Theorem \ref{theorem iff condition for wp in lac hyp gp 1} we formulate an "if and only if" condition for decidability of WP. Also we develop necessary tools to construct lacunary hyperbolic groups with decidable word problem and undecidable conjugacy problem; 
		\item From the perspective of computational complexity theory; and
		\item From the perspective of interconnection of WP and CP in the class of LHG, both in terms of computability and  computational complexity.

	\end{enumerate}
	
	 
	 Second, we use the developed framework to formulate the main theorems of this paper, that is Theorem \ref{theorem_verbally_complete}, Theorem \ref{theorem_tarskii_monsters}, Theorem \ref{theorem_about_connecton_of_word_and_conjugacy_problems} and Theorem \ref{theorem-schupp-miasnikov-question}. The first two theorems, in particular, show that versions of some of the most prominent groups of the class of LHG can be constructed in such a way that they will have fast WP and CP. The third theorem shows in particular that WP and CP are `almost' completely independent one of another in the class of lacunary hyperbolic groups, not only from the perspective of computability, but also from the perspective of computational complexity. \\
	 
	 
	 Below we describe the content of the paper in more details.\\
	~\\
	
	Even though the original definition of lacunary hyperbolic groups involves the concept of asymptotic cones, there exist equivalent and more algebraic definitions. In this work we employ the following definition (see Lemma \ref{lem lacunary hyp gps} and Remark \ref{remark 000}):  finitely presented group  $\bar{G}=\langle X \rangle$ is lacunary hyperbolic if and only if $\bar{G}$ is the inductive limit of a chain of epimorphisms
	\begin{align}
	\label{seq 000}
		G_1 \stackrel{\alpha_1}\twoheadrightarrow G_2 \stackrel{\alpha_2}\twoheadrightarrow \ldots,
	\end{align}
	where $\alpha_i: G_i \twoheadrightarrow G_{i+1}$ is the induced epimorphism from the identity map $id: X \rightarrow X$  for $i\in \mathbb{N}$, $G_i=\langle X \mid \bar{\mathcal{R}}_i \rangle$ is finitely presented hyperbolic group and   the hyperbolicity constant of $G_i$ (relative to $X$) is “little o” of
		the radius of $\alpha_i$. \textit{Radius} is defined as follows: For $G=\langle X \rangle$ and $\alpha: G \rightarrow G'$, radius  of $\alpha$ is the maximal radius of a ball in the Cayley graph $\Gamma(G, X)$ centered at $1_G$ such that all elements from that ball map to non-trivial elements in $G'$ except for $1_G$.
		
		We say that $\bar{G}$ has a \textit{graded recursive presentation} by hyperbolic groups with respect to \eqref{seq 000} if the map $i \mapsto \long\bar{\mathcal{R}}_i$ is computable.

		For $\Upsilon: \mathbb{N} \rightarrow \mathbb{N}$, we call $\Upsilon$ a \textit{supradius} \index{supradius} for \eqref{seq 000}, if for all $n \in \mathbb{N}$ and $i \in \mathbb{N}$ such that $i \geq \Upsilon(n)$, the radius of $\alpha_i: G_i \twoheadrightarrow G_{i+1}$ is greater than $n$.
We say that $\Upsilon: \mathbb{N} \rightarrow \mathbb{N}$ is a \textit{computable} supradius if $\Upsilon$ as a function is computable, i.e. the set $\{(i, \Upsilon(i)) \mid i \in \mathbb{N} \}$ is recursive.

	In Section \ref{section-lacunary-hyperbolic} we prove the following  theorem.
	\begin{theorem}[Theorem \ref{theorem iff condition for wp in lac hyp gp 1}]
		Let $\bar{G}$ be an inductive limit of hyperbolic groups connected by epimorphisms. Then $\bar{G}$ has decidable word problem if and only if it has a graded recursive presentation by hyperbolic groups and a recursively computable supradius function over that presentation.	
	\end{theorem}
	~\\
		
	The main object of investigation in this paper are the following type of chains of hyperbolic groups satisfying some special conditions.
	\begin{align}
	    \label{main seq of gps}
		G_0 \stackrel{\beta_0}\hookrightarrow H_1 \stackrel{\gamma_1}\twoheadrightarrow G_1 \stackrel{\beta_1}\hookrightarrow H_2 \stackrel{\gamma_2}\twoheadrightarrow \ldots.
	\end{align}
	If we denote $\alpha_i= \gamma_{i+1}\circ \beta_i$, then we always assume that $\alpha_i$ is surjective for $i=1,2, \ldots$. All the groups in this chain are assumed to be hyperbolic. Let $G_0 = \langle X \mid \mathcal{R}_0 \rangle$ be given with its initial finite presentation and let for all $i \in \mathbb{N}$,
	\begin{align}
				H_i =G_{i-1}*F(Y_i)/ \ll \mathcal{S}_i\gg, 
	\end{align}
	where $|Y_i|<\infty$, $Y_i \cap \beta_{i-1}(G_{i-1})=\emptyset$, $\mathcal{S}_i$ is a finite (symmetric) set of words from $(X \cup Y_i)^*$ and $F(Y_i)$ is the free group with basis $Y_i$. Also
	\begin{align}
		G_i &=H_i / \ll \mathcal{R}_i \gg,
	\end{align}
	where  $\mathcal{R}_i$ is a finite symmetric set of words from $(X \cup Y_i)^*$ satisfying certain small cancellation conditions.
	
	The main group of our interest is the group $\bar{G}=\langle X \rangle$, $|X|< \infty$, defined as the inductive limit
	\begin{align*}
		\bar{G} =  \lim_i (G_i, \alpha_i).
	\end{align*}
		
	In Section \ref{section_about_special_class_of_LHG}	we introduce the concepts of $G$- and $H$-conjugacies for the group $\bar{G}$ defined as follows: For $u, v \in X^*$ we say that $u$ is $H$-conjugate to $v$ if there exists $i\in \mathbb{N}$ such that $u$ is conjugate to $v$ in $H_i$ but nevertheless $u$ is not conjugate to $v$ in $G_{i-1}$. $G$-conjugacy is defined analogously, namely, $u$ is $G$-conjugate to $v$ in $\bar{G}$ if either $u$ is conjugate to $v$ in $G_0$ or there exists $i\in \mathbb{N}$ such that $u$ is conjugate to $v$ in $G_i$ but $u$ is not conjugate to $v$ in $H_i$. Clearly, $u$ is conjugate to $v$ in $\bar{G}$ if and only if either $u$ is $H$-conjugate to $v$ or $G$-conjugate to $v$ in $\bar{G}$.
	
	 In the same section we introduce a special small cancellation condition\\ 
	$C'\big(\mathcal{TM}, (g_i)_{i=1}^{\infty},  ({\rho}_i)_{i=1}^{\infty} \big)$ which assures that word problem and $G$-conjugacy problem for $\bar{G}$ can be solved in polynomial time granted that the words $\mathcal{R}_i$, $i=1,2, \ldots$, are effectively computable. Note that this condition does not tell us about effectiveness of the $H$-conjugacy problem. In fact, as the proof of Theorem \ref{theorem_about_connecton_of_word_and_conjugacy_problems} reveals, $H$-conjugacy problem in general can have an arbitrary behavior not depending on the behavior of, say, $G$-conjugacy problem.
	
	In Subsections \ref{words with small cancellation-1} and \ref{subsection-a-class-of-small-cancellation-words}, we describe constructions of words which can be easily constructed, have appropriate small-cancellation properties, and they will serve in Sections \ref{section-verbally_comlete_groups}, \ref{section-Tarskii_monsters} and  \ref{section_problem_7-5} as the main ingredient for describing the specifications of the words $\mathcal{R}_i$, $i=1,2, \ldots$ for appropriate constructions. It is worthwhile  to mention here that Sections \ref{section-verbally_comlete_groups}, \ref{section-Tarskii_monsters} and  \ref{section_problem_7-5} provide the proofs of the main applications of the general framework, that is the proofs of Theorems \ref{theorem_verbally_complete}, \ref{theorem_tarskii_monsters} and \ref{theorem_about_connecton_of_word_and_conjugacy_problems}, and all the proofs are constructive and based on a general scheme described in Section \ref{section-about-general-scheme}. On its own turn, the general scheme from 	Section \ref{section-about-general-scheme} is based on the already mentioned general framework developed mostly in Sections \ref{section Small cancellation conditions}, \ref{section-slender conjugacy dyagrams}, \ref{section-algorithms} and \ref{section_about_special_class_of_LHG}.

	
		Concerning the groups $H_i$, $i=1,2, \ldots$, in the main applications in Sections \ref{section-verbally_comlete_groups}, \ref{section-Tarskii_monsters}, \ref{section_problem_7-5}, we consider two main situations: First, when $H_i = G_{i-1}$ and $\beta_{i-1}=id$ and second, when $H_i$-s are obtained as HNN-extensions of $G_{i-1}$. 
		~\\
		\begin{definition}
		Let $f: \mathbb{N} \rightarrow \mathbb{N}$	be a positive integer valued function, and let $\mathcal{D}$ be any decision problem. Then we say that $\mathcal{D}$ can be solved in almost $f(n)$ time, if for any $\varepsilon>0$  the problem $\mathcal{D}$ belongs to $DTime\big(n^{\varepsilon}f(n)\big)$, or in other words, it belongs to $\bigcap_{k=1}^{\infty}DTime\big(n^{1/k}f(n)\big)$. If $f(n)=n$, $n\in \mathbb{N}$, then we say that $\mathcal{D}$ is decidable in almost linear time\index{almost linear time} (similarly we define almost quadratic time, etc).
		\end{definition}

		\subsection{Main theorems}
		The following theorem 
		concerns verbally complete infinite quotients of arbitrary torsion-free non-elementary hyperbolic groups with nice algorithmic properties.
		
		Recall that the group $G' = \langle X \rangle$ is {\textit{verbally complete}\index{verbally complete groups} if for any element $g\in G'$ and for any non-trivial element $w$ from a countably generated free group $F=F(y_1, y_2, \ldots)$, the equation $w=g$ has a solution in $G'$. In other words, there exists a homomorphism $h: F \rightarrow G'$ such that $h: w \mapsto g$.
		\begin{theorem}
		\label{theorem_verbally_complete}
		Let $G$ be an arbitrary torsion-free, non-elementary hyperbolic group. Then there exists a lacunary hyperbolic infinite torsion-free quotient $\check{G}$	of $G$ such that the following is true about $\check{G}$.
		\begin{enumerate}
			\item [(i).] $\check{G}$ is a verbally complete group,\\
			\item [(ii).] The word problem in $\check{G}$ is decidable in almost quadratic time and the conjugacy problem in $\check{G}$ is decidable in polynomial time.
		        
		\end{enumerate}
		\end{theorem}
         Note that part $(i)$ of Theorem \ref{theorem_verbally_complete} appears in the work of Mikhajlovskii and Olshanskii, \cite{mikhajlovskii-ol'shanskii}. Since verbally complete groups are divisible groups,  Mikhajlovskii and Olshanskii's work can be regarded as a generalization of a result of Guba from 1987, which is about the existence of finitely generated non-trivial divisible groups. Note that the question of   existence of such groups was regarded as a long standing open problem prior its solution. To achieve the result of Theorem \ref{theorem_verbally_complete}, we elaborate the original construction of Mikhajlovskii and Olshanskii and combine it with the machinery developed in this paper.
         
         Let us  mention that for the group $\check{G}$ from Theorem \ref{theorem_verbally_complete}, there exists an algorithm such that for all inputs $w \in F(y_1, y_2, \ldots) \setminus \{1\}$ and $\check{g} \in \check{G}$, the algorithm finds a solution for the equation $w=\check{g}$ in $\check{G}$.  Indeed, to solve the equation  $w=\check{g}$ in $\check{G}$, one can just check for all possible values of variables $y_1, y_2, \ldots$, whether $w=\check{g}$ in $\check{G}$ or not. Since the word problem in  $\check{G}$ is decidable and $\check{G}$ is verbally complete, this procedure will eventually halt.\\		
         

         ~\\
         ~\\
         \begin{theorem}
         \label{theorem_tarskii_monsters}
         	Let $G$ be an arbitrary torsion-free, non-elementary hyperbolic group. Then there exists a non-cyclic torsion-free lacunary hyperbolic quotient $\hat{G}$	of $G$ such that the following is true about $\hat{G}$.
         	\begin{enumerate}
         		\item[(i).] Every proper subgroup of $\hat{G}$ is an infinite cyclic group,\\
         		\item[(ii).] The word problem in $\hat{G}$  is decidable in almost quadratic time and the conjugacy problem in $\hat{G}$  is decidable in polynomial time.
         	\end{enumerate}
         \end{theorem}
         Note that the first example of an infinite non-cyclic group with the property of part $(i)$ appears in \cite{ol'shanskii - torsion-free Tarski monsters} and the exact statement of Theorem \ref{theorem_tarskii_monsters} but only with part $(i)$ appears in \cite{Olsh G-groups}. Construction of $\hat{G}$ can be regarded as a more elaborated version of the corresponding result from \cite{Olsh G-groups} combined with the machinery developed in this paper.\\
         
         Let us also mention that from the method by which the groups $\check{G}$ and $\hat{G}$ are constructed it follows that for every torsion-free, non-elementary hyperbolic $G$, there are continuum many pairwise non-isomorphic quotients of $G$ satisfying the statements (i) of Theorem \ref{theorem_verbally_complete} and Theorem \ref{theorem_tarskii_monsters}, respectively. However, the cardinality of  groups satisfying all the conditions of Theorem \ref{theorem_verbally_complete} and Theorem \ref{theorem_tarskii_monsters}, respectively, is $\aleph_0$. (In fact,  the cardinality of finitely generated groups with decidable word problem is  $\aleph_0$.)\\
         ~\\

\begin{definition}[Strong (many-one) reduction]
\index{strong reduction}
Let $\mathcal{L}_1 \subseteq \mathcal{A}_1^*$ and $\mathcal{L}_2\subseteq \mathcal{A}_2^*$, where $\mathcal{A}_1$ and $\mathcal{A}_2$ are finite alphabets. Then $\mathcal{L}_1$ is strongly (many-one) reducible to $\mathcal{L}_2$ if there exists a computable function $\phi: \mathcal{A}^*_1 \rightarrow \mathcal{A}^*_2$ and a constant $C>0$ such that for all $x \in \mathcal{A}_1^*$ we have $\|\phi(x)\|_{\mathcal{A}_2} \leq C\|x\|_{\mathcal{A}_1}$ and $\phi(\mathcal{L}_1) = \mathcal{L}_2$, $\phi( \mathcal{A}_1^* \setminus \mathcal{L}_1) \subseteq  \mathcal{A}_2^* \setminus \mathcal{L}_1$. Also if for some $g: \mathbb{N} \rightarrow \mathbb{N}$ and for all $ x\in \mathcal{L}_1 $ the value of $\phi(x)$ can be computed in time $\bigO(g(\|x\|_{\mathcal{A}_1}))$, then we say that $\mathcal{L}_1$ is strongly reducible to $\mathcal{L}_2$ in time $g(n)$.
\end{definition}

         \begin{theorem}
         \label{theorem_about_connecton_of_word_and_conjugacy_problems}
         Let $\mathcal{A}$ be any finite alphabet, and let $\mathcal{L} \subseteq \mathcal{A}^*$ be any recursively enumerable subset (i.e., r.e. language) of $\mathcal{A}^*$. Then there exists a lacunary hyperbolic group $G_{\mathcal{L}}$ 
         such that  the following is true about $G_{\mathcal{L} }$.
         \begin{enumerate}
         	\item[(I).] The word problem in $G_{\mathcal{L} }$ is decidable in almost linear time.\\

         \item[(II.i).] 
         The conjugacy problem in 	$G_{\mathcal{L} }$ can be strongly reduced to the decidability problem in $\mathcal{L}$ in almost linear time;\\
         \item[(II.ii).] 
         The decidability problem in $\mathcal{L}$ can be strongly reduced to the conjugacy problem in $G_{\mathcal{L} }$  in  linear time;\\
         
         In particular, if the membership problem for $\mathcal{L}$ belongs to $DTime(f(n))$, then the conjugacy problem in $G_{\mathcal{L}}$ is decidable in time almost $f(n)$, and if the conjugacy problem in $G_{\mathcal{L}}$ belongs to  $DTime(g(n))$, then the membership problem in $\mathcal{L}$  also belongs to $DTime(g(n))$.\\
         \item[(II.iii).] 
         For every fixed $g_0 \in G_{\mathcal{L}}$, the problem of deciding if an arbitrary $g \in G_{\mathcal{L}}$ is conjugate to $g_0$ is decidable in almost linear time.\\


                 
       
         \end{enumerate}

         \end{theorem}
         The individual conjugacy problem with regard to a fixed $g_0\in G$, shortly $ICP(g_0)$, asks if for any input element $g\in G$, $g$ is conjugate to $g_0$ in $G$. Note that $ICP(1)$ coincides with the word problem in $G$. The statement (II.iii) of Theorem \ref{theorem_about_connecton_of_word_and_conjugacy_problems} says that for every $g_0 \in G$, $ICP(g_0)$ belongs to  $\bigcap_{k=1}^{\infty}DTime\big(n^{1+\frac{1}{k}}\big)$.
                 
         Note that Theorem \ref{theorem_about_connecton_of_word_and_conjugacy_problems}  immediately implies that for every time complexity class $\mathcal{C}$, which is not contained in $\bigcap_{k=1}^{\infty}DTime\big(n^{1+\frac{1}{k}}\big)$, the question whether there exists a group $\tilde{G}$ with almost linear time word problem (even more generally, with almost linear time individual conjugacy problem) and with conjugacy problem belonging to $\mathcal{C}$ is equivalent to the question whether  $\mathcal{C}$ is empty or not. In case $\mathcal{C} \neq \emptyset$, one just can take $\mathcal{L} \in \mathcal{C}$ and consider the group $\tilde{G}=G_{\mathcal{L}}$.
         
         For example, the last observation implies that there exist finitely generated groups with almost linear time individual conjugacy problems and (uniform) conjugacy problem which belongs to one of the following time complexity classes:
         \begin{itemize}
         \item	NP-complete, co-NP-complete, PP-complete, PSpace-complete, etc; or
         \item belongs to $DTime(f(n)) \setminus DTime(g(n))$ where the time constructible functions $f$ and $g$ are such that $DTime(f(n)) \setminus DTime(g(n)) \neq \emptyset$ and $f(n)>n^{1+\varepsilon}$ for some $\varepsilon>0$; or
         \item the conjugacy problem is undecidable and has any given recursively enumerable Turing degree of undecidability.
         \end{itemize}
   In particular, parts (I), (II)  of Theorem \ref{theorem_about_connecton_of_word_and_conjugacy_problems} extend similar results of Miasnikov and Schupp from \cite{miasnikov-schupp}.
                      ~\\
                      
                      In \cite{cannonito}, Cannonito classified finitely generated groups with decidable word problem based on complexity of the word problem. As a measure of complexity the author considered \textit{Grzegorczyk hierarchy}. For the details of the results and definition of Grzegorczyk hierarchy and its  link to word problem, we refer to \cite{grzegorczyk} and \cite{cannonito}.
                      
                      In the same paper \cite{cannonito}, the author mentions a question posed by Boone (see page 391, \cite{cannonito}) which was formulated as follows:
                      \begin{quote}
                      	\textit{A very interesting problem suggested by W. W. Boone is the following: Do there exist any f.g. groups with conjugacy problem $\mathcal{E}_*^{\alpha}$-decidable, and word problem $\mathcal{E}^{\beta}$-decidable such that $\beta < \alpha$?}
                      \end{quote}
                      This question was also touched in \cite{lipton-logspace}.

         
         
         Parts (I) and (II) of Theorem \ref{theorem_about_connecton_of_word_and_conjugacy_problems} imply the following stronger statement.
         \begin{corollary}
         \label{corollary-22**}
          For every $\alpha \geq 3$, there exists a finitely generated (lacunary hyperbolic) group $\tilde{G}$ with   $\mathcal{E}^3 $-decidable  word problem and $\mathcal{E}_*^{\alpha}$-decidable conjugacy problem.          
          \end{corollary}
          \begin{remark}
          	We would like to note that Corollary \ref{corollary-22**} follows also from the main results of \cite{miasnikov-schupp}.

          \end{remark}
         


                  In \cite{olshanskii-osin-sapir}, the authors, Olshanskii, Osin and Sapir, asked about the existence of a lacunary hyperbolic group with decidable word problem but undecidable conjugacy problem. See Problem 7.5 in \cite{olshanskii-osin-sapir}.       
         Another immediate corollary from Theorem \ref{theorem_about_connecton_of_word_and_conjugacy_problems}, parts $(I)$, $(II.i)$ and $(II.ii)$, answers this question in positive.          
         
         \begin{corollary}
           There exists a lacunary hyperbolic group with decidable word problem but undecidable conjugacy problem.	
         \end{corollary}
         \begin{proof}
         	Indeed, take any recursively enumerable but not recursive set $\mathcal{L}$. Then, according to Theorem \ref{theorem_about_connecton_of_word_and_conjugacy_problems}, the group $G_{\mathcal{L}}$ has decidable word problem but undecidable conjugacy problem.
         \end{proof}
      ~\\
      ~\\
      Note that Theorem \ref{theorem_about_connecton_of_word_and_conjugacy_problems} provides a reasonably complete classification of conjugacy problem in finitely generated groups in terms of time computational complexity for groups with decidable word problem and in terms of recursively enumerable Turing degrees for recursively presented groups with undecidable conjugacy problem. It is worth mentioning that the analogous classifications were obtained for word problems, for example, by the following authors: by Cannonito \cite{cannonito} in terms of Grzegorczyk hierarchy; by Valiev and Trakhtenbrot \cite{valiev, trakhtenbrot} in terms of space  complexity, by Stillwell \cite{stillwell} in terms of time complexity. However, in spirit, probably the closest result to parts $(II.i)$ and $(II.ii)$ of Theorem \ref{theorem_about_connecton_of_word_and_conjugacy_problems} is the following result of Birget-Olshanskii-Rips-Sapir  from \cite{birget-olshanskii-rips-sapir} stated for the word problem in finitely presented groups and mentioned as ``an important corollary'' (Corollary 1.1).
      \begin{quote}
      	There exists a finitely presented group with NP-complete
 word problem. Moreover, for every language $L \subseteq \mathcal{A}^*$ from some finite alphabet $\mathcal{A}$,
 there exists a finitely presented group $G$ such that the nondeterministic time
 complexity of $G$ is polynomially equivalent to the nondeterministic time complexity of $L$.
      \end{quote}
      The first examples of groups with decidable word problem and undecidable conjugacy problem of arbitrary r.e. Turing degree for finitely generated groups  were constructed by Miller \cite{miller III}, and for finitely presented groups by Collins \cite{collins}. It was shown in \cite{borovik-miasnikov} that in Miller's group from \cite{miller III} even though the conjugacy problem is undecidable, the individual conjugacy problems $ICP(g)$ are solvable in polynomial time for all $g$ from an exponentially generic subset of $G$.  This and other observations led Miasnikov and Schupp to formulate the following question in \cite{miasnikov-schupp}.
      \begin{quote}      
\textit{Question. Are there recursively presented groups $G$ with solvable word problem such that if the individual conjugacy problems are decidable on a computably enumerable subset $Y \subseteq G$ then $Y$ is negligible, or indeed exponentially negligible?	}
      \end{quote} 
      We answer this question in positive by showing the following stronger result.
      \begin{theorem}
      \label{theorem-schupp-miasnikov-question}	
      There exist lacunary hyperbolic groups $\tilde{G}=\langle X \rangle$  with word problem decidable in almost linear time and such that for $g\in \tilde{G}$ the individual conjugacy problem $ICP(g)$ is decidable if and only if $g= 1$.
      \end{theorem}
      \begin{remark}
      	In fact, it is possible to show that the group $\tilde{G}$, which is constructed in Section \ref{section-MS} and satisfies Theorem \ref{theorem-schupp-miasnikov-question}, in addition, has exponential growth. This would imply that the set $\{ w\in X^* \mid  w=_{\tilde{G}} 1 \}$ is exponentially negligible.
      \end{remark}
~\\
\section{Preliminaries}
	\label{section-preliminaries}
	Let $(\mathcal{X}, d)$ be a geodesic metric space.
	Given a geodesic triangle $ABC$ in $\mathcal{X}$ with vertices $A$, $B$ and $C$, for any $\delta>0$, $ABC$ is called \textit{$\delta$-slim} \index{$\delta$-slim triangle}\index{slimness of a triangle}if each side of the triangle $ABC$ is contained in the $\delta$-neighborhood of the union of  other two sides of $ABC$.
	
	 For a given constant $\delta>0$, $\mathcal{X}$ is called \textit{$\delta$-hyperbolic} space, if all the geodesic triangles in $\mathcal{X}$ are $\delta$-slim. Throughout this text, when we consider a geodesic triangle with vertices $A$, $B$ and $C$, by $AB$, $BC$, $CA$ we denote the sides of the triangle joining the corresponding vertices. The same convention we use also for other polygons.
	
	There are other equivalent definitions of hyperbolic spaces characterizing hyperbolicity in therms of insize or thinness, which we briefly mention below. However, in this work we will exclusively stick to the definition through the notion of slimness (this is also called the Rips' definition). \\
	Easy to see that for any geodesic triangle $ABC$ in $\mathcal{X}$ there is a unique triple of points $(O_A, O_B, O_C)$ on the sides of $ABC$, such that $O_A \in BC$, $O_B \in AC$, $O_C \in AB$ and $d(A, O_B)=d(A, O_C)$,  $d(B, O_A)=d(A, O_C)$ ,  $d(C, O_A)=d(C, O_B)$. If for some $\delta'>0$, $d(O_A, O_B), d(O_B, O_C), d(O_C, O_A) \leq \delta'$, then, following a common terminology, we say that the \textit{insize} \index{insize} of the triangle $ABC$, defined as $insize(ABC)=\max\{ d(O_A, O_B), d(O_B, O_C), d(O_C, O_A)\}$, is bounded by $\delta'$.
	 
	If there exists $\delta''>0$, such that for any points $O_1 \in AO_B$ or $O_2 \in AO_C$; $O_1 \in BO_A$, $O_2 \in BO_C$ or $O_1 \in CO_A$, $O_2 \in CO_B$, the corresponding equation $d(A, O_1)=d(A, O_2)$ or $d(B, O_1)=d(B, O_2)$ or, respectively, $d(C, O_1)=d(C, O_2)$ implies that $d(O_1, O_2)\leq \delta''$, then the triangle $ABC$ is called \textit{$\delta''$-thin}. \index{thinness of a triangle}
	\begin{property}
	\label{property 1111}
	It follows from \cite{alonso brady} (see Proposition 2.1 in \cite{alonso brady}) that if the insize of $ABC$ is bounded from above by $\delta'>0$, then it is $3\delta'$-slim and $18\delta'$-thin. Meanwhile, if $ABC$ is $\delta$-slim, then it is $6\delta$-thin.	
	\end{property}

	
	Let $G = \langle X \rangle$ be a finitely generated group with a finite generating set $X$. Note that the Cayley graph $\Gamma(G, X)$ possesses a natural geodesic metric, $d_G$, called word metric. That is for any $g, h \in G$, $d_G(g, h)$ is the length of a smallest word from $X^*$ representing the word $g^{-1}h \in G$. By $|g|_G$ (or just by $|g|_X$ or $|g|$, depending on the context and convenience) we denote the distance $d_G(1, g)$. In the current work, whenever it does not lead to ambiguities, instead of using the notation $d_G$ we will simply write $d$. Depending on the
	 convenience  derived from the context, we will use sometimes instead of $d_G$, $d_X$ or simply $d$, if it does not lead to ambiguities.
	 \index{$d_G, d_X, d$}
	
	Note that, at the first glance, it would be more appropriate to use notations $d_X$ and $| \cdot |_X$ instead of $d_G$ and $|\cdot |_G$. However, this notation we use by purpose, because in many applications in this paper, we interchangeably consider metrics on different Cayley graphs of groups with presentations $\langle X \mid \mathcal{R}_1 \rangle$ and  $\langle X \mid \mathcal{R}_2 \rangle$, where $\mathcal{R}_1 \neq \mathcal{R}_1$.

	 The group  $G = \langle X \rangle$ is called $\delta$-hyperbolic, if its Cayley graph $\Gamma(G, X)$ is $\delta$-hyperbolic. In general, we say $G$ is hyperbolic if the Cayley graph $\Gamma(G,X)$ is $\delta$-hyperbolic for some $\delta\geq 0$. It is a well-know fact that the property of hyperbolicity does not depend on the choice of finite generating sets (see \cite{gromov-hyperbolic}). However, the hyperbolicity constant $\delta$ may depend on the choice of the generating set. In this text, whenever we say that some group or space is $\delta$-hyperbolic, by default we assume that $\delta$ is a positive integer.\\
	 ~\\
	 The following are well-known algorithmic properties of hyperbolic groups.
	 \begin{enumerate}
	    \item The calss of hyperbolic groups is exactly the class of finitely presented groups with Dehn presentation. See \cite{gromov-hyperbolic} and also \cite{lysenok}.\\
	    \item  It was established by Epstein and Holt in \cite{epstein holt} that given a hyperbolic group $G$ with finite Dehn presentation, there exists an algorithm solving the conjugacy problem in $G$ in linear time.\\
	   
	 	\item It was established by Papasoglu in \cite{papasoglu-detecting-hyperbolicity} (see also \cite{papasoglu-00} for background) that there exists a  partial algorithm which detects hyperbolicity of finitely presented hyperbolic groups. In other words, the set of finite presentations of hyperbolic groups is recursively enumerable. See also \cite{holt-detecting-hyperbolicity}. \\
	 	\item There exists an algorithm which computes a slimness constant $\delta$ for any finite presentation of a hyperbolic group. See, for example, \cite{holt-detecting-hyperbolicity}. \\
	 	\item There exists an algorithm which for any input of finite presentation of a hyperbolic group computes its Dehn presentation.
	 	 
	 \end{enumerate}
	~\\

	 \index{quasi-geodesic path}
	Now consider a path $p$ in  $(\mathcal{X}, d)$ with a natural parametrization by length. The path $p$ is called \textit{$(\lambda, c)$-quasi-geodesic} for some $\lambda \geq 1$ and $c\geq 0$, if for any points $p(s)$ and $p(t)$ on $p$, we have
	\begin{align*}
	|s-t| \leq \lambda d(p(s), p(t)) +c.
	\end{align*}
	Hereafter, whenever it is not stated otherwise, we assume that the quasi-geodesity constants $\lambda$ and $c$ are integers. We denote the origin of $p$ with respect to this parametrization (i.e. the point $p(0)$) by $p_-$ and the terminal point by $p_+$.\\

	We say that a word $W \in X^*$ is a \textit{geodesic word} \index{geodesic word} (in $\Gamma(G, X)$), if the paths in $\Gamma(G, X)$ with label $W$ are geodesics, and
	we say that a word $W \in X^*$ is \textit{cyclically geodesic} \index{cyclically geodesic word} if any cyclic shift of $W$ is a geodesic word in $\Gamma(G, X)$. Analogously, for $\lambda\geq 1, c\geq 0$, we say $W$ is \textit{$(\lambda, c)$-quasi-geodesic} \index{quasi-geodesic word} (in $\Gamma(G, X)$) if the corresponding paths in  $\Gamma(G, X)$ are $(\lambda, c)$-quasi-geodesic. 
	The length of  the  word $W$ we denote by $\|W\|$ and by $|W|$ we denote the length of the shortest word representing the same element as $W$ in $G$. Clearly, $W$ is a geodesic word if and only if $\|W\|=|W|$.
	
	For any $W' \in X^*$ the notation \textit{$W' \sim_{conj}W$ in $G$} means that $W'$ represents an element in $G$ conjugate to the element represented by $W$ in $G$. \index{$\sim_{conj}$}
	
	We say that $V \in X^*$ is a \textit{cyclically minimal} \index{cyclically minimal representative} representative of $W$ if $V \sim_{conj} W$ in $G$ and $V$ has the smallest length among all such words. For $V$ satisfying this assumption, we also define $|W|_c = \|V\|$. If $\|W\|=|W|_c$, then we say that $W$ is \textit{cyclically minimal}. Clearly, if $W$ is cyclically minimal,  then it is cyclically geodesic. \index{cyclically minimal word}
	
	Now suppose that $p$ is a path in $\Gamma(G, X)$. Then, as we said, we will denote its initial and terminal points by $p_-$ and $p_+$, respectively. If $A, B$ are some points on $p$, then by $[A, B]$ \index{$[A, B]$} we denote the subpath $q$ of $p$ between $A$ and $B$ such that $q_-=A$ and $q_+=B$. Also we denote the length of $p$ by $\|p\|$ and, context based, we denote the length of $q$ by $\|q\|$ or by $\big\|[A, B]\big\|$. Since all the edges in Cayley graphs are labeled by the letters of $X \cup X^{-1}$, any path $p$ in $\Gamma(G, X)$  in fact is a labeled path. We denote the label of $p$ by $lab(p)$.

		\begin{lemma}[Theorem III.1.7, \cite{bridson}]
		\label{lem hausdorff distance between quasi-geodesics}
		 Let $p$ be a $(\lambda, c)$-quasi-geodesic path in the Cayley graph $\Gamma(G,X)$, where $\lambda \geq 1, c\geq 0$ and $G = \langle X \rangle$ is a hyperbolic group. Then there exists an effectively calculable constant $R_{\lambda, c} \in \mathbb{N}$ depending on $\lambda, ~c$ and $G$, such that the Hausdorff distance between $p$ and any geodesic path joining $p_-$ to $p_+$ is bounded by $R_{\lambda, c}$. \index{$R_{\lambda, c}$}
	\end{lemma}
	
	In this text, whenever we write the notation $R_{\lambda, c}$, we refer to the constant from Lemma \ref{lem hausdorff distance between quasi-geodesics}. 
	
		\begin{corollary}
		\label{corollary on hausdorff distance between quasi-geodesics}
		Let $p$ and $q$ be $(\lambda_1, c_1)$- and $(\lambda_2, c_2)$-quasi-geodesic paths in $\Gamma(G,X)$ respectively. Also let $d(p_-, q_-) \leq L$,  $d(p_+, q_+) \leq L$ for some constant $L$, then the Hausdorff distance between $p$ and $q$ is bounded from above by $L+R_{\lambda_1, c_1}+R_{\lambda_2, c_2}+2\delta$, where $\delta$ is a hyperbolicity constant of $\Gamma(G, X)$. Moreover, if we join $p_-$ to $q_-$ and $p_+$ to $q_+$ by some geodesics, then we get a quadrangle such that the distance from any point on $p$ (or $q$) to the union of the other three sides is bounded from above by  $R_{\lambda_1, c_1}+R_{\lambda_2, c_2}+2\delta$. In case $p$ and $q$ are geodesics, this distance is bounded from above by $2\delta$.
	\end{corollary}
	\begin{proof}
	It follows from Lemma \ref{lem hausdorff distance between quasi-geodesics} that it would be enough to prove the statement for the case when $p$ and $q$ are geodesic paths and correspondingly $R_{\lambda_1, c_1}=R_{\lambda_2, c_2}=0$.
 
		Now assume that $p$ and $q$ are geodesics. Let $p_-$, $q_-$ and $p_+$, $q_+$ be joined by some geodesics $f_1$ and $f_2$, respectively. Also let $e$ be a geodesic path joining $q_-$ to $p_+$. 
			
		By the definition of hyperbolicity constant, for any point $o_1\in q$, there exists $o_2 \in e \cup f_2$ such that $d(o_1, o_2)\leq \delta$. Now, if $o_2 \in f_2$, then since $\|f_2\|\leq L$, the statement of the corollary follows for $o_1$ immediately. Otherwise, if $o_2 \in e$, the statement follows for $o_1$ immediately from the observation that $dist(o_2, f_1 \cup p)\leq \delta$ and $\|f_1\|\leq L$. If $o_1$ belongs to one of the other three sides, then we can deal with that case analogously.
	\end{proof}

	\begin{corollary}
		\label{another corollary about hausdorff distance}
		Let $p$ and $q$ be $(\lambda_1, c_1)$- and $(\lambda_2, c_2)$-quasi-geodesic paths in $\Gamma(G,X)$ respectively, and let $d(p_-, q_-) \leq L$,  $d(p_+, q_+) \leq L$ for some constants $\lambda_1 \geq 1, c_1 \geq 0$, $\lambda_2 \geq 1, c_2\geq 0$, $L \geq 0$. Then for any point $o\in p$ such that $d(o, p_-), d(o, p_+) \geq L+R_{\lambda_1, c_1}+2\delta$, we have
		$dist(o, q)\leq R_{\lambda_1, c_1} + R_{\lambda_2, c_2} + 2\delta$, where $\delta$ is the hyperbolicity constant of $\Gamma(G, X)$.
	\end{corollary}
	\begin{proof}
	 Let $p_-$, $q_-$ and $p_+$, $q_+$ be joined by some geodesics $f_1$ and $f_2$, respectively. Also let $p', q'$ be geodesic paths joining $p_-$ to $p_+$ and $q_-$ to $q_+$, respectively.
	 
	  By Lemma \ref{lem hausdorff distance between quasi-geodesics} there exists $o_1 \in p'$ such that $d(o, o_1) \leq R_{\lambda_1, c_1}$. Now, by Corollary \ref{corollary on hausdorff distance between quasi-geodesics}, $dist(o_1, f_1\cup f_2 \cup q') \leq 2\delta$. 
	  
	  On the other hand, if $dist(o_1, f_1)\leq 2\delta$, then $dist(o, f_1) \leq d(o, o_1)+dist(o_1, f_1) \leq R_{\lambda_1, c_1}+2\delta$. Hence, by the triangle inequality, this would imply $d(o, p_-) \leq L+R_{\lambda, c}+2\delta$, which is a contradiction. This contradiction implies that $dist(o_1, f_1)>2\delta$. Similarly, we get that $dist(o_1, f_2)>2\delta$. Therefore, $dist(o_1, q')\leq 2\delta$, and hence $dist(o_1, f_1\cup f_2 \cup q') \leq 2\delta$ implies that $dist( o_1, q')\leq 2\delta$.

	  Therefore, since  $d(o, o_1) \leq R_{\lambda_1, c_1}$ and the Hausdorff distance between $q'$ and $q$ is bounded from above by $R_{\lambda_2, c_2}$, we get that  $dist(o, q)\leq R_{\lambda_1, c_1} + R_{\lambda_2, c_2} + 2\delta$.
	\end{proof}
	~\\
	Given a path $p$ and $k \geq 0$, $\lambda \geq 1, c\geq 0$, we say that $p$ is\textit{ $k$-local $(\lambda, c)$-quasi-geodesic}, if each subpath of $p$, of length at most $k$, is $(\lambda, c)$-quasi-geodesic. In case $\lambda=1$, $c=0$, we say that $p$ is \textit{$k$-local geodesic.} \index{$k$-local (quasi-)geodesicity}\\
	
	\begin{lemma}[Theorem III.H.1.13, \cite{bridson}]
		\label{lemma_about_local_geodesicness}	
		Let $\mathcal{X}$ be a $\delta$-hyperbolic geodesic space and $p$ be a $k$-local geodesic, where $k> 8\delta$. Then for every geodesic segment $q$ joining $p_{-}$ to $p_{+}$ we have:\\
		(1) $p$ is contained in the $2\delta$-neighborhood of $q$;\\
		(2) $q$ is contained in the $3\delta$-neighborhood of $p$;\\
		(3) $p$ is a $(\lambda, c)$-quasi-geodesic, where $\lambda=(k+4\delta)/(k-4\delta) $ and $c = 2\delta $.
	\end{lemma}
	The next lemma is a generalization of the previous one. It can be found in \cite{ghys}.
	
	\begin{lemma}[See Theorem 25 in  \cite{ghys}]
		\label{lem 1}
		\index{$\mathcal{K}(\delta, \lambda, c)$}
		Let $\mathcal{X}$ be a $\delta$-hyperbolic  space. Then there exists an effectively computable constant $\mathcal{K}=\mathcal{K}(\delta, \lambda, c)\in \mathbb{N}$ such that for any $k\geq \mathcal{K}$, if $p$ is a $k$-local $(\lambda, c)$-quasi-geodesic path in $\mathcal{X}$, then $p$ is $(\mathcal{K}, \mathcal{K})$-quasi-geodesic.
	\end{lemma}
	~\\
	
	 For any metric space $(\mathcal{X}, d)$ and for any $x, y, z \in \mathcal{X}$ the \textit{Gromov product} \index{Gromov product} of $y$ and $z$ at $x$, denoted $(y \cdot z)_x$, is defined by
\begin{align*}
	(y \cdot z)_x = \frac{1}{2} \big( d(x, y) + d(x, z) - d(y, z) \big).
\end{align*}
	\begin{lemma}[see Lemma 5, \cite{ivanov ol'shanskii}]
		\label{lem Lemma 5}
		 Let $G = \langle X \rangle$ be a $\delta$-hyperbolic group. Let $\alpha\geq 14\delta$, $\alpha_1 \geq 12(\alpha + \delta)$, and a geodesic $n$-gon $A_1A_2\ldots A_n$ with
		$n\geq 3$ satisfies the following conditions: $d(A_{i-1}, A_i) > \alpha_1$ for $i = 2,...,n$ and
		$(A_{i-2} \cdot A_i)_{A_{i-1}} \leq \alpha$ for $i = 3,...,n$. Then the polygonal line $p = A_1 A_2 \cup \ldots \cup A_{n-1}A_n$ is
		contained in the closed $2\alpha$-neighborhood of the side $A_nA_1$ and the side $A_nA_1$ is
		contained in the closed $14\delta$-neighborhood of $p$. In addition, $d(A_1, A_n) > 6(n-1)(\alpha+\delta)$.
	\end{lemma}

	\begin{lemma}[see Lemma 1.17, \cite{Olsh G-groups}, also Lemma 8, \cite{ivanov ol'shanskii}]
		\label{lem Lemma 8}
		 Let $g$ be an element of infinite order in a hyperbolic group $G$ and an
		equality $x g^k x^{-1} = g^l$ holds in $G$, where $x \in G$, $l \neq 0$. Then $k = \pm l$.
	\end{lemma}
	
	\begin{lemma}
	\label{lemma_aux_q-g}
	Let $G = \langle X \rangle$ be a $\delta$-hyperbolic group, and let $W, V, T \in X^*$ be such that $V$ is freely cyclically reduced non-empty word and	
	\begin{align*}
		W=_G T^{-1} V T.
	\end{align*}
	Suppose that for some $k\in \mathbb{N}$ and $\lambda\geq 1$, $c\geq 0$, $V^k$ is a $(\lambda, c)$-quasi-geodesic word. Then $W^k$ is a $\big(\lambda \|W\|, (2\lambda\|T\|+ c +2)\|W\| \big)$-quasi-geodesic word.
	\end{lemma}
	\begin{proof}
	First of all, note that for all $l \geq 0$, we have
	\begin{align*}
		l\|V\|=\|V^{l}\|\leq \lambda |V^{l}|+c,
	\end{align*}
	hence
	\begin{align}
	\label{eq-++++}
		l\leq \frac{\lambda |V^{l}|+c}{\|V\|} \leq \lambda |V^{l}|+c.
	\end{align}
	
	Now note that every subword of $W^k$ is of the form $W_1 W^l W_2$, where $l\geq 0$ and $W_1$, $W_2$ are (possibly empty) suffix and prefix of $W$ respectively.
	
	Now for $W_1 W^l W_2$ we have
	\begin{align*}
		\|W_1 W^l W_2\| &\leq \|W_1\|+\|W^l\|+\|W_2\| = \|W_1\|+\|W_2\|+l\|W\|\\
		\text{by \eqref{eq-++++},~}&\leq \|W_1\|+\|W_2\|+(\lambda |V^{l}|+c)\|W\|\\
		&=\|W_1\|+\|W_2\|+(\lambda |TW^lT^{-1}|+c)\|W\|\\
		&\leq  2\|W\|+(\lambda |W^l|+ 2\lambda\|T\|+c)\|W\|\\
		& = \lambda \|W\||W^l|+2\|W\|+2\lambda\|T\|\|W\|+ c \|W\| \\
		&= \lambda \|W\||W^l|+(2\lambda\|T\|+ c +2)\|W\|.
	\end{align*}
	Now, since $W_1 W^l W_2$ was chosen to be an arbitrary subword of $W^k$, we conclude that $W^k$ is a $\big(\lambda \|W\|, (2\lambda\|T\|+ c +2)\|W\| \big)$-quasi-geodesic word.
	\end{proof}

\begin{lemma}
\label{lemma_aux_q-g_cyclically_minimal}
Let $G = \langle X \rangle$ be a $\delta$-hyperbolic group, and let $V \in X^*$ be a cyclically minimal word such that $\|V\|\geq \alpha$, where $\alpha=12\cdot 15 \delta=180\delta$. Then for each $k \in \mathbb{Z}$, $V^k$ is a $(4, 2520\delta)$-quasi-geodesic word.
\end{lemma}	

\begin{proof}
	Without loss of generality let us assume that $k\in \mathbb{N}$. We want to show that  $V^k$ is $(2, 1260\delta)$-quasi-geodesic. 
	
	For that reason, let us decompose $V$ as 
		\begin{equation*}
			V=V_1V_2\ldots V_s,
		\end{equation*}
		where $s = \Big\lfloor \frac{\|V\|}{\alpha} \Big\rfloor$ and $\alpha \leq \|V_i \| < 2\alpha$ for $i=1, \ldots, s$. Then, since $V$ is cyclically minimal and the word $V_sV_1$ along with the words $V_1V_2$, \ldots, $V_{s-1}V_s$  are subwords of (a cyclic shift of) $V$, we get
		\begin{equation*}
			|V_1|+|V_s|-|V_sV_1| =\|V_1\|+\|V_s\|-\|V_sV_1\| =0		\end{equation*}
					and 
		\begin{equation*}
			|V_i|+|V_{i+1}|-|V_iV_{i+1}| =\|V_i\|+\|V_{i+1}\|-\|V_iV_{i+1}\|=0, \text{~for $i=1, \ldots, s-1$}.
		\end{equation*}
		The last equations suggest that we can apply Lemma \ref{lem Lemma 5} on subwords of $V^k$ to conclude that for any subword $V'$ of $V^k$, which is indeed of the form
		\begin{equation*}
			V'=U_1 V_{i_1}\ldots V_{i_t} U_2,
		\end{equation*}
		where $U_1$ and $U_2$ are suffix and prefix of words from $\{ V_1, \ldots, V_s\}$, we have
		\begin{equation}
		\label{agagag}
			|V'|\geq |V_{i_1}\ldots V_{i_t}|-\|U_1\|-\|U_2\|> 6(t-1)15\delta-2\alpha=90\delta t - 450\delta.
		\end{equation}
		(Lemma \ref{lem Lemma 5} was used to obtain $|V_{i_1}\ldots V_{i_t}| >6(t-1)15\delta$).\\
		On the other hand
		\begin{equation*}
		\begin{aligned}
			\|V'\|\leq (t+2)2\alpha = (t+2)360\delta &=(360\delta t- 1800\delta)+1800\delta+720\delta\\
			\text{by (\ref{agagag}),~}&\leq 4|V'|+2520\delta.
		\end{aligned}		
		\end{equation*}
		Therefore, since $V'$ is an arbitrary subword of $V^k$, we conclude that $V^k$ is a 
		\begin{equation}
		\label{Vk_is_q-g}
			(4, 2520\delta)\text{-quasi-geodesic word}.
		\end{equation}
	
	\end{proof}

	\begin{lemma}
			\label{lem 1.11}
		Let $G = \langle X \rangle$ be a $\delta$-hyperbolic group, where $X$ is symmetric (i.e. $X=X^{-1}$), and let $W \in X^*$ be a geodesic word representing an element of $ G$ of infinite order. Then for every $k \in \mathbb{Z}$, the word $W^k$ is $(\lambda_W, c_W)$-quasi-geodesic in the Cayley graph $\Gamma(G, X)$, where $\lambda_W$ and $c_W$ are given by the formulas
		\begin{align}
		\label{eq def of lambda-g}
		\lambda_W= 	 4|X|^{\alpha}\|W\|,
		\end{align}
		and
		\begin{align}
		\label{eq def of c-g}
		c_W= 5|X|^{2\alpha}\|W\|^2
		\end{align}
		where $\alpha=180\delta$. Moreover, if $W$ is cyclically minimal, then
		$W^k$ is $(4\alpha|X|^{\alpha}, 5\alpha^2|X|^{2\alpha})$-quasi-geodesic.
	\end{lemma}
	\begin{proof}
		First, let us show that there exists an integer $ 1\leq m \leq |X|^{\alpha}$ such that $|W^m|_c >\alpha$ (recall that we assume $X=X^{-1}$). Indeed, assume that there is no such $m$. Then, by the pigeonhole principle, there exist $1 \leq m_1 < m_2 \leq |X|^{\alpha}$ and $V\in X^*$, $T_1, T_2 \in X^*$, such that $\|V\| \leq \alpha$, $\|V\|=|W^m|_c$ and
		\begin{align*}
		W^{m_1}=_G T^{-1}_1 V T_1, ~ W^{m_2}=_G  T^{-1}_2 V T_2.
		\end{align*}
		 But this means that $W^{m_1}$ and $W^{m_2}$ are conjugate in $G$, which on its own turn, by Lemma \ref{lem Lemma 8}, implies that $m_1 = \pm m_2$. A contradiction.
		
		Therefore, there exists $ 1 \leq m \leq |X|^{\alpha}$ such that 
		\begin{align}
		\label{equation*-*}
		W^m =_{G} T^{-1}VT,
		\end{align}
		where $T, V \in X^*$, $\|V\|=|W^m|_c$   and 
		\begin{align}
		\label{eq auxiliary 2}
		\|V\|>\alpha.
		\end{align}	
		
		Note that the equation $\|V\|=|W^m|_c$ implies that $V$ is cyclically geodesic.
		
		Without loss of generality assume that $T$ has the smallest length among all the words $T$ satisfying the equation (\ref{equation*-*}) for some $V$ with $\|V\|=|W^m|_c$.

		Let us assume that $W^m=_G U$ for some geodesic word $U \in X^*$. Let us consider a geodesic quadrangle $ABCD$ in $\Gamma(G,X) $ such that $lab(AB)=lab(DC)= T$, $lab(AD)=V$ and $lab(BC)= W^m$, i.e. the boundary of $ABCD$ corresponds to the equation $	W^m =_G T^{-1}VT$.
		
		 The first observation is that $\|T^{-1}\|=\|T\| = dist(B, AD) ~(=dist(C, AD))$. Indeed, if there exists a point $O\in AD$ such that $d(B, O) < \|T\|$, then there exists a path joining $B$ to $O$, whose label is a word $Q$ such that $\|Q\|<\|T\|$. Now, if we denote $lab(AO)=V_1$, $lab(O, D)=V_2$, we get $W^m=_G U =_G Q (V_2V_1) Q^{-1}$. But because of the minimality assumption on $\|T\|$, the inequality $\|Q\|<\|T\|$ leads to a contradiction. Thus the first observation is proved. See Figure \ref{fig:  for lemma 6}.
		\begin{figure}[H]
			\centering
			\includegraphics[clip, trim=2cm 14.3cm 1cm 4.8cm, width=.65\textwidth]{{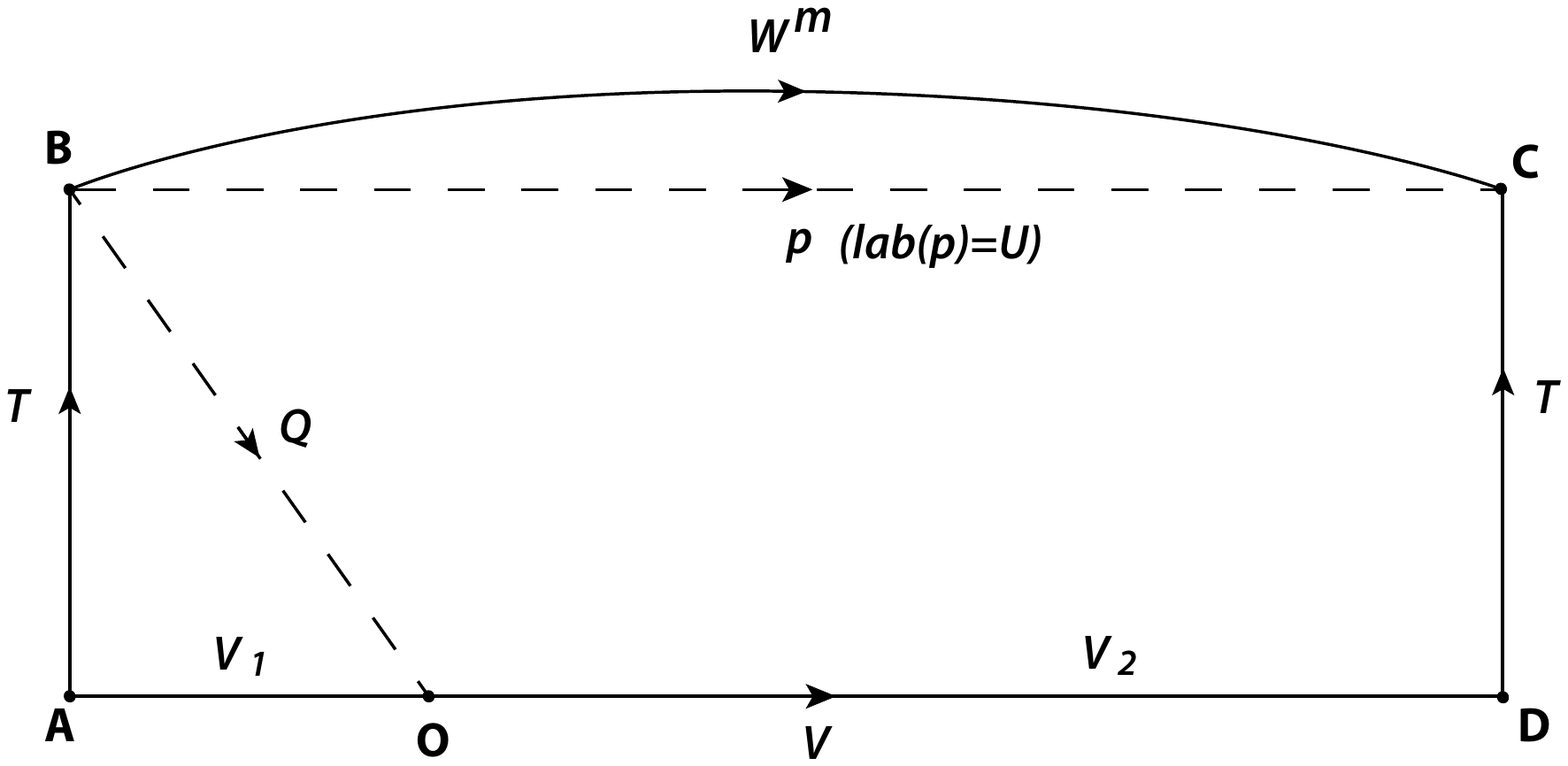}} 
			\caption{} 
			\label{fig:  for lemma 6}
		\end{figure}	
		
		The second observation is that for any point $O_1 \in AD$ such that $d(A,O_1), d(O_1, D) > 4\delta$ (note that such a point exists, because $\|V\|>\alpha$), we have $dist(O_1, p)\leq 2\delta$, where $p$ is the path joining $B$ to $C$ with the label $U$. To show this, first notice that $dist(O_1, AB\cup p \cup CD) \leq 2\delta$ (see Corollary \ref{corollary on hausdorff distance between quasi-geodesics}). Also, because of the minimality assumption on $\|T\|$, we get $d(B, O_1) \geq d(B, A)$. Now suppose that there is a point $O_2 \in AB$ such that $d(O_1, O_2) \leq 2\delta$. Then, since  $d(B, O_1) \geq d(B, A)$, we get 
		\begin{equation*}
		   d(B, A)=d(B, O_2)+d(O_2, A) \leq d(B, O_1) \leq d(B, O_2)+d(O_2, O_1).
		\end{equation*}
		 Therefore, $d(O_2, A) \leq d(O_2, O_1) \leq 2\delta$ and as a consequence,  by the triangle inequality, we get $d(A, O_1) \leq d(A, O_2)+d(O_2, O_1) \leq 4\delta$. But since $d(A, O_1)>4\delta$, we obtain a contradiction. 
		 
		 The last contradiction implies that $dist (O_1, AB)> 2\delta$. The same way we get $dist(O_1, CD) >2\delta$. Therefore, the inequality $dist (O_1, AB\cup p \cup CD) \leq 2\delta$ implies that $dist(O_1, p)\leq 2\delta$, and consequently, since the length of $p$ is bounded from above by $\|W^m\|$, we get that $d(O_1, B) \leq \|W^m\|+2\delta$. Therefore, from the minimality assumption on $\|T\|$, we get
		 \begin{equation}
			\label{estimation_of_T}
			 \|T\| \leq |X|^{\alpha}\|W\|+2\delta.
		\end{equation}
		 
		

Now, since $W^m = T^{-1}VT$ and $|V|>\alpha$, it follows immediately from Lemma \ref{lemma_aux_q-g}, Lemma \ref{lemma_aux_q-g_cyclically_minimal} and \eqref{estimation_of_T} that for all $k \in \mathbb{Z}$, $W^{km}$ is a
$\big( 4 \|W^m\|, (2|X|^{\alpha} \|W\|+8\delta+2520\delta+2)\|W^m\| \big)$-quasi-geodesic word. Also, taken into the account the fact that $W^k$ is a subword of $W^{km}$ and the inequalities $m\leq |X|^{\alpha}$ and $2520\delta+2 \leq |X|^{\alpha} \|W\|$, we conclude that $W^{k}$ is a
\begin{align}
		\label{equ}
		(4|X|^{\alpha}\|W\|, 5|X|^{2\alpha}\|W\|^2)\text{-quasi-geodesic}.	
		\end{align}

		 Finally, since for cyclically minimal words $V$ satisfying $\|V\|> \alpha$, we showed that $V^k$ is $(4, 2520\delta)$-quasi-geodesic, by taking in (\ref{equ}) $\|W\|=\alpha$, we get that for every cyclically minimal $V \in X^*$, regardless their lengths, $V^k$ is $(4\alpha|X|^{\alpha}, 5\alpha^2|X|^{2\alpha})$-quasi-geodesic.
	\end{proof}
	\subsection{Isoperimetric functions of hyperbolic groups}
	\label{subsec 3.1}
	Let $G$ be a group with finite presentation $G= \langle X \mid r_1, \ldots, r_k \rangle$. A function $f: \mathbb{N} \rightarrow \mathbb{N}$ is called an \textit{isoperimetric function} for $G$ (w.r.t. the given presentation), if for every reduced word $W \in X^*$ such that $W =_G 1$, $W$ can be presented as
	\begin{align*}
	W = \prod_{i=1}^{n} u_i r^{\pm 1}_{j_i} u_i^{-1}
	\end{align*}
	where $ n \leq f(\|W\|)$. The minimal among isoperimetric functions is sometimes called\textit{ Dehn function}. If $n$ is the minimal number for which such a decomposition exists, then $n$ is called \textit{the area} of $W$ and denoted $n = Area(W)$. Another, equivalent definition of the isoperimetric function is the following: let $p$ be a closed path in $\Gamma(G, X)$, then $p$ can be tessellated by at most $f(\|p\|)$ labeled discs whose labels belong to $\{r_1^{\pm 1}, \ldots, r_k^{\pm 1} \}$.
	
	\index{isoperimetric function} \index{Dehn's function}

	It is a well known fact that a group is hyperbolic if and only if it has a finite presentation with linear (equivalently, subquadratic) isoperimetric function. See for example \cite{bridson, Olsh isoperimetric}. Moreover, if $G = \langle X \rangle$ is $\delta$-hyperbolic and $\mathcal{F} = \{ U \in X^* \mid  \|U \| \leq 16 \delta+1, ~U=_G 1 \}$, then $G$ can be given by the following presentation
	\begin{align}
	\label{full-presentation}
	G = \langle X \mid \mathcal{F} \rangle
	\end{align}
	and for this presentation, for all reduced words $W \in \ll \mathcal{F}\gg$, we have $Area(W) \leq n$. Let us call this presentation the $(X, \delta)$-\textit{full presentation} of $G$ with respect to $X$ and $\delta$. If from the context it is clear what are $X$ and $\delta$, then we will just call it  \textit{the full presentation}. \index{$(X, \delta)$-full presentation} \index{full presentation for groups}
	
	An important observatoin about full-presentations follows from Lemma  \ref{lemma_about_local_geodesicness}. More precisely,  the full presentations \eqref{full-presentation} is in fact Dehn presentations. It follows from Lemma \ref{lemma_about_local_geodesicness} and from the observation that in the Cayley graph $\Gamma(G, X)$ the only $(8\delta+1)$-local geodesic loop is the loop with length $0$, i.e. a point. For more details see \cite{bridson} or Proposition \ref{proposition 1}.\\
	
	For a given presentation  $G= \langle X \mid \mathcal{R} \rangle$ of a hyperbolic group, let $f(n)\leq An$ for some constant $A>0$. Then we call $A$ an \textit{isoperimetry coefficient} (w.r.t. $G= \langle X \mid \mathcal{R} \rangle$). \index{isoperimetry coefficient}
	\begin{proposition}
		\label{proposition 1}
		(1). For any Dehn presentation $G=\langle X \mid \mathcal{R}\rangle$ the isoperimetry coefficient is equal to $1$.
		
		(2). If $G$ is $\delta$-hyperbolic, then the full presentation $G = \langle X \mid \mathcal{F} \rangle$ is a Dehn presentation.
	\end{proposition}
	\begin{proof}
	(1). Let $G=\langle X \mid \mathcal{R}\rangle$ be a Dehn presentation and let $p$ be a loop in $\Gamma(G, X)$. Then, since $G=\langle X \mid \mathcal{R}\rangle$  is a Dehn presentation, $p$ contains a subpath $q$ such that for another path $q'$ we have $\|q\|>\|q'\|$ and $lab(q^{-1}q')\in \mathcal{R}$. Then $q^{-1}q'$ can be filled with one cell from $\mathcal{R}$. Based on this observation, it is clear that there is a van Kampen diagram over $G=\langle X \mid \mathcal{R}\rangle$  with boundary $p$ and number of cells not exceeding $p$. Hence the first part of the proposition is proved.
	
	(2). 	Indeed, let $p$ be a closed path in $\Gamma(G, X)$ with its ends on $1$. Then, by Lemma \ref{lemma_about_local_geodesicness}, there exists a closed $8\delta+1$-local geodesic path $q$ with its ends on $1$ such that $pq$ can be tesselated by at most $\|p\|$ cells with labels from $\mathcal{F}$.
		
		On the other hand, again by Lemma \ref{lemma_about_local_geodesicness}, $q$ is $(3, 2\delta)$-quasi-geodesic. Now, since $q$ is a closed $8\delta$-local geodesic, we get that either $q$ has $0$ length, or $\|q\|\geq 8\delta$. But since $q$ is $(3, 2\delta)$-quasi-geodesic, the last inequality cannot happen. Hence $q$ has length $0$. This means that the loop $p$ can be tessellated by at most $\|p\|$ cells with labels from $\mathcal{F}$. Thus the proposition is proved.
	\end{proof}
	It is well-known that a finitely presentable group is hyperbolic if and only if with respect to any finite presentation the Dehn function of the group is linear. See, for example, \cite{gromov-hyperbolic, Olsh isoperimetric, alonso brady}. The next lemma tells that if with respect to some finite presentation $\langle X \mid r_1, r_2, \ldots, r_l \rangle$ of a hyperbolic group $G$, an isoperimetric coefficient $A$ is given, then one can effectively find $\delta>0$ such that $G$ will be $\delta$-hyperbolic with respect to the generating set $X$.
	\begin{lemma}[See \cite{lysenok}, \cite{alonso brady}]
		\label{lem about delta and isoperimetric index}
		Suppose $G$ is a hyperbolic group given with a finite presentation $G = \langle X \mid r_1, r_2, \ldots, r_l \rangle$. Also suppose that $f: \mathbb{N} \rightarrow \mathbb{N}$ is an isoperimetric function with respect to this presentation such that $f(n) \leq An$ for some positive integer $A$. Then  $G$ is $f_{\delta}(A, M)$-hyperbolic with respect to the generating set $X$, where $M=\max\{\|r_1\|, \ldots, \|r_l \| \}$ and $f_{\delta}: \mathbb{N}^2\rightarrow \mathbb{N}$ is a computable function independent  of $G$.
	\end{lemma}
	\subsection{Elementary subgroups of hyperbolic groups}
	\label{section elementary subgroups}
	A group is called \textit{elementary } \index{elementary groups} if it has a cyclic subgroup of finite index. It is a well know fact that in a hyperbolic group each element $g$ of infinite order is contained in a unique maximal elementary subgroup \index{$E(g)$, maximal elementary subgroup}, denoted by $E(g)$, see for example \cite{Olsh G-groups}.
	
	By the lemmas 1.16 and 1.17 of \cite{Olsh G-groups}, for a hyperbolic group $G$ and for $g\in G$ of infinite order, the following holds: 
	\begin{align*}
	E(g)=\{x \in G \mid xg^n x^{-1}= g ^{\pm n} \text{~for some~} n \in \mathbb{N} \}
	\end{align*}
	and
	\begin{align*}
	E(g)=\{x \in G \mid xg^k x^{-1}= g ^{l} \text{~for some~} k,l \in \mathbb{Z}\setminus \{0\}~ \}.
	\end{align*}
	Also we need the following definitions,
	\begin{align*}
	E&^-(g)=\{x \in G \mid xg^n x^{-1}= g ^{- n} \text{~for some~} n \in \mathbb{N} \},\\
	~\\
	E&^+(g)=\{x \in G \mid xg^n x^{-1}= g ^{n} \text{~for some~} n \in \mathbb{N} \}.
	\end{align*}
	Note that the equivalence of the two descriptions of $E(g)$ given above, follows from Lemma \ref{lem Lemma 8}.\\

	Since, as it is well-known, in every torsion-free hyperbolic group $G$ each elementary subgroup is cyclic, it follows that for all $g\in G\setminus \{1\}$ , the subgroup $E(g)$ is of the form $\langle g_0 \rangle$, where $g$ is a power of $g_0$ and $E(g_0)=\langle g_0 \rangle$.
	
	
	For any $U\in X^*$, we denote by $E(U)$ the group $E(g)$, where $g\in G$ and $U=_G g$. Similarly, we define $E^{\pm}(U)$. For $V\in X^*$, we say that $V \in E(U)$, if for some $h\in G$, $V=_G h$ and $h \in E(U)$.
	\begin{definition}
	\index{root element}
	If $G = \langle X \rangle$	is a torsion-free hyperbolic group, then for a word $U \in X^*$ we say that $U$ represents a root element in $G$, if $U=_G g_0$ and $E(g_0)=\langle g_0 \rangle$. Correspondingly, if $E(g_0)=\langle g_0 \rangle$, then $g_0$ is called root element.
	
	If for some $g \in G$, $E(g)=\langle g_0 \rangle$, then $g_0$ is called a root of $g_0$. (Note that each element $g \in G\setminus \{1\}$ has two different roots, $g_0$ and $g_0^{-1}$.)
	\end{definition}

	\begin{lemma}[see Lemma 2.1 in \cite{Olsh G-groups}]
		
		\label{lem 2.1}
		 Let $G=\langle X \rangle$ be a $\delta$-hyperbolic group, $X$ be symmetric, and let $U, V \in X^*$ be geodesic words with respect to $\Gamma(G, X)$. Let	
		 $\lambda\geq 1$ and $c \geq 0$ be constants such that $U^k$ and $V^k$ are $(\lambda, c)$-quasi-geodesic words w.r.t. $\Gamma(G, X)$ for all $k\in \mathbb{Z}$. (According to Lemma \ref{lem 1.11}, such $(\lambda, c)$  always exist.) Let $T_1, T_2 \in X^*$ be arbitrary elements in $G$. Denote $L=\max\{\|T_1\|, \|T_2\|\}$. Then, there exists a computable function $f: \mathbb{N}^5 \rightarrow \mathbb{N}$ independent of $G$ such that for any 
		integer $m$ satisfying the inequality
		\begin{align*}
		m \geq f(|X|, \delta, \lambda, c, \|V\|), 
		\end{align*}
		either
		\begin{align*}
		L > 	\frac{\|U\|}{12\lambda}m,
		\end{align*}
		or the equation 
		\begin{align*}
		T_1U^mT_2 =_G V^n
		\end{align*}
		implies that 
		$T_1UT_1^{-1}, T_2^{-1}UT_2 \in E(V)$. Moreover, if $U=_G V$, then $T_1, T_2 \in E(U) (=E(V))$. More precisely, $T_1,T_2 \in E^+(U)$ for $n>0$ and $T_1, T_2 \in E^-(U)$ for $n\leq 0$.
	\end{lemma}
	For the purpose of completeness we present a proof of  Lemma \ref{lem 2.1}  in Appendix.\\
	Also, for the reason of convenience, for the constants mentioned in Lemma \ref{lem 2.1} we introduce the following notations
	\begin{align}
	\label{def of theta}
	\index{$\upsilon(U)$}
	\upsilon = \upsilon(U)=	\frac{\|U\|}{12\lambda}.
	\end{align}
	and, assuming that the values of $|X|, \delta, \lambda, c, \|V\|$ are already known, we denote
	\begin{align}
	\label{def of M}
	\index{$\mathcal{M}(U, V)$}
	\mathcal{M}=	\mathcal{M}(U, V)=  f(|X|, \delta, \lambda, c, \|V\|). 
	\end{align}

	\begin{lemma}[See Theorem 2 and Theorem 3 in \cite{lysenok}]
		\label{lemma_about_finding_root_of_a_group_element}
		Let $G=\langle X \rangle$ be a torsion-free $\delta$-hyperbolic group given with its $(X, \delta)$-full-presentation. Then there exists an algorithm such that for any input $U \in X^*$ it finds a word $V\in X^*$ such that $E(U)=\langle V \rangle$, i.e. there exists an algorithm computing roots of the elements of $G$.
		\end{lemma}
		\begin{corollary}
		\label{corollary_about_finding_root_of_a_group_element} 
		There exists a (partial) algorithm	which for any input hyperbolic group $G=\langle X \mid \mathcal{R} \rangle$ given by a finite presentation and for any input word $U \in X^*$ finds $V \in X^*$ such that $V$ represents a root element of $U$ in $G$.
		\end{corollary}
\begin{proof}
	The set of finite group presentations for hyperbolic groups is recursively enumerable (Papasoglu \cite{papasoglu-detecting-hyperbolicity}) and
	there is an algorithm which finds a thinness constant $\delta$ for any input finitely presented hyperbolic group $G=\langle X \mid \mathcal{R} \rangle$ (see, for example, \cite{holt-detecting-hyperbolicity}) and moreover, with respect to this constant one can find the $(X, \delta)$-full-presentation of $G$. Combination of these observations with Lemma \ref{lemma_about_finding_root_of_a_group_element} implies Corollary \ref{corollary_about_finding_root_of_a_group_element}.	
		\end{proof}
		
\subsection{HNN-extensions of (hyperbolic) groups}
	\label{section-HNN}
	Let $G=\langle X \mid \mathcal{R} \rangle$ is a finitely generated group and $A, B \leq G$ are some isomorphic subgroups of $G$, and $\phi: A \rightarrow B$ is a group isomorphism between $A$ and $B$. Then the \textit{HNN-extension} \index{HNN-extension of a group} of $G$ with respect to $\phi: A \rightarrow B$ is defined as $H^G_{\phi}= \langle X \cup \{t \}\mid \mathcal{R}, t^{-1} a t = \phi(a) ~\forall a \in A \rangle$. Note that in this text, since mostly from the context it is clear what is $\phi$ , for the HNN-extension $H^G_{\phi}$ we will use the notation $H^G_{\phi}= H=\langle G, t  \mid t^{-1}At=B \rangle$.
	
	 We are mostly interested in the case when $A=\langle a \rangle$, $B=\langle b \rangle$ are infinite cyclic groups. For this case by the notation 
	$H=\langle G, t  \mid t^{-1}at=b \rangle$, we mean the HNN-extension $H^G_{\phi}$, where $\phi: A \rightarrow B$ is induced by the map $\phi: a \mapsto b$.
	
	
	Let us consider the product
	\begin{equation}
	\label{eq-t-form}
		u=g_0 t^{\epsilon_1} g_1 t^{\epsilon_2} \ldots t^{\epsilon_n} g_n,
	\end{equation}
	where for $0\leq i \leq n$, $g_i \in G$ and for $1\leq j \leq n$, $\epsilon_j \in \{\pm 1\}$. We say that this decomposition corresponds to the sequence $(g_0, t^{\epsilon_1}, g_1, t^{\epsilon_2}, \ldots, t^{\epsilon_n}, g_n)$ and we say that a decomposition is a cyclic shift of \eqref{eq-t-form} if is corresponds to a cyclic shift of the sequence $(g_0, t^{\epsilon_1}, g_1, t^{\epsilon_2}, \ldots, t^{\epsilon_n}, g_n)$.
	
Also the decomposition from \eqref{eq-t-form} is said to be \textit{$t$-reduced} if for $0\leq i \leq n$, $g_i \in G$ and no subproduct of the form $t^{-1}a t$, $a \in A$ or of the form  $t^{}b t^{-1}$, $b \in B$, appears in \eqref{eq-t-form}. And it is said to be \textit{cyclically $t$-reduced} if all cyclic shifts of the product \eqref{eq-t-form} are \textit{$t$-reduced}. 

The word 
\begin{align*}
	w=u_0 t^{\epsilon_1} u_1 t^{\epsilon_2} \ldots t^{\epsilon_n} g_n \in (X \cup \{t\})^*
\end{align*}
is called \textit{reduced word } \index{reduced word} with respect to the HNN-extension $H$, if $u_i \in X^*$ for $0 \leq i \leq n$ and the corresponding sequence $(u_0, t^{\epsilon_1}, \ldots, t^{\epsilon_n}, u_n)$ is $t$-reduced. Analogously, $w$ is said to be \textit{cyclically reduced} with respect to the HNN-extension $H$ if all cyclic shifts of $w$ are reduced. Also we define $\theta$ as
\begin{align*}
	\theta(w)=n
\end{align*}
and for $h\in H$, define
\begin{align*}
	\theta(h)= \min \{ \theta(w) \mid w \in (X \cup \{t\})^*, w=_H h \}.
\end{align*}

	
	An element $u \in H^G_{\phi}$ is said to be  \textit{cyclically $t$-reduced} if its $t$-reduced decomposition is in fact cyclically $t$-reduced}. Again, this is a well-defined definition. Also every element  $ u\in  H^G_{\phi}$ is conjugate to a cyclically $t$-reduced element $u'$ which we call \textit{$t$-cyclic-reduction} of $u$. See \cite{lyndon schupp}. \\
	

The next lemma is a very well-known and sometimes is  called Britton's Lemma in the literature.
\begin{lemma}[Britton's Lemma]
	\label{lem-britton}
	Let $w \in (X\cup \{t\})^*$ is a reduced word with respect to the HNN-extension $H=\langle G, t  \mid t^{-1}At=B \rangle$ and $\theta(w)>0$, then $w\neq_{H} 1$.
\end{lemma}

Then next lemma is a well-known fact as well and in literature is usually called Collins' Lemma. See, for example, \cite{miller-schupp, geometry-of-wp}
\begin{lemma}[Collins' Lemma]
\label{lemma-Collins}
\label{lemma_about_conjugacy_in_HNN_extensions}
	 Let $$u=u_0t^{\alpha_0} u_1t^{\alpha_1} \ldots u_{n}t^{\alpha_n} $$ and $$v=v_0t^{\beta_0} v_1 t^{\beta_1} \ldots v_{m}t^{\beta_m}$$
	  be cyclically reduced words with respect to the HNN-extension $H$ such that $u_i, v_j \in  X^*$ and $\alpha_i, \beta_j \in \mathbb{Z}$. 
If $u$ and $v$ are conjugate in $H =  \langle G, t \mid t^{-1}A t = B; \phi \rangle$, then one of the following holds:
\begin{itemize}
	\item $u$, $v$ are words in $X^*$ which are conjugate in $X^*$;
	\item There is a finite chain of words in $G$, 
	
	$$u = w_0, w'_1,w_1,w'_2,w_2,...,w'_k, w_k, w'_{k+1},$$ such that $w_i =\phi^{\pm 1}(w_i')$, as group elements, $w_i', w_i$ represent elements from $A \cup B$, and for each $i = 0,\ldots, k$, $w_i$ is conjugate to $w'_{i+1}$ in $G$;
	\item $\theta(u), \theta(v) >0$ and $p=q$, and $u$ is conjugate in $H$ to some cyclic shift of $v$ by an element from $A \cup B$.
\end{itemize}

\end{lemma}

\begin{corollary}
\label{corollary_about_proper_powers_in_hnn}
Let $H=\langle X \cup \{t \}\mid \mathcal{R}, t^{-1} a t = \phi(a) ~\forall a \in A \rangle$ be an HNN-extension of $G=\langle X \mid \mathcal{R} \rangle$, and suppose $g_1\in G$ is not a proper power of any element in $G$. Then  the image of $g_1$ in $H$ (which we again denote by $g_1$) is a proper power in $H$ if and only if there exists $k\geq 2$ and $g_2 \in G$ such that $g_1 \sim_{conj} g_2^k$ in $H$. 
\end{corollary}
\begin{proof}
	Let for some $u\in X^*$, $u=_G g_1$, and $w\in \big(X\cup \{t\}\big)^*$ such that $u=_H w^k$ for some $k\geq 2$. Now let $w'\in \big(X\cup \{t\}\big)^*$ be a $t$-cyclic reduction of $w$. Then for some $T \in \big(X\cup \{t\}\big)^*$, we have $w=Tw'T^{-1}$. Also note that for any $k\geq 2$, $(w')^k$ is also $t$-cyclically reduced. Therefore, by Lemma \ref{lemma_about_conjugacy_in_HNN_extensions}, it must be that $w' \in X^*$, namely $w'$ represents an element in $H$ which is an image of an element from $G$.
	
	The inverse statement of the corollary is obvious.
\end{proof}

\begin{lemma}
\label{lemma_about_proper_powers_in_hnn}
	Let $H=\langle G \cup \{t \}\mid   t^{-1}at = b \rangle$ be an HNN-extension of $G=\langle X\rangle$ where $a, b \in G$ are elements of infinite order which are not proper powers. Then, for any $g_0 \in G$, if $g_0$ is not a proper power in $G$, then its image in $H$ is also not a proper power.
\end{lemma}
\begin{proof}
	 Assume that for some $u \in X^*$, $u=_G g_0$ and also assume that there exists a word $w\in \big(X\cup \{t\}\big)^*$ such that $u=_H w^k$ for some $k\geq 2$. Then, by Corollary \ref{corollary_about_proper_powers_in_hnn}, there exists a word $T \in \big(X\cup \{t\}\big)^*$  and a word $w'\in X^*$ such that $u=_H T^{-1} (w')^k T$. If $T$ does not contain $t^{\pm 1}$, then clearly we get a contradiction to the fact that $u$ is not a proper power in $G$. Therefore, it must be that $\theta(T)\geq 1$, i.e. its $t$-reduced decompositions contains $t^{\pm 1}$. Assume that $w'$ and $T$ are chosen such that $\|T\|$ is minimal for all possible such triples $(u, w', T)$.
	
	Since $Tu^{-1}T^{-1}(w')^k=_H 1$, the word $Tu^{-1}T^{-1}(w')^k$ must contain a subword of the form $t^{\epsilon} v t^{-\epsilon}$, where for some $l \in \mathbb{Z}$, $v=_G a^l$ if $\epsilon=-1$ or $v=_G b^l$ if $\epsilon=1$. Moreover, $v$ is of the form $v_1 u^{-1} v_1^{-1}$, where $v_1 \in X^*$ is a suffix of $T$. But this contradicts the minimality assumption of $T$.


\end{proof}

\begin{lemma}
\label{lemma_about_conjugacy-commensurability_hnn-1}
	Let $H=\langle G \cup \{t \}\mid   t^{-1}u_0t = v_0 \rangle$ be an HNN-extension of $G=\langle X\rangle$, where $u_0, v_0 \in X^*$.  Suppose that $u, v \in X^*$ such that $u \sim_{conj} v$ in $H$. Then, either $u \sim_{conj} v$ in $G$ or $u$ and $v$ commensurate with at least one of $u_0$ and $v_0$ in $G$.
		\end{lemma}
\begin{proof}
	It follows immediately from Theorem 2 in \cite{miller-schupp}. 
\end{proof}

	The following theorem can be found in \cite{mikhajlovskii-ol'shanskii} (it can be also regarded as a corollary from the combination theorem of Bestvina and Feighn, \cite{bestvina feighn}).
	\begin{theorem}
		\label{th hnn extension}
		Let $G$ be a hyperbolic group with isomorphic infinite elementary subgroups $A$ and $B$, and let $\phi$ be an isomorphism from $A$ to $B$. Then the HNN-extension $H =
		\langle G, t \mid t^{-1}at = \phi(a), a \in A \rangle$ of $G$ with associated subgroups $A$ and $B$ is hyperbolic if and only if the following two conditions hold:
		\begin{enumerate}
			\item either $A$ or $B$ is a maximal elementary subgroup of $G$;
			\item for all $g \in G$ the subgroup $gAg^{-1} \cap B$ is finite.
		\end{enumerate}
	\end{theorem}
	\begin{remark}
		In this work we need Theorem \ref{th hnn extension} in case when $G$ is a torsion-free hyperbolic group. Note that in case $G$ is a torsion free hyperbolic group, the subgroups $A$ and $B$, being maximal elementary subgroups, are cyclic. Therefore, in this case, the second condition in the statement of Theorem \ref{th hnn extension} can be replaced with this: for all $g\in G$, the subgroup $gAg^{-1} \cap B$ is trivial.
	\end{remark}
	~\\
	
	 \section{Lacunary hyperbolic groups}
	 \label{section-lacunary-hyperbolic}
  Let $G=\langle X \rangle$, $|X|<\infty$. Let $\bar{d}=(d_i)_{i=1}^{\infty}$ be an unbounded sequence of positive constants, called scaling constants, and let $\bar{x}=(x_i)_{i=1}^{\infty}$ be any fixed sequence of points from $\Gamma(G, X)$, called observation points. Then the ultralimit of the sequence of spaces with basepoints $(\Gamma(G, X), d/d_i, x_i)$ with respect to some non-principal ultrafilter $\omega$ over $\mathbb{N}$ is called the asymptotic cone of $G=\langle X \rangle$ with respect to $\bar{d}$ and $\omega$, where $d$ is the word metric over $\Gamma(G, X)$. It is denoted by $Con^{\omega}(G, \bar{d})$. The term \textit{asymptotic cone} was first introduced by Gromov in \cite{gromov-cones}. Since, in this paper, we do not need much details about asymptotic cones, for more detailed definitions we refer to \cite{gromov-cones, olshanskii-osin-sapir}.
  
    As it was discovered by Gromov (see, for example, \cite{gromov-asymptotic, gromov-hyperbolic}) many basic algebraic properties of
  groups can be translated into geometric or topological ones via studying asymptotic cones of Cayley graphs of the groups. For example, hyperbolicity of a group is equivalent to the fact that all the asymptotic cones of the group are $\mathbb{R}$-trees. As it is shown by Kapovich and Kleiner (see \cite{olshanskii-osin-sapir}), if for a finitely presented group at least one of the asymptotic cones is an $\mathbb{R}$-tree, then the group is hyperbolic. However, if the group is not finitely presentable, then this statement is not true anymore. In fact, \textit{lacunary hyperbolic groups} \index{lacunary hyperbolic groups} are defined to be the groups which have at least one asymptotic cone that is an $\mathbb{R}$-tree, see \cite{olshanskii-osin-sapir}. 
 \begin{definition}[Lacunary hyperbolic groups]
 	A finitely presented group $G$ is lacunary hyperbolic if for some unbounded sequence $\bar{d}=(d_i)_{i=1}^{\infty}$ of scaling constants, $Con^{\omega}(G, \bar{d})$ is an $\mathcal{R}$-tree.
 \end{definition}

  Let $\alpha: G \rightarrow G'$ be a homomorphism, $G=\langle X \rangle$. The \textit{radius} \index{radius of homomorphism} of $\alpha$ is the maximal radius of a ball in the Cayley graph $\Gamma(G, X)$ centered at $1_G$ such that all elements from that ball map to non-trivial elements in $G'$ except for $1_G$.
  
  The next lemma is essentially Theorem 1.1 from \cite{olshanskii-osin-sapir}.
\index{lacunary hyperbolic groups}
	\begin{lemma}[Theorem 1.1, \cite{olshanskii-osin-sapir}]
	\label{lem lacunary hyp gps}
		A finitely generated group $G$ is lacunary hyperbolic if and only if $G$ is the direct limit of a sequence of hyperbolic groups $G_i=\langle X_i \rangle$ ($X_i$
		is finite) and epimorphisms
		\begin{align}
		\label{lacunary seq}
		G_1 \stackrel{\alpha_1}\twoheadrightarrow G_2 \stackrel{\alpha_2}\twoheadrightarrow \ldots,
		\end{align}
		where $\alpha_i(X_i)=X_{i+1}$, and the hyperbolicity constant of $G_i$ (relative to $X_i$) is “little o” of
		the radius of $\alpha_i$.
	\end{lemma}
\begin{remark}
\label{remark 000}
Note that in part (3) of Lemma \ref{lem lacunary hyp gps}, for almost all indices $i$, $|X_i|=|X_{i+1}|$, therefore we can identify $X_i$ with $X_{i+1}$ by $x=\alpha_i(x)$ for $x \in X_i$ and regard $\alpha_i$ as the identity map from $X_{i}$ to $X_{i+1}$.
\end{remark}		
\subsection{Word problem in lacunary hyperbolic groups.}
		Let $\bar{G}=\langle X \rangle$ be a finitely presented group given as an inductive limit of the chain of epimorphims
	\begin{align}
	\label{eq !}
		G_1 \stackrel{\alpha_1}\twoheadrightarrow G_2 \stackrel{\alpha_2}\twoheadrightarrow \ldots,
	\end{align}
	where $\alpha_i: G_i \twoheadrightarrow G_{i+1}$ is the induced epimorphism from the identity map $id: X \rightarrow X$, and for $i\in \mathbb{N}$, $G_i=\langle X \mid \mathcal{R}_i \rangle$ is finitely presented.
		
	We say that $\bar{G}$ has a \textit{graded recursive presentation} \index{graded recursive presentation} with respect to \eqref{eq !}  if the function $i\mapsto \mathcal{R}_i$ is computable. In general, if $\bar{G}$ has a graded recursive presentation with respect to some sequence of type \eqref{eq !} then we say that $\bar{G}$ has a graded recursive presentation. If, in addition, all the groups $G_i$, $i\in \mathbb{N}$, are hyperbolic, then we say that the presentation is a \textit{graded recursive presentation by hyperbolic groups}. \index{graded recursive presentation! by hyperbolic groups}

For $\Upsilon: \mathbb{N} \rightarrow \mathbb{N}$, we call $\Upsilon$ a \textit{supradius} \index{supradius} for \eqref{eq !}, if for all $n \in \mathbb{N}$ and $i \in \mathbb{N}$ such that $i \geq \Upsilon(n)$, the radius for $\alpha_i: G_i \twoheadrightarrow G_{i+1}$ is greater than $n$.
We say that $\Upsilon: \mathbb{N} \rightarrow \mathbb{N}$ is a \textit{computable} supradius if $\Upsilon$ as a function is computable, i.e. the set $\{(i, \Upsilon(i)) \mid i \in \mathbb{N} \}$ is recursively enumerable.
\begin{proposition}
\label{lemma-radius-delta}
	If the limit group $\bar{G}$ is lacunary hyperbolic and $G_i$ is hyperbolic for all $i\in \mathbb{N}$, then either $\bar{G}$  is finitely presented, hence hyperbolic, or 
	\begin{align*}
	\limsup_{i \rightarrow \infty} r_i = \infty \text{~and~}\limsup_{i \rightarrow \infty} \delta_i = \infty,
	\end{align*}
	where $r_i$ is the radius of $\alpha:G_i \twoheadrightarrow G_{i+1}$ and $\delta_i$ is a hyperbolicity constant for $G_i$.
\end{proposition}
\begin{proof}
	Indeed, if $\bar{G}$ is finitely presented, then starting from some $i\in \mathbb{N}$, for all $j> i$, the normal closure of $\mathcal{R}_{j+1}$ in $G_{i}$ coincides with the normal closure of $\mathcal{R}_{i+1}$ in $G_i$. Therefore, $\bar{G}$ coincides with $G_{i+1}$, hence is hyperbolic.
	
	Now let us assume that $\bar{G}$ is lacunary hyperbolic, but it is not hyperbolic. Then, since $\bar{G}$ is not finitely presented for each $N>0$ there is $n\in \mathbb{N}$ and $R \in \mathcal{R}_n$ such that there is no $U \in X^*$ such that $\|U\|<N$ and $R=_{G_{n-1}} U$. Therefore, $\limsup_{i \rightarrow \infty} r_i = \infty$ and by the statement $(3)$ of Lemma \ref{lem lacunary hyp gps}, also $\limsup_{i \rightarrow \infty} \delta_i = \infty$.
\end{proof}
In particular, from Lemma \ref{lemma-radius-delta} it follows that for infinitely presented lacunary hyperbolic groups all supradius functions are unbounded.\\

For the proof of the next theorem we need the follows definition from \cite{arzhantseva-dehn}, which is a slight generalization of the standard notion of Dehn's presentation.
\index{$\alpha$-Dehn presentation}\index{Dehn's presentation of groups! $\alpha$-Dehn presentation}
\begin{definition}[See Definition 1 in \cite{arzhantseva-dehn}]
	For $\frac{1}{2} \leq \alpha < 1$, the group $G=\langle X \mid \mathcal{R} \rangle$ given with a finite presentation, where $\mathcal{R}$ is symmetric, is said to be $\alpha$-Dehn presented, if for any freely cyclically reduced word $W \in X^*$ representing the trivial element of $G$, for some cyclic shift $W'$ of $W$, $W'$ contains a subword $u$, such that $u$ is a prefix of some word $R \in \mathcal{R}$ and $\|u\| > \alpha \|R\|$.
\end{definition} 
It is a well-known fact that hyperbolic groups admit $\alpha$-Dehn presentations for all $\frac{1}{2} \leq \alpha < 1$. See, for example, \cite{gromov-hyperbolic, alonso brady, arzhantseva-dehn}.

It was shown by Arzhantseva  in \cite{arzhantseva-dehn}  that the property that a finite presentation of a group is an $\alpha$-Dehn presentation for some   $\frac{3}{4} \leq \alpha <1$ can be detected algorithmically as it is stated below. 
\begin{theorem}[See \cite{arzhantseva-dehn}]
\label{theorem-dehn-presentation-detection}
	There exists an algorithm determining whether or not a finite presentation of a group is an $\alpha$-Dehn presentation for some $\frac{3}{4} \leq \alpha <1$.
\end{theorem}

Note that if $G=\langle X \mid \mathcal{R} \rangle = \langle X \mid \mathcal{R}' \rangle$ and $\mathcal{R} \subseteq \mathcal{R}'$, then the presentation $G=\langle X \mid \mathcal{R} \rangle$ is a $\alpha$-Dehn presentation implies that the presentation $\langle X \mid \mathcal{R}' \rangle$ is a $\alpha$-Dehn presentation too. Also, as we already mentioned in preliminaries, if $G=\langle X \mid \mathcal{R} \rangle$ is a finite presentation for a hyperbolic group, then there is an algorithm which constructs a Dehn presentation for $G$. 

\begin{theorem}
	\label{theorem iff condition for wp in lac hyp gp 1}
	Let $\bar{G}$ be an inductive limit of hyperbolic groups connected by epimorphisms. Then $\bar{G}$ has decidable word problem if and only if it has a graded recursive presentation by hyperbolic groups and a recursively computable supradius function over that presentation.	
	\end{theorem}
\begin{proof}
	First, let us show that if $\bar{G}$ has a decidable word problem, then $\bar{G}$ possesses the mentioned properties.
	
	Indeed, let $\bar{G}=\langle X \rangle$, $|X|<\infty$. For $n\in \mathbb{N}$, let us define  $\mathcal{S}_n=\{ W \in X^* \mid W=_{\bar{G}} 1, \|W\|\leq n \}$. Since the word problem in $\bar{G}$ is decidable, we get that the sets $\mathcal{S}_1, \mathcal{S}_2, \ldots $ are algorithmically constructible - one just needs to check for each $W\in X^*$, $\|W\|\leq n$, if $W=_{\bar{G}} 1$ or not, in order to construct $\mathcal{S}_n$. 
	
	Now, since $\bar{G}$ is an inductive limit of hyperbolic groups connected by epimorphisms and since all hyperbolic groups admit $\alpha$-Dehn presentations for some  $\alpha$ such that $\frac{3}{4} \leq \alpha <1$, we get that for some $n$, there exists $\mathcal{S}_n'\subseteq \mathcal{S}_n$  such that the group presentation $H_n=\langle X \mid \mathcal{S}'_n \rangle$ is an $\alpha$-Dehn presentations for some $\alpha$, $\frac{3}{4} \leq \alpha <1$. By Theorem \ref{theorem-dehn-presentation-detection}, for a fixed $\mathcal{S} \subseteq \mathcal{S}_n$ it can be algorithmically verified if $H=\langle X \mid \mathcal{S} \rangle$  is such a presentation or not. Therefore, one can algorithmically find a maximum subset $\mathcal{S}' \subseteq \mathcal{S}_n$ with respect to $\subseteq$ such that 	$\langle X \mid \mathcal{S}' \rangle$ is an $\alpha$-Dehn presentations for some $\alpha$, $\frac{3}{4} \leq \alpha <1$. Define $\mathcal{S}'_n=\mathcal{S}'$.
		
	Since $\bar{G}$ is an inductive limit of hyperbolic groups, there exists a strictly increasing sequence of natural numbers $(n_i)_{i=1}^{\infty}$ and a sequence of sets $(\mathcal{S}_{n_i}')_{i=1}^{\infty}$ such that $\mathcal{S}_{n_1}' \subseteq \mathcal{S}_{n_1} \subsetneq \mathcal{S}_{n_2}' \subseteq \mathcal{S}_{n_2} \ldots$ and for any $i\geq 1$ the presentation $\langle X \mid \mathcal{S}'_{n_i} \rangle$ is an 
	$\alpha$-Dehn presentation for some $\alpha$, $\frac{3}{4} \leq \alpha <1$. Moreover, such sequences $(n_i)_{i=1}^{\infty}$ and $(\mathcal{S}_{n_i}')_{i=1}^{\infty}$  can be algorithmically constructed. Indeed, the partially ordered set $(2^{X^*}, \subseteq)$ is recursively enumerable, hence, assuming that $\mathcal{S}_{n_i}'$ is already constructed, based on the algorithm from Theorem \ref{theorem-dehn-presentation-detection}, for each $\mathcal{S}$ and $\mathcal{S}_{n_i}$ such that $\mathcal{S}_{n_i} \subsetneq \mathcal{S}$, we can check if $\langle X \mid \mathcal{S} \rangle$ is an $\alpha$-Dehn presentations for some $\alpha$, $\frac{3}{4} \leq \alpha <1$. Then $\mathcal{S}'_{n_{i+1}}$ can be chosen to be the smallest such $\mathcal{S}$ and $n_{i+1}$ as the smallest index such that $\mathcal{S} \subseteq \mathcal{S}_{n_{i+1}}$.

	Now, for any such algorithmically constructible sequence $(\mathcal{S}_{n_i}')_{i=1}^{\infty}$ and for all $i \geq q$, let us define $\mathcal{R}_i=\mathcal{S}'_{n_i}$.  Also for all $m\in \mathbb{N}$, define $G_m = \langle X \mid \mathcal{R}_m \rangle$. Then clearly $G_m$ is hyperbolic and $\bar{G}$ is the inductive limit of
	\begin{align}
	\label{eq !!!}
		G_1 \stackrel{\alpha_1}\twoheadrightarrow G_2 \stackrel{\alpha_2}\twoheadrightarrow \ldots,
	\end{align}
	where $\alpha_i: G_i \twoheadrightarrow G_{i+1}$ is the induced epimorphism from the identity map $id: X \rightarrow X$. Also the presentation $\bar{G}=\langle X \mid \cup_{i=1}^{\infty} \mathcal{R}_i \rangle$ is a graded recursive presentation.
	
	The final observation for proving the first part of the theorem is the following: for any $r \in \mathbb{N}$, the smallest $n$ such that radius of $\alpha_n$ is larger than $r$ can be found algorithmically. Indeed, in order to find $n$ we can first   algorithmically construct the set $\mathcal{S}_r=\{W \in X^* \mid \|W\|\leq r, W=_{\bar{G}} 1 \}$, then for each $i=1,2, \ldots$, we can iteratively check whether for each $W \in \mathcal{S}_r$, $W=_{G_i} 1$. The smallest $n$ such that for all $W \in \mathcal{S}_r$, $W=_{G_n} 1$, will be the desired index. This means that there exists a computable supradius function for \eqref{eq !!!}.
	
	Thus the first part of the theorem is proved.\\
	
	Now assume that $\bar{G}$ is the inductive limit of 
	\begin{align}
	\label{eq !!!!}
		G_1 \stackrel{\alpha_1}\twoheadrightarrow G_2 \stackrel{\alpha_2}\twoheadrightarrow \ldots,
	\end{align}
	where for $i \in \mathbb{N}$, the groups $G_i = \langle X \mid \mathcal{R}_i \rangle$ are hyperbolic groups. Let $\bar{G}=\langle X \mid \cup_{i=1}^{\infty} \mathcal{R}_i \rangle$ be a graded recursive presentation and let $\Upsilon: \mathbb{N} \rightarrow \mathbb{N}$ be a computable supradius function for \eqref{eq !!!!}. Then for each $W \in X^*$ to check whether $W=_{\bar{G}}$ or not, it is enough to check if $W=_{G_{\Upsilon(n)}} 1$ or not, where $n=\|W\|$. Since the presentation $G_{\Upsilon(n)} = \langle X \mid \mathcal{R}_{\Upsilon(n)} \rangle$ can be recursively constructed and since $G_{\Upsilon(n)}$ is hyperbolic, we get that the word problem in $G_{\Upsilon(n)}$ is decidable. Therefore, indeed, the equation $W=_{G_{\Upsilon(n)}} 1$ can be algorithmically checked. This means that the word problem in $\bar{G}$ is decidable.
	
	Thus the theorem is proved.	
	\end{proof}

\begin{remark}
	Note that the proof of Theorem \ref{theorem iff condition for wp in lac hyp gp 1} does not give any idea about complexity of word problem in lacunary hyperbolic groups. Hence we need to obtain more detailed structure of presentations of classes of lacunary hyperbolic groups in order to describe efficient word problem solving algorithms on them. Description of subclasses of LHG with effective word (and conjugacy) problems is one of the primary goals in the next sections.
\end{remark}
\begin{corollary}
	A lacunary hyperbolic group has a solvable word problem if and only if it has a graded recursive presentation by hyperbolic groups and a computable supradius function over that presentation.
\end{corollary}

	\section{Small cancellation conditions} 
	\label{section Small cancellation conditions}
	In the first subsection we are going to recall small cancellation concepts for hyperbolic groups introduced in \cite{Olsh G-groups} and in the second section we extend this concept to chains of hyperbolic groups connected by epimorphisms.\\
	\subsection{Small cancellation in hyperbolic groups}(See \cite{Olsh G-groups}.)
	\label{subsection-small-cancellation}
	Let $G = \langle X \rangle$ be a finitely generated group, and let $\mathcal{R}$ be a symmetric set of words from $X^*$. A subword $U$ of a word $R \in \mathcal{R}$ is called an \textit{$\epsilon$-piece} \index{$\epsilon$-piece}  for $\epsilon \geq 0$ if there exists a word $R' \in \mathcal{R}$ such that
	\begin{enumerate}
		\item $R\equiv UV$, $R' \equiv U'V'$ for some $V, U', V' \in X^*$;
		\item $U'=_G YUZ$ for some $Y, Z \in X^*$ where $\| Y \|, \|Z \| \leq \epsilon$;
		\item $YRY^{-1} \neq_G R'$.
	\end{enumerate}
	
	It is said that the system $\mathcal{R}$ satisfies the $C(\lambda, c, \epsilon, \mu,  \rho)$-\textit{condition} \index{$C(\epsilon, \mu, \lambda, c, \rho)$-condition} for some $\lambda\geq 1$, $c\geq 0$,$\epsilon \geq 0$, $\mu >0$,  $\rho >0$, if 
	\begin{enumerate}
		\item[(1.1)] $\|R\| \geq \rho$ for any $R \in \mathcal{R}$;
		\item[(1.2)]  any word $R \in \mathcal{R}$ is $(\lambda, c)$-quasi-geodesic;
		\item[(1.3)]  for any $\epsilon$-piece of any word $R \in \mathcal{R}$, the inequalities $\| U \|, \|U'\| < \mu \|R\|$ hold.
	\end{enumerate}
	
	Now suppose that for a word $R \in \mathcal{R}$ we have
	\begin{enumerate}
		\item[(2.1)]  $R = UVU'V'$ for some $U,V, U', V' \in X^*$;
		\item[(2.2)]  $U' = Y U^{\pm 1} Z$ in the group $G$ for some words $Y,Z \in X^*$ where $\|Y\|$, $\|Z\|\leq \epsilon$;
	\end{enumerate}
	then the word $U$ is called an \textit{$\epsilon'$-piece} \index{$\epsilon'$-piece} of the word $R$. If $\mathcal{R}$ satisfies the $C(\lambda, c, \epsilon, \mu,  \rho)$-condition and, in addition, for all $R\in \mathcal{R}$, the above described decomposition of $R$ implies $\|U\|, \|U'\| < \mu \|R\|$ then, like in \cite{Olsh G-groups}, we say that $\mathcal{R}$ satisfies the \textit{$C'(\lambda, c, \epsilon, \mu,  \rho)$-condition} \index{$C'(\lambda, c, \epsilon, \mu,  \rho)$-condition}. 
	
	\subsection{Auxiliary parameters, lowest parameter principle (LPP) and the main conventions}
	\index{lowest parameter principle (LPP)}
	In Subsection \ref{subsection-small-cancellation}, in the context of the definition of the small cancellation condition $C(\lambda, c,  \epsilon, \mu,  \rho)$ the parameters 	$\delta, \lambda, c, \epsilon, \mu, \rho$ were introduced.
	In this paper, whenever we mention the small-cancellation condition $C(\lambda, c, \epsilon, \mu, \rho)$, we assume that the parameters $\delta,\lambda, c, \epsilon, \mu, \rho$ satisfy some relations. More specifically, $\epsilon$ depends on $\lambda$ and $c$; $\mu$ depends on $\lambda, c$ and $\epsilon$; and $\rho$ depends on $\lambda, c, \epsilon$ and $\mu$ (see, for example, Lemma \ref{lem 6.6} for an example where the condition $C(\lambda, c, \epsilon, \mu, \rho)$ is involved).
	
	Based on a similar concept introduced in \cite{olshanskii_the_book} (see \S 15 in \cite{olshanskii_the_book}), we introduce the notation $\succ$ between parameters defined as follows: if $ \alpha_1, \alpha_2, \ldots$ are some parameters, then $\alpha_1 \succ \alpha_2 \succ \ldots$ means that the value of  $\alpha_i$ is being chosen after the parameters $\alpha_1,  \ldots \alpha_{i-1}$ were chosen. In other words, the parameters  $\alpha_1,  \ldots \alpha_{i-1}$ are independent of $\alpha_i$, but $\alpha_i$ depends on the values of  $\alpha_1, \ldots, \alpha_{i-1}$. If $\alpha$ and $\beta$ are some parameters such that $\alpha \succ \beta$ then we say that $\alpha$ is a \textit{higher} parameter (correspondingly, $\beta$ is a \textit{lower} parameter), alternatively, we say that $\alpha$ has higher priority with respect to $\beta$ and $\beta$ has lower priority with respect to $\alpha$.\\ \index{parameters} \index{parameters! with lower(higher) priority}
	\begin{convention}
	\label{convention-lpp}
	Throughout this text we will deal with statements involving parameters $\lambda, c, \epsilon, \mu, \rho$ and their indexed versions  $\lambda_i, c_i, \epsilon_i, \mu_i, \rho_i$ for $i\in \mathbb{N}$. For all these parameters we assume that  $\delta \succ \lambda \succ c \succ \epsilon \succ \mu \succ \rho$. Analogously,  $\lambda_i \succ c_i \succ \epsilon_i \succ \mu_i \succ \rho_i$. 
	Also we assume that parameters with lower indexes are higher with respect to $\succ$.
	
	We also will deal with parameters $\delta_i$, $\delta_i'$. For them we assume $\rho_i \succ \delta_i $ and $\lambda_i \succ \delta_i' \succ \rho_{i-1}$ for $i=1,2, \ldots$.
  
	\end{convention}
	\begin{convention}
	\label{convention-large_enough}
		Throughout this text, for  parameters $\delta, \lambda, c, \epsilon, \mu^{-1}, \rho$ and their indexed versions  $\lambda_i, c_i, \epsilon_i, \mu_i, \rho_i$ when we say that some parameter, say $\alpha$, is \textbf{large enough} 
		then we mean that there is a finite number of parameters of higher priority, say $\beta_1, \ldots, \beta_k$, and a computable function $f_{\alpha, \beta_1, \ldots, \beta_k} : \mathbb{N}^k \rightarrow \mathbb{N}$ such that $\alpha$ can be chosen to have any value greater than $f(\beta_1, \ldots, \beta_k)$. 
		For example, if for $\rho_i$ "large enough" means $\rho_i> \lambda_i\mu_i$, then we think of $i$ to be an arbitrary index from $\mathbb{N}$. 
		
			\end{convention}
	\begin{definition}[The standard parameters]
	\index{standard parameters}
The parameters $\delta,  \lambda, c, \epsilon, \mu, \rho$ and the indexed parameters  $\delta_i, \delta_i', \lambda_i, c_i, \epsilon_i, \mu_i, \rho_i$, which are intensively used in this paper, we call the standard parameters. 
		\end{definition}
		\begin{definition}[Sparse enough standard parameters]
		\label{sparse parameters}
		\index{standard parameters! sparse enough}
		We will say that a sequence of parameters is \textit{sparse enough} if for each parameter $\alpha_{i_0}$, where $i_0$ is the index of the parameter, we assume that 	
		\begin{align}
		\label{relations}	
		\alpha_{i_0} > f_{i_0, i_1, \ldots, i_k}(\alpha_{i_1}, \ldots, \alpha_{i_k}),
		\end{align}	
		  where $\alpha_{i_1}, \ldots, \alpha_{i_k}$ are parameters with smaller indices $i_1, \ldots, i_k$ (hence, of higher priority) and $f_{i_0, i_1, \ldots, i_k}$ is a computable function such that   $f_{i_0, i_1, \ldots, i_k}=f_{i_0+t, i_1+t, \ldots, i_k+t}$ for all $t\geq 0$, and the map $ i_0 \mapsto f_{i_0, i_1, \ldots, i_k}$ is computable as well. 
	\end{definition}
	
	\begin{convention}[Lowest parameter principle (LPP)]
	\index{lowest parameter principle (LPP)}
		In order many results of the current paper to hold (for example, theorem \ref{theorem-effective-G-conjugacy-problem}, \ref{theorem_about_connecton_of_word_and_conjugacy_problems}, etc.), we require from the standard parameters to be sparse enough. Therefore, whenever we mention some relation of the form \eqref{relations} involving the standard parameters, for example, $\epsilon> \lambda\delta+c$ or $\epsilon_i> \mu_i \rho_i+c$ (the last one is equivalent to $\mu_i^{-1}>(\epsilon_i-c)^{-1}$), then we say that this relation holds by lower parameter principle -- simply, by LPP.
	\end{convention}

	\subsection{Words with small cancellation conditions}
	\label{words with small cancellation-1}
	Hereafter, if it is not stated otherwise, we assume that $G=\langle X \rangle$ is a non-trivial, non-elementary, torsion free $\delta$-hyperbolic group for some $\delta>0$.
	
	Let us consider a set $\mathcal{R}$ consisting of words of the form
	\begin{align}
	\label{eq words with cancellation condition}
	R_i=z_i U^{m_{i, 1}} V U^{m_{i, 2}} V  U^{m_{i, 3}}\ldots V U^{m_{i, j_i}}, ~~i=1,2, \ldots, k
	\end{align}
	and their cyclic shifts, where $k \in \mathbb{N}$, $U, V,z_1, \ldots, z_k \in X^*$ are geodesic words, $U, V \neq_G 1$, and $m_{i, t} \in \mathbb{N}$ for $1\leq i \leq k$, $1 \leq t \leq j_i$.
	Denote $Z=\{z_1, \ldots, z_k\}$, $L = \max\{\|U\|, \|V\|, \|z_1\|,\ldots, \|z_k\|\}$.
	
	Let $\tilde{\lambda}, \tilde{c} \in \mathbb{N}$ be such that $U^n$ is $(\tilde{\lambda}, \tilde{c})$-quasi-geodesic in $\Gamma(G,X)$ for all $n\in \mathbb{Z}$. Note that the existence of $\tilde{\lambda}$ and $\tilde{c}$ follows from \eqref{lem 1.11}. Moreover, given the $\delta$-hyperbolic group $G=\langle X \rangle$ and the word $U$, one can find such a pair $(\tilde{\lambda}, \tilde{c})$ algorithmically.
	
	
	Now let $\underbar{m}=\min\{m_{i, t} \mid 1\leq i \leq k ~\text{and}~ 1 \leq t \leq j_i\}$, $\overline{m}_i=\max\{m_{i, t} \mid  1 \leq t \leq j_i\}$ for $1\leq i \leq k$. Then the following holds.
	\begin{lemma}
		\label{lem 1111 on cancellation words}
		For the set of words $\mathcal{R}$ suppose that $V \notin E(U)$, $z_i \notin E(U)$ for $1\leq i \leq k$. Then there exist constants $\lambda=c=\tilde{K} \in \mathbb{N}$, computably depending on $G$, $U$, $V$ and $Z$, such that the words of the system \eqref{eq words with cancellation condition} are $(\lambda, c)$-quasi-geodesic in $\Gamma(G,X)$, provided that $\underbar{m} \geq \tilde{K}$. 	\end{lemma}
	\begin{proof}
	We will show that $\lambda,c$ and $\tilde{K}$ can be effectively computed by the following formulas
		\begin{equation}
		\begin{aligned}
		\label{eq finding m_0 to make a system of words quasigeodesic}
		\lambda=c=\tilde{K}=\mathcal{K}\big(24\tilde{\lambda}, (2\bar{\mathcal{M}}+2)L\big)\\
		\end{aligned}
		\end{equation}	
		where $\mathcal{K}( ~)$ is defined as in Lemma \ref{lem 1}, $L = \max\{\|U\|, \|V\|, \|z_1\|,\ldots, \|z_k\|\}$ and 
		\begin{align*}
			\bar{\mathcal{M}}=\max\big\{24\tilde{\lambda}+\tilde{c} , ~\mathcal{M}(U, V_1)\mid V_1\in \big\{V^{\pm 1}, z_1^{\pm 1}, \ldots, z_k^{\pm 1} \big\}\big\}
		\end{align*}
		where $\mathcal{M}()$ is given  by the formula \eqref{def of M}.

	First, we will show that all the paths  in $\Gamma(G, X)$ with labels of the form
	\begin{equation}
	\label{auxiliary_form_1}
		W_1 U^{a_1} V_1^b U^{a_2}W_2,
	\end{equation}
	where $W_1$ and $W_2$ are subwords of some words from $ \big\{U^{\pm 1}, V^{\pm 1}, z_1^{\pm 1}, \ldots, z_k^{\pm 1} \big\}$, $V_1\in \big\{V^{\pm 1}, z_1^{\pm 1}, \ldots, z_k^{\pm 1} \big\}$ and $b \in \{0, \pm 1\}$, are 
	$\Big(24  \tilde{\lambda}, (2\bar{\mathcal{M}}+2)L \Big)$-quasi-geodesic. 
	
	For that let us fix an arbitrary such path $q$, with $lab(q)=W_1 U^{a_1} V_1^b U^{a_2}W_2$.
	Note that since all the subwords of $lab(q)$ are also of the form (\ref{auxiliary_form_1}), to show that $q$ is $\Big(24\tilde{\lambda}, (2\bar{\mathcal{M}}+2)L \Big)$-quasi-geodesic, it is enough to show that 
	\begin{equation*}
		\|W_1 U^{a_1} V_1^b U^{a_2}W_2\| \leq 24\lambda|W_1 U^{a_1} V_1^b U^{a_2}W_2| + (2\mathcal{M}+2)L.
	\end{equation*}
	
	To this end we will separately consider three cases: 
	\begin{enumerate}
		\item[1.] when $b\neq 0$ and $\max\{a_1, a_2\}<\bar{\mathcal{M}}$;
		\item[2.] when $b \neq 0$ and $\max\{a_1, a_2\}\geq \bar{\mathcal{M}}$; and
		\item[3.] when $b=0$. 
	\end{enumerate}	 
	\textit{Case 1.} If $b\neq 0$ and $\max\{a_1, a_2\}< \bar{\mathcal{M}}$, then 
	\begin{align*}
		\|W_1 U^{a_1} V_1^b U^{a_2}W_2\| \leq &\|W_1\|+a_1\|U\|+\|V_1\|+a_2\|U\|+\|W_2\|\\
		\leq &\|W_1\|+\|V_1\|+\|W_2\|+(2\bar{\mathcal{M}}-1)\|U\| \leq (2\bar{\mathcal{M}}+2)L.
	\end{align*}
	\textit{Case 2.} If $b \neq 0$ and $\max\{a_1, a_2\}\geq \bar{\mathcal{M}}$, then, by Lemma \ref{lem 2.1}, either $|U^{a_1} V_1^b U^{a_2}| \geq \upsilon \max\{a_1, a_2\}$  or $V_1 \in E(U)$, where $\upsilon=\|U\|/12\tilde{\lambda}$.
	
	 Since, by our assumptions, $V_1\notin E(U)$, we get that $|U^{a_1} V_1^b U^{a_2}| \geq \upsilon \max\{a_1, a_2\}$. Therefore,
	\begin{equation}
		\label{auxiliary_2}
		|W_1 U^{a_1} V_1^b U^{a_2}W_2| \geq |U^{a_1} V_1^b U^{a_2}|-|W_1|-|W_2|\\
		\geq \upsilon \max\{a_1, a_2\} -2L.
	\end{equation}
	
	   On the other hand,
	   \begin{align*}
	   	\|W_1 U^{a_1} V_1^b U^{a_2}W_2\| &\leq \|W_1\|+a_1\|U\|+\|V_1\|+a_2\|U\|+\|W_2\| \\
	   	&\leq 2\max\{a_1, a_2\} \|U\|+3L \\
	   	\text{~ by (\ref{auxiliary_2}),~} &\leq 2 \bigg(\frac{|W_1 U^{a_1} V_1^b U^{a_2}W_2|+2L}{\upsilon}\bigg)\|U\|+3L\\
	   	&\leq 24\tilde{\lambda}|W_1 U^{a_1} V_1^b U^{a_2}W_2| + (48\tilde{\lambda}+1)L\\
	   	 	&\leq 24\tilde{\lambda}|W_1 U^{a_1} V_1^b U^{a_2}W_2| + (2\bar{\mathcal{M}}+2)L.
	   \end{align*}
	\textit{Case 3.} If $b=0$, then since $U^{a_1+a_2}$ is a $(\tilde{\lambda}, \tilde{c})$-quasi-geodesic word,
	   we get
	   \begin{align*}
	   	\|W_1 U^{a_1} V_1^b &U^{a_2}W_2\|=\|W_1 U^{a_1+a_2}W_2\| \leq \|U^{a_1+a_2}\|+\|W_1\|+\|W_2\|\\
	   	 &\leq \lambda  |U^{a_1+a_2}|+c+2L < \lambda(|W_1 U^{a_1+a_2}W_2|+2L)+c+2L \\	 
	   	&< 24\lambda|W_1  U^{a_1+a_2}W_2| + (2\bar{\mathcal{M}}+2)L. 	   \end{align*}
	   	Formula (\ref{def of M}) implies that  $2(\tilde{\lambda}+1)L+\tilde{c} < (2\bar{\mathcal{M}}+2)L$, hence the last inequality is true.\\

	Now, let $p$ be a path in $\Gamma(G,X)$ whose label corresponds to a word from $\mathcal{R}$. Since $\tilde{K} \leq m$, all the subpaths of $p$ of the lengths bounded from above by $\tilde{K}$ are of the form (\ref{auxiliary_form_1}). Therefore, $p$ is $\tilde{K}$-local $\Big(24\tilde{\lambda}, (2\bar{\mathcal{M}}+2)L \Big)$-quasi-geodesic. Therefore, taken into account the formula for $\tilde{K}$ from (\ref{eq finding m_0 to make a system of words quasigeodesic}) and the inequality $m \geq \tilde{K}$, by Lemma \ref{lem 1}, $p$ is $(\tilde{K}, \tilde{K})$-quasi-geodesic.
	\end{proof}
	
	Assume that in the system (\ref{eq words with cancellation condition}), for all $1\leq i, i' \leq k$ and $1 \leq t \leq j_{i}$, $1 \leq t' \leq j_{i'}$, $m_{i,t} \neq m_{i', t'}$ if $(i, t) \neq (i', t')$. 
	
	Recall that in Lemma \ref{lem 1111 on cancellation words} we required
	\begin{align}
	\label{a condition 1}
	V \notin E(U) \text{~and $z_i \notin E(U)$ for $1\leq i \leq k$}.	
	\end{align}

		Let us introduce the following notations: 
		 For a given $\epsilon>0$, $\epsilon_0 = \epsilon+2L$, $\epsilon_i = \epsilon_0+i\big(2R_{\lambda, c}+182\delta+\frac{L}{2} \big)$ for $1\leq i \leq 5 $, where $R_{\lambda, c}$ is defined as in Lemma \ref{lem hausdorff distance between quasi-geodesics} and, as before, $L = \max\{\|U\|, \|V\|, \|z_1\|,\ldots, \|z_k\|\}$. Let $\tilde{K}$ be defined by the formula (\ref{eq finding m_0 to make a system of words quasigeodesic}) and $\lambda=c=\tilde{K}$. Now, with respect to the given constants $\epsilon \geq 0$, $\mu >0$, $\rho >0$ assume that
		\begin{equation}
			\label{inequality -1}
			\|R\|\geq \rho, \text{~for all~} R \in \mathcal{R},
		\end{equation}
		
		\begin{align}
		\label{mistery}
		\underbar{m} \geq \tilde{K},
		\end{align}
		hence, by Lemma \ref{lem 1111 on cancellation words}, the words from $\mathcal{R}$ are $(\lambda, c)$-quasi-geodesics in $\Gamma(G, X)$.
		Next, we require the following.
		\begin{align}
		\label{inequal 1}
		\mu \|R_i\| \geq 6L(\overline{m}_i+1)
		\end{align}		
		and
		\begin{align}
		\label{inequality mmmm}
		 \underbar{m} \geq \frac{2\epsilon_5}{\upsilon}
		\end{align}	
		where $\upsilon=\upsilon(U)$ is defined by formula (\ref{def of theta}).	
	\begin{lemma}
		\label{lem on words with cancellation condition}
		Using the setting of the previous lemma and assuming that the above described conditions take place, let us consider the system of words $\mathcal{R}$ given by \eqref{eq words with cancellation condition}. Let $\lambda, c$ be defined by the formulas \eqref{eq finding m_0 to make a system of words quasigeodesic}.
		Then, if for the given constants $\epsilon \geq 0$, $\mu >0$, $\rho >0$, the conditions \eqref{inequality -1}, \eqref{mistery}, \eqref{inequal 1} and  \eqref{inequality mmmm} are satisfied, then the system $\mathcal{R}$ satisfies the $C'(\lambda, c, \epsilon, \mu, \rho)$-condition.
		
		Moreover, if two words $R_1, R_2 \in \mathcal{R}$ are not equal up to cyclic shifts, then there are no subwords $U_1$ and $U_2$ of $R_1$ and $R_2$, respectively, such that $\|U_1\| \geq \rho \|R_1\|$ and for some $T_1, T_2 \in X^*$, $\|T_1\|, \|T_2\| \leq \epsilon$ and 
		\[
		T_1^{-1} U_1 T_2 =_G U_2.
		\] 
		
	\end{lemma}
	\begin{proof}
		First of all, let us assume that the constants $\epsilon \geq 0$, $\mu >0$, $\rho >0$ are already given.

		Assume by contradiction that there exist two different words from $\mathcal{R}$, $W_1$ and $W_2$, which have common $\epsilon$-pieces. Suppose that $W_1$ and $W_2$ are cyclic shifts of the words 
		\begin{align*}
		z_i U^{m_{i, 1}} V U^{m_{i, 2}} V  U^{m_{i, 3}}\ldots V U^{m_{i, j_i}}		\end{align*}
		and
		\begin{align*}
		z_{i'} U^{m_{i', 1}} V U^{m_{i', 2}} V  U^{m_{i', 3}}\ldots V U^{m_{i', j_{i'}}} 
		\end{align*}
		or of their inverses, for some $1\leq i, i' \leq k$.
		
		Existence of a common $\epsilon$-piece for the words $W_1$ and $W_2$ implies that there is a rectangle $ABCD$ in $\Gamma(G, X)$ such that the labels of $AD$ and $BC$ are prefixes of $W_1$ and $W_2$ with length at least $\mu\|W_1\|$ and $\mu \|W_2 \|$, respectively, and $AB$, $CD$ are geodesics with length at most $\epsilon$. 
		
		Let us call vertices on $AD$ and $BC$ \textit{phase vertices} if they are either origin or endpoint of a subpath with label $U^{\pm 1}$, $V^{\pm 1}$, $z_i^{\pm 1}$ or $z_{i'}^{\pm 1}$. 
		
		Note that after making $AB$ and $CD$ longer by at most $2y$, we can ensure that $A$, $B$, $C$ and $D$ are phase vertices. Hereafter, let us assume that the length of $AB$ and $CD$ are bounded by $\epsilon+2L= \epsilon_0$ and the vertices $A, B, C$ and $ D$ are phase vertices.
		
		 We will call a subpath of $AD$ or $BC$ \textit{special} if it is labeled by $V$, $z_i$ or $z_{i'}$. If a special segment on $AD$ or $BC$ is between other special segments then we call this special segment \textit{inner}, otherwise, we call it \textit{boundary special segment}. Note that for any point $O \in AD$ (or $O \in BC$), there is a phase vertex $O' \in AD$ (or, correspondingly,  $O' \in BC$), such that $\big\|[O, O']\big\| \leq L/2$.
		
		Before proceeding further, let us state and prove the following auxiliary claims.\\
		~\\
		\textit{Claim 1.}
		For the rectangle $ABCD$ let us consider any inner special segment $P_1P_2$ on one of the sides $AD$ or $BC$. For concreteness let us assume that $P_1P_2$ belongs to $AD$. Then  for any phase vertex $Q_1 \in BC$, if $d(P_1, Q_1) \leq \epsilon_3$, then either $lab(P_1Q_1) \in E(g)$ in $G$ or $lab(P_2Q_1) \in E(g)$ in $G$, where by $lab(P_1Q_1)$ and $lab(P_2Q_1)$ we mean the labels of any paths joining $P_1$ to $Q_1$ and $P_2$ to $Q_1$, respectively.
		\begin{proof}
			Let $P_3P_4$ and $P_5P_6$ be the closest to $P_1P_2$ special segments on $AD$ such that $P_1P_2$ is between $P_3P_4$ and $P_5P_6$ (their existence follows from the assumption that $P_1P_2$ is an inner special segment). Let $Q_1 \in BC$ be a fixed phase vertex such that $d(P_1, Q_1)\leq \epsilon_3$, and let $Q_4$ be the closest to $P_4$ phase vertex on $BQ_1$. See Figure \ref{fig: small cancellation words 1}.
			\begin{figure}[H]
				\centering
				\includegraphics[clip, trim=1cm 19cm 1cm 4.9cm, width=1.\textwidth]{{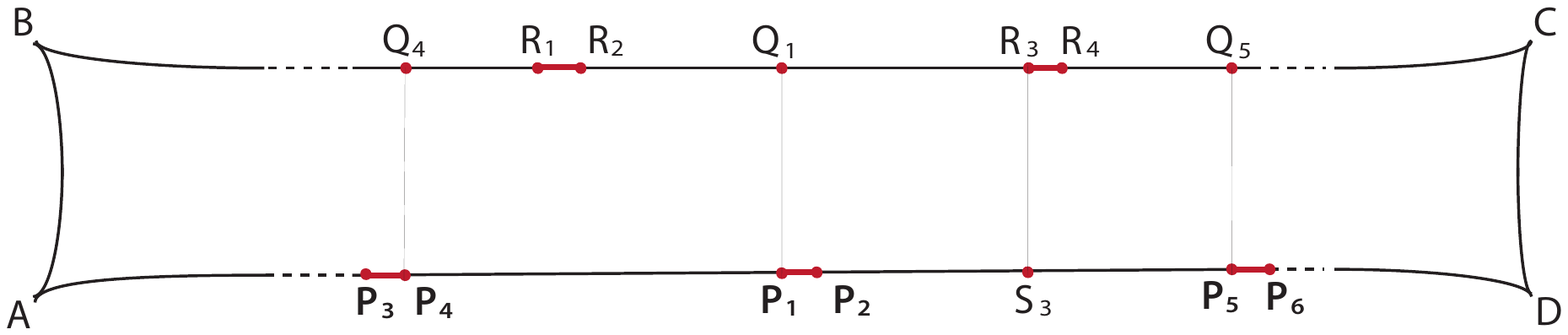}} 
				\caption{} 
				\label{fig:  small cancellation words 1}
			\end{figure}
			Since $d(P_1, Q_1), d(A, B) \leq \epsilon_3$, by Corollary \ref{corollary on hausdorff distance between quasi-geodesics}, we get $dist(P_4, BQ_1)\leq \epsilon_3+2R_{\lambda, c}+2\delta$. Therefore, $d(P_4, Q_4) \leq \epsilon_3+2R_{\lambda, c}+2\delta+\frac{L}{2} = \epsilon_4$.
			
			There are two possibilities which we are going to discuss separately: either $lab(Q_4Q_1)$ is a power of $U$ or $Q_4Q_1$ contains a special segment.
			
			 In case $lab(Q_4Q_1)$ is a power of $U$, since $lab(P_4P_1)=U^{\xi_1}$ for $\xi_1\geq \underbar{m}\geq \mathcal{M}$ and $d(P_4, Q_4), d(P_1, Q_1)\leq \epsilon_4 \leq \upsilon\underbar{m}$ (the last inequality follows from \eqref{inequality mmmm}), by Lemma \ref{lem 2.1}, the equality 
			 \begin{align*}
			 	lab(P_4Q_4)lab(Q_4Q_1)lab(Q_1P_1)lab(P_1P_4)=_G 1
			 \end{align*}
			  implies that $lab(P_1Q_1)\in E(U)$.
			  Thus we are done with this case.
			
			Now let us consider the case when $Q_4Q_1$ contains a special segment. Let $Q_5$ be a phase vertex on $Q_1C$ closest to $P_5$. By Corollary \ref{corollary on hausdorff distance between quasi-geodesics}, we again get $d(P_5, Q_5)\leq \epsilon_4$. Again, if $lab(Q_1Q_5)$ is a power of $U$, then, similarly to the previous case, by Lemma \ref{lem 2.1}, $lab(P_2Q_1)\in E(g)$ in $G$. Thus we are  left only with the case when both $Q_4Q_1$ and $Q_1Q_5$ contain special segments. Let us consider this case in more details.
			
			 Let $R_1R_2$ and $R_3R_4$ be the closest to $Q_1$ special segments on $Q_4Q_1$ and $Q_1Q_5$, respectively. See Figure \ref{fig:  small cancellation words 1}. Since $lab(R_2R_3)$ has a form $U^{\xi_2}$, where $|\xi_2| \geq \underbar{m}$, at least one of $lab(R_2Q_1)$ and $lab(Q_1R_3)$ is of the form $U^{\xi_3}$, where $|\xi_3| \geq \underbar{m}/2 \geq \mathcal{M}$. Without loss of generality, assume that $lab(Q_1R_3)=U^{\xi_3}$ for $|\xi_3| \geq \underbar{m}/2$. Then, let $S_3$ be a phase vertex on $P_1P_5$ closest to $R_3$. Then, by Corollary \ref{corollary on hausdorff distance between quasi-geodesics}, $d(R_3, Q_3)\leq \epsilon_5$. Therefore, since by \eqref{inequality mmmm}, $ \upsilon\underbar{m} \geq 2\epsilon_5$, by Lemma \ref{lem 2.1}, the equality
			 \begin{align*}
			 	lab(P_2Q_1)lab(Q_1R_3)lab(R_3S_3)lab(S_3P_2)=_G 1
			 \end{align*}
			  implies that $lab(Q_1P_2) \in E(U)$.\\ 
		\end{proof}
		~\\
		\textit{Claim 2.} If $A_1A_2$, $A_3 A_4$ and $A_5A_6$ are three consecutive inner special segments belonging either to $AD$ or to $BC$, then $A_3A_4$ is a special segment on $AD \cap BC$.
		
		\begin{proof}
			Firstly, without loss of generality let us assume that $A_1A_2$, $A_3 A_4$ and $A_5A_6$  belong to $AD$. Let $B_3$ be the closest to $A_3$ phase vertex on $BC$, $B_2$ be the closest to $A_2$ phase vertex on $BB_3$ and $B_5$ be the closest to $A_5$ phase vertex on $B_3C$. See Figure \ref{fig:  small cancellation words 2}.\\
			\begin{figure}[H]
				\centering
				\includegraphics[clip, trim=1cm 19cm 1cm 4.9cm, width=1.\textwidth]{{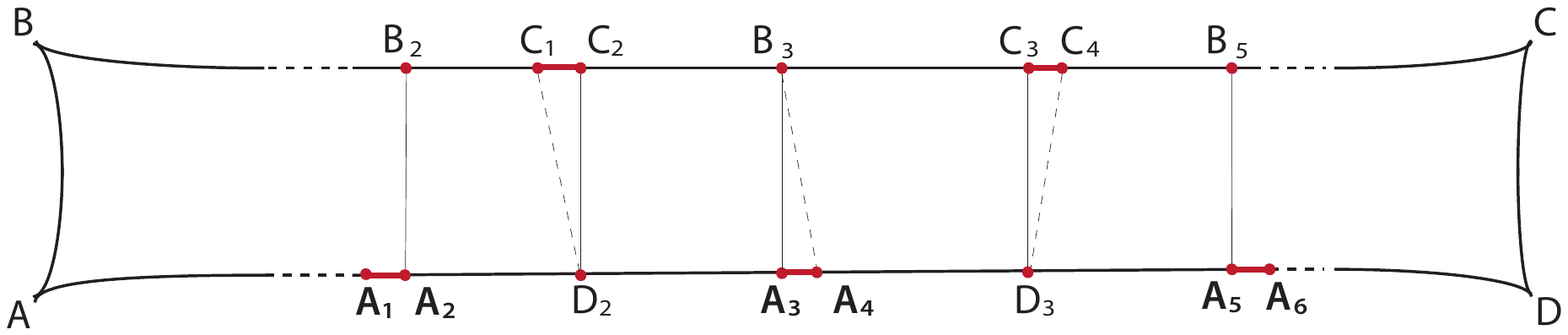}} 
				\caption{} 
				\label{fig:  small cancellation words 2}
			\end{figure}	
			We will consider the case when $B_3B_5$ contains special segment(s) and the case when it does not contain any special segment separetely.
			
			First let us consider the case when $B_3B_5$ contains special segment(s). Let $C_1C_2$ be the closest to $B_3$ special segment on $B_2B_3$ and $C_3C_4$ be the closest to $B_3$ special segment on $B_3B_5$. In case $B_2B_3$ does not contain any special segments, we take $C_2=B_2$.
						
			Correspondingly, let $D_2$ be the closest to $C_2$ phase vertex on $A_2A_3$ and $D_3$ be the closest to $C_3$ phase vertex on $A_4A_5$. Then, by Corollary \ref{corollary on hausdorff distance between quasi-geodesics}, $d(A_3, B_3) \leq \epsilon_1$, $d(A_2, B_2), d(A_5, B_5)\leq \epsilon_2$, and hence $d(C_2, D_2), d(C_3, D_3) \leq \epsilon_3$. Therefore, by Claim 1, one word from each pair $(lab(C_1D_2)$, $lab(C_2D_2))$; $(lab(A_3B_3)$, $lab(A_4, B_3))$ and $(lab(D_3, C_3)$, $lab(D_3, C_4))$ belongs to $E(U)$.
			
			Note that if $lab(D_3C_4) \in E(U)$, then it cannot be so that $lab(A_4B_3) \in E(U)$, because otherwise it would imply that $lab(C_3C_4)\in E(U)$ as well, which is not true by our assumptions (see the condition (\ref{a condition 1})). Therefore, in case $lab(D_3C_4) \in E(U)$, it must be that $lab(A_3B_3) \in E(U)$. But, since $lab(A_3C_3), lab(A_4 C_4) \in E(U)$ in that case, by condition (\ref{a condition 1}), it would mean that $d(A_3, C_3) = d(A_4, C_4) = 0$ or, in other words, $A_3A_4$ coincides with $C_3C_4$.
			
			Now, if $lab(D_3C_3) \in E(U)$, then $lab(A_4B_3) \in E(U)$. Therefore, because of the condition (\ref{a condition 1}), $lab(C_2D_2)$ cannot belong to $E(U)$. Finally, in case $C_2=B_2$, by Claim 1, this would mean that $lab(A_1B_2) \in E(U)$, which is impossible because of the condition (\ref{a condition 1}). Otherwise, again by Claim 1, $lab(C_1D_2) \in E(g)$ in $G$. Therefore, by the condition (\ref{a condition 1}), since in this case $lab(A_3C_1) \in E(U)$ and $lab(A_4C_4)\in E(U)$, we would get $A_3=C_1$ and $A_4=C_2$.
			
			Now let us turn to the case when $B_3B_5$ does not contain a special segment. In this case, by applying Lemma \ref{lem 2.1} to the boundary label of the rectangle $A_4B_3B_5A_5$ we get that $lab(A_4B_3)$ and $lab(A_5B_5)$ belong to $E(U)$. Then, by repeating previous arguments, we obtain that $lab(D_2C_1) \in E(U)$ and consequently $A_3=C_1$ and $A_4=C_2$. Thus Claim 2 is proved.
		\end{proof}
		Inequality \eqref{inequal 1} assures us that on $AD$ one can find six consecutive special segments $A_1A_2$, $A_3A_4$,  $A_5A_6$,  $A_7A_8$,  $A_9A_{10}$ and  $A_{11}A_{12}$. By Claim 2, $A_5A_6$,  $A_7A_8$ belong to $AD \cap BC$. See Figure \ref{fig:  small cancellation words 3}.
		\begin{figure}[H]
			\centering
			\includegraphics[clip, trim=1cm 19cm 1.2cm 4.5cm, width=1.\textwidth]{{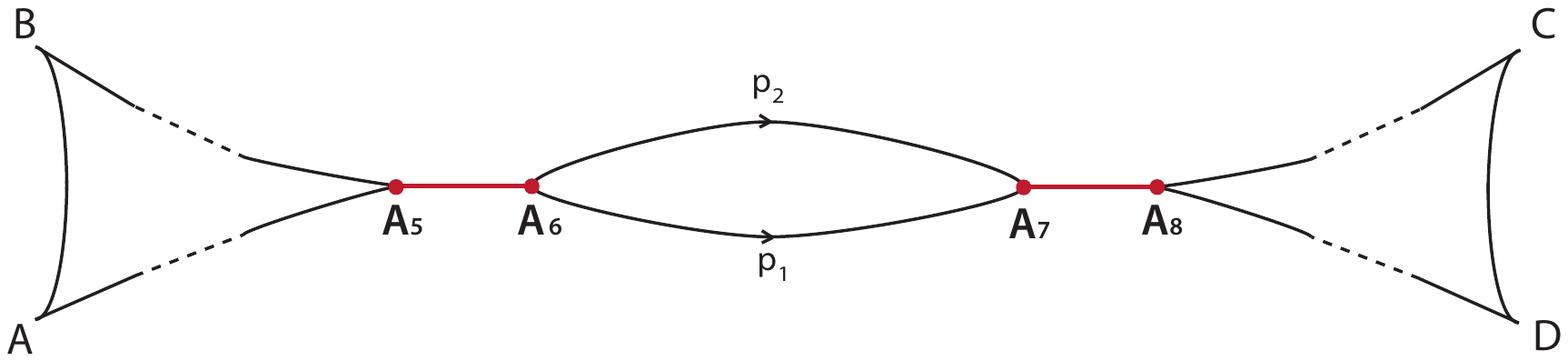}} 

 			\caption{} 
			\label{fig:  small cancellation words 3}
		\end{figure}	
		As it is shown in Figure  \ref{fig:  small cancellation words 3}, let us denote the subpaths of $AD$ and $BC$ restricted between $A_6$ and $A_7$ by $p_1$ and $p_2$, respectively. Since $A_5A_6$ and $A_7A_8$ are consecutive special segments, the label of $p_1$ is a power of $U$. Now, assuming that $p_2$ contains a special segment, just like it was done in the proof of Claim 2, we can show that that special segment must also belong to $p_1$, which is impossible since $p_1$ does not contain any special segments. Therefore, it must be that $p_2$ also does not contain any special segments. In other words, the label of $p_2$ is a power of $U$ as well. This means that the label of the closed path $p_2p_1^{-1}$ is also a power of $U$. But since $U$ represents an element $g\in G$ of infinite order, this can happen if and only if the label of $p_2p_1^{-1}$ is the empty word, i.e., when $p_2$ coincides with $p_1$. 
		
		Now, since for all  $1 \leq t \leq j_{i}$, $1 \leq t' \leq j_{i'}$, $m_{i,t} \neq m_{i', t'}$ if $(i, t) \neq (i', t')$, the last observation implies that, in fact, $i=i'$ and $W_1$ is a cyclic shift of $W_2$. Moreover, we get that either $lab(BA)$ is a suffix of $lab(AA_5)$ or $lab(AA_5)$ is a suffix of $lab(BA)$. This means that either $lab(AB)$ is equal to a prefix of $W_1$ in $G$ or $lab(BA)$ is equal to a prefix of $W_2$ in $G$; but this is impossible, because it
		contradicts condition (3) in the definition of $\epsilon$-pieces.
		~\\
		
		Now it follows from \eqref{inequality -1} and Lemma \ref{lem 1111 on cancellation words} that $\mathcal{R}$ satisfies the small cancellation condition $C'(\lambda, c, \epsilon, \mu, \rho)$.
		
	\end{proof}
	
	\subsection{A special subclass of small cancellation words}
	\label{subsection-a-class-of-small-cancellation-words}
	
	Using the already established setting of Subsection \ref{words with small cancellation-1}, let us define the positive integer $m_{1, 1}$ as the smallest integer satisfying all the constraints put on it in Subsection \ref{words with small cancellation-1}. 

	Now let us assume that in the set of words $\mathcal{R}$ we have that $m_{1,1}$ is defined as above and for all $1 \leq i \leq k$,  $m_{i, 1}=2^{i-1}m_{1, 1}$, $j_i=m_{i, 1}-1$ and for all $1 \leq t \leq j_{i}$, $m_{i, t}=m_{i,1}+(t-1)$. 

	If all these equations are satisfied, then we denote the system of words $\mathcal{R}$ by
	\begin{align}
	\index{$\mathcal{R} \big((z_j)_{j=1}^k, U, V,  \delta, \tilde{\lambda}, \tilde{c}, \epsilon,  \mu, \rho \big)$, a set of words}
	\label{definition_of_special_words}
	\mathcal{R}=\mathcal{R} \big(\boldsymbol{Z}, U, V,  \delta, \lambda, c, \epsilon,  \mu, \rho \big).
	\end{align}
	where $\boldsymbol{Z}$ is the ordered set $\{z_1<z_2<\ldots<z_k \}$.
	
	Note that the set of words $\mathcal{R}$ defined this way satisfies all the conditions prescribed for Lemmas \ref{lem 1111 on cancellation words} and  \ref{lem on words with cancellation condition}. A little bit less obvious among this conditions seems to be condition \eqref{inequal 1}. Let us show that condition \eqref{inequal 1} holds as well.
		
	Indeed, the length of each word $R_i$ from $\mathcal{R} \big(\boldsymbol{Z}, U, V,  \delta, \lambda, c, \epsilon,  \mu, \rho \big)$ is not smaller than $m_{i,1}+(m_{i,1}+1)+\ldots+ (2m_{i,1}-1) > m_{i,1}^2$ and $\overline{m}_i = 2m_{i,1}-1$. Therefore, for each $1\leq i \leq k$, $\mu \|R_i\|\geq \mu m_{i,1}^2$. 
	Now we have $\mu \|R_i\| \geq \mu m_{i,1}^2 \geq 12Lm_{i,1}=6L(\overline{m}_i+1)$. Note that the last inequality follows from the condition A2.
	
	Thus, by Lemma \ref{lem 1111 on cancellation words} and Lemma \ref{lem on words with cancellation condition}, the set of words $\mathcal{R} \big(\boldsymbol{Z}, U, V,  \delta, \lambda, c, \epsilon,  \mu, \rho \big)$ satisfies the small cancellation condition $C'(\lambda, c, \epsilon, \mu, \rho)$.\\

Let $f: \mathbb{N} \rightarrow \mathbb{N}$ be a linear time computable function. Then for all $n \in \mathbb{N}$, define
   \begin{equation*}
   	_n\mathcal{R} = \big\{ R \in \mathcal{R} \mid \|R\| \leq f(n) \big\}.
   \end{equation*}
   Assuming that $\boldsymbol{Z}, U, V,  \delta, \lambda, c, \epsilon,  \mu, \rho$ are already computed, from the structure of \eqref{definition_of_special_words}, it is not hard to see that the set $_n\mathcal{R}$ can be computed in time bounded from above by $An$, where $A>0$ does not depend on the parameters of $_n\mathcal{R}$. Thus we get the following property.
  \begin{property}
  \label{property-time-for-special-words}
  $_n\mathcal{R}$ can be computed in time bounded from above by $An$, where $A$ is a positive constant not depending on $\boldsymbol{Z}, U, V,  \delta, \lambda, c, \epsilon,  \mu, \rho$.
  \end{property}
For the applications, let us introduce the following convention: 
\begin{align}
\label{definition_of_special_words-2}
	\mathcal{R} \big( \emptyset, U, V,  \delta, \lambda, c, \epsilon,  \mu, \rho \big) = \emptyset.
\end{align}
	
	
	\subsection{Planar diagrams over hyperbolic groups and van Kampen's lemma}
	Let $H=\langle X \mid \mathcal{O}\rangle$, where $|X|<\infty$.
	
	A \textit{map} is a finite, planar connected $2$-complex. A \textit{diagram} $\Delta$ over $X$ is a map whose edges $e$ are labeled by letters $lab(e) \in X^{\pm 1}$ such that $lab(e)^{-1}= lab(e^{-1})$. 
	 The label of a path $p=e_1 \ldots e_n$ in $\Delta$ is, by definition, the word $lab(e_1) \ldots lab(e_n)$. A diagram over $X$ is called a \textit{ diagram over the group} $H=\langle X \mid \mathcal{O} \rangle$ if the label of the boundary path of every cell of $\Delta$ is a cyclic shift of some relator from $R$.  \index{diagram}
	
	A van Kampen lemma \index{van Kampen's lemma} states that a word $W \in X^*$ represents the identity of the group $H$ if and only if there is a simply connected diagram $\Delta$ over $H=\langle X \mid \mathcal{O} \rangle$  such that the boundary label of $\Delta$ is $W$. Hence, for a given $W\in X^*$ we call such a $\Delta$\textit{ van Kampen's diagram} \index{van Kampen's diagram} with label $W$ over $H=\langle X \mid \mathcal{O} \rangle$. In this paper we only use simply connected diagrams. Therefore, hereafter by diagrams we will mean simply connected diagrams.
	
	Note that for any diagram $\Delta$ over $H=\langle X \mid \bigO \rangle$, $\Delta$ can be naturally projected into the Cayley graph $\Gamma(H, X)$ such that all the labels are preserved. Moreover, if we fix arbitrary vertex $o_1$ of $\Delta$ and arbitrary vertex $o_2$ of $\Gamma(H, X)$, then the projection which maps $o_1$ to $o_2$ is defined uniquely. We denote this projection by $Proj_{o_1}^{o_2}(\Delta)$. Since in the applications of the current paper, we do not need to specify $o_1$ and $o_2$, we will simply use the notation $Proj(\Delta)$. This projection allows us to consider word metric on $\Delta$, by simply considering the word metric on the projection of $\Delta$. \index{$Proj(\Delta)$}

	Let $G$ be a quotient of $H$. When considering group $G$ we will partition the defining relators into two sets. The first set $\mathcal{O}$ will consist of all relators (not only defining) of $H$ with a fixed generating set $X$. The second set, $\mathcal{R}$, will be some symmetrized set of additional relators. We shall write
	
	\begin{align}
	\label{presentation of G1}
	G = \langle X \mid \mathcal{O} \cup \mathcal{R} \rangle =\langle H \mid \mathcal{R} \rangle.
	\end{align}
	
	\index{$\mathcal{R}$-cell} \index{$0$-cell}
	Using the terminology of \cite{Olsh G-groups} we call the cells of a diagram with boundary labels from $\mathcal{O}$ (from $\mathcal{R}$) $0$-cells ($\mathcal{R}$-cells). 
	Diagram is called \textit{reduced} if it contains minimal number of $\mathcal{R}$-cells among all diagrams with the same boundary label.
	
	
	\index{$\epsilon$-contiguity subdiagram} \index{contiguity subdiagram}
	\index{contiguity arc}	Now consider  a simple closed path $w = p_1q_1p_2 q_2$ in a diagram $\Delta$ over $G$, such that $q_1$ and $q_2$ are subpaths of boundary cycles of $\mathcal{R}$-cells $\Pi_1$ and $\Pi_2$, and $\|p_1\|$, $\|p_2\| \leq \epsilon$ for a fixed constant $\epsilon$. Assuming that the subdiagram $\Gamma$ of $\Delta$ bounded by $w$ has no hole and no $\mathcal{R}$-cell and $\Pi_1 \neq \Pi_2$, following Olshanskii, we call $\Gamma$ a \textit{$\epsilon$-contiguity}  (or simply, \textit{contiguity}) subdiagram of $\Pi_1$ and $\Pi_2$. The same term will be used if $\Pi_1 = \Pi_2$ and $\Gamma$ contains no holes. 
	In case $q_2$ instead of being a subpath of $\Pi_2$ is a subpath of of a connected path $q$ on $\partial\Delta$, $\Gamma$ is called \textit{outer $\epsilon$-contiguity subdiagram} (from $\Pi_1$ to $\partial\Delta$ or to $q$). 
	The notation $\partial (\Pi_1, \Gamma, \Pi_2)$ (or $\partial(\Pi_1, \Gamma, q)$) $=p_1q_1p_2q_2$ will define the partition of the contour $w$ of $\Gamma$. The above subpaths $q_1$ and $q_2$ are called the contiguity arcs while $p_1$ and $p_2$ are called the side arcs of the contiguity subdiagram $\Gamma$.
	
	Hereafter we will denote by $\partial\Pi$ the loop in $\Gamma(G, X)$ with the label equal to the label of $\Pi$. By $\| \Pi \|$ we denote the length of the boundary label of a cell $\Pi$. The ratio $\|q_1\|/\| \Pi_1\|$ for a contiguity subdiagram of a cell $\Pi_1$ to a cell $\Pi_2$ (or to a section $q$), is called the contiguity degree of $\Pi_1$ to $\Pi_2$ via $\Gamma$ (or of $\Pi_1$ to $q$). It is denoted $(\Pi_1, \Gamma,\Pi_2)$ (or $(\Pi_1, \Gamma, q)$). For a matter of convenience, instead of the notation $(\Pi_1, \Gamma, q)$ we will simply use the notation $\partial\Gamma$  if it does not lead to ambiguities.
	\index{$(\Pi_1, \Gamma,\Pi_2)$} \index{$(\Pi_1, \Gamma, q)$}
	
	
	\index{contiguity arc!inner}\index{contiguity arc!outer}

	If for a contiguity subdiagram $\Gamma \in \mathcal{M}$, $p_1^{-1}q_1p_2q_2^{-1} = \partial\Gamma$, $q_2$ belongs to $\partial\Delta$, then $q_2$ is called \textit{outer} contiguity arc, and correspondingly $q_1$ is called \textit{inner} contiguity arc. 
	Whenever it is not mentioned otherwise, hereafter we will denote the outer arc of $\Gamma$ by $\hat{q}_{\Gamma}$ and the inner arc by $\check{q}_{\Gamma}$. Also let us denote $p_1 = p_{\Gamma}$ and $p_2=p'_{\Gamma}$.\\
	\index{$\check{q}_{\Gamma}$, $\hat{q}_{\Gamma}$, $p_{\Gamma}$, $p'_{\Gamma}$}
	
			\subsection{Quotients of hyperbolic groups by normal closures of words with small cancellation conditions}
	\label{subsection about quotients of hyp gps}
Hereafter, if $\Delta$ is a diagram over the quotient $G=H/\ll \mathcal{R} \gg$, then by saying that the boundary $\partial\Delta$ of  $\Delta$ is a $(\lambda, c)$-quasi-geodesic $t$-gon, we mean that $\partial\Delta$ is partitioned into $t$ connected pieces such that they are $(\lambda, c)$-quasi-geodesic in $\Gamma(H, X)$.

	\begin{lemma}[see Lemma 4.6 in \cite{olshanskii-osin-sapir} and  Lemma 6.6 in \cite{Olsh G-groups}]
		\label{lem 6.6}
		
		For appropriately chosen parameters based on the lowest parameter principle with respect to the order $\lambda \succ c \succ \epsilon \succ \mu \succ \rho$, if the presentation (\ref{presentation of G1}) satisfies the $C( \lambda, c, \epsilon, \mu, \rho)$-condition,  then for any reduced disk diagram $\Delta$ over the presentation (\ref{presentation of G1})  whose boundary is a $(\lambda, c)$-quasi-geodesic $t$-gon for $1\leq t \leq 12$ and which contains an $\mathcal{R}$-cell, there exists an $\mathcal{R}$-cell $\Pi$ in $\Delta$ and disjoint outer $\epsilon$-contiguity subdiagrams $\Gamma_{1}, \ldots, \Gamma_{t}$ of $\Pi$ to different sides of the $(\lambda, c)$-q.g. $t$-gon $\partial\Delta$, such that
		\begin{align}
		\label{greendlinger inequality}
		\sum_{i=1}^t(\Pi, \Gamma_{i}, \hat{q}_{\Gamma_i}) > 1- 23 \mu.
		\end{align}
		Moreover, the quotient $G=H/\ll \mathcal{R} \gg$ is $4L$-hyperbolic, where $L=\max\{ \|R\| \mid R \in \mathcal{R} \}$.
	\end{lemma}
	~\\
	\textbf{Remark.}
	Note that, in fact, in Lemma \ref{lem 6.6} some of the subdiagrams $\Gamma_{1}, \ldots, \Gamma_{t}$, say $\Gamma_1$,  may not exist, in which case we would call $\Gamma_1$ \textit{empty} \index{empty contiguity subdiagram} \index{contiguity subdiagram!empty} contiguity subdiagram and take $(\Pi, \Gamma_{1}, \hat{q}_{\Gamma_1}) =0$. The important thing is that, according to Lemma \ref{lem 6.6}, some of $\Gamma_{1}, \ldots, \Gamma_{t}$  are not empty, so that the inequality \eqref{greendlinger inequality} holds.
	
	\begin{lemma}[Lemma 7.2, \cite{Olsh G-groups}]
		\label{lem 7.2}
		 Let $H=\langle X \rangle$ by a non-elementary hyperbolic group. Let $G$ be a group with a presentation (\ref{presentation of G1}) such that $\mathcal{R}$ satisfies the $C'(\lambda, c, \epsilon, \mu,  \rho)$-condition for appropriately chosen parameters $\lambda \succ c \succ \epsilon \succ \mu \succ \rho$. Then $G$ is non-cyclic, each $R\in \mathcal{R}$ represents an element of infinite order in $G$, and a word $W\in X^*$ has a finite order in $G$ if and only if $W$ is trivial in $G$ or is conjugate in $G$ to an element having finite order in $H$.
	\end{lemma}
	\begin{remark}
	\label{remarm 7.2}
		Note that if $H$ is a non-elementary torsion-free hyperbolic group, then Lemma \ref{lem 7.2} implies that $G$ is also a non-elementary torsion-free hyperbolic group. 
	\end{remark}

	\begin{definition}[\textit{Essential} cells and contiguity subdiagrams]
		Let $\Pi$ be an $\mathcal{R}$-cell in a reduced van Kampen diagram $\Delta$ with $(\lambda, c)$-quasi-geodesic $t$-gon boundary for $1\leq t \leq 12$, and let $\Pi$ be connected to the sides of the $t$-gon $\partial\Delta$  by disjoint outer $\epsilon$-contiguity subdiagrams $\Gamma_1$, \ldots, $\Gamma_t$  such that
		\begin{align}
		\label{greendlinger inequality}
		\sum_{i=1}^t(\Pi, \Gamma_{i}, \hat{q}_{\Gamma_i}) > 1- 23 \mu.
		\end{align}
		Then we call $\Pi$ \textit{an essential} cell,\index{essential cell} and the contiguity subdiagrams $\Gamma_{1}, \ldots, \Gamma_{t}$ -- \textit{essential} contiguity subdiagrams. \index{essential contiguity subdiagrams}		
	\end{definition}

\subsection{Auxiliary definitions and lemmas}	
In this subsection we discuss some auxiliary lemmas and definitions for  $G=H/\ll \mathcal{R} \gg$, where $H=\langle X \rangle$ is hyperbolic and $ \mathcal{R}$ satisfies the $C(\lambda, c, \epsilon, \mu,  \rho)$-condition. Also, $0\leq \epsilon_0 \leq \eta \leq 1$ are some constants.

\begin{definition}[$(\epsilon_0, \eta)$-arcs and $(\epsilon_0, \eta)$-words]
	\label{()-arcs}
	\index{$(\epsilon_0, \eta)$-arc, $(\epsilon_0, \eta)$-word }
		$W_0\in X^*$ is an $(\epsilon_0, \eta)$-word (associated with a word $R\in \mathcal{R}$) with respect to the quotient $G=H/\ll \mathcal{R} \gg$, if there exist words $T_1, T_2 \in X^*$, $\|T_1\|, \|T_2\| \leq \epsilon_0$ and a word $R \in \mathcal{R}$ such that $R=UV$, $\|U\| \geq \eta\|R\|$ and
	\begin{align*}
		 W_0 =_H T_1^{-1} U T_2.
	\end{align*}
	
	A subpath $p'$ of a path $p$ from $\Gamma(G, X)$ is called $(\epsilon_0, \eta)$-arc (or $(\epsilon_0, \eta)$-subpath) if its label is a $(\epsilon_0, \eta)$-word.
	

\end{definition}

\begin{lemma}
\label{lemma-arcs-associated-cells-are-short}
Suppose that $W\in X^*$ contains  a $(\epsilon_0, \eta)$-subword  associated with some word $R \in \mathcal{R}$. Then 
$$\|R\| \leq \frac{\lambda(\|W\|+2\epsilon_0)+c)}{\eta}.$$
\end{lemma}
\begin{proof}
	Follows from the definition of the $(\epsilon_0, \eta)$-subwords and the fact that the word from $\mathcal{R}$ are $(\lambda, c)$-quasi-geodesics in $\Gamma(H, X)$. We just need to apply the triangle inequality.
\end{proof}

\begin{lemma}
		\label{lemma on Greendlinger's cell}
		Let us consider the quotient $G=H/\ll \mathcal{R} \gg$, where $H=\langle X \rangle$ and $ \mathcal{R}$ satisfies the $C(\lambda, c, \epsilon, \mu,  \rho)$-condition. Then for any constants $\epsilon_0\geq 0$ and $K>0$, if $\rho$ is small enough and $\rho$ is large enough, then there is no $(\lambda, c)$-quasi-geodesic path in $\Gamma(G, X)$ containing an $(\epsilon_0, 1-K\lambda \mu)$-arc.
\end{lemma}

\begin{proof}
	All the metric notations   which we use in this proof are in Cayley graph $\Gamma(G, X)$.
	
	First of all, assume that $0< \mu < \frac{1}{K\lambda}$ so that we have $0< 1-K\lambda \mu <1$.
	
	Now assume that there exists a $(\lambda, c)$-quasi-geodesic path $p$ in $\Gamma(H, X)$ which contains an $(\epsilon_0, 1-K\lambda \mu)$-arc $p'$. 
	Then, by definition, there exist words $T_1, T_2 \in X^*$, $\|T_1\|, \|T_2\| \leq \epsilon_0$ and a word $R \in \mathcal{R}$, such that $R=UV$, $\|U\| \geq (1-K\lambda \mu)\|R\|$ and
	\begin{align*}
		lab(p') =_H T_1^{-1} U T_2.
	\end{align*}
	Then, combining the last equation with triangle inequality and with the inequality $\|U\| \geq (1-K\lambda \mu)\|R\|$, we get 
	\begin{align}
	\label{aux-maux-1}
		|p'| \geq \frac{(1-K\lambda \mu)\|R\|-c}{\lambda}-2\epsilon_0.
	\end{align}
	On the other hand, by triangle inequality, we have
	\begin{align}
	\label{aux-maux-2}
	|p'|\leq 2\epsilon_0 + \|V\| \leq 2\epsilon_0 +(1-(1-K\lambda \mu)) \|R\| = 2\epsilon_0+K\lambda \mu \|R\|.	
	\end{align}
	Finally, note that, since $\|R\|\geq \rho$, if $\rho$ is large enough, then the system of inequalities \eqref{aux-maux-1} and \eqref{aux-maux-2}  is not consistent, which contradicts the existence of $p'$
	\end{proof}

	\begin{definition}[Truncated diagrams]
	\label{def-truncated diagrams}
	\index{truncated diagrams}
	If a van Kampen diagram $\Delta$ over $G=H/\ll \mathcal{R} \gg$ has a rectangular boundary $\partial\Delta=ABCD$ such that the following conditions hold.
	\begin{enumerate}
		\item $[A, D]$ and $[B, C]$ are $(\lambda, c)$-quasi-geodesics in $\Gamma(G, X)$,
			\item $[A, B]$ and $[D, C]$ are geodesic,
			\item $d_G(A, B)=dist_G(A, [B, C])$, $d_G(D, C)=dist_G(D, [B, C])$.

	\end{enumerate}
	Then $\Delta$ is called truncated diagram.
	\end{definition}

	\begin{lemma}
	\label{lemma-truncated-diagrams}
			Suppose that $\Delta$ is a reduced diagram over $G=H/\ll \mathcal{R} \gg$ such that $\partial\Delta = ABCD$, $\Delta$ is truncated and the following holds:
			\begin{align}
			\label{cond-(4)}
				 d_G(A, D) \geq \lambda\big( L+\big\|[A, B]\big\|+\big\|[D, C]\big\|+2\epsilon \big)+c, 
			\end{align}
			\text{~where~} $L=\max\{\|R\| \mid R \in \mathcal{R} \}$.
		
		Suppose that $\Delta$ contains an essential $\mathcal{R}$-cell $\Pi$ connected to $[A, B]$, $[B, C]$, $[C, D]$ and $[D, A]$ by essential $\epsilon$-contiguity subdiagrams $\Gamma_1$, $\Gamma_2$, $\Gamma_3$ and $\Gamma_4$, respectively. Then, if the standard parameters are large enough, we have
		\begin{enumerate}
			\item[(i)] either $\Gamma_1$ or $\Gamma_3$ is empty;
			\item[(ii)]  $(\Pi, \Gamma_1, [A, B])+(\Pi, \Gamma_2, [B,C]) + (\Pi, \Gamma_3, [B,C]) \leq 1-26\mu$; and
			\item[(iii)] $(\Pi, \Gamma_4, [A, D])>\mu$.
		\end{enumerate}
	\end{lemma}
	\begin{proof}
		First of all, if both $\Gamma_1$ and $\Gamma_3$ are not empty, then the distance between $[A, B]$ and $[D, C]$ is bounded by $2\epsilon+\|\Pi\| \leq 2\epsilon+L$. Therefore, since $[A, D]$ is $(\lambda, c)$-quasi-geodesic in $\Gamma(G, X)$, by the triangle inequality, we have $d_G(A, D)\leq \lambda\big(\big\|[A, B]\big\|+\big\|[D, C]\big\|+2\epsilon+L \big)+c$, which contradicts the condition \eqref{cond-(4)} in the statement of the lemma.
		Therefore, without loss of generality we can  assume that $\Gamma_3$ is empty.
		
		
		Now let us prove that $$(\Pi, \Gamma_1, [A, B])+(\Pi, \Gamma_2, [B, C])<1-26\mu.$$ 
		For that let us denote $\kappa_1=(\Pi, \Gamma_1, [A, B])$ and $\kappa_2=(\Pi, \Gamma_2, [B, C])$, and by contradiction, assume that $\kappa_1+\kappa_2 \geq 1-26\mu$. Then, since $d_G(A, B)=dist_G(A, [B, C])$, we get $d_G(A, B)\leq d_G(A, (\hat{q}_{\Gamma_2})_+)$, and consequently, 
		\begin{align}
		\label{Equation111}
			d_G\big((\hat{q}_{\Gamma_1})_-, B \big)\leq d_G\big((\hat{q}_{\Gamma_1})_-, (\hat{q}_{\Gamma_2})_+ \big) \leq 2\epsilon+(1-\kappa_1 - \kappa_2)\|\Pi\| < 2\epsilon + 26\mu\|\Pi\|.
		\end{align}
		 See Figure \ref{fig:  for lemma 18}.
		 Since $\partial\Pi$ is $(\lambda, c)$-quasi-geodesic, we also have
		 \begin{align}
		 \label{Equation222}
		 \frac{\kappa_1 \|\Pi\|-c}{\lambda}-2\epsilon \leq \| \hat{q}_{\Gamma_1} \| \leq d_G((\hat{q}_{\Gamma_1})_-, B).	
		 \end{align}
		 
		 Combining \eqref{Equation111} and \eqref{Equation222}, we get $\kappa_1 \|\Pi\| \leq \lambda(4\epsilon + 26 \mu\|\Pi\|) + c <^{\text{by LPP}}  27 \lambda \mu \|\Pi \|$, and consequently, we get $\kappa_1 < 27 \lambda \mu$. Therefore, $\kappa_2 > 1-26\mu -27 \lambda\mu > 1-53 \lambda \mu$, or in other words, $\hat{q}_{\Gamma_2}$ is a $(\epsilon, 1-53\lambda\mu)$-arc, by Lemma \ref{lemma on Greendlinger's cell}, but for large enough $\rho$ this  is impossible, because $[B, C]$ is $(\lambda, c)$-quasi-geodesic. 

		\begin{figure}[H]
				\centering
				\includegraphics[clip, trim=1.9cm 19cm 1cm 0cm, width=.8\textwidth]{{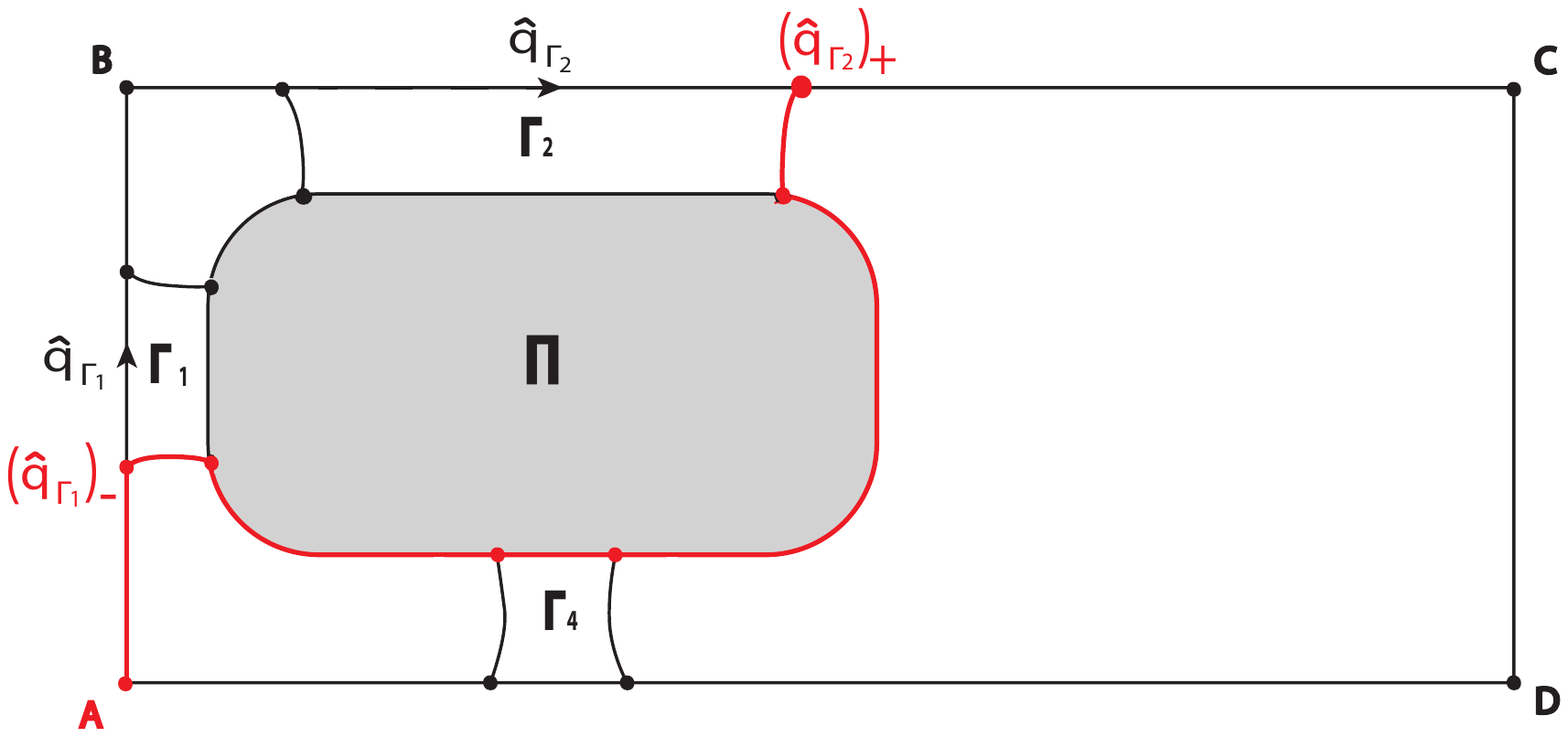}} 
				\caption{} 
				\label{fig:  for lemma 18}
			\end{figure}

	\end{proof}
	
	Finally, since the system of contiguity diagrams $\Gamma_i$, $i=1,2,3,4$, is essential and $\Gamma_3$ is empty, we get that $(\Pi, \Gamma_4, [A, D]) \geq 1-23\mu - (\kappa_1+\kappa_2) > 3\mu >\mu$.

	\begin{lemma}
		\label{lemma_about_contiguity_arcs_with_the_same_label}
		
		Suppose that $\mathcal{R}$ satisfies $C'(\lambda, c, \epsilon, \mu,  \rho)$-condition and $\epsilon$, $\rho$ are large enough.
		Suppose that $R \in \mathcal{R}$ and $U, V$ are disjoint subwords from $R$ such that for some words $T_1, T_2\in X^*$, $\|T_1\|,  \|T_2\| \leq 2\epsilon $ and $T_1^{-1} U T_2 =_H V^{\pm 1}$. 
		
		Then $\|U\|, \|V\| \leq 2\mu \|R\|$.
		
	\end{lemma}
		 \begin{proof}
		   The statement follows from the definition of the small cancellation condition $C'(\lambda, c, \epsilon, \mu,  \rho)$ (see properties (2.1) and (2.2) in the definition of $C'(\lambda, c, \epsilon, \mu,  \rho)$) and Corollary \ref{another corollary about hausdorff distance}.
		 \end{proof}

	\section{van Kampen diagrams over HNN-extensions}
	Let $ G=\langle X \mid \mathcal{R} \rangle$  be a group presentation.
Let $S$ be a subset of $X$. Then an \textit{$S$-band} $B$ is a sequence of cells $\Pi_1$,...,$\Pi_n$ in a van Kampen diagram such that
	
	\begin{itemize}
	\item Each two consecutive cells in this sequence have a common edge labeled by a letter from $S$;
	\item Each cell $\Pi_i$, $i = 1, ..., n$, has exactly two $S$-edges (i.e. edges labeled by a letter from $S$) having opposite orientations.
	\end{itemize}
	
	We call the sides of an $S$-band with labels from $S$ \textit{ends} of the band, and the sides of the $S$-band which are separated by the ends  \textit{sides} of the $S$-band. \index{bands} \index{bands! ends of bands} \index{bands! sides of bands}\\
	
	The concept of bands naturally occurs when one considers HNN-extensions of groups as follows.
	
	Let $H=\langle X, t \mid t^{-1} A t=B \rangle$, where $A, B \leq G=\langle X \mid R \rangle$ are isomorphic subgroups by some isomorphism $\phi: A \rightarrow B$. Then, from van Kampen's lemma it follows that for each $W \in (X \cup \{t\})^*$ such that $W=_H 1$, there exists a disc diagram (=van Kampen diagram) $\Delta$ over the presentation 
	\begin{align}
	\label{HNN-pres}
	H=\big\langle X \cup \{t\} \mid \mathcal{R} \cup \{t^{-1} a t \phi(a)^{-1} \mid a \in A \} \big\rangle	
	\end{align}
such that $lab(\partial \Delta)= W$ and $\Delta$ is reduced, in the sense that it contains minimal number of $t$-bands among all van Kampen diagrams with boundary label $W$. A well-known fact is that either $\Delta$ does not contain $t$-bands  (which implies that $W \in X^*$ and $W=_G1$) or all the $t$-bands of $\Delta$ have their edges with label $t^{\pm 1}$ on the boundary $\partial\Delta$ of $\Delta$. For more details see, for example, \cite{miller-schupp, sapir-book}.
	
	Analogously, if $H$ is obtained from $G$ by multiple HNN-extensions with respect to isomorphic subgroups $\phi_1: A_1 \rightarrow B_1$, \ldots, $\phi_n: A_n \rightarrow B_n$, namely,
	$$H=\big\langle X \cup \{t_1, \ldots, t_n\} \mid \mathcal{R} \cup \{t_i^{-1} a_i t_i \phi_i(a_i)^{-1} \mid 1\leq i \leq n, a_i \in A_i \} \big\rangle, $$
	Then for each $1\leq i \leq n$, either $W$ does not contain letters from $\{t_i^{\pm 1} \}$ or in $\Delta$ all $t_i$-bands have their ends on $\partial \Delta$ and moreover, every edge of $\partial \Delta$ with a label from $\{t_i^{\pm 1}\}$ is connected with  a $t_i$-band to another edge on $\partial \Delta$ with the same label.


	\section{Slender conjugacy diagrams and their geometry}
\label{section-slender conjugacy dyagrams}
\subsection{Slender conjugacy diagrams over hyperbolic groups}
	Let $H = \langle X \mid \mathcal{O} \rangle$, $|\mathcal{O}|<\infty$, $|X|<\infty$, be a non-cyclic $\delta$-hyperbolic group with respect to $\Gamma(G, X)$ for some $\delta \in \mathbb{N}$.
	
	We call a disk diagram $\Delta$ with quadrilateral boundary $ABCD$, a d$(U,V)$-conjugacy diagram \index{$(U,V)$-conjugacy diagram} over  $\langle X \mid \mathcal{O} \rangle$ if $lab(AB)=lab(DC)$ and $lab(BC)=U$, $lab(AD)=V$.
	
	We say that $\Delta$ is a \index{slender conjugacy diagram} \textit{slender} $(U, V)$-conjugacy diagram over  $\langle X \mid \mathcal{O}\rangle$, if $AB$ has minimal length among all $(U, V)$-conjugacy diagrams over  $\langle X \mid \mathcal{O} \rangle$.
	Also we say $\Delta$ is a \index{slender conjugacy diagram! cyclically slender conjugacy diagram}  $\textit{cyclically slender}$ $(U,V)$-conjugacy diagram over  $\langle X \mid \mathcal{O} \rangle$ if it is a $(U',V')$-conjugacy diagram  for some cyclic shifts $U'$ and $V'$ of $U$ and $V$, respectively, and in addition,	\begin{align*}
	\|lab(AB)\|=\\
	\min\{&\|lab(A'B')\| \mid ~\forall (U',V')\text{-conj. diagram~} 
	\Delta' \text{~with~} \partial\Delta'=A'B'C'D',\\
	& \text{where $U'$ and $V'$ are, respectively, cyclic shifts of $U$ and $V$}\}.
	\end{align*}
	
	For arbitrary points $O \in AB$ and $O' \in DC$,  let us call them \textit{mirroring} points if $lab(AO)= lab(DO')$. \index{mirroring points}
	\begin{lemma}
	\label{lemma_about_opposite_points_in_slender_conjugacy_diagrams}
	  If $(U, V)$-conjugacy diagram $\Delta$ has two different pairs	 of mirroring points $(O_1, O'_1)$ and $(O_2, O'_2)$ such that in $Proj(\Delta)$,  $O_1$ is joined to $O'_1$ by a path $p_1$ and $O_2$ is joined to $O'_2$ by a path $p_2$ such that $lab(p_1) \equiv lab(p_2)$, then $\Delta$ is not slender.
	\end{lemma}
	\begin{proof}
	Indeed, if the statement of Lemma \ref{lemma_about_opposite_points_in_slender_conjugacy_diagrams} holds, then we can remove the subdiagram in $\Delta$ bounded between $O_1, O_2, O_2'$ and $O_1'$ and obtain a new  diagram $\Delta'$ with $\partial \Delta'=A'B'C'D'$, where $A'B'$ is shorter than $AB$. This procedure is depicted in Figure \ref{fig:  surgery on conjugacy diagrams}. Since the boundary label of the newly obtained diagram $\Delta'$ represents the trivial element of $G$, by van Kampen's lemma, there exists a disk diagram over $\langle X \mid \mathcal{O} \rangle$ with boundary of $\Delta'$. Since $lab(A'B')=lab(D'C')$, then, in fact, the new disk diagram is a $(U, V)$-conjugacy diagram over $\langle X \mid \mathcal{O} \rangle$ as well. Finally, since the length of $A'B'$ is strictly shorter than the length of $AB$, by definition, $\Delta$ is not a slender diagram.
	\begin{figure}[H]
				\centering
				\includegraphics[clip, trim=-0cm 15.9cm 0cm 7.1cm, width=1.13\textwidth]{{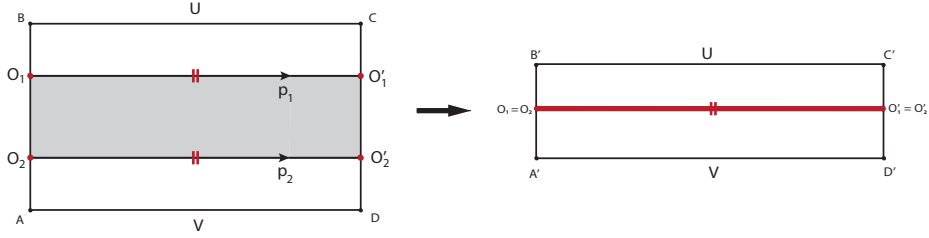}} 
				\caption{The left diagram is $Proj(\Delta)$, $lab(p_1)=lab(p_2)$. The right diagram is $\Delta'$, which is obtained after making a surgery on $\Delta$ to remove the colored subdiagram.} 
				\label{fig:  surgery on conjugacy diagrams}
			\end{figure}
	\end{proof}
	
   Based on Lemma \ref{lemma_about_opposite_points_in_slender_conjugacy_diagrams} and Corollary \ref{another corollary about hausdorff distance}, it is not hard to see that the following is true.  
   \begin{lemma}
   \label{lemma-side length of a slender diagram}
   	Let $U, V \in X^*$ be cyclically $(\lambda, c)$-geodesic words such that $U \sim_{conj} V$ in $H$. Suppose that $\Delta$ is a slender $(U, V)$-conjugcy diagram with the standard boundary $ABCD$. Then $\|AB\|=\|DC\| \leq \tau(|X|, \delta, \lambda, c)$, where $\tau: \mathbb{N}^4 \rightarrow \mathbb{N}$ is a computable function independent of $H$. In other words, there exist cyclic shifts $U', V' \in X^*$ of $U$ and $V$, respectively, and a word $T \in X^*$ such that $\|T\|\leq \tau(|X|, \delta, \lambda, c)$ and $U'=_H T^{-1}V'T$.
   \end{lemma}   
   In fact, Lemma \ref{lemma-side length of a slender diagram} is a slight variation of Lemma 10  in \cite{lysenok} and Proposition 3 in \cite{antolin sale}.
   
   \begin{lemma}[About conjugacy diagrams over HNN-extensions]
	Let $\Delta$ be a reduced conjugacy diagram over the presentation \eqref{HNN-pres} such that $\partial\Delta =ABCD$, $lab(AB)=lab(DC)$. Then either $\Delta$ does not contain $t$-bands joining $AB$ to $DC$ or if $B$ is such a band with its ends $e_1 \in AB$ and $e_2 \in DC$, then $(e_1)_+$ and $(e_2)_+$ are mirroring pair of points.
		\end{lemma}
	\begin{proof}
		It follows immediately from Collins' Lemma (see Lemma \ref{lemma-Collins}).
		\end{proof}
~\\
	\subsection{Cyclically slender conjugacy diagrams over quotient groups with small cancellation conditions}
	~\\	
\begin{definition}[(cyclically) $(\lambda, c, \epsilon,  \eta)$-reduced words]
\index{ $(\lambda, c, \epsilon,  \eta)$-reduced word}
\index{ $(\lambda, c, \epsilon,  \eta)$-reduced word! cyclically -}
 For $\epsilon>0$, $0< \eta \leq 1$, a cyclically reduced word $W \in X^*$ is called $(\lambda, c, \epsilon, \eta)$-reduced over the quotient $G= H / \ll \mathcal{R} \gg$ if $W$ is $(\lambda, c)$-quasi-geodesic in $\Gamma(H, X)$ and moreover, $W$ does not contain a  $(\epsilon, \eta)$-subword.	And it is called cyclically $(\lambda, c, \epsilon,  \eta)$-reduced, if all cyclic shifts of $W$ are $(\lambda, c, \epsilon,  \eta)$-reduced.
\end{definition}
For the next lemma, let $H = \langle X \rangle$ be a $\delta$-hyperbolic group with respect to the generating set $X$, and let $G= H / \ll \mathcal{R} \gg=\langle H \mid \mathcal{R} \rangle$, where $\mathcal{R}$ is a finite symmetric set of words satisfying the small cancellation condition $C'(\lambda, c, \epsilon, \mu, \rho)$ for appropriately chosen parameters  $\lambda \succ c \succ \epsilon \succ \mu \succ \rho$.
	\begin{lemma}
		\label{lemma_about_slender_conjugacy_diagrams}
 Let $U, V \in X^*$ be cyclically $(\lambda, c, \epsilon, 1-121\lambda\mu)$-reduced words. Then for any  reduced  cyclically slender $(U,V)$-conjugacy diagram $\Delta$ with $\partial\Delta=ABCD$, assuming that $\Delta$  contains an $\mathcal{R}$-cell, we get that $\Delta$ contains an essential $\mathcal{R}$-cell $\Pi$ which is connected to $AB$, $BC$, $CD$ and $DA$ by contiguity subdiagrams $\Gamma_1$, $\Gamma_2$, $\Gamma_3$ and $\Gamma_4$, respectively, and the following hold		
		\begin{enumerate}
			\item $\Gamma_2$ and $\Gamma_4$ are non-empty;
			\item $(\Pi, \Gamma_2, BC)+(\Pi, \Gamma_4, DA) \geq 1-121 \lambda \mu$; and
			\item $(\Pi, \Gamma_1, AB)$ and $ (\Pi, \Gamma_3, CD)$ are either empty or smaller than  $49 \lambda \mu $.
		\end{enumerate}
	\end{lemma}
	\begin{proof}
		Proof of Lemma \ref{lemma_about_slender_conjugacy_diagrams} is given in Appendix (see Subsection \ref{subsection-proof-of-lemma-slender-c-d}).
	\end{proof}
	
	\begin{definition}
	\label{definition-types-of-slender-diagrams}
	\index{slender conjugacy diagram! $H$-diagram}
	\index{slender conjugacy diagram! $G$-diagram}
	If $\Delta$ is a cyclically slender $(U, V)$-conjugacy diagram over the quotient $G=\langle H \mid \mathcal{R} \rangle$, then we say that $\Delta$ is a $(U, V)$-conjugacy $H$-diagram if $G$ does not contain an $\mathcal{R}$-cell, otherwise, we say that $\Delta$ is a $(U, V)$-conjugacy $G$-diagram.	
	\end{definition}

\begin{convention}
\label{convention-8-1}
	In the rest of the text for the quotient $G=H/\ll \mathcal{R} \gg$ we assume that the parameters $\lambda, c, \epsilon, \mu, \rho$ are chosen so that $1>1-122\lambda\mu>0$ and no $(\lambda, c)$-quasi-geodesic path in $\Gamma(G, X)$ contains an $(\epsilon, 1-122\lambda\mu)$-arc (with respect to $G=H/\ll \mathcal{R} \gg$).
	
	Note that we can make this assumptions without loss of generality because of Lemma \ref{lemma on Greendlinger's cell}.
\end{convention}

		  
	\subsection{An application of Lemma \ref{lemma_about_slender_conjugacy_diagrams}}
	 Lemma \ref{lemma_about_slender_conjugacy_diagrams} together with Lemma \ref{lemma_about_contiguity_arcs_with_the_same_label} implies the following.
	\begin{lemma}
	\label{lemma_about_proper_powers_in_quotients}
		Let $H=\langle X \rangle$ be a torsion-free non-elementary hyperbolic group and $G= H / \ll \mathcal{R} \gg$ satisfies the $C'(\lambda, c, \epsilon, \mu, \rho)$-condition for sparse enough parameters  $\lambda \succ c \succ \epsilon \succ \mu \succ \rho$. 
\begin{enumerate}[label=(\roman*)]
\item For any $U \in X^*$ and $k\in \mathbb{N}$ such that $U$ is a cyclically minimal word in $\Gamma(G, X)$, $U^k$ does not contain a $(\epsilon, 1-122\lambda\mu)$-subword with respect to the quotient $G= H / \ll \mathcal{R} \gg$.\\

	\item Suppose that $U, W \in X^*$  are such that $U=_G W^k$  for some $k\geq 2$ and 
		 		\begin{align}
		\label{inequality_aaaaa}
			\|U\|< \frac{\mu\rho-c}{\lambda}-2\epsilon.
		\end{align}		
		Then $U=_H W^k$. In particular, if $U \notin E(W)$ in $G$, then $U \notin E(W)$ in $H$.
		\end{enumerate}	
		\end{lemma}
	\begin{proof}
		
		(i). Let $U$ be as in the statement of the lemma. By contradiction, assume that $U$ contains a $(\epsilon, 1-122\lambda \mu)$-subword $V$.
		
		For sparse enough standard parameters $\lambda, c, \mu, \rho$, by Lemma  \ref{lem 1.11}, $U^k$ is a $(\lambda, c)$-quasi-geodesic word in $\Gamma(G, X)$.
	Then, since by the assumption $U$ is cyclically minimal in $\Gamma \big(G, X \big)$, by Convention \ref{convention-8-1}, we get that $V$ is not a subword of a cyclic shift of $U$. 
This means that $V$ is of the form
		\begin{align*}
			lab(\hat{q}_{\Gamma_2}) = (U')^n Q,
		\end{align*}
		where $U'$ is a cyclic shift of $U$, $n \geq 1$, and $Q$ is a prefix of $U'$.
		
		By the definition of $(\epsilon, 1-122\lambda \mu)$-subwords, in the Cayley graph $\Gamma(G, X)$, there exist paths $p$ and $q$ such that $lab(p)$ is a subword of a word $R$ from $\mathcal{R}$, $lab(q)=U^k$ and $d(p_-, q_-), d(p_+, q_+) \leq \epsilon$. 
		
		Note that, by Corollary \ref{corollary on hausdorff distance between quasi-geodesics}, the Hausdorff distance between $p$ and $q$ is bounded from above by $\epsilon+2R_{\lambda, c}+2\delta <^{\text{by LPP}} 2\epsilon$.\\ 

		Now let us separately consider the cases when $n=1$ and when $n>1$.\\
		~\\
		\textit{Case 1:}($n=1$). In this case, let us partition $q =  q_{1} q_{2} q_{3}$, where $lab(q_{1})=lab(q_{3}) = Q$.  Let us also partition $p=  p_{1} p_{2} p_{3}$ such that $(p_{1})_+$ and $(p_{2})_+$ are the closest points on $p_{\Gamma_2}$ correspondingly to $(q_{1})_+$ and to $(q_{2})_+$. Since $lab(q_{1})=lab(q_{3})$ and $d_{Haus}(p, q)<2\epsilon$, from  Lemma \ref{lemma_about_contiguity_arcs_with_the_same_label} it follows that $\|p_{1}\|, \|p_{ 3}\| \leq 2\mu\|R\| < \mu \lambda \|R\|$. Then, combining this with (\ref{aineq ****}), we get that $\|p_{1} p_{2}\| > 1-23\lambda\mu$. But, since $\|U'\|= \big\|p_{1} p_{2} \big\|$ and $U'$ is a geodesic word in $\Gamma \big(G, X \big)$, by Convention \ref{convention-8-1}, we get a contradiction. 
		  Thus we are done with the case $n=1$.\\
		~\\
		\textit{Case 2:}($n>1$). In this case, again we partition $q$ into three parts $q =  q_{1}q_{2}q_{3}$ such that  $lab(q_{1})=lab(q_{3})$ and  $lab(q_{2})$ is a suffix of $U'$. Then, since $lab(q) = (U')^n Q $ and $n \geq 2$, we get that $\| q_{1}\|=\|q_{3}\| > \frac{1}{3} \|q_{2}\|$, hence $\| q\| < 3 \| q_{1}\|$. Also just like we showed in case $n=1$, by Lemma \ref{lemma_about_contiguity_arcs_with_the_same_label}, in this case $\| q_{1}\|, \| q_{3}\| \leq 2\mu \|R\|$ as well. Therefore, $\|q\|< 6\mu \|R\|$. But if $\rho$ and $\mu$ are chosen sparse enough, then the last inequality, combined with $\|R\| \geq \rho$, contradicts the assumption that $lab(q)$ contains a $(\epsilon, 1-122\lambda \mu)$-subword associated with $R$.\\
~\\	
(ii).	 Suppose that $U$ and $W$ are as in the statement of the lemma and we have $U=_G W^k$ for some $k\geq 2$. Also, by contradiction, assume that $U\neq_H W^k$.
	
	Let $U'\in X^*$ be a cyclically $(\lambda, c, \epsilon, 1-121\lambda\mu)$-reduced word such that  $U' \sim_{conj} U$ in $G$ (clearly such a word exists). Then there exists a word $W' \in X^*$ such that $U' =_G (W')^k$.
	
		Now, let $W'' \in X^*$ be a cyclically minimal representative of $W'$ with respect to $G$. This means that there exists $T\in X^*$ such that $W'=_{G} T W'' T^{-1}$ and $W''$ has minimal length among all such words. In particular, this means that $U'=_{G}  T(W'')^k T^{-1}$ and $W''$ is cyclically geodesic in $\Gamma \big(G, X\big)$. Note that, since $G$ is a quotient of $H$, we get that $W''$ is also cyclically geodesic in $\Gamma \big(H, X\big)$. Therefore, if $\lambda$ and $c$ are large enough, then by Lemma \ref{lem 1.11},  $(W'')^k$ is cyclically $(\lambda, c)$-quasi-geodesic in $\Gamma\big(H, X\big)$.
		
	Since $W''$ is conjugate to $W'$ in $G$ and $U'=_{G} (W')^k$, there exists a $(U', (W'')^k)$-conjugacy diagram over $G$. Hence there exists a cyclically slender $(U', (W'')^k)$-conjugacy diagram over $G$. Let $\Delta$ be such a diagram. As before, let us denote $\partial \Delta = ABCD$, where $lab(BC), lab(AD)$ are cyclic shifts of $(W'')^k$ and $U'$, respectively, and $lab(AB)=lab(DC)$ are geodesic words in $\Gamma\big(G, X\big)$. 
	
	Since $U'$ is cyclically $(\lambda, c)$-quasi-geodesic in $\Gamma(H, X)$, by Lemma \ref{lem 6.6}, $\Delta$ contains an essential $\mathcal{R}$-cell, $\Pi$. Let  $\Gamma_1$, $\Gamma_2$, $\Gamma_3$ and $\Gamma_4$ be essential $\epsilon$-contiguity subdiagrams connecting $\Pi$ to $AB$, $BC$, $CD$ and $DA$, respectively. Since we chose $\Delta$ to be cyclically slender, by Lemma \ref{lemma_about_slender_conjugacy_diagrams}, $\Gamma_2$ and $\Gamma_4$ are non-empty and 
	\begin{equation}
	\label{aiiiii}
		(\Pi, \Gamma_2, BC)+ (\Pi, \Gamma_4, DA) \geq 1-121 \lambda \mu.
	\end{equation}
		Also, by (\ref{inequality_aaaaa}), using triangle inequalities and the fact that $\partial\Pi$ is $(\lambda, c)$-quasi-geodesic, we get 
		\begin{equation*}
			\|\check{q}_{\Gamma_4}\| \leq \big(\|U\|+2\epsilon \big)\lambda+c \leq^{\text{by (\ref{inequality_aaaaa})}}  \mu \rho.
			\end{equation*}
		Therefore, 
		\begin{equation*}
			(\Pi, \Gamma_4, DA) < \frac{\mu\rho}{\rho} =\mu.
		\end{equation*}				
		Combining this with (\ref{aiiiii}), we get
		\begin{equation}
		\label{aineq ****}
			(\Pi, \Gamma_2, BC)  > (1-121 \lambda \mu) - \mu > 1-122 \lambda \mu.
		\end{equation}
		
Since $lab(BC)= (W'')^k$ and $W''$ is cyclically minimal in $\Gamma(G, X)$,  by Part (i) of the current lemma, we get a contradiction.

	\end{proof}
~\\

\section{Algorithms}
\label{section-algorithms}
This section is dedicated to the description of the core algorithms needed for the further exposition.

\textbf{ Convention:} In this section, for our purposes, it will be convenient to consider any word $W \in X^*$ as  a labeled circle $\sigma$ such that its label is equal to $W$ when we read it in the clockwise direction starting from some point on it. We denote the length of $\sigma$ by $\|\sigma\|$.
 
 Throughout this subsection we are interested  in words up to their cyclic shifts. Taken this into account, for a \textit{labeled circle} \index{labeled circle}  $\sigma$, we will say that $lab(\sigma)=W$ if its label is equal to $W$ if we read it in the clockwise direction starting from some point on it. For the circle $\sigma$ we introduce the following quasi-metrics, $\overrightarrow{d}$ and $\overleftarrow{d}$: \index{$\overrightarrow{d}$, $\overleftarrow{d}$}for any points $A, B \in \sigma$, $\overrightarrow{d}(A, B)$ is the length of the arc connecting $A$ to $B$ in the clockwise direction and $\overleftarrow{d}(A, B)=\|\sigma\|-\overrightarrow{d}(A, B)$. Also, for $\varepsilon>0$, we say that $B$ is in $\varepsilon$-neighborhood of $A$ (denoted $B \in \mathcal{N}_{\varepsilon}(A)$) \index{$\mathcal{N}_{\varepsilon}( ~)$} if either $\overrightarrow{d}(A, B)\leq \epsilon$ or  $\overleftarrow{d}(A, B)\leq \epsilon$.
 
 Oriented arc on $\sigma$ which, in a clockwise direction, starts at $A$ and ends at $B$ we denote by $[A, B]$. The length of the label of $[A, B]$ we denote by $\big\|[A, B] \big\|$.\\
 
 In the further exposition, we will regards $\mathcal{R}$-cells, defined in Section \ref{section Small cancellation conditions}, as labeled circles. Therefore, all the notations on circles are applicable for $\mathcal{R}$-cells.
 
 Also, throughout this section $H = \langle X \mid \mathcal{F} \rangle$, $|X|<\infty$, is a $\delta$-hyperbolic group with respect to the generating set $X$ given with its $(X, \delta)$-full-presentation.
 ~\\

\subsection{A few auxiliary algorithms}




Below we are going to describe a few auxiliary algorithms which will be used to construct effective algorithms for word and conjugacy problems in the limit groups over chains of type (\ref{main seq of gps}). 

~\\
~\\
\underline{Algorithm \texttt{$(\lambda, c)$-smoothing}.} Let us assume that $\sigma$ is a labeled circle with a label from $X^*$, and let $A_1, \ldots, A_K$ be points on $\sigma$ such that the labels of the arcs $[A_1, A_2]$, $[A_2, A_{3}]$, \ldots $[A_K, A_1]$ are  $(8\delta+1)$-local geodesics in $\Gamma(H, X)$. Then we call the points \textit{$A_1, \ldots, A_K$  $(\lambda, c)$-break-points} of $\sigma$. \index{$(\lambda, c)$-break-points}

Below we describe an algorithm which on input receives $\sigma$ along with the break points $A_1, \ldots, A_K$ and outputs another labeled circle $\sigma'$ such that $lab(\sigma')$ is $(8\delta+1)$-local geodesic  in $\Gamma(H, X)$ and $lab(\sigma') \sim_{conj} lab(\sigma)$ in $H$.

 First, suppose that $lab(\sigma)=W_0=_H h_0$, and suppose that the break points $A_1, \ldots, A_K$ are recorded in a list which may change after each step of the following described procedure.\\
~\\
\underline{Step 1.} Chose a break point $A$ on the circle $\sigma$ and search for  a pair of points $B, B' \in \mathcal{N}_{8\delta+1}$ such that $\overrightarrow{d}(B, B')=8\delta+1$ but $[B, B']$ is not a geodesic arc (i.e. $lab([B, B'])$ is not a geodesic word). If such a pair is found, then go to Step 1.1, otherwise, go to Step 1.2.\\
~\\
\underline{Step 1.1.} If $[B, B']$ is not a geodesic arc, then replace the arc $[B, B']$ of $\sigma$ with a shorter arc whose label represents the same element of $H$. As a result, we obtain a new labeled circle whose label represents an element of $H$ conjugate to $h_0$. Also, add the points $B, B'$ to the list of the break points and remove the break point which are not on the newly obtained circle from the current list of break points.\\
~\\
\underline{Step 1.2.} If such points $B, B'$ are not found, then remove $A$ from the list of break points.\\
~\\
\underline{Step 2.} Repeat the procedure of Step 1. for the next break point until there is no break point left in the list.\\
~\\
\underline{Step 3.} If there is no break point left, then return the current circle.\\
~\\

Clearly the procedure described in Step 1 and Step 2 will eventually halt, since after each call of Step 1 either the newly obtained circle gets shorter (Step 1.1.) or the number of break point in the list decreases.
Also, it is clear that the newly obtained labeled circle $\sigma'$ is such that $lab(\sigma')$ is $8\delta+1$-local geodesic and $lab(\sigma) \sim_{conj} lab(\sigma')$ in $H$.

The following observation will be used for the main algorithm.\\
\textit{Observation 8.1}. Suppose that the $A_1, \ldots, A_K\in \sigma$ are enumerated in the clockwise direction and $[A_1, A_K]$ is marked with some number, say, with $0$. Suppose also that after replacing an arc of the current circle with a new arc on Step 1.1, we mark the edges of the new arc with the label $0$. If $\sigma'$ is the returned labeled circle of\texttt{$(\lambda, c)$-smoothing} algorithm, then, clearly, the edges of $\sigma'$ which are marked with $0$ compose a connected arc.\\

Now, suppose that $\|\sigma\|-\|\sigma'\|=d$, $d\geq 0$. Then, note that Step 1 was called during this procedure not more than $d$ times, and since after Step 1.1 the number of break points in the list increases at most by 1, while after Step 1.2 it decreases by 1, we get that Step 1.2. was called during this procedure at most $K+d$ times. Therefore, Step 1 was called during this procedure at most $d+(K+d)=K+2d$ times. Consequently, the total time required for this procedure is bounded from above by
\begin{align}
\label{complexity of smoothing}
	f_{S}(\delta, |X|) (K+d),
\end{align}
where $f_{S}: \mathbb{N} \rightarrow \mathbb{N}$ is a computable function depending only on $\delta$ and $|X|$.
After summarizing, we get to the following lemma.
\begin{lemma}
	\label{lemma-smoothing-complexity}
	For input $\sigma$, $A_1, \ldots, A_K \in \sigma$ if the output of the algorithm\\ \texttt{$(\lambda, c)$-smoothing} is shorter than $\sigma$ by $d$, then the time which \texttt{$(\lambda, c)$-smoothing} spent before halting is bounded from above by $f_{S}(\delta, |X|) (K+d)$, where $f_{S}: \mathbb{N} \rightarrow \mathbb{N}$ is a  computable function depending only on $\delta$ and $|X|$.
\end{lemma}

 As we already mentioned the procedure of Step 1 and Step 2 ends up with a cyclically $8\delta+1$-local geodesic word. However, as it is apparent from Lemma \ref{lemma_about_local_geodesicness}, for large enough constants $\lambda, c$, $8\delta+1$-local geodesiceness implies cyclically $(\lambda, c)$-quasi-geodesicness. Therefore, since our primary interest in this procedure is about obtaining cyclically $(\lambda, c)$-quasi-geodesic word conjugate to $W_0$ in $H$, we name  this algorithm \texttt{$(\lambda, c)$-smoothing} \index{algorithm!\texttt{$(\lambda, c)$-smoothing}}(with respect to the input $\sigma$ and $A_1, \ldots, A_K \in \sigma$). \\
 
 Also, note that if we consider all the points on $\sigma$ as break points, then the $(\lambda, c)$-\texttt{smoothing} algorithm becomes the well-known algorithm for finding $(8\delta+1)$-local geodesic word conjugate to the given word (see, for example, \cite{bridson}). In case all the points on $\sigma$ are regarded as break points then we call this algorithm  $(\lambda_i, c_i)$-\texttt{cyclic-reduction}. \index{$(\lambda_i, c_i)$-\texttt{cyclic-reduction}}\\

 \underline{Algorithm \texttt{ShortLex}.} (Shapiro's Algorithm on ShortLex normal forms). 
 \index{\texttt{ShortLex}}
 As in \cite{epstein holt}, for a given element $g \in H$, we define the \textit{$ShortLex_H$ normal form} \index{ShortLex normal form}  of $g$ with respect to $X$ as lexicographically the least word $W' \in X^*$ such that $W'=_H g$. Analogously, for $W \in X^*$, we denote $W'=ShortLex_H(W)$ \index{$ShortLex_H(W)$} if $W'$ is lexicographically the least word  in $X^*$ such that $W=_H W'$. 
 
 According to Shapiro's theorem described in \cite{epstein holt}, there is a linear-time algorithm which for any input $W\in X^*$ finds $W'=ShortLex_H(W)$. Moreover, as it follows from the proof of the theorem in \cite{epstein holt}, the time complexity of this procedure is bounded from above by
 \begin{align}
 \label{complexity_of_the_shortlex_form_algorithm}
 	f_{\texttt{SL}}(|X|, \delta)\|W\| ,
 \end{align}
 where $f_{\texttt{SL}}$ is a computable function independent of $H$ and $W$. We name this algorithm simply \texttt{ShortLex}.
 ~\\

\subsection{The main algorithm}
\label{subsection-the main algorithm}
 (\texttt{$(\lambda, c, \epsilon,  \eta)$-cyclic-reduction}.) \index{algorithm!\texttt{$(\lambda, c, \epsilon,  \eta)$-cyclic-reduction}}
  
  As above, let $H=\langle X \rangle$ be a $\delta$-hyperbolic group with $(X, \delta)$-full presentation $H = \langle X \mid \mathcal{F} \rangle$. Let $G= H / \ll \mathcal{R} \gg$, where $\mathcal{R}$ is a finite set of words satisfying the small cancellation condition $C'(\lambda, c, \epsilon, \mu, \rho)$ for large enough standard parameters  $\lambda \succ c \succ \epsilon \succ \mu \succ \rho$.
  Note that, as it follows from \ref{lem 6.6}, the group $G$ is a non-elementary hyperbolic group.


  Let $0<\eta<1$ be a fixed rational constant such that 
  	\begin{align}
  	\label{condition aa}
  	2\eta-3/2 > 3\lambda(1-\eta).
  	\end{align}
~\\

  
  In this subsection, our goal is to describe an algorithm (see \texttt{$(\lambda, c, \epsilon, \eta)$-cyclic-reduction}  algorithm below)  which  for an input word $W \in X^*$ (inputed as a labeled circle), outputs a word $W' \in X^*$ such that $W \sim_{conj} W'$ in $G$ and $W'$ is  cyclically $(\lambda, c, \epsilon,  \eta)$-reduced in $\Gamma(G, X)$. 
   Moreover, if $W$ and all its cyclic shifts do not contain $(\epsilon, \eta)$-subwords, then $W \sim_{conj} W'$ in $H$.
  The algorithm will be universal in the sense that it does not depend  on the choice of $H$ and $G$.
 We also would like to note that the main technical difficulties for this algorithm are connected with making it as fast as possible. 
 
  \begin{lemma}
  	\label{lemma-a1}
  	Suppose that $W \in X^*$ is $(\lambda, c)$-quasi-geodesic word in $H$ and $W'$ is a geodesic word in $H$ such that $W=_H W'$. If $W$ contains an $(\epsilon, \eta)$-subword, then $W'$ contains a $(2\epsilon, \eta)$-subword.
  \end{lemma}
  
  \begin{lemma}
  	\label{lemma-a2}
  	Let $\sigma$ be a labeled circle which contains an $(\epsilon, \eta)$-arc. Then there exist points $A, B \in \sigma$ such that $[A, B]$ contains an $(\epsilon, \eta)$-arc and $\|[A, B]\| \leq \tilde{N}$.
  \end{lemma}	
  ~\\
  \textbf{Observation 8.1.} As it follows from Lemma \ref{lemma-arcs-associated-cells-are-short}, if a cyclic shift of a word $W \in X^*$ contains an $(\epsilon, \eta)$-subword, then this subword must be associated with a word from $\mathcal{R}$ whose lengths is bounded from above by $\frac{\lambda(\|W\|+2\epsilon)+c }{\eta}$. This boservation leads us to the following definition.
  
  	   Define 
   \begin{equation*}
   	_{n}\mathcal{R} = \bigg\{ R \in \mathcal{R} \mid \|R\| \leq \frac{\lambda(\|W\|+2\epsilon)+c }{\eta} \bigg\}
   \end{equation*}
   where $n=\|W\|$, and let $\mathcal{C}(_n\mathcal{R})$ be an upper bound of time required for constructing a set of  representatives of $_n\mathcal{R}$ up to cyclic shift. 
    
   Let us denote 
   \begin{align}
   \label{definition of L and l}
   \tilde{L}_n = \lceil \lambda(\eta L_n+ 2\epsilon)+c \rceil \text{~and~} l_n= \min\{\|R\| \mid R \in _n\mathcal{R} \}.	
   \end{align}
 The following simple key observation is the main motivation for considering $\tilde{L}_n$.\\
  \textbf{Observation 8.2.} Suppose that $A_1, \ldots, A_s \sigma$ are such that $\overrightarrow{d}(A_i, A_{i+1}) \leq \tilde{L}_n$ for $1\leq i <s$ and $\overrightarrow{d}(A_s, A_{1})\leq \tilde{L}_n$. Then, if $\sigma$ contains an $(\epsilon, \eta)$-arc, there exists a point $A_i \in \{A_1, \ldots, A_s \}$ such that the $\tilde{L}_n$-neighborhood of $A_i$ contains an $(\epsilon, \eta)$-arc.
   
   
   
   ~\\
   Let $k_n=\#\big\{~_n\mathcal{R}/ \sim_{{cyclic~shifts}} \big\}$, where $\sim_{{cyclic~shifts}}$ is the equivalence relation for cyclic shifts. Now let $ ~_n\tilde{\mathcal{R}} \stackrel{def}=\{ R_1, \ldots, R_{k_n}\} \subset ~_n\mathcal{R}$ be a set of representatives of the set $~_n\mathcal{R}/ \sim_{{cyclic~shifts}}$. 
   
   Next, let us partition the elements from $ ~_n\tilde{\mathcal{R}}$ in the following way: partition the words $R_i \in _n\tilde{\mathcal{R}}$, $i=1, \ldots, k_n$, in the form
				
				\begin{equation*}
					R_i= U_i^1U_i^2\ldots U_i^{s_i},
				\end{equation*}
				where $\big\lfloor \frac{1}{1-\eta} \big\rfloor -1 < s_i \leq \big\lceil \frac{1}{1-\eta} \big\rceil$, $\|U_i^j\| = \big\lfloor ({1-\eta})\|R_i\| \big\rfloor$, for $j=1, \ldots, s_i-1$, and $\big\lfloor ({1-\eta})\|R_i\| \big\rfloor \leq \|U_i^{s_i}\| < 2\big\lfloor ({1-\eta})\|R_i\| \big\rfloor $.
				
				Now let us define
				\begin{align*}
				~_n\mathcal{R}'=\big\{\hat{U}_i^j \stackrel{\text{def}}= U_i^1\ldots U_i^{j-2} U_i^{j+1}U_i^{j+2} \ldots U_i^{s_i} \mid  (i, j), 1\leq i \leq k_n, ~&1< j < s_i \big\}\\
				 \bigcup \{ \hat{U}_i^1 \stackrel{\text{def}}= U_i^{2}U_i^{3} \ldots U_i^{s_i-1} \} \bigcup \{ \hat{U}_i^{s_i}\stackrel{\text{def}}=U_i^{1}U_i^{2} \ldots U_i^{s_i-2} \}
				\end{align*}
				and also, by using the convention $U_i^0 \stackrel{\text{def}}= U_i^{s_i}$, define
				\begin{align}
				\label{auxiliary-set-of-words}
				~_n\mathcal{R}''=\{{U}_i^{j-1}{U}_i^j \mid (i, j), 1 \leq i \leq k_n, ~&1\leq j \leq s_i \}.
				\end{align}
				
				Then we have
				\begin{equation}
				\label{inequality-8.1}
				\begin{aligned}
					(2\eta-1)\|R_i\| \leq \hat{U}_i^j \leq (3\eta-1)\|R_i\|,\\
					2(1-\eta) \|R_i\| \leq  \|{U}_i^{j-1}{U}_i^j\|\leq 3(1-\eta)\|R_i\|.
				\end{aligned}
				\end{equation}
	~\\			
		One of the motivations for considering the sets $~_n\mathcal{R}'$ and $~_n\mathcal{R}''$ is revealed in the following proposition.
		\begin{proposition}
		\label{proposition - aux 1}
			Let $W\in X^*$ be a word containing a $(\epsilon, \eta)$-subword $V$. Suppose $n=\|W\|$. Then $V$ contains a subword $V'$ of the form 
		\begin{align*}
			V'=_H E_1^{-1} U'E_2,
		\end{align*}
		where $U' \in~_n\mathcal{R}'$ and $E_1, E_2 \in X^*$, $\|E_1\|, \|E_2\| \leq 2\epsilon$. Hence $V'$ is a $(2\epsilon, 3\eta-2)$-subword of $W$.
		\end{proposition}
		\begin{proof}
			By definition and Observation 8.1, there exist $T_1, T_2 \in X^*$, $\|T_1\|, \|T_2\| \leq \epsilon$ and $R  \in ~_n\mathcal{R}$ such that for some subword $U$ of $R$, $\|U\|\geq \eta \|R\|$ and
			\begin{align*}
				V=_H T_1^{-1} UT_2.
			\end{align*}
			Let $R = U \bar{U}$. Then, since $\|\bar {U} \| \leq (1-\eta)\|R\|$, there exists a word $U_i^{j-1}U_i^j$ such that $\bar{U}$ is a subword of $U_i^{j-1}U_i^j$. But this means that $\hat{U}_i^j$ is a subword of $U$. Therefore, by the inequality $\epsilon >^{\text{by LPP}} 2R_{\lambda, c}+2\delta$ and by Corollary \ref{corollary on hausdorff distance between quasi-geodesics}, we get that there exists a subword $V'$ of $V$ and $E_1, E_2 \in X^*$, $\|E_1\|, \|E_2\| \leq 2\epsilon$, such that $V'=_H E_1^{-1} U'E_2$. Finally, since we have $\| \hat{U}_i^j \| >\|R\|-3\lfloor (1-\eta) \|R\| \rfloor \geq (3\eta -2)\|R\|$, we get that $V'$ is a  $(2\epsilon, 3\eta-2)$-subword of $W$.
			
		\end{proof}

		Now let us define
		\begin{align*}
		\mathcal{E}_0(~_n\mathcal{R}')=\{ShortLex (T_1^{-1} \hat{U}_i^j T_2) \mid (i, j), 1\leq i \leq k_n, ~&1\leq j \leq s_i,\\
		& T_1, T_2 \in X^*,  \|T_1\|, \|T_2\| \leq 3\epsilon \}.
		\end{align*}

	\begin{lemma}
		\label{lemma-epsilon_0}
		The sum of the lengths of the words from $\mathcal{E}_0(~_n\mathcal{R}')$ is  bounded from above by $f_{\mathcal{E}_0}(|X|, \epsilon, \eta) L_nk_n$, where $f_{\mathcal{E}_0}: \mathbb{N} \rightarrow \mathbb{N}$ is a computable function depending only on $|X|, \epsilon$ and $ \eta$.
	\end{lemma}
	\begin{proof}
	Indeed, first of all note that
	
		\begin{align*}
		\label{estimate for Epsilon_0}
			\#\mathcal{E}_0(~_n\mathcal{R}')  \leq k_n \bigg\lceil \frac{1}{2(1-\eta)} \bigg\rceil|X|^{6\epsilon} .
		\end{align*}
		Therefore,
		\begin{equation*}
\label{eq bb}
\begin{aligned}
 \sum_{W \in \mathcal{E}_0(~_n\mathcal{R}')}\|W\|
&\leq (\eta'L_n+6\epsilon) \#\mathcal{E}_0(~_n\mathcal{R}')
\leq   (\eta'L_n+6\epsilon)k_n \bigg\lceil \frac{1}{2(1-\eta)} \bigg\rceil|X|^{6\epsilon}\\
&< (\eta'+6\epsilon) \bigg\lceil \frac{1}{2(1-\eta)} \bigg\rceil|X|^{6\epsilon} L_nk_n.
\end{aligned}
\end{equation*}

Now define $f_{\mathcal{E}_0}(|X|, \epsilon, \eta)=(\eta'+6\epsilon) \bigg\lceil \frac{1}{2(1-\eta)} \bigg\rceil|X|^{6\epsilon}$.
	\end{proof}
~\\
The main motivation for considering the set $\mathcal{E}_0(~_n\mathcal{R}')$ is observed in the following proposition.
\begin{proposition}
\label{proposition - aux - 4}
	Suppose that $W \in X^*$ is a $(\lambda, c)$-quasi-geodesic word in $\Gamma(H, X)$ containing an $(\epsilon, \eta)$-subword and $\|W\|\leq n$. Then the word $W'=ShortLex_H(W)$ contains a subword from  $\mathcal{E}_0(~_n\mathcal{R}')$ which is also a $(3\epsilon, \eta')$-subword.
\end{proposition}
\begin{proof}
First of all, note that, by Proposition \ref{proposition - aux 1}, $W$ contains a $(2\epsilon, \eta')$-subword, say, $V$.

	Let us consider a bigon in $\Gamma(H, X)$ with boundary $pq^{-1}$ such that $lab(p)=W$ and $lab(q)=W'$. Also, let $p_1$ be a subpath on $p$ such that $lab(p_1)=V$. By Corollary \ref{corollary on hausdorff distance between quasi-geodesics}, the Hausdorff distance between $p$ and $q$ is bounded from above by $2\delta+R_{\lambda, c}$. Therefore, since $\epsilon>2\delta+R_{\lambda, c}$, we get that there is a subpath $q_1$ on $q$ such that $d((p_1)_-, (q_1)_-), d((p_1)_+, (q_1)_+) < \epsilon$, which implies that $lab(q_1)$ is a $(3\epsilon, \eta')$-subword of $W'$.
	
	Now, since all subword of a word in ShortLex form are also in ShortLex form, combining with Observation 8.1, we get that $lab(q_1) \in \mathcal{E}_0(~_n\mathcal{R}')$.
\end{proof}
~\\
~\\
Now let us describe the \texttt{$(\lambda, c, \epsilon, \eta)$-cyclic-reduction}  algorithm.\\
\subsubsection{Description of \texttt{$(\lambda, c, \epsilon, \eta)$-cyclic-reduction}}
\label{subsubsection-main-algorithm}
~\\
\underline{Input/Output.} As an input the algorithm receives a labeled circle $\sigma$ with $lab(W)\in X^*$ and outputs a word $W'$ such that $W' \sim_{conj} W$ in $G$ and $W'$ is cyclically $(\lambda, c, \epsilon, \eta)$-reduced. Let $\|W\|=n$.\\
~\\
\underline{Step 0.} Compute $W_0$ such that $W_0$ is cyclically $8\delta+1$-local geodesic (hence, $W_0$ is $(\lambda, c)$-quasi-geodesic in $\Gamma(H, X)$) and $W_0 \sim_{conj} W$ in $H$.

Let $\sigma_0$ be a labeled circle such that $lab(\sigma_0)=W_0$.\\
~\\
\underline{Step 1.}
		If $\|\sigma_0\| \geq 2\tilde{L}_n$, then  partition $\sigma_0$ by points $A_1, A_2, \ldots, A_s \in \sigma_0$ such that $\overrightarrow{d}(A_i, A_{i+1})=\tilde{L}_n$ for $1\leq i < s$ and $\overrightarrow{d}(A_s, A_1) \leq \tilde{L}_n$. Then $s=\lfloor \frac{\|\sigma_0\|}{\tilde{L}_m} \rfloor +1$.\\
		Otherwise, if  $\|\sigma_0\| \leq 2\tilde{L}_n$ take $A_1 \in \sigma_0$ arbitrarily and define $A_2\in \sigma_0$ as the opposite to $A_1$ point on $\sigma_0$ in the sense that $\overrightarrow{d}(A_1, A)=\overrightarrow{d}(A, A_1)\pm 1$.
		~\\
		Include the elements $A_1, A_2, \ldots, A_s$ in a list of special points which we simply call \texttt{List}.\\
~\\
Let us save the value of $\sigma_0$ in a special variable $\sigma'$ which is by itself a labeled circle.\\

Now for all elements $A \in \texttt{List}$ do the procedure of Step 2 as follows.\\
~\\
\underline{Step 2.}  If $\|\sigma'\| < 2\tilde{L}_n$, then consider the points $B_1, B_2 \in \sigma'$ such that $B_1=B_2$ and $\overrightarrow{d}(B_1, A)=\overrightarrow{d}(A, B_1)\pm 1$ (thus $B_1$ is the opposite vertex point of $A$ on $\sigma'$). Otherwise, if $\|\sigma'\| \geq 2\tilde{L}_n$ choose $B_1, B_2 \in \sigma'$ such that $\overrightarrow{d}(B_1, A) =\overrightarrow{d}(A, B_2)=\tilde{L}_n$. Then go to Step 2.1 as follows.\\
~\\
\underline{Step 2.1.}  
Compute $W_A \stackrel{\text{def}} = ShortLex (lab[B_1, B_2])$ and go to Step 2.2.\\
~\\
\underline{Step 2.2.}  
Search for a subword from $\mathcal{E}_0(~_n\mathcal{R}')$ in $W_A$ using Aho-Corasick's string search algorithm. (A formal description of Step 2.2 via pseudo-code is given in Algorithm \ref{algorithm_Aho_Corasick}).\\ 

 If such a subword is not found, then conclude that  $[B_1, B_2]$ does not contain a $(\epsilon, \eta)$-subword and go to Step 2.2.1 as follows, otherwise go to Step 2.2.2.\\
~\\
\underline{Step 2.2.1.} Remove $A$ from \texttt{List}. Then, if \texttt{List} is not empty, choose another point from \texttt{List} and return to Step 2 with the chosen point as the input. Otherwise, return $lab(\sigma')$ and halt.\\
~\\
\underline{Step 2.2.2.} Suppose that $W_A$ contains a subword from $\mathcal{E}_0(~_n\mathcal{R}')$ of the form\\ $ShortLex (T_1^{-1} \hat{U}_i^j T_2) $. Then, 
\begin{enumerate}
	\item In $W_A$ replace the subword $ShortLex (T_1^{-1} \hat{U}_i^j T_2) $ with the word $T_1^{-1} U_i^{j-1}U_i^j T_2$. Denote the new word by $W'_A$,\\
	\item Compute $W''_A \stackrel{\text{def}}= ShortLex(W_A')$ and replace the arc $[B_1, B_2]$ of $\sigma'$ with a new arc $[B_1', B_2']$ such that $lab([B_1', B_2'])=W''_A$,\\
	\item Change the value of $\sigma'$ by prescribing to it the newly obtained labeled circle,\\
	\item Add the points $B_1'$ and $B_2'$ to \texttt{List},\\
	\item If the point $A$ was the $i$-th point which was checked in Step 2, then mark the arc $[B_1', B_2']$ of $\sigma'$ with  $i$,\\
	\item Go to Step 2.2.3 as follows.
\end{enumerate}
~\\
\underline{Step 2.2.3.} Apply the \texttt{$(\lambda, c)$-Smoothing} algorithm with inputs $lab(\sigma'), B_1', B_2'$ and then mark all the newly obtained edges during the process of running\\ \texttt{$(\lambda, c)$-Smoothing}$(lab(\sigma'), B_1', B_2')$  with $i$. Save the newly obtained labeled circle again in the variable $\sigma'$.\\
~\\
\underline{Step 2.2.4.} Suppose that the new labeled circle $\sigma'$, obtained after Steps 2.2.2 and  2.2.3, has an arc marked with $i$ which is bounded between some points $O, O'\in \sigma'$ (the fact that the edges marked with $i$ form an arc follows from Observation 8.1). Then partition the arc $[O, O']$ with the points $O_1, \ldots, O_t \in [O, O']$ such that 
$O_1=O$, $O_2=O'$ and for $1\leq i <t$, $\overrightarrow{d}(O_i, O_{i+1}) = \tilde{L}_n$ and $\overrightarrow{d}(O_{t-1}, O_{t}) \leq \tilde{L}_n$. 

Add the points $O_1, \ldots, O_t$ to \texttt{List} and then choose another point from \texttt{List} and go to Step 2 with the chosen point as the input.\\
~\\
\begin{figure}[H]
					\centering
					\includegraphics[clip, trim=0.2cm 9cm .73cm 11.5cm, width=1\textwidth]{{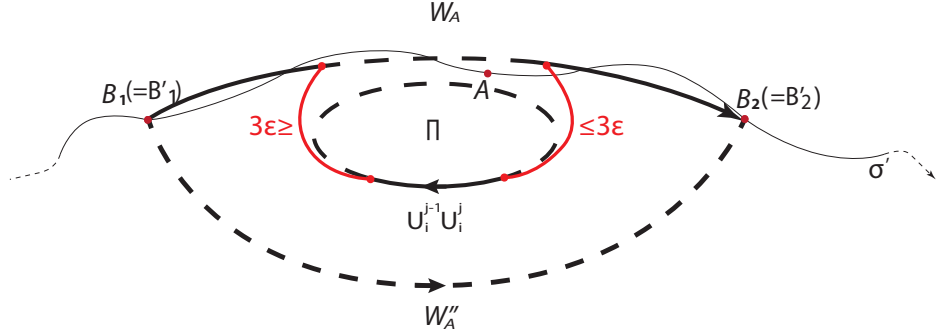}} 
					\caption{Graphical explanation of Steps 2-2.2.4 of \texttt{$(\lambda, c, \epsilon, \eta)$-cyclic-reduction} algorithm in case when the result of search in Step 2.2 is positive. In the figure $W_A$ is the ShortLex for of the label of the arc $[B_1, B_2]$ which gets replaced with a shorter arc labeled with label $W_A''$ as in Step 2.2.2.}

					\label{fig:  for break points}
				\end{figure}
Below we give a more formal description of Step 2.2 based on Aho-Corasick's famous  algorithm (see \cite{aho corasick} for the description of the algorithm) for multiple string search.

\begin{algorithm}
\caption{Searching for $(3\epsilon, \eta')$-arcs}
\label{algorithm_Aho_Corasick}
\begin{algorithmic}[1]
\State $\textbf{Input:} ~W\in X^*, ~\mathcal{E}_0(~_n\mathcal{R}')$. \Comment{\textit{$W$ is given in its $ShortLex_H$ form.}}
\State $\textbf{Output:}\text{~An $(3\epsilon, \eta')$-subword of $W$ if such a subword exists and $0$ otherwise} .$
\State Apply Aho-Corasick's string searching algorithm to find all subwords of $W$ from $~\mathcal{E}_0(~_n\mathcal{R}')$ \Comment{\textit{For the description of Aho-Corasick's algorithmic see} \cite{aho corasick}.}
\If {at least one such subword is found} 
\State \textbf{return } one of the found subwords \Comment{\textit{By  definition, this subword will be a $(3\epsilon, \eta')$-subword.} }
\Else
\State \textbf{return } $0$
\EndIf
\end{algorithmic}
\end{algorithm}				
~\\
\subsubsection{Time complexity of  \texttt{$(\lambda, c, \epsilon, \eta)$-cyclic-reduction}}

First, we will estimate the time that \texttt{$(\lambda, c, \epsilon, \eta)$-cyclic-reduction} algorithm spends on Step 2.
To this end suppose $\sigma_1, \sigma_2, \ldots, \sigma_s$ are the circles which Step 2 outputs in cases when it changes the input circle (i.e. when on Step 2.2 algorithm finds a subword from $\mathcal{E}_0(~_n\mathcal{R}')$). According to Step 2.2.2 (5), this means that for each $1\leq i \leq s$, $\sigma_i$ contains an arc whose edges are marked by $i$ and $\sigma_i$ does not contain edges marked with $i+1$. Let us denote this arc by $q_i$ and suppose that $q_i$ replaced an arc $p_i$ of $\sigma_{i-1}$.


Note that on Step 2.2.2 the algorithm replaces an arc $[B_1, B_2]$ with an arc $[B_1', B_2']$. For sparse enough standard parameters $\lambda, c, \epsilon, \mu, \rho$ we have $\|[B_1', B_2']\| < \|[B_1, B_2]\|$. Therefore, after each run of Steps 2-2.2.2 either the number of elements in \texttt{List} is decreasing or the newly obtained circle is shorter and hence the algorithm eventually halts. On the other hand, it follows from Observation 8.2 and Proposition \ref{proposition - aux - 4} that the output circle of \texttt{$(\lambda, c, \epsilon, \eta)$-cyclic-reduction} does not contain an $(\epsilon, \eta)$-arc.
\begin{lemma}
\label{shortening-coefficient}
	$\frac{\|[B'_1, B'_2]\|}{\|[B_1, B_2]\|} \leq \Lambda_0<1$, where $\Lambda_0=\Lambda_0(\lambda, c, \epsilon, \eta, \frac{L_n}{l_n})$ is a constant depending only on $\lambda, c, \epsilon, \eta$ and $ \frac{L_n}{l_n}$.
\end{lemma}
\begin{proof}
Indeed, direct computations show
\begin{align*}
	\|[B'_1, B'_2]\| &\leq \|[B_1, B_2]\| - |T_1^{-1} \hat{U}_i^j T_2| +\|T_1^{-1} {U}_i^{j-1}{U}_i^j T_2\|\\
	&\leq^{\text{we use \eqref{inequality-8.1}}} 2\tilde{L}_n - \frac{(2\eta-1)\|R_i\|}{\lambda}+12\epsilon+3(1-\eta)\|R_i\|\\
	&\leq^{\text{by LPP}} 2\tilde{L}_n-\frac{(2\eta-3/2+3\lambda(1-\eta))\|R_i\|}{\lambda} <^{\text{by \eqref{condition aa}}}2\tilde{L}_n.
\end{align*}
Therefore, we get that $\frac{\|[B'_1, B'_2]\|}{\|[B_1, B_2]\|}$ is of the forms described in the statement of the lemma.
\end{proof}
\begin{lemma}
\label{shortening-coefficient-2}
	For each $1\leq i \leq s$, $\frac{\|q_i\|}{\|p_i\|} \leq \Lambda<1$, where $\Lambda=\Lambda(\lambda, c, \epsilon, \eta, \frac{L_n}{l_n})$ is a constant depending only on $\lambda, c, \epsilon, \eta$ and $ \frac{L_n}{l_n}$.
\end{lemma}
\begin{proof}
	$q_i$ is obtained from $p_i$ after replacing arcs of $p_i$ after applying Step 2.2.2. (5) and replacing arcs of  lengths $8\delta+1$ with shorter arcs after applying Step 2.2.3.
	Therefore, taken into account Lemma \ref{shortening-coefficient}, $\Lambda$ can be taken as $\Lambda=\min\{\Lambda_0, 8\delta/(8\delta+1) \}$, where $\Lambda_0$ is defined as in Lemma \ref{shortening-coefficient}.
\end{proof}
\begin{corollary}
\label{a main corollary}
	$\sum_{i=1}^s \|q_i\| \leq \frac{\Lambda}{1-\Lambda}n$.
\end{corollary}
\begin{proof}
	Indeed, since by Lemma \ref{shortening-coefficient-2}, $q_i$, $1\leq i \leq s$ are obtained by replacing an arc $p_i$ of $\sigma_{i-1}$ of lengths at least $\|q_i\|/\Lambda$, we get that 
	\begin{align*}
		\sum_{i=1}^s \|q_i\|\leq \Lambda n+ \Lambda^2 n + \ldots =  \frac{\Lambda}{1-\Lambda}n.
	\end{align*}
\end{proof}
\begin{lemma}
\label{a main lemma}	
	During the run of \texttt{$(\lambda, c, \epsilon, \eta)$-cyclic-reduction} algorithm, the total number of points added to \texttt{List} (See Steps 1, 2.2.2 (5) and 2.2.4) is bounded from above by
	\begin{align*}
		\frac{2}{(1-\Lambda)\tilde{L}_n}n.
	\end{align*}
	Hence Step 2 of \texttt{$(\lambda, c, \epsilon, \eta)$-cyclic-reduction} algorithm is being called not more than $\frac{2}{(1-\Lambda)\tilde{L}_n}n$ times for input of lengths $n$.
\end{lemma}
\begin{proof}
It follows directly from the description of \texttt{$(\lambda, c, \epsilon, \eta)$-cyclic-reduction} algorithm and Corollary \ref{a main corollary}.
\end{proof}

\begin{corollary}
\label{cor-1}
	On Step 2.1 in summary ---- spends time bounded from above by $f_1(|X|, \lambda, c, \epsilon, \mu,  \eta, k)n^{1+\upsilon} $, where $f_1: \mathbb{N}^7 \rightarrow \mathbb{N}$ is a  computable function depending only on $\delta$ and $|X|$,  $k$ is the number of elements in $\mathcal{R}$ up to cyclic shifts,  and $\upsilon=0$ if $k=1$, otherwise $\upsilon=1$.
\end{corollary}
\begin{proof}
	It follows directly from Lemma \ref{a main lemma} and from the time complexity properties of the \texttt{ShortLex} algorithm. See \eqref{complexity_of_the_shortlex_form_algorithm}.

\end{proof}

\begin{corollary}
\label{cor-2}
	On Step 2.2  in summary the algorithm spends time bounded from above by $f_2(|X|, \lambda, c, \epsilon, \mu,  \eta, k)n^{1+\upsilon} $, where $f_2: \mathbb{N}^7 \rightarrow \mathbb{N}$ is a  computable function depending only on $\delta$ and $|X|$,  $k$ is the number of elements in $\mathcal{R}$ up to cyclic shifts,  and $\upsilon=0$ if $k=1$, otherwise $\upsilon=1$.
\end{corollary}
\begin{proof}
	Indeed, it follows from Lemma \ref{a main lemma} and from the time complexity properties of Aho-Corasick's string search algorithm. See \cite{aho corasick}.
\end{proof}
\begin{corollary}
\label{cor-3}
	On Steps 2.2.1-2.2.4  in summary the algorithm spends time bounded from above by $f_3(|X|, \lambda, c, \epsilon, \mu,  \eta, k)n^{1+\upsilon} $, where $f_3: \mathbb{N}^7 \rightarrow \mathbb{N}$ is a  computable function depending only on $\delta$ and $|X|$,  $k$ is the number of elements in $\mathcal{R}$ up to cyclic shifts,  and $\upsilon=0$ if $k=1$, otherwise $\upsilon=1$.

\end{corollary}
\begin{proof}
	Indeed,  in terms of time complexity, the hardest part among the Steps 2.2.1-2.2.4 is  Step 2.2.3, and taken this into account, the claim of the corollary follows immediately from Lemma \ref{lemma-smoothing-complexity}.   
\end{proof}
~\\
\begin{lemma}
\label{lemma_about_complexity_of_arc_reduction}
	Suppose that the above define set  $~_n\mathcal{R}'$  is already computed. Then there exists a computable function $\Psi: \mathbb{Q}^7 \rightarrow \mathbb{N}$ such that if the constants $\lambda, c, \epsilon, \mu, \rho, \eta$ are fixed and large enough, then for any word $W \in X^*$, $\|W\|=n$, a cyclic $(\lambda, c, \epsilon, \eta)$-reduction of $W$ can be computed in time bounded from above by
	\begin{equation}
	\label{eq 00}
		 \Psi(|X|, \lambda, c, \epsilon, \mu,  \eta, k) n^{1+\upsilon}	
	\end{equation}	
where  $k$ is the number of elements in $\mathcal{R}$ up to cyclic shifts,  and $\upsilon=0$ if $k=1$, otherwise $\upsilon=1$.
\end{lemma}
\begin{proof}
	Indeed, it directly follows from Corollaries \ref{cor-1}, \ref{cor-2} and \ref{cor-3}.
\end{proof}
~\\
~\\	
\section{A subclass of lacunary hyperbolic groups with effectively decidable word problem and $G$-conjugacy problem.}
\label{section_about_special_class_of_LHG}
In this section we describe a subclass of lacunary hyperbolic groups for which there is a fast algorithm solving the word problem. As it will be shown in corresponding sections, the groups which are described in theorems \ref{theorem_verbally_complete}, \ref{theorem_tarskii_monsters} and \ref{theorem_about_connecton_of_word_and_conjugacy_problems} are constructed so that they belong to that subclass.

\subsection{Small cancellation conditions in chains of hyperbolic groups}
	\label{subsection about small cancellation in chains}
	Let us consider the chain of hyperbolic groups (\ref{main seq of gps}), that is
	\begin{align}
	    \label{main seq of gps-2*}
		G_0 \stackrel{\beta_0}\hookrightarrow H_1 \stackrel{\gamma_1}\twoheadrightarrow G_1 \stackrel{\beta_1}\hookrightarrow H_2 \stackrel{\gamma_2}\twoheadrightarrow \ldots,
	\end{align}
	where $\alpha_i= \gamma_{i+1}\circ \beta_i$ is surjective for $i=1,2, \ldots$. All the groups in this chain are hyperbolic.
	
	 Suppose that for all $i  \geq 0$, $G_i$ is $\delta_i$ hyperbolic and for all $j \geq 1$, $H_i$ is $\delta'_j$ hyperbolic, where $\delta_i, \delta_j' \in \mathbb{N}$. 
	 Also suppose $G_0 = \langle X \mid \mathcal{R}_0 \rangle$ is given with its initial symmetric finite presentation and for all $i \in \mathbb{N}$
	\begin{align}
				H_i =G_{i-1}*F(Y_i)/ \ll \mathcal{S}_i\gg,
	 	\end{align}
	where $|Y_i|<\infty$, $Y_i \cap \beta_{i-1}(G_{i-1})=\emptyset$ and $\mathcal{S}_i$ is a finite symmetric set of words from $(X \cup Y_i)^*$, and 
	\begin{align}
		G_i &=H_i / \ll \mathcal{R}_i \gg,
	\end{align}
	where  $\mathcal{R}_i$ is a finite symmetric set of words from $(X \cup Y_i)^*$ as well.\\
	~\\
	~\\
	Let us denote the sequences $(\lambda_i)_{i=1}^{\infty}, (c_i)_{i=1}^{\infty}, (\epsilon_i)_{i=1}^{\infty}$, $(\mu_i)_{i=1}^{\infty}$, $(\rho_i)_{i=1}^{\infty}$ by
	$\boldsymbol{\lambda}$, $\boldsymbol{ c}$, $\boldsymbol{\epsilon}$, $\boldsymbol{ \mu}$, $\boldsymbol{\rho}$, respectively.
	~\\
	~\\
	\begin{definition}[Small cancellation conditions  $C'\big( \boldsymbol{\lambda, c, \epsilon, \mu, \rho} \big)$ and $C''\big(\boldsymbol{\lambda, c, \epsilon, \mu, \rho}\big)$\big ]
	
	\label{definition_of_the_sequential_small_cancellation_condition_C'}
		We say that the chain (\ref{main seq of gps-2*}) satisfies (alternatively, based on the context, we may say $\bar{G}=\lim_{i}(G_i, \alpha_i)$ satisfies) 
		 the
		 \index{$C'\big(\boldsymbol{\lambda, c, \epsilon, \mu, \rho}\big)$}
		\index{$C''\big(\boldsymbol{\lambda, c, \epsilon, \mu, \rho}\big)$}
		
	\begin{align*}
		C'\big(\boldsymbol{\lambda, c, \epsilon, \mu, \rho}\big)\text{-condition}
	\end{align*}
		 of small cancellation, if
		 \begin{enumerate}
		  \item[(a)] The set $\mathcal{R}_i$ satisfies the $C'(\lambda_i, c_i, \epsilon_i, \mu_i,  \rho_i)$-condition with respect to $H_i=\langle X \bigcup \cup_{j=1}^i Y_j \rangle$; \\
		  	 
		  \item[(b)] The following sequences are computable:
		  \begin{equation}
\begin{aligned}
\label{computable sequences}
(Y_i)_{i=1}^{\infty}, (\mathcal{R}_i)_{i=0}^{\infty}, (\mathcal{S}_i)_{i=1}^{\infty}& \text{~and~}\\
 (\delta_i)_{i=0}^{\infty}, (\delta'_i)_{i=1}^{\infty}, (\lambda_i)_{i=1}^{\infty}, (c_i)_{i=1}^{\infty},& (\epsilon_i)_{i=1}^{\infty}, (\mu_i)_{i=1}^{\infty}, (\rho_i)_{i=1}^{\infty}.
\end{aligned}
\end{equation}
That is there exists an algorithm which on input $i\geq1$ returns\\ $(Y_i, \mathcal{R}_i,  \mathcal{S}_i, \delta_i, \delta'_i,\lambda_i, c_i, \epsilon_i, \mu_i, \rho_i)$;\\
		   \end{enumerate}
		   If, in addition, the following condition is satisfied, then in the notations we replace $C'(\boldsymbol{\lambda, c, \epsilon, \mu, \rho})$ with $C''(\boldsymbol{\lambda, c, \epsilon, \mu, \rho})$.
		   \begin{enumerate}
		    \item[(d)]
		    For each pair $(i, j)$, $0<i<j$, and each $R_i\in \mathcal{R}_i$, $R_j \in \mathcal{R}_j$, there are no subwords $U_i$ and $U_j$ of $R_i$ and $R_j$, respectively, such that $\|U_i\| \geq \mu_i \|R_i\|$ and there exist $v_1, v_2 \in (X\cup Y_i)^*$, $\|v_1\|, \|v_2\|\leq \epsilon_i$, such that $v_1U_iv_2U_j=_{H_{i} }1$.
		    \end{enumerate}
\end{definition}

\begin{definition}[Rank of contiguity subdiagrams]
\index{rank of contiguity subdiagrams}
We say that a diagram over \eqref{main seq of gps-2*} is a $\epsilon_i$-contiguity subdiagram of rank $i$ if the diagram can be regarded as a $\epsilon_i$-contiguity subdiagram over the quotient $G_i=H_i /\ll \mathcal{R}_i \gg$.
	
\end{definition}
~\\

\subsection{An auxiliary theorem}
Now let  $G=\langle H \mid \mathcal{R} \rangle$ be fixed and suppose $\mathcal{R}$ satisfies the $C'(\lambda, c, \epsilon, \mu, \rho)$ small cancellation condition. Let $\eta=1-23\mu$. Then the following theorem holds.

\begin{theorem}
\label{theorem-main-theorem-about-wp-in-quotients}
Using the above described setting, suppose that $\lambda \succ c \succ \epsilon \succ \mu \succ \rho$ are large enough. Then
there exists a computable function $\Psi: \mathbb{Q}^6 \rightarrow \mathbb{Q}$ such that for any given $W \in X^*$, the checking $W=_{G} 1$ can be done in time bounded from above by 

	\begin{equation}
	\label{equation*******}	
		a \Psi(|X|, \lambda, c, \epsilon, \mu,  k) n^{1+\upsilon}	
\end{equation}	
where $n=\|W\|$, $k$ is the number of elements  in $\mathcal{R}$ up to cyclic shifts, and $\varepsilon=0$ when $k=1$ and $\varepsilon=1$ when $k>1$.
\end{theorem}
\begin{proof}
First of all, notice that if $\rho$ are large enough then the restrictions put on $\eta$ and $\eta'$ in the beginning of Subsection \ref{subsection-the main algorithm} are satisfied if $\eta$ is defined as $\eta=1-23\mu$.
 
    For the given word $W \in X^*$, let $\sigma$ be a labeled circle such that $lab(\sigma)=W$. Let $\sigma'$ be the output of the $(\lambda, c, \epsilon, \eta)$-cyclic-reduction algorithm (i.e. Algorithm \ref{algorithm-2}), and let $lab(\sigma')=W'$ for $W' \in X^*$.
    We claim that $W'$ is empty if and only if $W=_{G} 1$.
    
    First of all, since $W'\sim_{conj} W$ in $G$, the emptiness of $W'$  would imply that $W=_{G} 1$. Now let us prove the opposite.
    Suppose that $W'$ is not empty. Then $W'\neq_{H} 1$, because $W'$ is $8\delta+1$-local geodesic word in $\Gamma(H, X)$ and the only 
    $8\delta+1$-geodesic word in $\Gamma(H, X)$ which represents the trivial element of $H$ is the empty word.
    
    On the other hand, since $W'$ is a $(\lambda, c)$-quasi-geodesic word in $\Gamma(H, X)$, by Lemma \ref{lem 6.6}, the equation $W'=_{G} 1$ would imply that $W'$ contains a $(\epsilon, 1-23\mu)$-subword. But since $W'$ is an output of the $(\lambda, c, \epsilon, \eta)$-cyclic-reduction algorithm, this cannot happen. A contradiction. Therefore, it must be that $W'$ is empty if and only if $W=_{G} 1$.
    
    Now the complexity formula in the statement of the lemma directly follows from Lemma \ref{lemma_about_complexity_of_arc_reduction}.
\end{proof}

\begin{remark}
	\label{remark-about-wp-in-quotients}
In the settings of the current subsection, for any word $W \in X^*$, $W=_{G} 1$ if and only if  the 	$(\lambda, c, \epsilon, \eta)$-cyclic-reduction algorithm (i.e. Algorithm \ref{algorithm-2}) returns the empty word for input $W$, as it was shown in the proof of Theorem \ref{theorem-main-theorem-about-wp-in-quotients}.
\end{remark}

\subsection{The definition of the subclass}
\label{subsection_the-definition-of-the-subclass}

Let us return to the chain of hyperbolic groups given by \eqref{main seq of gps} in the introduction.  That is the chain of hyperbolic group homomorphisms:
	\begin{align}
	    \label{main seq of gps***}
		G_0 \stackrel{\beta_0}\hookrightarrow H_1 \stackrel{\gamma_1}\twoheadrightarrow G_1 \stackrel{\beta_1}\hookrightarrow H_2 \stackrel{\gamma_2}\twoheadrightarrow \ldots,
	\end{align}
	where $\alpha_i= \gamma_{i+1}\circ \beta_i$ is surjective for $i=1,2, \ldots$.  
	Recall that in Subsection \ref{subsection about small cancellation in chains} we described additional settings for this chain. Namely,
	for all integers $i$, $i  \geq 0$, $G_i$ is $\delta_i$ hyperbolic and for all $j \geq 1$, $H_i$ is $\delta'_j$ hyperbolic with respect to the generating set $X\cup\bar{Y}_i$, where $ \bar{Y}_i= \cup_{j=1}^i Y_j$ and $\delta_i, \delta_j' \in \mathbb{N}$. 
	 Also we suppose $G_0 = \langle X \mid \mathcal{R}_0 \rangle$ is given with its initial symmetric finite presentation and for all $i \in \mathbb{N}$,
	\begin{align}
				H_i =G_{i-1}*F(Y_i)/ \ll \mathcal{S}_i\gg,	
	\end{align}
	where $|Y_i|<\infty$, $Y_i \cap \beta_{i-1}(G_{i-1})=\emptyset$ and $\mathcal{S}_i$ is a finite symmetric set of words from $(X \cup Y_i)^*$, and 
	\begin{align}
		G_i &=H_i / \ll \mathcal{R}_i \gg,
	\end{align}
	where  $\mathcal{R}_i$ is a finite symmetric set of words from $(X \cup Y_i)^*$. Then
	\begin{align*}
		\bar{G}\stackrel{\text{def}}= \lim_i (G_i, \alpha_i).
	\end{align*}
In addition, we suppose that the fixed sequences   $(\lambda_i)_{i=1}^{\infty}$, $(c_i)_{i=1}^{\infty}$, $(\epsilon_i)_{i=1}^{\infty}$, $(1/\mu_i)_{i=1}^{\infty}$, $(\rho_i)_{i=1}^{\infty}$ of positive integers are such that the chain \eqref{main seq of gps***} satisfies the small cancellation condition $C'\big(\boldsymbol{\lambda, c, \epsilon, \mu, \rho}\big)$. Even more, hereafter we  will assume that for all $i\geq 1$,  the 5-tuple $(\lambda_i, c_i, \epsilon_i, 1-121\lambda_i\mu_i)$ satisfies the $SP$-relation.

Also suppose that the sequences 
\begin{equation}
\begin{aligned}
\label{computable sequences}
(Y_i)_{i=1}^{\infty}, (\mathcal{R}_i)_{i=0}^{\infty}, (\mathcal{S}_i)_{i=1}^{\infty}& \text{~and~}\\
 (\delta_i)_{i=0}^{\infty}, (\delta'_i)_{i=1}^{\infty}, (\lambda_i)_{i=1}^{\infty}, (c_i)_{i=1}^{\infty},& (\epsilon_i)_{i=1}^{\infty}, (\mu_i)_{i=1}^{\infty}, (\rho_i)_{i=1}^{\infty} 
\end{aligned}
\end{equation}
are computable sequences of integers, i.e. there exists an algorithm which on input $i\geq1$ returns $(Y_i, \mathcal{R}_i,  \mathcal{S}_i, \delta_i, \delta'_i,\lambda_i, c_i, \epsilon_i, \mu_i, \rho_i)$.\\
	~\\
	
	

Let us assume that $\Psi: \mathbb{Q}^7 \rightarrow \mathbb{Q}$ is a fixed computable function satisfying the conditions defined in Theorem \ref{theorem-main-theorem-about-wp-in-quotients} (see expression \eqref{equation*******}). Let us denote $\Psi_i=\Psi(|X \cup \bar{Y}_i|, \lambda_i, c_i, \epsilon_i, \mu_i,  k_i)$, where $k_i$ is the cardinality of $\mathcal{R}_i$ up to cyclic shifts of its elements.\\

\begin{definition}[$SP$-relation]

\label{definition-coefficients-relation}
 \index{$SP$-relation}
		Let $H=\langle X \rangle$, $|X|<\infty$, be a $\delta$-hyperbolic group, where $\delta$ is a given positive integer.  We say that the 5-tuple of positive number 
		$(\lambda, c, \epsilon, \mu, \rho)$ 
		satisfies the standard parameters relation, or briefly $SP$-relation with respect to $(H, \delta)$ if the following holds.
		\begin{enumerate}
			\item $\lambda, c, \epsilon, \mu, \rho$ with respect to $(H, \delta)$ satisfy all the restrictions and relations put on $\lambda, c, \epsilon, \mu, \rho$ for all the lemmas and theorems above (see Section \ref{section Small cancellation conditions}),
			\item  $\lambda, c, \epsilon, 1/\mu, \rho \in \mathbb{N}$, and
			\item if we define $\eta = 1-121\lambda\mu$ and $\eta'=3\eta-2$, then  $\eta$ and $\eta'$ satisfy all the restrictions put on $\eta$ and $\eta'$ in the beginning of Subsection \ref{subsection-the main algorithm}. 
		\end{enumerate}
	\end{definition}

Let $g_1, g_2, \ldots : \mathbb{R}_+ \rightarrow \mathbb{R}_+$ be a sequence of  functions such that for all $i\in \mathbb{N}$, $g_i^{-1}$ is computable and
\begin{align*}
	\bigO(g_i) \subseteq \bigO(g_j) \text{~whenever~} i>j. 
\end{align*}

\index{$\bar{\rho}_i$}
For all $i\in \mathbb{N}$, let $f_{\bar{\rho}}:\mathbb{N} \rightarrow \mathbb{N}$ be a fixed computable function such that for $\bar{\rho}_i =f_{\bar{\rho}}(i)$ 
the 5-tuple $( \lambda_i, c_i, \epsilon_i, \mu_i, \bar{\rho}_i)$ satisfies the SP-relation, 
and in addition, 
\begin{align}
\label{additional condition on bar-rho}
	\bar{\rho}_i \geq \frac{\lambda_i(g_i^{-1}(\Psi_i)+2\epsilon_i)+c_i}{1-23\mu_i}. 
\end{align}
Clearly, since $i \mapsto \frac{\lambda_i(g_i^{-1}(\Psi_i)+2\epsilon_i)+c_i}{1-23\mu_i}$ is computable, such functions $f_{\bar{\rho}}$ do exist.

Hereafter, by lowest parameter principle, we will always assume that $\rho_i \geq \bar{\rho}_i$ for all $i\in \mathbb{N}$.
~\\

Let us define \index{$\xi$, $\bar{\xi}$} $\xi, \bar{\xi}:\mathbb{N} \rightarrow \mathbb{N}$ as follows
\begin{align*}
	\bar{\xi}(i)=\frac{(1-23\mu_{i}) \bar{\rho}_{i} - c_{i} }{\lambda_{i}} - 2\epsilon_{i} \text{~and~}
	\xi(i)=\frac{(1-23\mu_{i}) \rho_{i} - c_{i} }{\lambda_{i}} - 2\epsilon_{i}.
\end{align*}
Note that, in this notations, \eqref{additional condition on bar-rho} immediately implies 
\begin{align}
\label{inequality-an-additional-inequality-001}
	g_i\big(\bar{\xi}(i) \big) \geq \Psi_i \text{~for ~} i=1,2, \ldots.
\end{align}
\begin{lemma}
\label{property 0}
	Let $W \in X^*$ and $W=_{\bar{G}} 1$, but $W\neq_{G_0} 1$. Suppose $i\geq 1$ is such that $W=_{G_i} 1$, but $W \neq_{G_{i-1}} 1$ (i.e. $i$ is the minimum index such that $W=_{G_i} 1$). Then
	\begin{align*}
	\|W\| > \xi(i).
	\end{align*}
\end{lemma}
\begin{proof}
	First, note that, since the map $\beta_i: G_{i-1} \rightarrow H_i$ is an embedding, the relation $W \neq_{G_{i-1}} 1$ implies that $W \neq_{H_{i}} 1$.
	
	 Now let $W' \in (X\cup Y_i)^*$ be the $(\lambda_i, c_i)$-cyclic-reduction of $W$ over $\Gamma(H_i, X \cup Y_i)$. Then, since $W \sim_{conj} W'$ in $H_i$ and $W \neq_{H_{i}} 1$, we get that $W' \neq_{H_{i}} 1$. Therefore, if $\Delta$ is a reduced disk diagram over $G_i$ with the boundary label $W'$ then, by Lemma \ref{lem 6.6}, $\Delta$ contains an  $\mathcal{R}_i$-cell $\Pi$ connected to $\partial\Delta$ by a $\epsilon_i$-contiguity subdiagram $\Gamma$ such that $(\Pi, \Gamma, \partial\Delta)> 1-23\mu_i> \eta_i$. Therefore, by the triangle inequality, we have
	 \begin{align*}
	 	\|W\|\geq \|W'\| \geq \|\hat{q}_{\Gamma}\| \geq \frac{\|\check{q}_{\Gamma}\|-c_i}{\lambda_i}-2\epsilon_i> \frac{\eta_i\|\Pi\|-c_i}{\lambda_i}-2\epsilon_i \geq \frac{\eta_i\rho_i-c_i}{\lambda_i}-2\epsilon_i=\xi(i).
	 \end{align*}
\end{proof}

\begin{remark}
	\label{remark-on-lemma-property 0}
	Note that, since $\xi(i)\geq \bar{\xi}(i)$, in the setting of Lemma \ref{property 0}, Lemma \ref{property 0} implies $\|W\| > \bar{\xi}(i)$.
\end{remark}
~\\

\begin{definition}[$i$-th level data] \index{$i$-th level data} \index{level data}
\label{definition-i-th-level-data}
 For any fixed $i\in \mathbb{N}$, the below described list of data we call the $i$-th level data for the chain \eqref{main seq of gps***}.
\begin{enumerate}
\item The slimness constant $\delta'_i \in \mathbb{N}$ of $\Gamma(H_i, X \cup \bar{Y}_i)$, where recall that $ \bar{Y}_i= \cup_{j=1}^i Y_j$; 
\item The $(X\cup \bar{Y}_i, \delta_i')$-full-presentation $H_i=\langle X \cup \bar{Y}_i \mid \mathcal{F}_i \rangle$ of $H_i$;

\item The constants $\delta_{i-1}$, $\delta'_i$, $\lambda_i$, $c_i$, $\epsilon_i$, $\mu_i$, $\bar{\rho}_i$;
\item $\bar{\xi}(i)$.
\end{enumerate}
\end{definition}

Note that, since the sequences \eqref{computable sequences} are computable by our assumption, there exists an algorithm which computes the $i$-the level data, i.e. there exists a (deterministic) Turing machine $\mathcal{TM}$ which, for the input $i\in \mathbb{N}$, outputs the $i$-th level data. Indeed, the computability of the data from parts (1), (2) and (4) of Definition \ref{definition-i-th-level-data} is straightforward. 

Let $\mathcal{TM}$ be a fixed deterministic Turing machine which,  for all inputs $i\geq 1$, computes the $i$-th level data for the presentation \eqref{main seq of gps***} of $\bar{G}$.
 Suppose that $\phi: \mathbb{N} \rightarrow \mathbb{N}$  is a recursive function, such that for $i\in \mathbb{N}$, $\phi(i)$ is the number of steps $\mathcal{TM}$ makes after input $i$ before it halts. Let $\Phi: \mathbb{N} \rightarrow \mathbb{N}$ be a function defined as $\Phi(i)= \sum_{j=1}^i \phi(i)$, for $i\in \mathbb{N}$.
\index{$\mathcal{I}:\mathbb{N} \rightarrow \mathbb{N}$}

Let $\mathcal{I}: \mathbb{N}\rightarrow \mathbb{N}$ be the integer valued function such  that $$\Phi(\mathcal{I}(n))\leq n < \Phi(\mathcal{I}(n)+1).$$ In other words, if we run $\mathcal{TM}$ consecutively for inputs $i=1, 2, \ldots$, then after the $n$-th step $\mathcal{I}(n)$-th level data will be computed, but $\mathcal{I}(n)+1$-th level data will not.

~\\

Now for $i, n \in \mathbb{N}$, in analogy with the set \eqref{auxiliary-set-of-words}, let us define  
\begin{align*}
	_n\mathcal{R}_i= \bigg\{ R  \mid R \in \mathcal{R}_i, \|R\| \leq \frac{\lambda_i( n+2\epsilon_i) + c_i}{1-23\mu_i} \bigg\}.
\end{align*}
The motivation behind the definition of $_n\mathcal{R}_i$ is that if a word $W \in X^*$, $\|W\|=n$, is not trivial in $H_i$ but trivial in $G_i$, then $W$ must be trivial also in the group $H_i/\ll_n\mathcal{R}_i \gg$. This follows from Lemma \ref{lem 6.6}.

Since, by our assumptions, the sequences \eqref{computable sequences} are computable, the sets $_n\mathcal{R}_i$ are computable too, i.e. there exists an algorithm which for input $(n, i)$ returns $_n\mathcal{R}_i$.  Let  $\mathcal{C}: \mathbb{N} \rightarrow \mathbb{N}$ be a (time-constructible) function such that, for some fixed Turing machine computing the words $_n\mathcal{R}_i$, $\mathcal{C}(n, i)$ is the time the machine spends after input $(n, i)$ before it halts. Define \index{$\mathcal{C}_{_n\mathcal{R}_i}$} 
\begin{align*}
\mathcal{C}_{_n\mathcal{R}_i}=\max\{\mathcal{C}(n, j) \mid 1\leq j \leq i \}.	
\end{align*}

The main theorem of this section is the following.
\begin{theorem}
\label{theorem_about_WP_in_barG_2}
If the standard parameters are sparse enough,  the word problem in $\bar{G}$ can be solved  in time   
  	\begin{equation*}
	    	  \bigO \big( \mathcal{C}_{_n\mathcal{R}_{\mathcal{I}(n)}} + g_k(n) n^{1+\upsilon} \big),
	    \end{equation*}
where $k \in \mathbb{N}$ is any positive integer, $n$ is the length of the input word from $X^*$  and $\upsilon=0$ if for all but finitely many $i\geq 1$, $\mathcal{R}_i$ contains one word up to cyclic shifts, otherwise,  $\upsilon=1$. 
\end{theorem}
\begin{proof}

For a given word $W \in X^*$, first of all, without loss of generality assume that $\|W\| \geq \xi(1)$ and $W\neq_{G_0} 1$. Now, to check whether $W=_{\bar{G}} 1$ or not, we can apply the following procedure:
\begin{enumerate}
	\item[S-1.] Run the Turing machine $\mathcal{TM}$ consecutively for inputs $k=1,2,\ldots$ and stop after exactly $\|W\|$ steps. Suppose that, as a result,  the $i_0$-th level data  is constructed, but the $(i_0+1)$-st level data is not constructed, i.e. $i_0 = \mathcal{I}(\|W\|)$;
	
	\item[S-2.] Find the maximum integer index $i_1$ from the interval $[1, i_0]$ such that $\bar{\xi}(i_1) \leq \|W\|$;
	\item[S-3.] Construct the set $_n\mathcal{R}_{i_1}$, where $n=\|W\|$;
	\item[S-4.] Run the \texttt{$(\lambda_{i_1}, c_{i_1}, \epsilon_{i_1}, 1-121\lambda_{i_1}\mu_{i_1})$-cyclic-reduction} algorithm with input circle $\sigma$ such that $lab(\sigma)=W$. Note that, in order to run this algorithm, we need the $i_1$-th level data and the set $_n\mathcal{R}_{i_1}$.
\end{enumerate}
Note that if $W=_{\bar{G}} 1$, then there is minimum $j_0\geq 1$ (recall that we assumed $W\neq_{G_0} 1$) such that $W=_{G_{j_0}} 1$. 



~\\
\textit{Claim.} $j_0 \leq i_0$.
\begin{proof}[Proof of the claim]
First of all, we have $\Phi(i_0)\leq \|W\| < \Phi(i_0+1)$. 
We have
\begin{align*}
	\xi({j_0})\stackrel{\text{by def}}{=} \frac{(1-23\mu_{j_0})\rho_{j_0}-2\lambda_{j_0}\epsilon_{j_0} - c_{j_0}}{\lambda_{j_0}} \geq^{by LPP} \Phi({j_0}).
\end{align*}
Therefore, since by Lemma \ref{property 0} we have $\|W\|>\xi({j_0})$, we get $\|W\|>\Phi({j_0})$. On the other hand, since $\|W\| < \Phi(i_0+1)$ and $\Phi$ is an increasing function, from the last inequality we get $i_0+1>j_0$. Therefore, $i_0\geq j_0$. The claim is proved.
\end{proof}
Since, by the above claim, $j_0<i_0$, and by Remark \ref{remark-on-lemma-property 0}, $\bar{\eta}(j_0) < \|W\|$, in view of the way $i_1$ was defined, we get that $j_0 \leq i_1$. Therefore, the equality  $W=_{G_{j_0}} 1$ implies  $W=_{G_{i_1}} 1$. Thus $W=_{\bar{G}} 1$ if and only if $W=_{G_{i_1}} 1$, hence
on step S-4 the $(\lambda_{i_1}, c_{i_1}, \epsilon_{i_1}, \eta_{i_1})$-cyclic-reduction algorithm returns empty word for input $W$ (see Remark \ref{remark-about-wp-in-quotients}).

Now we are in a position to show that the time complexity estimations in the statement of the theorem, in fact, are true. For that reason, first, notice that on steps S-1 and S-3 the procedure spends $\|W\|+\mathcal{C}_{\mathcal{R}^n_m} =\bigO \big( \mathcal{C}_{\mathcal{R}^n_m} + n^{1+\upsilon}g_{j_0}(n) \big)$ time. Next, since $i_0< \|W\|$ and since before the step S-2 the $i$-th level data already was constructed for $i= 1, 2, \ldots, i_0$, we get that on  step S-2 the procedure spends $\bigO(\|W\|)$ time. Finally, on step S-4, by Theorem \ref{theorem-main-theorem-about-wp-in-quotients}, the procedure spends 
$a \Psi\big(|X\cup Y_{i_1}|, \lambda_{i_1}, c_{i_1}, \epsilon_{i_1}, \mu_{i_1},  k_{i_1} \big) n^{1+\upsilon} = a \Psi_{i_1} n^{1+\upsilon}$ time, where $a$ is a constant not depending on $\bar{G}$ and $k_{i_1}$ is the number of elements in $\mathcal{R}_{i_1}$ up to cyclic shifts. Since $\|W\| > \bar{\xi}(i_1)$ and $g_{i_1}$ is increasing, by \eqref{inequality-an-additional-inequality-001}   we get $a \Psi_{i_1} n^{1+\upsilon} =\bigO \big(g_{i_1}(n) n^{1+\upsilon} \big)$. Thus we confirmed the estimations in the statement of the lemma.

\end{proof}
\begin{remark}
	Note that in Theorem \ref{theorem_about_WP_in_barG_2}, we did not put any restrictions on the relators $\mathcal{S}_i$, $i=1,2,\ldots$,  other then that they are recursively enumerable and make the groups $H_i=G_{i-1}*F(Y_i)/ \ll \mathcal{S}_i\gg$ hyperbolic. However, when instead of the word problem we consider conjugacy problem in $\bar{G}$, an analogue to the statement of Theorem \ref{theorem_about_WP_in_barG_2} no longer holds unless the sets $\mathcal{S}_i$  possess additional properties. In fact, the group $G_{\mathcal{L}}$, constructed in the proof of Theorem \ref{theorem_about_connecton_of_word_and_conjugacy_problems}, has the structural properties of the group $\bar{G}$ from Theorem \ref{theorem_about_WP_in_barG_2} but, nevertheless, the conjugacy problem is undecidable in it whenever the underlying set $\mathcal{L}$ is not recursive.
	
	However, if we restrict ourselves from the conjugacy problem to the so called $G$-conjugacy problem, then the analogue of Theorem \ref{theorem_about_WP_in_barG_2} holds as it is shown in Theorem \ref{theorem-effective-G-conjugacy-problem}.
\end{remark}
\begin{remark}
\label{an-important-remark}
Note that in the proof of Theorem \ref{theorem_about_WP_in_barG_2} we, in particular, showed that the construction of the $\mathcal{I}(n)$-th level data and the implementation of the $(\lambda_{i_1}, c_{i_1}, \epsilon_{i_1}, \eta_{i_1})$-\texttt{cyclic-reduction} algorithm, whenever $_n\mathcal{R}_{i_1}$ is not empty, can be done in time $\bigO \big( \mathcal{C}_{_n\mathcal{R}_{\mathcal{I}(n)}} + g_k(n) n^{1+\upsilon} \big)$.
\end{remark}


\begin{theorem}
\label{theorem_about_lacunary_hyperbolicity_of_barG}
If the standard parameters are sparse enough, then the group $\bar{G}$ from Theorem \ref{theorem_about_WP_in_barG_2} is lacunary hyperbolic.
\end{theorem}
\begin{proof}
First of all, the group $\bar{G}$ is an inductive limit of groups $G_i$, $i\in \mathbb{N}$, all of which are hyperbolic. More precisely, $\bar{G}$ is the inductive limit of the following sequence
\begin{align*}
	G_0  \stackrel{\alpha_0}\twoheadrightarrow G_1  \stackrel{\alpha_1}\twoheadrightarrow \ldots.
\end{align*}

   An immediate corollary of Lemma \ref{property 0} is that, for all $i \in \mathbb{N}$, the radius of $\alpha_{i}: G_{i}\rightarrow G_{i+1}$, which we denote by $r_{i}$, satisfies the following inequality
      \begin{align*}
   \xi(i+1)\stackrel{\text{by def}}{=}\frac{(1-23\mu_{i+1})\rho_{i+1} -c_{i+1}}{\lambda_{i+1}} - 2\epsilon_{i+1}<r_i.
   \end{align*}
   Combining the last inequality with the inequality $\xi(i+1) \leq \delta_i \Phi(i+1)$, we get $\delta_i \Phi(i+1) < r_i$. Therefore,
   \begin{equation*}
   	\lim_{i\rightarrow \infty} \frac{\delta_i}{r_i} \leq \lim_{i\rightarrow \infty} \frac{\delta_i}{\delta_i \Phi(i+1)}=\lim_{i\rightarrow \infty} \frac{1}{ \Phi(i+1)}=0,
   \end{equation*}

   hence, by Lemma \ref{lem lacunary hyp gps}, this means that $\bar{G}$ is lacunary hyperbolic.
\end{proof}
~\\
~\\




\subsection{$G$- and $H$- conjugacy problems in $\bar{G}$. Effectiveness of the $G$-conjugacy problem in $\bar{G}$}

The main goal of this subsection is to define the $G$-conjugacy problem for sequences of type \eqref{main seq of gps***} (see Definition \ref{definition- G- and H- conjugacy problems}) and then show that the $G$-conjugacy problem is effectively solvable when  the sequences $(\lambda_i, c_i, \epsilon_i, \mu_i, \rho_i)_{i=1}^{\infty}$ are sparse enough (see Theorem \ref{theorem-effective-G-conjugacy-problem}).\\

Let $\bar{G}$ be the group defined in Subsection \ref{subsection_the-definition-of-the-subclass} which also carries all the properties described there. 
\begin{definition} [$G$- and $H$-conjugates]
\index{$G$- and $H$-conjugates}
	\label{definition- G- and H- conjugacy}
	Let $U, V \in X^*$. Then we say that $U$ is $G$-conjugate to $V$ in $\bar{G}$ if either $U \sim_{conj} V$ in $G_0$ or there exists $i \in \mathbb{N}$ such that $U \sim_{conj} V$ in $G_i$ but $ U \not\sim_{conj} V$ in $H_i$.
	
	Analogously, if there exists $i \in \mathbb{N}$ such that $U \sim_{conj} V$ in $H_i$, but $U \not\sim_{conj} V$ in $G_{i-1}$, then we say that $U$ is $H$-conjugate to $V$ in $\bar{G}$.
\end{definition}

\begin{definition}[$G$- and $H$-conjugacy problems]
	\index{$G$- and $H$-conjugacy problems} \index{conjugacy problem! $H$-conjugacy problem} \index{conjugacy problem! $G$-conjugacy problem}
	\label{definition- G- and H- conjugacy problems} 
	For the presentation \eqref{main seq of gps***} of $\bar{G}$ the $G$-conjugacy problem  asks whether there is an algorithm which for any pair of  input words $U, V \in X^*$, decides whether $U$ is $G$-conjugate to $V$ in $\bar{G}$ or not. $H$-conjugacy problem is defined analogously.
\end{definition}
~\\

Let us define $\zeta: \mathbb{N} \rightarrow \mathbb{N}$ as
\begin{align*}
	\zeta(i)=\frac{(1-121\lambda_i \mu_i )\rho_i - 2c_i }{\lambda_i}  - 4\epsilon_i.
\end{align*}

\begin{lemma}
\label{lemma-about-structure-of-CP-in-barG}
 Suppose that the standard parameters are sparse enough, and  $U, V \in X^*$ are such that $U$ is $G$-conjugate to $V$ in $\bar{G}$. Then there exists $i\in \mathbb{N}$ such that $\zeta(i) \leq \|U\|+\|V\|$, $i \leq \mathcal{I} (\|U\|+\|V\| )$ and $U \sim_{conj} V$ in $G_i$, but $U \not\sim_{conj} V$ in $H_{i}$.
\end{lemma}
\begin{proof}
If $U\sim_{conj} V$ in $G_0$ then the statement is obvious. Now, without loss of generality assume that $U\not\sim_{conj} V$ in $G_0$. Then there exists a minimal $i\in \mathbb{N}$ such that $U \sim_{conj} V$ in $G_i$, but $U\not\sim_{conj} V$ in $H_i$.
	Suppose that $U', V' \in X^*$ are the $(\lambda_i, c_i)$-cyclic-reductions of $U$ and $V$, respectively. 
	
	First, let us show that $\zeta(i) \leq \|U\|+\|V\|$. For that purpose, let us separately consider two different cases. The first case is when at least one of $U', V'$, say $U'$, is not cyclically $(\lambda_i, c_i, \epsilon_i,  1-121\lambda_i\mu_i)$-reduced. The second case is when both $U'$ and $V'$ are cyclically $(\lambda_i, c_i, \epsilon_i,  1-121\lambda_i\mu_i)$-reduced.
	
	For the first case, by definition, some cyclic shift $U''$ of $U'$ contains a $(\epsilon_i, 1-121\lambda_i\mu_i)$-subword. Therefore, by definition and by triangle inequality,
	\begin{align}
	\label{inequality-1234}
		\|U\|+\|V\| \geq \|U''\| \geq \frac{(1-121\lambda_i \mu_i) \rho_i -c_i}{\lambda_i} -2\epsilon_i > \zeta(i).
	\end{align}
	~\\
	
	Now let us consider the second case, i.e. when both $U'$ and $V'$ are $(\lambda_i, c_i, \epsilon_i,  1-121\lambda_i\mu_i)$-reduced. In this case, there exists a reduced cyclically slender $(U', V')$-conjugacy diagram $\Delta$ over $G_i=H_i/\ll \mathcal{R}_i \gg$ which contains an $\mathcal{R}_i$-cell.
 Let $\partial \Delta = ABCD$ and $lab (BC) = U'', lab (AD) = V''$, where $U''$ and $V''$ are some cyclic shifts of $U'$ and $V'$, respectively. Then, by Lemma \ref{lemma_about_slender_conjugacy_diagrams}, there exists an essential $\mathcal{R}_i$-cell  $\Pi$ in $\Delta$ connected to $AB$, $BC$, $CD$ and $DA$ by $\Gamma_1$, $\Gamma_2$, $\Gamma_3$ and $\Gamma_4$, respectively, such that
	\begin{enumerate}
		\item $\Gamma_2$ and $\Gamma_4$ are non-empty;
		\item $(\Pi, \Gamma_2, BC)+(\Pi, \Gamma_4, DA) \geq 1-121 \lambda_i \mu_i$; and
	\end{enumerate}
	Therefore,
	\begin{equation}
		\label{inequality-12345}
	\begin{aligned}
		\|U\|&+\|V\| \geq \|U'\|+\|V'\| = \|U''\|+\|V''\|
		\geq \| 
		\hat{q}_{\Gamma_2} \| + \| \hat{q}_{\Gamma_4} \|\\
		&\geq \Bigg(\frac{(\Pi, \Gamma_2, \hat{q}_{\Gamma_2})\|\Pi\| - c_i }{\lambda_i}  - 2\epsilon_i \Bigg) + 	
		 \Bigg(\frac{(\Pi, \Gamma_4, \hat{q}_{\Gamma_4})\|\Pi\| - c_i }{\lambda_i}  - 2\epsilon_i \Bigg)\\
		 &\geq \frac{(1-121\lambda_i \mu_i )\|\Pi\| - 2c_i }{\lambda_i}  - 4\epsilon_i \geq \frac{(1-121\lambda_i \mu_i )\rho_i - 2c_i }{\lambda_i}  - 4\epsilon_i=\zeta(i).
	\end{aligned}
	\end{equation}
~\\
The conclusion from \eqref{inequality-1234} and \eqref{inequality-12345} is that if $i\in \mathbb{N}$, $U \sim_{conj} V$ in $G_i$, but $U \not\sim_{conj} V$ in $H_{i}$, then
\begin{align*}
	\|U\|+\|V\| \geq \zeta(i).
\end{align*}
~\\

Now let us show that $i \leq \mathcal{I} (\|U\|+\|V\| )$. 

From the definition of $\mathcal{I} (\|U\|+\|V\|)$ it follows that  $\Phi(\mathcal{I} (\|U\|+\|V\|)+1)> \|U\|+\|V\|$. Therefore, from the last two inequalities we get  
\begin{align*}
	\frac{\delta_{i-1}\Phi(\mathcal{I} (\|U\|+\|V\|)+1)+4\lambda_i\epsilon_i+2c_i}{1-121\lambda_i\mu_i} > \rho_i \geq^{\text{by LPP}} \frac{\delta_{i-1}\Phi(i)+4\lambda_i\epsilon_i+2c_i}{1-121\lambda_i\mu_i},
\end{align*}
which implies that $\mathcal{I} (\|U\|+\|V\|)\geq i$. Thus the lemma is proved.
	
\end{proof}

An obvious corollary from Lemma \ref{lemma-about-structure-of-CP-in-barG} is the following lemma.
\begin{lemma}
\label{lemma-1*1}
	If $U \not\sim_{conj} V$ in $G_{\mathcal{I}(n)}$, but $U \sim_{conj} V$ in $\bar{G}$, then $U$ is $H$-conjugate to $V$ in $\bar{G}$.
\end{lemma}

\begin{theorem}
\label{theorem-effective-G-conjugacy-problem}
	If the standard parameters are sparse enough and the function $f(n)\stackrel{\text{def}}{=}\mathcal{C}_{_n\mathcal{R}_{\mathcal{I}(n)}}$ is bounded by a polynomial, then the $G$-conjugacy problem in $\bar{G}$ is solvable in polynomial time.
	\end{theorem}
	\begin{proof}
		For any given words $U, V \in X^*$, by definition, $U$ being $G$-conjugate to $V$ in $\bar{G}$ means that either $U\sim_{conj} V$ in $G_0$ or  there exists $i\geq 1$ such that $U \sim_{conj} V$ in $G_i$ but $U \not\sim_{conj} V$ in $H_{i}$. If it is so, then, by Lemma \ref{lemma-1*1}, $i \leq \mathcal{I}(n)$, where $n=\|U\|+\|V\|$. 
		
		From what we said, it becomes apparent that in order to show that $U$ is $G$-conjugate to $V$ in $\bar{G}$ it is enough to check if $U\sim_{conj} V$ in $G_0$ and if it is not, then for each $1\leq i \leq \mathcal{I}(n)$ check whether 
		\begin{itemize}
			\item $U \not\sim_{conj} V$ in $H_{i}$, and
			\item $U \sim_{conj} V$ in $G_i$.
		\end{itemize}
		
		Now without loss of generality let us assume that $U \not\sim_{conj} V$ in $G_0$.
		
		Let $U', V' \in (X \cup Y_i)^*$ be cyclically $(\lambda_i, c_i)$-quasi-quasi geodesic word obtained by applying the $(\lambda_i, c_i)$-\texttt{cyclic-reduction} algorithm on $U$ and $V$, respectively. Then, since $U'$ and $V'$ are conjugate to $U$ and $V$ in $H_i$ respectively, we get that $U \sim_{conj} V$ in $H_{i}$ if and only if $U' \sim_{conj} V'$ in $H_{i}$.
		
		To check whether $U' \sim_{conj} V'$ in $H_{i}$,
	by Lemma \ref{lemma-side length of a slender diagram}, it is enough  to check for all 3-tuples $(T, U'', V'')$, where $T, U'', V'' \in (X\cup Y_{i})^*$, $U'', V''$ are some cyclic shifts of $U$, $V$ and $\|T\|\leq \tau(|X|, \delta_i', \lambda_i, c_i)$ ( where $\tau$ is defined as in Lemma \ref{lemma-side length of a slender diagram}) the equality
	\begin{align}
	\label{equality-for-H-conjugacy}
	T^{-1} U'' T =_{H_{i}} V''.
	\end{align}
	Clearly, since for large enough standard parameters, the word problem in $\bar{G}$ is decidable in polynomial time, then for large enough values of $\rho_i$ this checking can be done in polynomial time. 
	
	Now, assuming that  $U \not\sim_{conj} V$ in $H_{i}$ is already verified, in order to check whether $U \sim_{conj} V$ in $G_i$, we can apply \texttt{$(\lambda, c, \epsilon, \eta)$-cyclic-reduction} algorithm for $\eta= 1-121\lambda_i\mu_i$ to find cyclic $(\lambda_{i}, c_{i}, \epsilon_{i},  1-121\lambda_i\mu_i)$-reductions $U'$ and $V'$ of $U$ and $V$, respectively, and then check whether $U' \sim_{conj} V'$ in $H_i$ or in $G_i$. Without loss of generality assume that $U' \not\sim_{conj} V'$ in $H_i$, then, by Lemma \ref{lemma_about_slender_conjugacy_diagrams}, there exist $T_1, T_2, W \in (X \cup Y_{i})^*$ such that $\|T_1\|, \|T_2\| \leq 2\epsilon_{i}$, $W$ is a subword of a word $R \in ~_n\mathcal{R}_{i}$ of length $\|W\| \leq \lambda_i \mu_i$, and 
   \begin{align}
   \label{identity-1234}
   	(T_1 W T_2)^{-1} U'' (T_1 W T_2) =_{G_{i}} V'' 
   	   \end{align}
    for some cyclic shifts $U''$, $V''$ of $U'$ and $V'$, respectively. Therefore, in order to check whether $U' \sim_{conj} V'$ in $G_i$, it is enough to check equality \eqref{identity-1234} for all mentioned collection of words $(T_1, T_2, W, U'', V'')$. Clearly, this checking can be done in polynomial time, provided that the standard parameters are sparse enough and 
   $f(n)=\mathcal{C}_{_n\mathcal{R}_{\mathcal{I}(n)}}$ is bounded by a polynomial

\end{proof}

\subsection{The condition $C'\big(\mathcal{TM},  (g_i)_{i=1}^{\infty}, ({\rho}_i)_{i=1}^{\infty} \big)$}
\begin{definition}
\label{definition-new-small-cancellation-conditions}
\index{$C'\big(\mathcal{TM}, (g_i)_{i=1}^{\infty}$}  
\index{$C'\big(\mathcal{TM}, (g_i)_{i=1}^{\infty} \big)$}
If for fixed sequence $(g_i)_{i=1}^{\infty}$, fixed function $f_{\bar{\rho}}: \mathbb{N} \rightarrow \mathbb{N}$ and fixed Turing machine $\mathcal{TM}$ (all are defined is Subsection \ref{subsection_the-definition-of-the-subclass}), elements of the sequence $(\rho_i)_{i=1}^{\infty}$ are large enough so that  Theorem \ref{theorem_about_WP_in_barG_2} and Theorem \ref{theorem-effective-G-conjugacy-problem} hold, then we say that the presentation \eqref{main seq of gps***}  of $\bar{G}$ satisfies the condition
\begin{align*}
	C'\big(\mathcal{TM}, (g_i)_{i=1}^{\infty},  ({\rho}_i)_{i=1}^{\infty} \big).
\end{align*}
\end{definition}
\begin{property}
\label{property for rho}
	As it follows from the proof of theorems \ref{theorem_about_WP_in_barG_2} and \ref{theorem-effective-G-conjugacy-problem}, there exist a linear time computable function $f_{\rho}: \mathbb{N}^6 \rightarrow \mathbb{N}$ such that in order Theorem \ref{theorem_about_WP_in_barG_2} and Theorem \ref{theorem-effective-G-conjugacy-problem} to hold it is enough to require  $\rho_i \geq f_{\rho}(\delta'_{i}, \lambda_i, c_i, \epsilon_i, \mu_i, \Phi(i))$ for all $i\in \mathbb{N}$.
	\end{property}
\section{The general scheme for group constructions of Sections \ref{section-verbally_comlete_groups}, \ref{section-Tarskii_monsters}, \ref{section_problem_7-5} and \ref{section-MS}. }
\label{section-about-general-scheme}
The proofs of theorems \ref{theorem_verbally_complete}, \ref{theorem_tarskii_monsters} and \ref{theorem_about_connecton_of_word_and_conjugacy_problems}, given in sections \ref{section-verbally_comlete_groups}, \ref{section-Tarskii_monsters} and \ref{section_problem_7-5} respectively, are constructive and the corresponding groups constructed in these sections are inductive limits of presentations of type \eqref{main seq of gps***} satisfying the condition $$C'\big(\mathcal{TM}, (g_i)_{i=1}^{\infty},  ({\rho}_i)_{i=1}^{\infty} \big)$$ for $g_i(n)=n^{\frac{1}{i}}$. Moreover, the presentation
\begin{align}
	    \label{main seq of gps----}
		G_0 \stackrel{\beta_0}\hookrightarrow H_1 \stackrel{\gamma_1}\twoheadrightarrow G_1 \stackrel{\beta_1}\hookrightarrow H_2 \stackrel{\gamma_2}\twoheadrightarrow \ldots,
	\end{align} 
	for the corresponding constructions is such that if the group $G_i=\langle X \rangle / \ll \mathcal{R}_i \gg$ is already constructed, then the group $H_i=G_{i-1}*  F( Y_i)  / \ll \mathcal{S}_i \gg$ is defined uniformly, in the sense that the definition of $H_i$ does not depend on specific values of $i$. The standard parameters $(\delta_{i-1}, \delta_i', \lambda_i, c_i, \epsilon_i, \mu_i, \rho_i)_{i=1}^{\infty}$ are different in the corresponding constructions only in terms of their ``sparseness'', however, since we are not interested in their specific values, we will not go into the details of defining them, instead we will assume that they are sparse enough. 
	
	The groups $G_i$, $i=1,2, \ldots$ in all  those constructions will be defined uniformly as $G_i=H_i / \ll \mathcal{R}_i \gg$, where 
	\begin{align}
	\label{abcdef}
		 \mathcal{R}_i = \mathcal{R} \big(Z_i, U_i, V_i,  \delta'_i, \lambda_i, c_i, \epsilon_i,  \mu_i, \rho_i \big), 	
	\end{align}
according to the definition \eqref{definition_of_special_words} in Subsection \ref{subsection-a-class-of-small-cancellation-words}. Also $Z_i$ is a set of elements of $H_i$ with a ``natural'' order such that $\cup_{j=1}^i Y_i \subseteq Z_i$ and $(\lambda_i, c_i, \epsilon_i,  \mu_i, \rho_i)$ are sparse enough so that Lemma \ref{lem on words with cancellation condition}
 guarantees that $\mathcal{R}_i$ satisfies the small-cancellation condition $C'(\lambda_i, c_i, \epsilon_i, \mu_i, \rho_i)$ and the chain \eqref{main seq of gps----} satisfies the condition $C'\big(\mathcal{TM}, (g_i)_{i=1}^{\infty},  ({\rho}_i)_{i=1}^{\infty} \big)$ for $g_i(n)=n^{\frac{1}{i}}$ for some fixed Turing machine $\mathcal{TM}$ computing the $i$-th level data for \eqref{main seq of gps----}. Note that we will note specify the details about $\mathcal{TM}$ in the constructions of Sections \ref{section-verbally_comlete_groups}, \ref{section-Tarskii_monsters} and \ref{section_problem_7-5}   since what we need is actually the only fact that such a Turing machine $\mathcal{TM}$ indeed exists.
 
 As a conclusion of what is said in this section, in Sections \ref{section-verbally_comlete_groups}, \ref{section-Tarskii_monsters} and \ref{section_problem_7-5} we will only specify description of the following:
 \begin{itemize}
 	\item $G_0$;
 	\item How does $H_i$ relate to $G_{i-1}$;
 	\item Precise definitions of $Z_i$, $U_i$ and $V_i$ from \eqref{abcdef}.
 \end{itemize} 
 
 \begin{lemma}
 \label{lemma-word-problem-scheme}
 	The groups constructed according to the above described scheme have word problem decidable in almost linear time when for all $i\in \mathbb{N}$, $\mathcal{R}_i$ contains one element up to cyclic shift, and in almost quadratic time otherwise.
 \end{lemma}
 \begin{proof}
 	It follows directly from Property \ref{property-time-for-special-words} and Theorem \ref{theorem_about_WP_in_barG_2}.
 \end{proof}

		\section{Proof of Theorem \ref{theorem_verbally_complete}}
	\label{section-verbally_comlete_groups}
		In this section we are going to show that for any given non-elementary, torsion-free $\delta_0$-hyperbolic group $G_0$, there exists a lacunary hyperbolic quotient of $G_0$, denoted by $\check{G}$, which satisfies the conditions of Theorem \ref{theorem_verbally_complete}.
	
	 Our approach is constructive and will be based on the scheme described in Section \ref{section-about-general-scheme}. First of all, this means that $\check{G}$ will be constructed as an inductive limit of a chain of hyperbolic groups of type \eqref{main seq of gps***}, that is 
	\begin{align}
	    \label{main seq of gps-verbal}
		G_0 \stackrel{\beta_0}\hookrightarrow H_1 \stackrel{\gamma_1}\twoheadrightarrow G_1 \stackrel{\beta_1}\hookrightarrow H_2 \stackrel{\gamma_2}\twoheadrightarrow \ldots.
	\end{align}
	
	In our construction below we will inductively show that the groups $H_i$ and $G_i$, $i=1, \ldots$ are non-elementary torsion-free hyperbolic groups. In this section, the limit of \eqref{main seq of gps-verbal} we denote by $\check{G}$.
	
	Let the finite symmetric set $X = \{ x_1^{\pm 1}, x_2^{\pm 1}, \ldots, x_{n_0}^{\pm 1}\}$ be a generating set of $G_0$ such that $\Gamma(G_0, X)$ is $\delta$-hyperbolic for some $\delta \in \mathbb{N}$.  Let us denote $X^-= \{ x_1^{- 1}, x_2^{-1}, \ldots, x_{n_0}^{- 1}\}$ and $X^+=\{ x_1^{+ 1}, x_2^{+1}, \ldots, x_{n_0}^{+ 1}\}$. Also let us order $X$ in the following natural way: $x_i^{-1}<x_j^{-1}<x_i^{1} < x_j^{1}$ if $i<j$,  and the elements of $X^{-}$ precede the elements of $X^+$, i.e., for all $i, j \in \mathbb{N}$, $x_i^{-1} < x_j$. Hereafter whenever we consider an indexed alphabet $X'$, the order of the set $(X')^{\pm 1}$ will be defined just like it was done for $X =X^{-} \cup X^+$.
	
	Let us consider the free group $F_1=F(Y)$ of infinite rank, where $Y = \{y_1, y_2, \ldots \}$ and let us introduce an order on the set of reduced words from $F$ in the following natural way: for reduced words $u, v \in F$, we define $u>v$ if either $\|u\| = \|v\|$ and $u>v$ lexicographically (here we regard words as vectors of letters from $Y^{\pm 1}$) or $\|u\|>\|v\|$. In the analogous way, we order elements of the free group $F_2 = F(X)$. For a reduced word $u\in F(Y)$, we say that $u$ is a \textit{dense} word, if there exists $i \in \mathbb{N}$ such that $u$ contains at least one letter from each of the following sets $\{y_1^{\pm 1}\}, \ldots, \{y_i^{\pm 1}\}$ and does not contain any other letters.
	
	Finally, let us introduce a partial linear order on the set $F_1 \times F_2$ in the following way: let $u, u'$ be reduced words in $F_1=F(X)\setminus\{1\}$ and $v, v'$ be dense words in $F_2=F(Y)\setminus\{1\}$, then we define $(u, v)< (u', v')$ if either $\|u\|+ \|v\| < \|u'\|+\|v'\|$ or 
	\begin{align*}
		\|u\|+\|v\|=\|u'\|+\|v'\| \text{~and~} u<u'.
	\end{align*}
	 Denote the $i^{th}$ element of the set  $F_1 \times F_2$ with respect to this partial order by $(u_i, v_i)$. The reason for considering only the dense words from $F_2$ (and also the partial order) is that for dense words, balls of finite radius with respect to the word metric have finite volume and hence, by the introduced partial order, we will be able to effectively enumerate all the aforementioned pairs $(u, v)\in F_1 \times F_2$. (Also note that the map $i\mapsto (u_i, v_i)$ is not bijective). 
As it will be clear from what is discussed below, this is important for the construction of machines $TM_1$ and $TM_2$. 	
	Without loss of generality we assume that $x_1$ and $x_2$ are different elements of infinite order in $G_0$.

\subsubsection{Definition of $H_{i+1}$ ($i\geq 0$) for $\check{G}$}
Suppose that the non-elementary torsion-free $\delta_{i}$-hyperbolic  group $G_{i}$ is already constructed. Let $(u_{i+1}, v_{i+1}) \in F_1 \times F_2$ be a pair of dense word as defined above. 
~\\
	Case 1. ($u_{i+1}=_{G_i} 1$). In this case define $H_{i+1}=G_i$;\\
	Case 2. ($u_{i+1} \neq_{G_i} 1$). In this case let us define  $n_{i+1}$ to be the least positive integer such that $v_{i+1} \in F(y_1, \ldots, y_{n_{i+1}})$ (and consequently, $v_{i+1} \not\in F(y_1, \ldots, y_{n_{i+1}-1})$). Note that the existence of $n_{i+1}$ follows from the fact that $v_{i+1}$ is dense by definition.\\
	Let us define $G_i'=G_i*F(y_1, \ldots, y_{n_{i+1}})$. By Corollary \ref{corollary_about_finding_root_of_a_group_element}, there exists a Turing machine which for input $(G_i, \delta_i; u_{i+1}, v_{i+1})$ outputs the pairs $(u'_{i+1}, k)$ and $(v'_{i+1}, l)$, where $u'_{i+1}, v'_{i+1} \in (X \cup \{(y_1, \ldots, y_{n_{i+1}}\})^*$, $k, l \in \mathbb{N}$ are such that $u_{i+1} =_{G_i'} (u'_{i+1})^k$ and $v_{i+1} =_{G_i'} (v'_{i+1})^l$  and $u'_{i+1}, v'_{i+1}$ represent root elements of $u_{i+1}$ and $v_{i+1}$ in $G_i'$ (i.e. $E(u_{i+1}) = \langle u'_{i+1} \rangle$ and $E(v_{i+1}) = \langle v'_{i+1} \rangle$ in $G'_i$). We will use the standard notation $v_{i+1} = v_{i+1} ( y_1, \ldots, y_{n_{i+1}})$ to emphasize that $v_{i+1}$ is formed by the letters (or, in the context of diophantine equations, by variables) $y_1, \ldots, y_{n_{i+1}}$.
	
	Let $\langle z_{i+1} \rangle$ be an infinite cyclic group disjoint from $G_i'$. Define the group $H_{i+1}^0$ as an HNN-extension of $G_i$ as follows.
	\begin{align*}
		H_{i+1}^0= \langle G_i* \langle z_{i+1} \rangle, t_{i+1} \mid  t_{i+1}^{-1} u'_{i+1} t_{i+1} =  z_{i+1}^l\rangle.
	\end{align*}
   Now define $H$ as an HNN-extension of $H^0$ as follows.
   \begin{align}
   \label{definition of H}
		H_{i+1}= \langle H_{i+1}^0*F(y_1, \ldots, y_{n_{i+1}}), s_{i+1} \mid  s_{i+1}^{-1} v'_{i+1} s_{i+1} =  z_{i+1}^k\rangle.
	\end{align}
	Finally, define $Y_{i+1} = \{y_1, \ldots, y_{n_{i+1}}\} \cup \{z_{i+1}\} \cup \{ t_{i+1} \} \cup \{s_{i+1}\}$.\\
	\begin{proposition}
	The group $H_{i+1}$  is a torsion-free non-elementary hyperbolic group and the identity map $id: X \rightarrow X$ induces an embedding of $G_{i}$ into $H_{i+1}$.	
	\end{proposition}
   \begin{proof}
   First of all, if $H_{i+1}=G_i$ then the statement follow from the inductive assumption that $G_i$ is a torsion-free non-elementary hyperbolic group.
   
   Now assume that $H_{i+1}$ is defined by \eqref{definition of H}. Then, since $\langle u'_{i+1} \rangle$ is a maximal elementary subgroup of $G_i * \langle z_{i+1} \rangle$ and since for all $g \in G_i * \langle z_{i+1} \rangle$ we have $g^{-1} \langle u'_{i+1} \rangle g \cap \langle z_{i+1} \rangle = \{1\}$, by Theorem \ref{th hnn extension}, $H_{i+1}^0$ is a hyperbolic group.
   
   Now, since $\langle v'_{i+1} \rangle$ is a maximal elementary subgroup in $F(y_1, \ldots, y_{n_{i+1}})$, we get that $\langle v'_{i+1} \rangle$ is a maximal elementary subgroup in $H_{i+1}^0*F(y_1, \ldots, y_{n_{i+1}})$ and for all $h \in H_{i+1}^0*F(y_1, \ldots, y_{n_{i+1}})$, we have $h^{-1} \langle v'_{i+1} \rangle h \cap \langle z_{i+1} \rangle = \{1\}$. Therefore, by Theorem \ref{th hnn extension}, $H_{i+1}$ is a hyperbolic group.
   
   The fact that $H_{i+1}$ is torsion free follows from the fact that $G_i$ is torsion free and from Lemma \ref{lemma_about_proper_powers_in_hnn}. 
   
   	The part of the statement that the identity map $id: X \rightarrow X$ induces an embedding of $G_{i}$ into $H_{i+1}$
 follows from the basic properties of HNN-extensions. See \cite{lyndon schupp}.
 
 Finally, since $H_{i+1}$ contains an isomorphic copy of $G_{i}$ and $G_{i}$ is non-elementary, it follows that $H_{i+1}$ is non-elementary as well.
   \end{proof}
\begin{proposition}
   \label{proposition-verbally-completeness-in-checkG}
   	The equation $v_{i+1}( y_1^{s_{i+1}t_{i+1}^{-1}},  \ldots, y_{n_{i+1}}^{s_{i+1}t_{i+1}^{-1}} ) = u_{i+1}$ holds in $H_{i+1}$. In other words, $y_1 \mapsto y_1^{s_{i+1}t_{i+1}^{-1}}, \ldots,  y_{n_{i+1}} \mapsto y_{n{i+1}}^{s_{i+1}t_{i+1}^{-1}}$ is a solution to the diophantine equation
   	\begin{align*}
   		v_{i+1} (y_1, \ldots, y_{n_{i+1}}) = u_{i+1}
   	\end{align*}
   	in $H_{i+1}$.
   \end{proposition}
   \begin{proof}
   	Indeed, first of all, the relations $t_{i+1}^{-1} u'_{i+1} t_{i+1} = z_{i+1}^l$ and $s_{i+1}^{-1} v'_{i+1} z_{i+1}^k$ imply that
   	\begin{align*}
   		(t_{i+1}^{-1} u'_{i+1} t_{i+1})^k=t_{i+1}^{-1} u_{i+1} t_{i+1} = z_{i+1}^{lk}=(z_{i+1}^k)^l = (s_{i+1}^{-1} v s_{i+1})l = s_{i+1}^{-1} v_{i+1} s_{i+1}.
   	\end{align*}
   	Therefore, $t_{i+1}s_{i+1}^{-1} v_{i+1} s_{i+1}t_{i+1}^{-1} = v_{i+1}^{s_{i+1}t_{i+1}^{-1}}=u_{i+1}$. Now, since $$v_{i+1}^{s_{i+1}t_{i+1}^{-1}}=v_{i+1}( y_1^{s_{i+1}t_{i+1}^{-1}},  \ldots, y_{n_{i+1}}^{s_{i+1}t_{i+1}^{-1}} ),$$ we get that $y_1 \mapsto y_1^{s_{i+1}t_{i+1}^{-1}}, \ldots,  y_{n_{i+1}} \mapsto y_{n_{i+1}}^{s_{i+1}t_{i+1}^{-1}}$ is a solution of the diophantine equation
$v_{i+1} (y_1, \ldots, y_{n_{i+1}}) = u_{i+1}$.
   \end{proof}
\begin{proposition}
	\label{aux-prop-111}
		If $x_2 \notin E(x_1)$ in $G$, then $x_2 \notin E(x_1)$ in $H_{i+1}$. Also, for all $y \in Y_{i+1}$, $y \notin E(x_1)$ in $H_{i+1}$.
	\end{proposition}
	\begin{proof}
		This fact immediately follows from Lemma \ref{lemma_about_proper_powers_in_hnn}.
	\end{proof}
	
	\begin{proposition}
	\label{proposition-H-conjugation-in-checkG}
		Let $U, V \in X^*$ be such that $U \sim_{conj} V$ in $H_{i+1}$. Then $U \sim_{conj} V$ in $G_i$.
	\end{proposition}
	\begin{proof}
	Suppose that $U \not\sim_{conj} V$ in $G_i$. Then we want to show that $U \not\sim_{conj} V$ in $H$.
	
		By contradiction let us assume that  $U \sim_{conj} V$ in $H_{i+1}$. Then there exists a $(U, V)$-conjugacy diagram $\Delta$ over $H_{i+1}$ with boundary $ABCD$, $lab(AD)=V$, $lab(BC)=U$. Note that since $U \not\sim_{conj} V$ in $G_i$, $\Delta$ must contain at least one $t_{i+1}$- or $s_{i+1}$-band which has its ends on different sides of $ABCD$. Also, since $U$ and $V$ do not contain edges with labels from $\{ s_{i+1}^{\pm 1}, t_{i+1}^{\pm 1} \}$, it must be that all these bands are horizontal, i.e., have their ends on $AB$ and $DC$.
        
        Next, we will show that $\Delta$ cannot contain horizontal bands. By contradiction let us assume that it contains horizontal bands. 
        
        First, suppose that $\Delta$ contains more than one horizontal bands. In this case, let us choose edges $e_1, e_2 \in AB$ and $e'_1, e'_2 \in CD$ such that they have labels from $\{ s_{i+1}^{\pm 1}, t_{i+1}^{\pm 1} \}$ and $e_1$ and $e_2$ are connected by horizontal bands to $e_1'$ and $e_2'$, respectively. Additionally, without loss of generality let us assume that there is no horizontal band between these two bands. See Figure \ref{fig:  Verbally_complete-conjugacy}.
       \begin{figure}[H]
						\centering
						\includegraphics[clip, trim=0cm 8.5cm .7cm 7.7cm, width=.6\textwidth]{{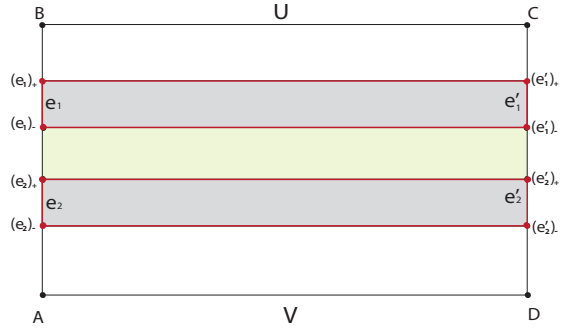}} 
						\caption{$\Delta$ with the two horizontal bands, depicted as grey areas.} 
						\label{fig:  Verbally_complete-conjugacy}
					\end{figure}
        Note that then $(e_1)_-$, $(e_1')_-$ and $(e_2)_+$, $(e'_2)_+$ are pairs of mirroring point, respectively. Therefore,
            the subdiagram of $\Delta$ bounded between $(e_2)_+$, $(e_1)_-$, $(e_1')_-$ and $(e'_2)_+$ is a $\Big(lab\big([(e_1)_-, (e'_1)_- ]\big), lab\big([(e_2)_+, (e'_2)_+ ]\big) \Big)$-conjugacy diagram over $G_i'$. In particular, $lab\big([(e_1)_-, (e'_1)_- ]\big)$ and $ lab\big([(e_2)_+, (e'_2)_+ ]\big)$ are conjugate in $G_i'$. Now, since $lab([(e_1)_-, (e'_1)_- ])$ and $lab([(e_2)_+, (e'_2)_+ ])$ are (free) powers of elements from $\{u'_{i+1}, v'_{i+1}, z_{i+1} \}$, it follows that, in fact, they must be freely equal. This means that $lab([(e_1)_+, (e'_1)_+ ])$ and $lab([(e_2)_-, (e'_2)_- ])$ are also freely equal. But, since $\big((e_1)_+, (e'_1)_+ \big)$ and $\big((e_2)_-, (e'_2)_- \big)$ are pairs of mirroring points on $\partial\Delta$, by Lemma \ref{lemma_about_opposite_points_in_slender_conjugacy_diagrams}, this contradicts  the assumption that $\Delta$ is chosen to be slender.         Therefore, $\Delta$ cannot contain two horizontal bands, hence, since by our assumptions, it contains at least one horizontal $\{ s_{i+1}^{\pm 1}, t_{i+1}^{\pm 1} \}$-band, it means that the number of such bands is exactly one.
        
        Now suppose that $\Delta$ contains only one horizontal band and that only horizontal band of $\Delta$ has its ends on edges $e_1$ and $e_1'$, i.e. in Figure \ref{fig:  Verbally_complete-conjugacy} just neglect the bottom band. Since one of $lab\big([(e_1)_-, (e'_1)_- ]\big)$ and $ lab\big([(e_1')_+, (e'_1)_+ ]\big)$ belong to $Y_{i+1}^*$, without loss of generality assume that $lab\big([(e_1)_+, (e'_1)_+ ]\big)$ $\in Y_{i+1}^*$. Then, since $lab([B, (e_1)_+])\equiv lab([C, (e'_1)_+])$, we get that $u'_{i+1} \sim_{conj} lab\big([(e_1)_+, (e'_1)_+ ]\big)$ in $G_i'$, which is impossible, since $G_i' = G_i * F(Y_{i+1})$ and $u'_{i+1} \in X^*$. A contradiction.

	\end{proof}
	\subsubsection{Definition of $G_{i+1}$ ($i\geq 0$) for $\check{G}$}
	Assuming that the torsion-free non-elementary hyperbolic group $H_{i+1}$ is already defined, $G_{i+1}$ we define as
	$$G_{i+1}=H_{i+1} / \ll \mathcal{R} \big(Y_{i+1}, x_1, x_2,  \delta'_{i+1}, \lambda_{i+1}, c_{i+1}, \epsilon_{i+1},  \mu_{i+1}, \rho_{i+1} \big)\gg.$$
	Note that, if $x_1 \notin E(x_2)$ in $G_{i}$, then,  by Proposition \ref{aux-prop-111}, $x_1 \notin E(x_2)$ in $H_{i+1}$, hence for sparse enough standard parameters $\delta'_{i+1}, \lambda_{i+1}, c_{i+1}, \epsilon_{i+1},  \mu_{i+1}, \rho_{i+1}$, the set of words $$\mathcal{R}_{i+1}=\mathcal{R} \big(Y_{i+1}, x_1, x_2,  \delta'_{i+1}, \lambda_{i+1}, c_{i+1}, \epsilon_{i+1},  \mu_{i+1}, \rho_{i+1} \big)$$ satisfies the small-cancellation condition $C'(\lambda_{i+1}, c_{i+1}, \epsilon_{i+1}, \mu_{i+1}, \rho_{i+1})$. Therefore,
	\begin{enumerate}
		\item By Lemma \ref{lem 7.2}, $G_{i+1}$ will be non-elementary torsion-free hyperbolic group;
		\item By Lemma \ref{lemma_about_proper_powers_in_quotients}, $x_1$ is not a proper power in $G_{i+1}$. Therefore, since $G_{i+1}$ is a torsion-free hyperbolic group, we get $x_1 \notin E(x_2)$ in $G_{i+1}$. Thus, by inductive hypothesis, $G_{i+1}$ is well-defined non-elementary torsion-free hyperbolic group.
	\end{enumerate}
	\subsection{Main properties of $\check{G}$}
	Note that since the groups $G_i$ are torsion-free non-elementary hyperbolic groups, the group $\check{G}$ is torsion-free infinite lacunary hyperbolic  group (recall that we assume that the standard parameters are sparse enough). 
	
	From Proposition \ref{proposition-verbally-completeness-in-checkG} if follows that $\check{G}$ is verbally complete.
	
	From Property \ref{property-time-for-special-words} and Theorem \ref{theorem_about_WP_in_barG_2} it follows that the word problem in $\check{G}$ is almost quadratic.
	
	From Proposition \ref{proposition-H-conjugation-in-checkG} it follows that the set of $H$-conjugate pairs in $\check{G}\times \check{G}$ is empty, hence combined this with Theorem \ref{theorem-effective-G-conjugacy-problem} and Property \ref{property-time-for-special-words}, we get that the conjugacy problem in $\check{G}$ is polynomial.
	
	Thus Theorem \ref{theorem_verbally_complete} is proved.

\section{Proof of Theorem \ref{theorem_tarskii_monsters}}
\label{section-Tarskii_monsters}
Let $G_0=\langle X \rangle $, $X=\{ x_1, x_2, \ldots, x_n \}$, be a torsion-free non-elementary $\delta$-hyperbolic group with respect to $X$. Without loss of generality we assume that $E(x_i) \cap E(x_j) = \{1\}$ if $i\neq j$.

 Let $X$ be linearly ordered such that $x_i^{-1}<x_j^{-1}<x_i< x_j$ if $i<j$.
We denote the set of reduced non-empty words of $X^*$ by $F'(X)$. Let us enumerate the set $F'(X)$ as $F'(X)=\{ w_1, w_2, \ldots \}$ where for $i<j$, $w_i < w_j$ according to the lexicographical order induced from the order on $X$. Then clearly $w_1=x_1$, $w_2=x_2$. Now, based on this order of $F'(X)$ let us lexicographically order the set $F'(X) \times F'(X) \setminus \{(w, w) \mid w \in F'(X) \}$ and enumerate it according to that order. Let
\begin{align*}
	F'(X) \times F'(X) \setminus \{(w, w) \mid w \in F'(X) \}= \{(u_1, v_1), (u_2, v_2) \ldots \},
\end{align*}
 where for $i<j$, we have $(u_i, v_i) < (u_j, v_j)$.
 
 As it was mentioned in Section \ref{section-about-general-scheme}, in this section we are going to construct the group $\hat{G}$ from Theorem \ref{theorem_tarskii_monsters} which will be a direct limit of a chain of non-elementary torsion-free hyperbolic groups of the form \eqref{main seq of gps***}, that is 
 \begin{align}
	    \label{main seq of gps-tarski}
		G_0 \stackrel{\beta_0}\hookrightarrow H_1 \stackrel{\gamma_1}\twoheadrightarrow G_1 \stackrel{\beta_1}\hookrightarrow H_2 \stackrel{\gamma_2}\twoheadrightarrow \ldots.
	\end{align}
	
 In this section we define $H_i=G_{i-1}$ for all $i\in \mathbb{N}$ and the map $\beta_{i-1}: G_{i-1} \rightarrow H_i$ is the identity map.

  Now let us describe how $G_{i}$ is obtained from $G_{i-1}$ for $i \in \mathbb{N}$. For that purpose by induction let us assume that $G_{i-1}$ is a non-elementary torsion-free hyperbolic group (below we will show that for $G_i$ the same property holds as well). Then, there exists smallest index ${j_i} \geq i$ such that $v_{j_i} \notin E(u_{j_i})$ in $G_{i-1}$, and the set $$Z_i\stackrel{\text{def}}{=} \{ x \in X \mid x \notin E(u_{j_i})\}$$ is non-empty. By Corollary \ref{corollary_about_finding_root_of_a_group_element}, $(u_{j_i}, v_{j_i})$ and $Z_i$ can be found algorithmically.
  
  Now define $$G_i = G_{i-1} / \ll \mathcal{R} \big(Z_i, u_{j_i}, v_{j_i},  \delta'_i, \lambda_i, c_i, \epsilon_i,  \mu_i, \rho_i \big) \gg.$$
Note that by Lemma \ref{lem on words with cancellation condition}, if the standard parameters $\delta'_i, \lambda_i, c_i, \epsilon_i,  \mu_i, \rho_i $ are sparse enough, then the set of words $$\mathcal{R}_i\stackrel{\text{def}}{=} \mathcal{R} \big(Z_i, u_{j_i}, v_{j_i},  \delta'_i, \lambda_i, c_i, \epsilon_i,  \mu_i, \rho_i \big)$$ satisfies the small-cancellation condition $C'(\lambda_i, c_i, \epsilon_i, \mu_i, \rho_i)$. Also note that, by the definition of $G_i$, $G_i = \langle u_{j_i}, v_{j_i} \rangle$.
 \begin{lemma}
\label{property-structure-G-i-tarskii}
	The following are true about $G_i$.
	\begin{enumerate}
		\item $G_i$ is a torsion-free non-elementary hyperbolic group;
		\item Either $v_i \in E(u_i)$ in $G_i$ or $\langle u_i, v_i \rangle = G_i$.
		\item For each $x \in X$, $E(x)= \langle x \rangle$ in $G_i$ (we assume that for $G_{i-1}$ this is already shown).
	\end{enumerate}
\end{lemma}
\begin{proof}

	Part (1) of the statement follows from Lemma \ref{lem 7.2}. 
	
	 For part (2) simply notice that, by our definition of $j_i$ if $j_i>i$ then $u_i \in E(v_i)$ in $G_i$, otherwise if $j_i=i$ then $v_i \notin E(u_i)$ in $G_i$ and $G_i = \langle u_{i}, v_{i} \rangle$.
	 
	 For Part (3), first, note that it immediately follows from Lemma \ref{lemma_about_proper_powers_in_quotients} that $x$ is not a proper power in $G_i$. Therefore, since by Part (1) of the current lemma, $G_i$ is a torsion-free abelian group, we get that $E(x)=\langle x \rangle$.
	 \end{proof}
 \begin{proposition}
	The group $\hat{G}$, which is defined as inductive limit of  $(G_i)_{i=1}^{\infty}$, satisfies the statement of Theorem \ref{theorem_tarskii_monsters}. That is
	\begin{enumerate}
         		\item[(i).] Every proper subgroup of $\hat{G}$ is an infinite cyclic group, while $\hat{G}$ is not cyclic;\\
         		\item[(ii).] The word problem in $\hat{G}$  is decidable in almost quadratic time and the conjugacy problem in $\hat{G}$  is decidable in polynomial time.
         	\end{enumerate}
\end{proposition}
\begin{proof}

	(i). First of all, notice that $\hat{G}$ is not cyclic, because otherwise, by part (3) of Lemma \ref{property-structure-G-i-tarskii} it follows that, for example, $x_1=x_2$ in some $G_i$, which is impossible if the standard parameters $(\lambda_i, c_i, \epsilon_i, \mu_i, \rho_i)_{i=1}{\infty}$ are sparse enough (recall that for sparse enough standard parameters, $\hat{G}$ is lacunary hyperbolic by 	Theorem \ref{theorem_about_lacunary_hyperbolicity_of_barG}).
	
	Now, by contradiction let us assume that $\hat{G}$ contains a proper non-cyclic subgroup $K$. Then, since $K$ is a proper subgroup, by part (2) of Lemma \ref{property-structure-G-i-tarskii}, $K$ is abelian (even more, each finitely generated subgroup of $K$ is cyclic). Let us fix any non-trivial element $x \in K$. Then in each of the groups $G_i$ the centralizer of $x$ coincides with $E(x)$, hence it is cyclic. This means that in the inductive limit $\hat{G}$ the centralizer of $x$ is again cyclic. Therefore, since $K$ is contained in the centralizer of $x$, $K$ is cyclic as well. A contradiction.
	
	(ii). Second part of Proposition follows from Property  \ref{property-time-for-special-words} and theorems \ref{theorem_about_WP_in_barG_2} and \ref{theorem-effective-G-conjugacy-problem}. As for conjugacy problem, let us notice that two elements of $\hat{G}$ are conjugate if and only if they are $G$-conjugate, hence Theorem \ref{theorem-effective-G-conjugacy-problem} implies that conjugacy problem in $\hat{G}$ is polynomial.

\end{proof}
Thus Theorem \ref{theorem_tarskii_monsters} is proved.
\section{Proof of Theorem \ref{theorem_about_connecton_of_word_and_conjugacy_problems}} 
	\label{section_problem_7-5}
	Let $\mathcal{A}$ be any finite alphabet, and let $\mathcal{L}\subseteq \mathcal{A}^*$ be any recursively enumerable subset of  $\mathcal{A}^*$.
	
	 For the two generated free group $F(x_1, x_2)$, let us denote by $F^+(x_1, x_2)$ the set of words from $F(x_1, x_2)$ which do not contain the letters $x_1^{-1}$ and $x_2^{-1}$.
	
	  Let us also fix a bijective map $\Lambda_0:\mathcal{A}^*  \rightarrow F^+(x_1, x_2)$ such that $\Lambda_0$ and $\Lambda_0^{-1}$ are computable in linear time. Construction of such a map can be easily achieved through a standard binary encoding of the set $\mathcal{A}^*$. 
	
	Let us define $G_0 = F_1 * F_2 * F_3$, where $F_1 = F(x_1, x_2, x_3)$,	$F_2 = F(y_1, y_2, y_3)$, $F_3 = F(z_1, z_2)$ are free groups with freely generating sets $X_0=\{x_1, x_2, x_3\}$, $Y_0=\{y_1, y_2, y_3\}$ and $Z_0=\{z_1, z_2\}$, respectively. For the convenience in the further exposition, let us also introduce the following notations: $G_{0, 1}=F_1$, $G_{0, 2}=F_2$ and $G_{0, 3} = F_3$.
	
	 
	
	Let $\varsigma: F_1 \rightarrow F_2$ be the isomorphism between $F_1$ and $F_2$ induced by the map $x_1 \mapsto y_1$, $x_2 \mapsto y_2$, $x_3 \mapsto y_3$. 
	
	Define $\Lambda: \mathcal{L} \rightarrow X_0^* \times Y_0^*$ as follows: For all $\omega \in \mathcal{L}$, 
	\begin{align*}
		\Lambda(\omega )=( ~\Lambda_0(\omega)x_3, ~\varsigma( \Lambda_0(\omega))y_3  ).
	\end{align*}

		Clearly, $\Lambda$ in an injection. 
		Let
		$$ \Lambda(\mathcal{L} ) = \{(u_1, v_1), (u_2, v_2), \ldots \},$$
		where the enumeration is with respect to some fixed Turing machine $M_{\Lambda}$ which on input $i\in \mathbb{N}$ outputs $(u_i, v_i)$. Note that such an enumeration exists since $\mathcal{L}$ is recursively enumerable.
		
	~\\
	
	As it was mentioned in Section \ref{section-about-general-scheme}, in this section we are going to construct the group $G_{\mathcal{L}}$ from Theorem \ref{theorem_tarskii_monsters} which will be a direct limit of a chain of non-elementary torsion-free hyperbolic groups of the form \eqref{main seq of gps***}, that is 
 \begin{align}
	    \label{main seq of gps-wp-cp}
		G_0 \stackrel{\beta_0}\hookrightarrow H_1 \stackrel{\gamma_1}\twoheadrightarrow G_1 \stackrel{\beta_1}\hookrightarrow H_2 \stackrel{\gamma_2}\twoheadrightarrow \ldots.
	\end{align}
	More specifically, $G_0 \stackrel{def}= F(X)$, where 
	$$X \stackrel{def}=X_0 \cup Y_0 \cup Z_0 =\{x_1, x_2, x_3, y_1, y_2, y_3, z_1, z_2 \},$$
	where $X_0=\{x_1, x_2, x_3\}$, $Y_0= \{ y_1, y_2, y_3\}$ and $Z_0=\{z_1, z_2 \}$.
	
	\subsubsection{Definition of $H_{i+1}$ ($i\geq 0$) for $\check{G}$}
	Assuming that $G_i$ is already constructed we define $H_{i+1}$ as an HNN-extension of $G_i$. More precisely,
	\begin{equation}
	\label{definitino-H_i}
		H_{i+1}=\langle G_i, t_{i+1} \mid u_{i+1}=t_{i+1}^{-1}v_{i+1} t_{i+1} \rangle.
	\end{equation}
Then, clearly the identity map $id: X \rightarrow X$ induces an embedding $\beta_i: G_i \hookrightarrow H_{i+1}$.
	Define $$Y_{i+1} =\{t_{i+1}\}.$$
	
	We will show by induction that for all $i \geq 0$, $H_{i+1}$ is torsion-free, non-elementary $\delta'_{i+1}$-hyperbolic group (for some $\delta'_{i+1} \in \mathbb{N}$ such that the map $i+1 \mapsto \delta'_{i+1}$ is computable) with respect to the generating set $X \cup \{t_1, \ldots, t_i\}$.
\subsubsection{Definition of $G_{i+1}$ ($i\geq 0$) for $G_{\mathcal{L}}$}
Suppose that $H_{i+1}$ is already constructed and it is non-elementary, torsion-free $\delta'_{i+1}$-hyperbolic with respect to the generating set $X \cup \{t_1, \ldots, t_i\}$ for $\delta'_{i+1}\in \mathbb{N}$. Then, we define $G_{i+1}$ as follows 
$$G_{i+1} \stackrel{\text{def}}= G_{i} / \ll \mathcal{R} \big(\{t_{i+1}\}, z_1, z_2,  \delta'_{i+1}, \lambda_{i+1}, c_{i+1}, \epsilon_{i+1},  \mu_{i+1}, \rho_{i+1} \big) \gg,$$
where $\lambda_{i+1} \succ c_{i+1} \succ \epsilon_{i+1} \succ  \mu_{i+1} \succ \rho_{i+1}$ are some sparse enough standard parameters. Denote $\mathcal{R}_{i+1} =\mathcal{R} \big(\{t_{i+1}\}, z_1, z_2,  \delta'_{i+1}, \lambda_{i+1}, c_{i+1}, \epsilon_{i+1},  \mu_{i+1}, \rho_{i+1} \big)$, and let $R_{i+1} \in \mathcal{R}_{i+1}$ be any fixed representative of $\mathcal{R}_{i+1}$ (then,  $\mathcal{R}_{i+1}$ is the set of cyclic shifts of $R_{i+1}$).

Note that $G_{i+1}$ is generated by the image of $X$ (which we denote by $X$ too) under the natural  homomorphism from $G_i$ to $G_{i+1}$.
 We will show by induction that for all $i \geq 0$, $G_{i+1}$ is torsion-free, non-elementary $\delta_{i+1}$-hyperbolic group (for some $\delta_{i+1} \in \mathbb{N}$ such that the map $i+1 \mapsto \delta_{i+1}$ is computable) with respect to the generating set $X$.

	~\\
	
	For the further exposition let us define the concept of \textit{truncated} contiguity diagrams as follows: In a van Kampen diagram $\Delta$ over $G_i=H_i/\ll \mathcal{R}_i \gg$ which contains an essential cell $\Pi$ and an outer contiguity diagram  $\Gamma$  connecting an arc $\check{q}_{\Gamma}$ of $\Pi$ to an arc $\hat{q}_{\Gamma}$ of $\partial \Delta$, we say that $\Gamma$ is \textit{truncated} if $p_{\Gamma}$ and $p'_{\Gamma}$ are the shortest paths in $Proj(\Delta)$ joining, respectively, $(\check{q}_{\Gamma})_-$ and $(\check{q}_{\Gamma})_+$ to $\partial\Delta$. 

Note that truncated contiguity diagrams are truncated diagrams according to Definition \ref{def-truncated diagrams}.

	\subsubsection{Main properties of the chain \eqref{main seq of gps-wp-cp}}
	\begin{enumerate}
	\item[(a$_{i}$).] \label{lemma_auxiliary_for_problem_7-5}
	Let $W \in (X_0 \cup Y_0 \cup Z_0)^*$ and for some $i \geq 1$, $W =_{H_{i}} W'$, where $W'$ is a geodesic word in $\Gamma\big(H_{i}, X_0 \cup Y_0 \cup Z_0 \cup \{t_{i}\} \big)$. Then $W'$ does not contain the letter $t_{i}^{\pm 1}$, i.e. $W' \in (X_0 \cup Y_0 \cup Z_0)^*$. Also, if $W \in X_0^* \cup Y_0^*$ is a freely reduced word, then $W$ is geodesic in $\Gamma(H_{i}, X_0 \cup Y_0 \cup Z_0 \cup \{t_{i}\} )$;\\
	\item[(b$_{i}$).] \label{contiguity} There is no $\epsilon_{i}$-contiguity subdiagram $\Gamma$ of rank $i$ such that $lab(\hat{q}_{\Gamma}) \in X_0^* \cup Y_0^*$ and $\| \check{q}_{\Gamma}\| \geq \mu_{i} \|R_{i}\|$. Moreover, if $\check{q}_{\Gamma}$ does not contain an edge labeled by $t_{i}^{\pm 1}$, then it is enough to require $\| \check{q}_{\Gamma}\| \geq \mu_{i} \|R_{i}\|/2$;\\
	\item[(b$'_{i}$).] If a truncated $\epsilon_i$-contiguity subdiagram $\Gamma$ of rank $i$ is such that $\check{q}_{\Gamma}$ is geodesic in $\Gamma\big(H_i, X_0\cup Y_0 \cup Z_0 \cup \{t_1, \ldots, t_i \}\big)$ and $\| \check{q}_{\Gamma}\| \geq \mu_{i} \|R_{i}\|$, then $Area(\Gamma)=0$. \\
	
	\item[(c$_{i}$).] \label{geodesic} If $w \in  X_0^* \cup Y_0^*$ is a reduced word, then it is a geodesic word in $\Gamma\big(G_{i}, X_0 \cup Y_0 \cup Z_0 \cup \{t_{i}\}\big)$. Moreover,  if for some word $u \in (X_0 \cup Y_0 \cup Z_0 \cup \{t_{i}\})^*$, $u$ is geodesic in $\Gamma\big(G_{i}, X_0 \cup Y_0 \cup Z_0 \cup \{t_{i}\}\big)$ and $u=_{G_{i}} w$, then $u \equiv w$ (i.e. $u$ is freely equal to $w$) ;\\
	
	\item[(d$_{i}$).] \label{torsion} If $U \in X_0^* \cup Y_0^*$ is a reduced word which is not a proper power of another word from $G_0$, then it represents an element in $G_{i}$ which is not a proper power of another element from $G_{i}$; \\

	\item[(e$_{i}$).] \label{disjointness} $G_{i, 1} \cap G_{i, 2} = \{1\}$;\\
	
	\item[(f$_{i}$).] Assuming that $G_{i-1}$ is a $\delta_{i-1}$-hyperbolic group with respect to the generating set $X$, we have that $H_{i}$ is a $\delta_{i}'$-hyperbolic group with respect to the generating set $X \cup \{t_1, \ldots, t_{i}\}$ for some computable $\delta_i' \in \mathbb{N}$.
\end{enumerate}

Clearly this properties are true for $i=0$. Next, based on induction on $i$ we will prove that they are true for every $i$.
 	
\subsubsection{Proof of the properties (a$_{i+1}$)-(f$_{i+1}$).}

\label{subsection_proof_of_properties}
	\begin{lemma}
		\label{lem_main_structural_lemma_for_problem_7-5}
Assuming that the statements (a$_{i}$)-(f$_{i}$) are true, the following properties hold.
\begin{enumerate}
	\item[(a$_{i+1}$).] \label{lemma_auxiliary_for_problem_7-5}
	Let $W \in (X_0 \cup Y_0 \cup Z_0)^*$ and for some $i \geq 1$, $W =_{H_{i+1}} W'$, where $W'$ is a geodesic word in $\Gamma\big(H_{i+1}, X_0 \cup Y_0 \cup Z_0 \cup \{t_{i+1}\} \big)$. Then $W'$ does not contain the letter $t_{i+1}^{\pm 1}$, i.e. $W' \in (X_0 \cup Y_0 \cup Z_0)^*$. Also, if $W \in X_0^* \cup Y_0^*$ is a freely reduced word, then $W$ is geodesic in $\Gamma(H_{i+1}, X_0 \cup Y_0 \cup Z_0 \cup \{t_{i+1}\} )$;\\
	\item[(b$_{i+1}$).] \label{contiguity} There is no $\epsilon_{i+1}$-contiguity subdiagram $\Gamma$ of rank $i+1$ such that $lab(\hat{q}_{\Gamma}) \in X_0^* \cup Y_0^*$ and $\| \check{q}_{\Gamma}\| \geq \mu_{i+1} \|R_{i+1}\|$. Moreover, if $\check{q}_{\Gamma}$ does not contain an edge labeled by $t_{i+1}^{\pm 1}$, then it is enough to require $\| \check{q}_{\Gamma}\| \geq \mu_{i+1} \|R_{i+1}\|/2$;\\
	\item[(b$'_{i+1}$).] If a {truncated} $\epsilon_{i+1}$-contiguity diagram $\Gamma$ of rank $i+1$ is such that $\| \check{q}_{\Gamma}\| \geq \mu_{i+1} \|R_{i+1}\|$, then $Area(\Gamma)=0$. \\
	
	\item[(c$_{i+1}$).] \label{geodesic} If $w \in  X_0^* \cup Y_0^*$ is a reduced word, then it is a geodesic word in $\Gamma\big(G_{i+1}, X_0 \cup Y_0 \cup Z_0 \cup \{t_{i+1}\}\big)$. 
	Moreover,  if for some word $u \in (X_0 \cup Y_0 \cup Z_0 \cup \{t_{i+1}\})^*$, $u$ is geodesic in $\Gamma\big(G_{i+1}, X_0 \cup Y_0 \cup Z_0 \cup \{t_{i+1}\}\big)$ and $u=_{G_{i+1}} w$, then $u \equiv w$ (i.e. $u$ is freely equal to $w$) ;\\
	
	\item[(d$_{i+1}$).] \label{torsion} If $U \in X_0^*\cup Y_0^*$ is a reduced word which is not a proper power of another word $G_0$, then it represents an element in $G_{i+1}$ which is not a proper power of another element from $G_{i+1}$; \\

	\item[(e$_{i+1}$).] \label{disjointness} $G_{i+1, 1} \cap G_{i+1, 2} = \{1\}$;\\
	
	\item[(f$_{i+1}$).] Assuming that $G_{i}$ is a non-elementary torsion-free $\delta_{i}$-hyperbolic group with respect to the generating set $X_0 \cup Y_0 \cup Z_0 $, we have that $H_{i+1}$ is a non-elementary torsion-free $\delta_{i+1}'$-hyperbolic group with respect to the generating set $X \cup \{t_1, \ldots, t_{i+1}\}$, where $\delta'_{i+1}$ is some (computable) positive integer. Also, the group $G_{i+1}$ is non-elementary, torsion-free hyperbolic group.
	
	\end{enumerate}
	\end{lemma}
	\begin{proof}
	
	
	Based on the inductive assumption we will prove Lemma \ref{lemma_auxiliary_for_problem_7-5} using the following scheme: $the~inductive~hypothesis\implies (a_{i+1}) \implies (b_{i+1}) \implies (b'_{i+1}),~ (c_{i+1}) \implies (d_{i+1}) \implies (e_{i+1}) \implies (f_{i+1})$.\\
	~\\
	\textbf{(a$_{i+1}$).}  If $W =_{H_{i+1}} W'$ and $W'$ is a geodesic word in $\Gamma\big(H_{i+1}, X_0\cup Y_0 \cup Z_0\cup \{t_{i+1} \} \big)$, then there is a reduced van Kampen diagram $\Delta$ over $H_{i+1}$ such that $\partial \Delta = pq^{-1}$, where $lab(p)= W$ and $lab(q)=W'$.
	
	 If $W'$ contains a letter from $t_{i+1}^{\pm 1}$, then $q$ contains an edge with label from $t_{i+1}^{\pm 1}$, hence
	$\Delta$ contains a $t_{i+1}$-band. Therefore, since $W$ does not contain $t_{i+1}^{\pm 1}$ (or equivalently, $p$ does not contain edges with labels from $t_{i+1}^{\pm 1}$) we get that the $t_{i+1}$-bands of $\Delta$ must start and end on $q$. Let us consider edges $e$ and $e'$ on $q$ such that they are connected by a $t_{i+1}$-band and between them there is no other edge labeled by $t_{i+1}^{\pm 1}$. Let us denote the sides of this $t_{i+1}$-band which are not on $q$ by $q_1$ and $q_2$ as in Figure \ref{fig:  figure-a-i+1}. Note that since in the definition \eqref{definitino-H_i} of $H_{i+1}$ the words $u_{i+1}$ and $v_{i+1}$ are freely cyclically reduced and $\|u_{i+1}\|=\|v_{i+1}\|$, we get $\|q_1\|=\|q_2\|$.
	 Let us also denote by $q'$ the subpath of $q$ between $e_+$ and $(e')_-$. 
	 
	 By our assumptions,   there is no  edge on $q'$ labeled by $t_{i+1}^{\pm 1}$. Therefore, since $lab(q_2(q')^{-1})$ does not contain edges with labels from $t_{i+1}^{\pm 1}$, we get that the subdiagram of $\Delta$ with the boundary $q_2(q')^{-1}$ is a diagram over $G_{i}$ (see Figure \ref{fig:  figure-a-i+1}). Therefore, since by our assumptions $q'$, as a subpath of $q$,  is geodesic in $\Gamma(H_i, X_0\cup Y_0 \cup Z_0 \cup \{t_{i+1}\})$, it is also geodesic in $\Gamma(G_i, X_0\cup Y_0 \cup Z_0)$. Also, since by the statement of (c$_i$), $q_2$ is geodesic in $\Gamma(G_i, X_0\cup Y_0 \cup Z_0)$ too,  we get that $\|q_2\|=\|q'\|$. Also,  since $\|q_1\|=\|q_2\|$, we get $\|q_1\|=\|q'\|$.  Therefore, if we replace the subpath $eq'e'$ of $q$ with $q_1$, then $q$ will be shortened by $2$. The last observation contradicts the assumption that $q$ is geodesic in $\Gamma\big(H_{i+1}, X_0\cup Y_0 \cup Z_0\cup \{t_{i+1} \} \big)$. Therefore, it must be that $W'$ does not contain $t_{i+1}^{\pm 1}$, i.e. $W' \in (X_0\cup Y_0 \cup Z_0)^*$. 
	\begin{figure}[H]
						\centering
						\includegraphics[clip, trim=1cm 14cm 0cm 5cm, width=.75\textwidth]{{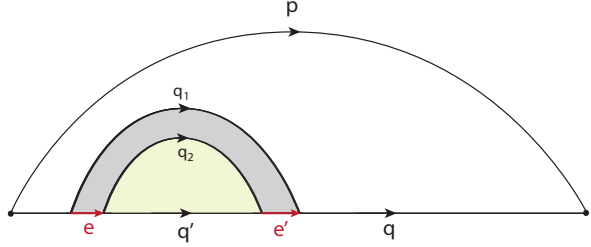}} 
						\caption{$\Delta$: $lab(p)=W$, $lab(q)=W'$, $lab(e)\in \{t_{i+1}^{\pm 1}\}$, $lab(e)\in \{t_{i+1}^{\mp 1}\}$.}
						\label{fig:  figure-a-i+1}
	\end{figure}	
	
	Now let us turn to the last statement of part (a$_{i+1}$). Namely, if $W \in X_0^* \cup Y_0^*$, then $W$ is geodesic in $\Gamma(H_{i+1}, X_0\cup Y_0 \cup Z_0 \cup \{t_{i+1} \} )$. 
	
	Suppose that $W' \in (X_0\cup Y_0 \cup Z_0 \cup \{t_{i+1} \})^*$ is a geodesic word in $\Gamma(H_{i+1}, X_0\cup Y_0 \cup Z_0 \cup \{t_{i+1} \} )$ such that $W'=_{H_{i+1}} W$. Then, by the first part of (a$_{i+1}$), $W'$ does not contain $t_{i+1}^{\pm 1}$, which implies that $W'=_{G_i} W$. By inductive hypothesis (more precisely, by  (c$_{i}$)), since  $W \in X_0^* \cup Y_0^*$, we get that $W$ is geodesic in $\Gamma(G_i, X_0\cup Y_0 \cup Z_0)$. Therefore, $W=_{G_i} W'$ implies $\|W\|=\|W'\|$ and since $W'$ is  geodesic in $\Gamma(H_{i+1}, X_0\cup Y_0 \cup Z_0 \cup \{t_{i+1} \} )$, we get that $W$ is geodesic in $\Gamma(H_{i+1}, X_0\cup Y_0 \cup Z_0 \cup \{t_{i+1} \} )$ as well.\\
	~\\
	\textbf{(b$_{i+1}$).} Suppose $\Gamma$ is a contiguity subdiagram satisfying the conditions described in the statement of (b$_{i+1}$), which, in particular, means that $lab(\hat{q}_{\Gamma}) \in X_0^* \cup Y_0^*$. 
	
	First of all, let us notice that since $lab(\hat{q}_{\Gamma}) \in X_0^* \cup Y_0^*$, by (a$_{i+1}$) we get that $\hat{q}_{\Gamma}$ is  geodesic in $\Gamma\big(H_{i+1}, X_0 \cup Y_0 \cup Z_0\cup \{t_{i+1} \} \big)$.\\
	Now, let $\partial \Gamma = ABCD$, where $AB = p_{\Gamma}$, $BC = \hat{q}_{\Gamma}$, $DC = p'_{\Gamma}$ and $AD = \check{q}_{\Gamma}$. Without loss of generality assume that $\|AB\|+\|DC\|$ is minimal among all  contiguity subdiagrams satisfying the conditions stated in (b$_{i+1}$).
	
	 Now we are going to show that $\Gamma$ does not contain any $t_{i+1}$-bands with both ends on $AB \cup BC \cup CD$. For that purpose, let us notice that since by definition $AB$ and $DC$ are geodesics, there is no $t_{i+1}$-band with both ends on $AB$ or on $DC$. Also, since $BC$ does not contain an edge with label $t_{i+1}^{\pm 1}$, there is no $t_{i+1}$-band which ends on $BC$. Also, since $AB$ and $DC$ are geodesics, there is no $t_{i+1}$-band with both of its ends on $AB$ or on $DC$ (the impossibility of such scenario is explained in the proof of part (a$_{i+1}$)). Thus the only possible way for a $t_{i+1}$-band to have both of  its ends on $AB \cup BC \cup CD$ is when one end is on $AB$ and the other one is on $DC$.\\	 
	  Now assume that there are edges $e$ and $e'$ on $AB$ and $DC$, respectively, such that their labels belong to $\{t_{i+1}^{\pm 1} \}$ and they are connected by a $t_{i+1}$-band. Suppose $e$ belongs to $[A, e_+]$ and $e'$ belongs to $[D, e'_+]$.
	   Then denote $B'=e_-$ and $C'=e'_-$. See Figure \ref{fig:  Contiguity-diagram_Problem-7_5_N1}. Then, since the labels of sides of $t_{i+1}$-bands belong to $X_0^*$ or $Y_0^*$, we get that the subdiagram $A B' C' D$ is another $\epsilon_{i+1}$-contiguity subdiagram which satisfies all the conditions put on $\Gamma$ in (b$_{i+1}$). But since $\|AB'\|+\|DC'\| < \|AB\|+\|DC\|$, this contradicts the minimality assumption on $\|AB\|+\|DC\|$. Therefore, there is no $t_{i+1}$-band with both of its ends on  $AB \cup BC \cup CD$. 
	   \begin{figure}[H]
						\centering
						\includegraphics[clip, trim=2cm 19.2cm 3cm 3.0cm, width=1\textwidth]{{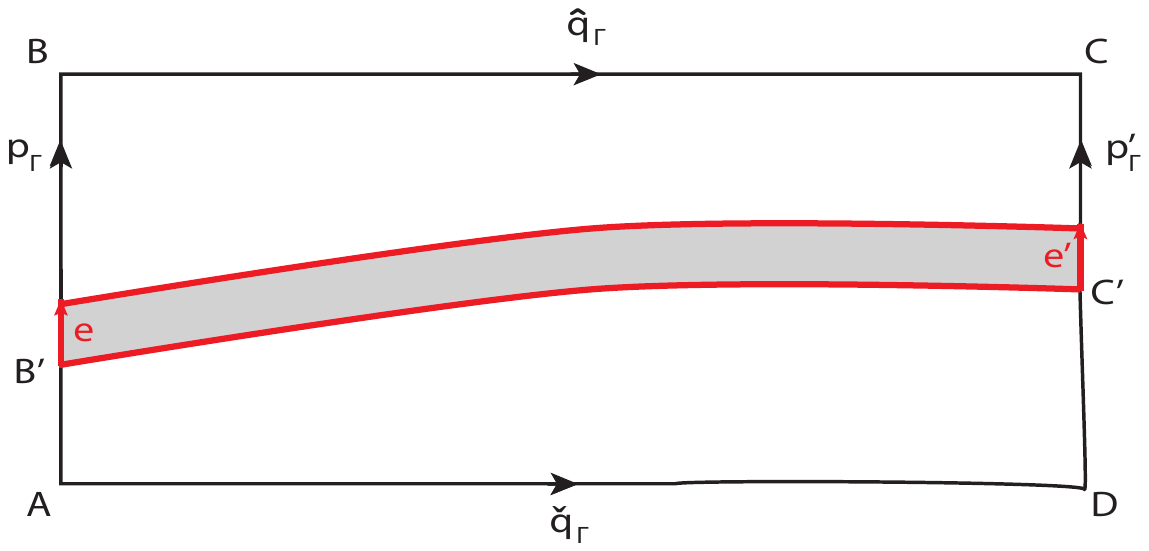}} 
						\caption{}
						\label{fig:  Contiguity-diagram_Problem-7_5_N1}
	\end{figure}	

	 Now let us consider the cases when $\Gamma$ contains a $t_{i+1}$-band with one of its ends on $AD$ and the other one on $AB \cup DC$ or it does not contain any $t_{i+1}$-band. 
	   
	 From the structure of the words $\mathcal{R}_{i+1}$ it follows that they contain exactly one letter from $t_{i+1}^{\pm 1}$. Therefore, since $lab(AD)$ is a subword of some word from $\mathcal{R}_{i+1}$, we get that $AD$ contains maximum one edge with a label from $\{t_{i+1}^{\pm 1} \}$. Hence in $\Gamma$ there is no $t_{i+1}$-band with both of its ends on $AD$.
	  Thus the only possible $t_{i+1}$-band in $\Gamma$ starts on $AD$ and ends on $AB \cup DC$ as it is depicted in Figure \ref{fig:  Contiguity-diagram_Problem-7_5}.
	  
	   Below we discuss in more details the only two possible cases: Case 1 -- when $\check{q}_{\Gamma}$ does not contain an edge with label from $\{t_{i+1}^{\pm 1} \}$ and $\Gamma$ does not contain a $t_{i+1}$-band, and Case 2 --  when $\check{q}_{\Gamma}$ contains exactly one edge with label from $\{t_{i+1}^{\pm 1} \}$.\\
	  ~\\
	\textit{Case 1.} ($\check{q}_{\Gamma}$ does not contain an edge with a label from $\{t_{i+1}^{\pm 1} \}$). In this case, clearly there is no $t_{i+1}$-band in $\Gamma$ which starts on $AD$ and ends on $AB \cup DC$.  Therefore, $\Gamma$ does not contain any $t_{i+1}$-band and $lab(\partial\Gamma) \in (X_0 \cup Y_0 \cup Z_0)^*$. Now, since the boundary of $\Gamma$ does not contain an edge with label $t_{i+1}^{\pm 1}$, clearly, for some $0<j\leq i$, $\Gamma$ is a diagram over $G_j$. Let us assume that $j$ is chosen to be minimal (since $G_0=F_1*F_2*F_3$, $j$ cannot be $0$). Then, clearly there exists a reduced diagram over the quotient $G_j=H_j/\ll \mathcal{R}_j \gg$ with the boundary $\partial\Gamma$. Therefore, let us regard $\Gamma$ as a reduced diagram over $G_j=H_j/\ll \mathcal{R}_j \gg$. From Lemma \ref{lem on words with cancellation condition} and from the structure of the words $\bigcup_k \mathcal{R}_k$, it follows that there is no $\mathcal{R}_j$-cell $\Pi_0$ in $\Gamma$ connected to $[A, D]$ by a $\epsilon_j$-contiguity subdiagram $\Gamma_0$ such that $(\Pi_0, \Gamma_0, [A, D])\geq \rho_j$. 
	
	 Let us choose $B', C' \in [B, C]$ such that $d(A, B')=dist(A, [B, C])$ and $d(D, C')=dist(D, [B, C])$ in $\Gamma(H_j, (X_0 \cup Y_0 \cup Z_0 \cup \{t_j \})^*$. Let $[A, B']$ and $[D, C']$ be geodesics in $\Gamma(H_j, (X_0 \cup Y_0 \cup Z_0 \cup \{t_j \})^*$ joining $A$ to $B'$ and $D$ to $C'$, respectively.  Note that, since $lab([A, B]), lab([B, B']) \in (X_0 \cup Y_0 \cup Z_0)^*$, by the property (a$_j$), $lab([A, B']) \in (X_0 \cup Y_0 \cup Z_0)^*$. The same way we get $lab([D, C']) \in (X_0 \cup Y_0 \cup Z_0)^*$. Therefore, from the minimality assumption on $\|[A, B]\|+\|[C, D]\| (=\|p_{\Gamma}\| + \|p'_{\Gamma}\|)$ it follows that $\|[A, B]\| = \|[A, B']\|$ and $\|[D, C]\|=\|[D, C']\|$, which means that we can simply assume that $B=B'$ and $C=C'$. Consequently, combining this observation with Lemma \ref{lemma-truncated-diagrams} (note that since $\| \check{q}_{\Gamma}\| \geq \mu_{i+1} \|R_{i+1}\|/2$, by LPP, we can assume that $\| \check{q}_{\Gamma}\|\geq \lambda_{i+1} (2\epsilon_{i+1}+ 2\epsilon_{i}+24\mu_i\|R_{i}\|)+c_{i+1}$, so that Lemma \ref{lemma-truncated-diagrams} can be applied) and with the observation that there is no $\mathcal{R}_j$-cell $\Pi_0$ in $\Gamma$ connected to $[A, D]$ by a $\epsilon_j$-contiguity subdiagram $\Gamma_0$ such that $(\Pi_0, \Gamma_0, [A, D])\geq \rho_j$, we conclude that $\Gamma$ does not contain an $\mathcal{R}_j$-cell. Therefore, $\Gamma$ is a diagram over $H_j=\langle X \cup \{t_j\} \rangle$. But since $\partial\Gamma$ does not contain an edge with a label from $\{t_j^{\pm 1}\}$, we conclude that, in fact, $\Gamma$ is a diagram over $G_{j-1}$, which contradicts the minimality assumption on $j$. Since $\Gamma$ cannot be a diagram over $G_0$, we conclude that such a $\Gamma$ does not exist. Thus Case 1 is proved.

~\\
\textit{Case 2.} ($\check{q}_{\Gamma}$ contains exactly one edge with label from $\{t_{i+1}^{\pm 1} \}$). In this case, there exists exactly one $t_{i+1}$-band joining $AD$ to $AB$ or to $DC$. Without loss of generality let us assume that there is an edge $e$ on $AD$ and an edge $e'$ on $DC$ labeled by $t_{i+1}^{\pm 1}$ such that they are connected by a $t_{i+1}$-band. Let us denote the side $[e_+, e'_+]$ of this $t_{i+1}$-band by $q_2$. Also let us denote the diagram between $q_2$, $\check{q}_{\Gamma}$ and $p'_{\Gamma}$ by $\Gamma'$. See Figure \ref{fig:  Contiguity-diagram_Problem-7_5}.
	
	Then, since $\partial\Gamma'$ does not contain an edge with label $t_{i+1}^{\pm 1}$, from Case 1 it follows that $\big\| [e_+,  D]\big\| < \mu_{i+1} \|R_{i+1}\|/2$. Otherwise, since $\Gamma'$ is  $\epsilon_{i+1}$-contiguity subdiagram as well and $\partial \Gamma'$ does not contain edges labeled by $t_{i+1}^{\pm 1}$, its existence is impossible as it is shown in Case 1.
		
	  Therefore, since $\|AD\| \geq \mu_{i+1} \|R_{i+1}\|$, for the point $D'' \in AD$ such that $\|AD''\|=\lceil \mu_{i+1} \|R_{i+1}\|/2\rceil$, we get that $D''$ is between $A$ and $e_-$, i.e., $AD''$ does not contain an edge with label $t_{i+1}^{\pm 1}$, and  by Corollary \ref{corollary on hausdorff distance between quasi-geodesics}, we get that there is a point $C'' \in BC$ such that $d(D'', C'') \leq 2R_{\lambda_{i+1}, c_{i+1}}+2\delta'_{i+1} \leq^{\text{by LPP}}  \epsilon_{i+1}$. This means that the $\epsilon_{i+1}$-contiguity subdiagram $ABC''D''$ satisfies all the conditions put on $\Gamma$, and since $AD''$ does not contain an edge with label $t_{i+1}^{\pm 1}$, we already showed that this cannot happen. See Figure \ref{fig:  Contiguity-diagram_Problem-7_5} for visual description.\\
	\begin{figure}[H]
						\centering
						\includegraphics[clip, trim=2cm 19.3cm 3cm 3.35cm, width=1\textwidth]{{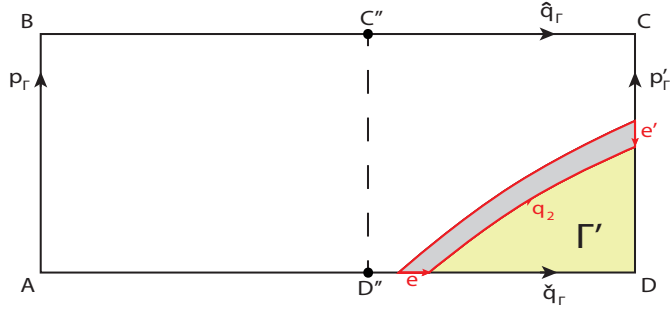}} 
						\caption{The case when $\check{q}_{\Gamma}$ contains an edge $e$ with a label from $t_{i+1}^{\pm 1}$ joined by a  $t_{i+1}$-band to $p'_{\Gamma}$.}
						\label{fig:  Contiguity-diagram_Problem-7_5}
	\end{figure}	
	~\\
	\textbf{(b$'_{i+1}$)}. Suppose that $\Gamma$ is a \text{truncated} $\epsilon_{i+1}$-contiguity subdiagram satisfying the conditions from the statement of (b$'_{i+1}$). Now let $\partial \Gamma = ABCD$, where $AB = p_{\Gamma}$, $BC = \hat{q}_{\Gamma}$, $DC = p'_{\Gamma}$ and $AD = \check{q}_{\Gamma}$ as it was in in (b$_{i+1}$) (see Figure \ref{fig:  Contiguity-diagram_Problem-7_5}).
	
		 Assume that $\Gamma$ contains $t_{i'}$-bands with both ends on $\partial\Gamma$ for some $1 \leq i' \leq i+1$. By  (b$_{i+1}$), there is no $t_j$-band in $\Gamma$ with both of its ends on $AB \cup DC$. Also, since $lab(\check{q}_{\Gamma})$ is a subword of a word $R_i \in \mathcal{R}_i$, we get that it can contain maximum one end of $t$-bands, where $t \in \{t_1, \ldots, t_{i+1} \}$  (more precisely, it must be that $t=t_{i+1}$).
		 
		 First, let us assume that there is no $t_{i+1}$-band with one of its ends on $AD=\check{q}_{\Gamma}$. Then all  $t$-bands of $\Gamma$ have their ends on $AB \cup BC \cup DC$, and no band has its sides on the same edge. Let $e_1$ and $e_2$ be edges on $AB$ and $DC$, respectively, such that they are ends of some $t$-bands and $[A, A']$ and $[D, D']$ do not contain  ends of $t$-bands, where $A'=(e_1)_-$ and $D'=(e_2)_-$. Let $e_1', e_2' \in BC$ be the other ends of these bands, respectively. Denote  $B'=(e_1')_-$ and $D'=(e_2)_-$. Also denote the subdiagram $AA'B'C'D'D$ by $\Gamma'$. See Figure \ref{fig:  b'_i-1}.
		 
		 \begin{figure}[H]
						\centering
						\includegraphics[clip, trim=2cm 10.3cm 3cm 10.35cm, width=1\textwidth]{{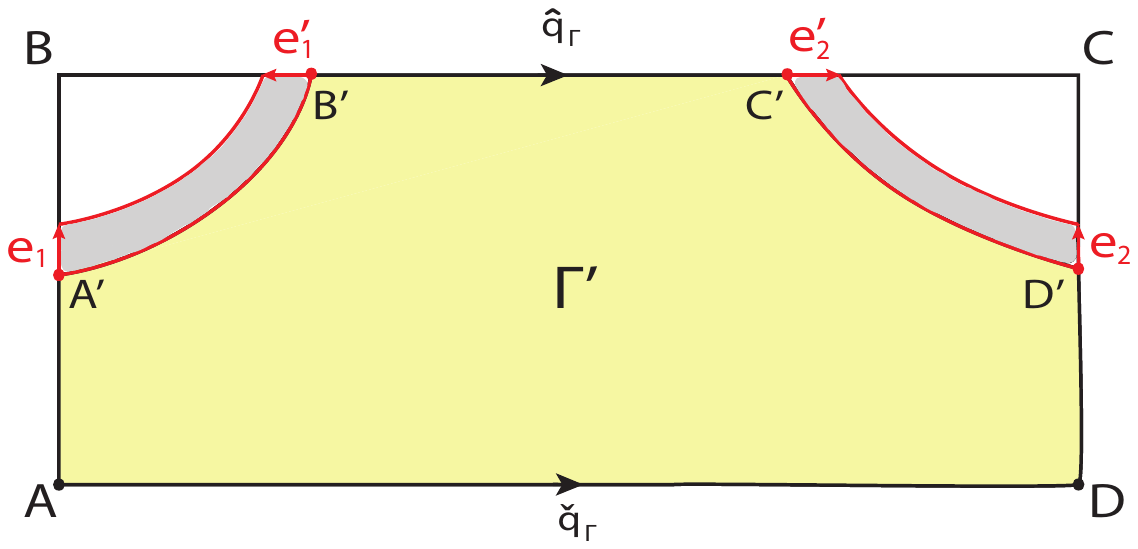}} 
						\caption{}
						\label{fig:  b'_i-1}
	\end{figure}	
		  Since $\Gamma'$ does not contain $t$-bands, it is either a diagram over $G_0$ or $\Gamma'$ contains an $\mathcal{R}_j$-cell for some $1\leq j \leq i$. Let us consider these two cases separately.
		  ~\\
		  \textit{Case 1.} If the first case holds, then, since $G_0$ is a free group, we get $Area(\Gamma')=0$, in which case, since $lab([A', B']), lab([D', C']) \in X_0^* \cup Y_0^*$ and $lab([A, D]) \in Z_0^*$, we get that $A, D \in [B', C']$, but this contradicts with the assumption that $\Gamma$ is truncated.
		  ~\\
		  \textit{Case 2.} Now assume that $\Gamma'$ contains an $\mathcal{R}_j$-cell for some $1\leq j \leq i$ and $j$ is chosen to be maximal. Then, since the sides $[A, A']$, $[A', B']$, $[B', C']$, $[C', D']$, $[D', D]$ and $[A, D]$ of $\Gamma'$ are $(\lambda_j, c_j)$-quasi-geodesic in $\Gamma(H_j, X_0\cup Y_0\cup Z_0 \cup \{t_1, \ldots, t_j \} )$, by Lemma \ref{lem 6.6}, we get that $\Gamma'$ contains an essential $\mathcal{R}_j$-cell $\Pi$, connected to $[A, A']$, $[A', B']$, $[B', C']$, $[C', D']$, $[D', D]$ and $[A, D]$ by essential $\epsilon_j$-contiguity subdiagrams $\Gamma_1$, $\Gamma_2$, $\Gamma_3$, $\Gamma_4$, $\Gamma_5$ and $\Gamma_6$, respectively. See Figure \ref{fig:  b'_i-2}.
		 \begin{figure}[H]
						\centering
						\includegraphics[clip, trim=2cm 10.3cm 3cm 10.35cm, width=1\textwidth]{{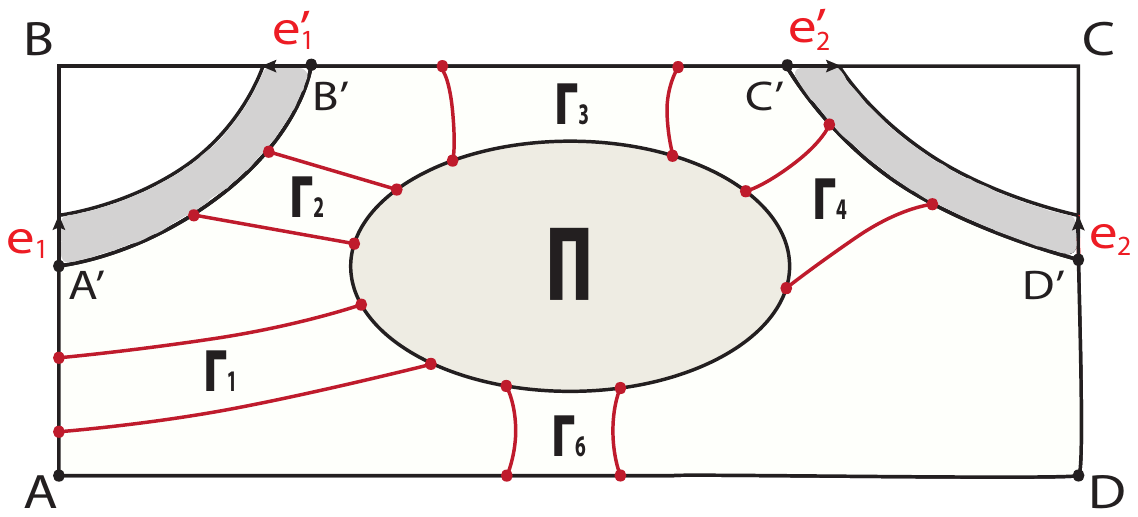}} 
						\caption{}
						\label{fig:  b'_i-2}
	\end{figure}	
		 Since $\|AD\|\geq \mu_{i+1}\rho_{i+1}$, for sparse enough standard parameters we get that at least one of $\Gamma_1$ and $\Gamma_5$ must be empty, because, otherwise, by triangle inequality it would be that
		  $$\|AD\|\geq \mu_{i+1}\rho_{i+1} \geq 2\epsilon_{i+1}+2\epsilon_j+\|\Pi\|,$$
		 which we can assume to be wrong  by LPP.
		 
		 Therefore,  without loss of generality we can assume that $\Gamma_5$ is empty. Now, from Lemma \ref{lem on words with cancellation condition} it follows that $(\Pi, \Gamma_6, [A, D]) \leq \rho_j$. Also,  because of the fact that $lab([A', B']), lab([D', C']) \in X_0^* \cup Y_0^*$, by (b$_{i+1}$), we get that 
		 \begin{align*}
		 	(\Pi, \Gamma_2, [A', B']), (\Pi, \Gamma_4, [C', D'])\leq \rho_j.
		 \end{align*}	
		 Therefore, we get
		 \begin{align*}
		 	(\Pi, \Gamma_1, [A, A'])+(\Pi, \Gamma_3, [B', C']) > 1-26 \rho_j,
		 \end{align*}
		 which is impossible because of Lemma \ref{lemma-truncated-diagrams}. Thus we showed that $Area(\Gamma)=0$ when $\check{q}_{\Gamma}$ does not contain an edge which is an end of a $t$-band for $t \in \{t_1, \ldots, t_{i+1} \}$. The case when 	 $\check{q}_{\Gamma}$ contains such an edge can be treated in a similar way.
	~\\
\textbf{(c$_{i+1}$)}. Now let us turn to the part (c$_{i+1}$) of the statement. By contradiction assume that there exists a reduced word $w  \in X_0^*$ which is not geodesic in  $\Gamma\big(G_{i+1}, X_0 \cup Y_0 \cup Z_0 \cup \{t_{i+1}\}\big)$. Without loss of generality we can assume that $w$ is the shortest one among such words. Then, since $w$ is not geodesic in $\Gamma\big(G_{i+1}, X_0 \cup Y_0 \cup Z_0 \cup \{t_{i+1}\}\big)$,  there exists a word $w' \in \big(X_0 \cup Y_0 \cup Z_0 \cup \{t_{i+1}\} \big)^*$ which is geodesic in   $\Gamma\big(G_{i+1}, X_0 \cup Y_0 \cup Z_0 \cup \{t_{i+1}\}\big)$ and  $\|w'\|< \|w\|$ and $w'=_{G_{i+1}} w$. 
 Since $w^{-1}w' = _{G_{i+1}} 1$, there exists a reduced diagram $\Delta$ over $G_{i+1}$ with the boundary label $w^{-1}w' $. Let us denote $\partial \Delta = pq$, where $lab(p) = w'$, $lab(q)=w$. By (c$_i$), $w$ is geodesic in $\Gamma\big(G_{i}, X_0 \cup Y_0 \cup Z_0 \cup \{t_{i}\}\big)$, hence the inequality $\|w'\| < \|w\|$ implies $w'\not\in (X_0 \cup Y_0 \cup Z_0)^*$, i.e. $w'$ contains a letter from $\{t_{i+1}^{\pm 1} \}$.
In particular, this means that $w \neq_{G_i} w'$.

 On the other hand, since $w'$ is geodesic in $\Gamma\big(G_{i+1}, X_0 \cup Y_0 \cup Z_0 \cup \{ t_{i+1}\}\big)$, there is no $t_{i+1}$-band in $\Delta$ which starts and ends on $p$ (otherwise, we will obtain a contradiction as in the proof of part (a$_{i+1}$)). Also, since $w$ does not contain any letter from $\{t_{i+1}^{\pm 1}\}$, by (a$_{i+1}$) it follows that $\Delta$ does not contain $t_{i+1}$-bands at all. Therefore, $w \neq_{H_{i+1}}w'$, because,  since $w \neq_{G_{i}} w'$, if $w=_{H_{i+1}}w'$ then $\Delta$ would contain a $t_{i+1}$-band. Therefore, $\Delta$ contains an $\mathcal{R}_{i+1}$-cell.
 
  Let $w=w_0x$, where $x \in X_0$. Denote the subword of $q$ with the label $w_0$ by $q_0$ and the one with the label $x$ by $q_1$. Since we chose $w$ to be of minimal length with the mentioned properties, it must be that $w_0$ is a geodesic word in $\Gamma\big(G_{i+1}, X_0 \cup Y_0 \cup Z_0 \cup \{t_{i+1}\}\big)$.  Therefore, $\partial \Delta$ is a geodesic triangle in $\Gamma\big(G_{i+1}, X_0 \cup Y_0 \cup Z_0 \cup \{ t_{i+1}\}\big)$ with geodesic sides $p$, $q_0$ and $q_1$. Therefore, by Lemma \ref{lem 6.6}, $\Delta$ contains an essential $\mathcal{R}_{i+1}$-cell  $\Pi$ connected to $p$, $q_0$ and $q_1$ by a system of essential $\epsilon_{i+1}$-contiguity subdiagrams $\Gamma_1$, $\Gamma_2$ and $\Gamma_3$, respectively. 
   See Figure \ref{fig:  problem_7-5_diagram_1}. 
		\begin{figure}[H]
			\centering
			\includegraphics[clip, trim=0.8cm 16.9cm 0cm 5.9cm, width=1\textwidth]{{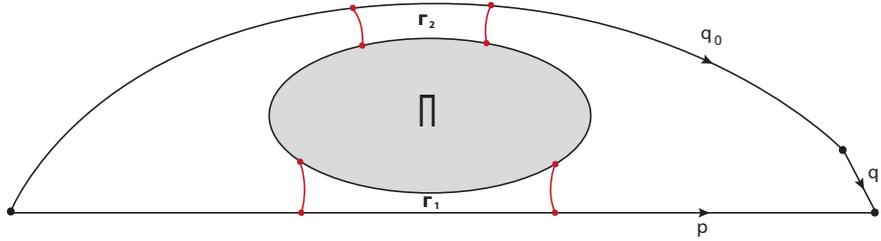}} 
			\caption{$\Delta:$ $lab(q_0q_1)=w$, $lab(p)=w'$.} 
			\label{fig:  problem_7-5_diagram_1}
		\end{figure}
		
	From (b$_{i+1}$) it follows that $(\Pi, \Gamma_2, q_0) < \mu_{i+1}$ and $(\Pi, \Gamma_3, q_1) < \mu_{i+1}$. 	Therefore,
	$(\Pi, \Gamma_1, p) > (1-23\mu_{i+1})-2\mu_{i+1} = 1-25\mu_{i+1}$. But since $p$ is geodesic in  $\Gamma\big(G_{i+1}, X_0 \cup Y_0 \cup Z_0 \cup \{t_{i+1}\}\big)$, by Lemma \ref{lemma on Greendlinger's cell}, for sparse enough standard parameters, this is impossible. A contradiction. 
	
		~\\
		
	\textbf{(d$_{i+1}$)}. Let $U \in X_0^*$ be  a reduced word which is not a proper power of any other word from $G_0$. Without loss of generality assume that $U$ is a freely cyclically reduced word. By contradiction let us assume that for some $k \geq 2$ and $W \in (X_0 \cup Y_0 \cup Z_0 \cup \{t_{i+1}\})^*$, $U=_{G_{i+1}} W^k$.
	First of all, it directly follows from Lemma \ref{lemma_about_proper_powers_in_hnn} and from the inductive hypothesis (i.e. by the statement (d$_i$)) that
	$U$ is not a proper power in $H_{i+1}$.	
	
	Now, let $W' \in \big(X_0 \cup Y_0 \cup Z_0 \cup \{t_{i+1}\}\big)^*$ be a cyclically minimal representative of $W$ in $G_{i+1}$. This means that there exists $T\in \big(X_0 \cup Y_0 \cup Z_0 \cup \{t_{i+1}\}\big)^*$ such that $W=_{G_{i+1}} T W' T^{-1}$ and $W'$ has minimal length among such words. In particular, this means that $U=_{G_{i+1}}  T(W')^k T^{-1}$, and $W'$ is cyclically geodesic in $\Gamma \big(G_{i+1}, X_0 \cup Y_0 \cup Z_0\cup \{t_{i+1}\}\big)$. Note that, since $G_{i+1}$ is a quotient of $H_{i+1}$, we get that $W'$ is cyclically geodesic in $\Gamma \big(H_{i+1}, X_0 \cup Y_0 \cup Z_0 \cup \{t_{i+1}\}\big)$ as well. Therefore, by Lemma \ref{lem 1.11} and by LPP, $(W')^k$ is cyclically $(\lambda_{i+1}, c_{i+1})$-quasi-geodesic in $\Gamma\big(H_{i+1}, X_0 \cup Y_0 \cup Z_0 \cup \{t_{i+1}\}\big)$.
	
	Since $W'$ is conjugate to $W$ in $G_{i+1}$ and $U=_{G_{i+1}} W^k$, there exists a $(U, (W')^k)$-conjugacy diagram over $G_{i}$. Hence there exists a cyclically slender $(U, (W')^k)$-conjugacy diagram over $G_{i}$. Let $\Delta$ be such a diagram. As before, let us denote $\partial \Delta = ABCD$, where $lab(BC), lab(AD)$ are cyclic shifts of $(W')^k$ and $U$, respectively, and $lab(AB)=lab(DC)$ are geodesic words in $\Gamma\big(G_{i+1}, X_0 \cup Y_0 \cup Z_0 \cup \{t_{i+1}\}\big)$. Note that by  (c$_{i+1}$), $U$ is also cyclically geodesic in $\Gamma\big(H_{i+1}, X_0 \cup Y_0 \cup Z_0 \cup \{t_{i+1}\}\big)$. Therefore, by Lemma \ref{lem 6.6}, $\Delta$ contains an essential $\mathcal{R}_{i+1}$-cell, $\Pi$. Let  $\Gamma_1$, $\Gamma_2$, $\Gamma_3$ and $\Gamma_4$ be essential $\epsilon_{i+1}$-contiguity subdiagrams connecting $\Pi$ to $AB$, $BC$, $CD$ and $DA$, respectively. Since we chose $\Delta$ to be cyclically slender, by Lemma \ref{lemma_about_slender_conjugacy_diagrams}, $\Gamma_2$ and $\Gamma_4$ are non-empty and 
	\begin{equation}
	\label{iiiii}
		(\Pi, \Gamma_2, BC)+ (\Pi, \Gamma_4, DA) \geq 1-121 \lambda_{i+1} \mu_{i+1}.
	\end{equation}
		Also, by statement (b$_{i+1}$) and (c$_{i+1}$) of the current lemma and by  LPP, since $lab(\hat{q}_{\Gamma_4}) \in X_0^*$, it follows that 
		\begin{equation*}
			(\Pi, \Gamma_4, DA) < \mu_{i+1}.
		\end{equation*}				
		Combining this with (\ref{iiiii}), we get
		\begin{equation}
		\label{ineq ****}
			(\Pi, \Gamma_2, BC)  > (1-121 \lambda_{i+1} \mu_{i+1}) - \mu_{i+1} > 1-122 \lambda_{i+1} \mu_{i+1}.
		\end{equation}
Therefore, since $W'$ is cyclically geodesic in $\Gamma \big(G_{i+1}, X_0 \cup Y_0 \cup Z_0 \cup \{t_{i+1}\}\big)$, by LPP and by Lemma \ref{lemma on Greendlinger's cell}, we get that $lab(\hat{q}_{\Gamma_2})$ is not a subword of a cyclic shift of $W'$. 
 This means that $lab(\hat{q}_{\Gamma_2})$ is of the form
		\begin{align*}
			lab(\hat{q}_{\Gamma_2}) = (W'')^n Q,
		\end{align*}
		where $W''$ is a cyclic shift of $W'$, $n \geq 1$, and $Q$ is a prefix of $W''$.
		
		Let us separately consider the cases when $n=1$ and when $n>1$.
		
		Before that, let us notice that by Corollary \ref{corollary on hausdorff distance between quasi-geodesics}, the Hausdorff distance between $\check{q}_{\Gamma_2}$ and $\hat{q}_{\Gamma_2}$ is bounded from above by $\epsilon_{i+1}+2R_{\lambda_{i+1}, c_{i+1}}+2\delta_{i+1} <^{\text{by LPP}} 2\epsilon_{i+1}$.\\
		~\\
		\textit{Case 1.} ($n=1$). For this case, let us partition $\hat{q}_{\Gamma_2} =  \hat{q}_{1} \hat{q}_{2} \hat{q}_{3}$, where $lab(\hat{q}_{1})=lab(\hat{q}_{3}) = Q$.  Let us also partition $\check{q}_{\Gamma_2}=  \check{q}_{1} \check{q}_{2} \check{q}_{3}$ such that $(\check{q}_{1})_+$ and $(\check{q}_{2})_+$ are the closest to $(\hat{q}_{1})_+$ and to $(\hat{q}_{2})_+$ points on $\check{q}_{\Gamma_2}$, respectively.  Since $lab(\hat{q}_{1})=lab(\hat{q}_{3})$, from the observation right above Case 1 and from Lemma \ref{lemma_about_contiguity_arcs_with_the_same_label}, it follows that $\|\check{q}_{1}\|, \|\check{q}_{ 3}\| \leq 2\mu_{i+1} \|\Pi\| < \mu_{i+1} \lambda_{i+1} \|\Pi\|$. Then combining this with (\ref{ineq ****}), we get that $\|\check{q}_{1} \check{q}_{2}\| > (1-23\lambda_{i+1}\mu_{i+1})\|\Pi\|$. But, since $\|W''\|= \big\|\check{q}_{1} \check{q}_{2} \big\|$ and $W''$ is a geodesic word in $\Gamma \big(G_{i+1}, X_0 \cup Y_0 \cup Z_0 \cup \{t_{i+1}\} \big)$, this is impossible for sparse enough standard parameters. 
		Thus we are done with  Case 1.\\
		~\\
		\textit{Case 2.} ($n>1$). For this case, again we partition $\hat{q}_{\Gamma_2}$ into three parts $\hat{q}_{\Gamma_2} =  \hat{q}_{1} \hat{q}_{2} \hat{q}_{3}$ such that  $lab(\hat{q}_{1})=lab(\hat{q}_{3})$ and  $lab(\hat{q}_{2})$ is a suffix of $W''$. Then, since $lab(\hat{q}_{\Gamma_2}) = (W'')^n Q $ and $n \geq 2$, we get that $\| \hat{q}_{1}\|=\|\hat{q}_{3}\| > \frac{1}{3} \|\hat{q}_{2}\|$, hence $\| \hat{q}_{\Gamma_2}\| < 3 \| \hat{q}_{1}\|$. Also, just like we showed in case $n=1$, by Lemma \ref{lemma_about_contiguity_arcs_with_the_same_label}, in this case also $\| \hat{q}_{1}\|, \| \hat{q}_{3}\| \leq 2\mu_{i+1} \|\Pi\|$. Therefore, $\| \hat{q}_{\Gamma_2}\|< 6\mu_{i+1} \|\Pi\|$. But, since {\text{by LPP}}  $1-122 \lambda_{i+1} \mu_{i+1}> 6\mu_{i+1} $, we get a contradiction with (\ref{ineq ****}). The case when $U \in Y_0^*$ can be dealt in the same way. Thus we are done with this case as well.\\	
		~\\
		\textbf{(e$_{i+1}$)}. By contradiction, let us assume that for some non-trivial reduced words $U \in X_0^*$, $V \in Y_0^*$ we have $U = _{G_{i+1}} V$. Then there exists a reduced van Kampen diagram $\Delta$ such that $\partial \Delta = q_1 q_2^{-1}$ and $lab(q_1)=U$, $lab(q_2)=V$. Since  $U$ and $V$ do not contain letters from $\{t_{i+1}^{\pm 1} \}$ and $U \neq _{G_{i}} V$, by  (e$_i$) we have that $\Delta$ contains an $\mathcal{R}_{i+1}$-cell. Therefore, since by (c$_{i+1}$) $U$ and $V$ are geodesic word in $\Gamma\big(H_{i+1}, X_0\cup Y_0 \cup Z_0 \cup \{t_{i+1}\}\big)$, by Lemma \ref{lem 6.6}, $\Delta$ contains an essential $\mathcal{R}_{i+1}$-cell, $ \Pi$. Let us assume that $\Pi$ is connected to $q_1$ and $q_2$ by $\epsilon_{i+1}$-contiguity subdiagrams $\Gamma_1$ and $\Gamma_2$, respectively. Then we have $(\Pi, \Gamma_1, q_1)+ (\Pi, \Gamma_2, q_2) \geq 1-23\mu_{i+1}$. But, on the other hand, by (b$_{i+1}$) we have that
		    $	(\Pi, \Gamma_1, q_1)+ (\Pi, \Gamma_2, q_2) < 2\mu_{i+1}$.
		 		    But since by LPP we can assume
		    \begin{equation*}
		    	2\mu_{i+1} < 1-23 \mu_i,
		    \end{equation*}
		    we get a contradiction.\\
		    ~\\
		   \textbf{ (f$_{i+1}$)}. The fact that $H_i$ is a hyperbolic group follows from Theorem \ref{th hnn extension} and parts (d$_{i+1}$) and (e$_{i+1}$) of the current lemma. 
	\end{proof}
	
	\begin{corollary}
	\label{corollary---*}
		Suppose $U \in X^*$ is a $(\lambda_i, c_i, \epsilon_i, 1-122\lambda_i\mu_i)$-cyclic-reduced word for $i=\mathcal{I}(\|U\|)$, and $U =_{G_{\mathcal{L}}} V$ for some $V \in X_0^* \cup Y_0^*$. Then $U \in X_0^* \cup Y_0^*$.
	\end{corollary}
	\begin{proof}
		This directly follows from properties (a$_i$),(b$_i$) and Lemma \ref{lem 6.6}.
	\end{proof}

			\begin{lemma}
				\label{lemma-*}
				Two cyclically reduced non-empty words $U, V \in X^* \cup Y^*$ are conjugate one to each other in $G_{\mathcal{L}}$ if and only if 
				$(U, V)$ is a $\Lambda$-pair. Moreover, if $U$ is not a cyclic shift of $V$, then $U$ is not $G$-conjugate to $V$ in $G_{\mathcal{L}}$.
			\end{lemma}
			\begin{proof}
				Assume that $U$ and $V$ freely reduced words which are conjugate in $G_{\mathcal{L}}$. The case when $U$ is a cyclic shift of $V$ is trivial. Therefore, without loss of generality assume that $U$ is not a cyclic shift of $V$. Then, by the definition of $G_0$, it is clear that $U$ and $V$ are not conjugate in $G_0$. Therefore, there exists a minimal index $i \geq 1$ such that $U$ and $V$ are conjugate in $G_i$. This means that there exists a minimal cyclically slender $(U, V)$-conjugacy diagram $\Delta$ of rank $i$. If $i=0$ then the statement of the lemma follows from basic properties of free groups. Suppose that $U, V$ are chosen so that the corresponding index $i\in \mathbb{N}$ is minimal. Now, let us assume that $i \geq 1$ and apply induction on $i$.\\
				
				As usual, let us denote the boundary $\partial\Delta$ of $\Delta$ by $ABCD$. Let $U'=lab(BC)$ and $V'=lab(AD)$, where $U'$ and $V'$ are some cyclic shifts of $U$ and $V$, respectively. \\
				~\\
				\textbf{Claim 1.} $U$ is not $G$-conjugate to $V$ in $G_{\mathcal{L}}$.
				\begin{proof}[Proof of the claim]
					Indeed, assume that $U$ is $G$-conjugate to $V$ in $G_{\mathcal{L}}$. Then, since by the property (c$_{i}$), $U$ and $V$ are geodesic words in $\Gamma(H_i, X\cup \{t_1, \ldots, t_j\})$, according to Lemma \ref{lemma_about_slender_conjugacy_diagrams}, we get that at least one of $U'$ and $V'$ must contain a $(\epsilon_i, (1-121\lambda_i \mu_i)/2)$-subword, which contradicts to property (b$_{i}$).
				\end{proof}

				By Claim 1, we get that $\Delta$ is a slender $(U, V)$-conjugacy diagram over $H_i$. Therefore, since we chose the index $i$ to be minimal, $\Delta$ contains $t_i$-bands. Since $lab(AB)$ and $lab(DC)$ do not contain letters from $\{t_i^{\pm 1} \}$, we get that the $t_i$-bands of $\Delta$ must be horizontal, i.e. their ends belong to $[A, B]$ and $[D, C]$. 
				
								Now let us choose an edge $e_1$ on the side $AB$ such that $lab(e_1) \in \{t_i^{\pm 1}\}$ and $lab([(e_1)_+, B])$ does not contain $t_i^{\pm 1}$. From the basic properties of HNN-extensions, it it follows that there exists an edge $e_1'$ on $DC$ such that $lab(e_1) \in \{t_i^{\pm 1}\}$  and $e_1$ is connected to $e_1'$ by a $t_i$-band. Moreover, $lab([(e_1)_+, B]) = lab([(e_1')_+, C])$.
								
								 Let us denote the side of the $t_i$-band connecting $(e_1)_+$ to $(e'_1)_+$ by $p_1$ and the side connecting $(e_1)_-$ to $(e_1')_-$ by $q_1$. See Figure \ref{fig:  Conjugacy-diagram_Problem-7_5}. Then $lab(p_1)$ belongs to either $X_0^*$ or $Y_0^*$. Denote $U''=lab(p_1)$. 
					\begin{figure}[H]
						\centering
						\includegraphics[clip, trim=0cm 12.1cm 1cm 4.2cm, width=.6\textwidth]{{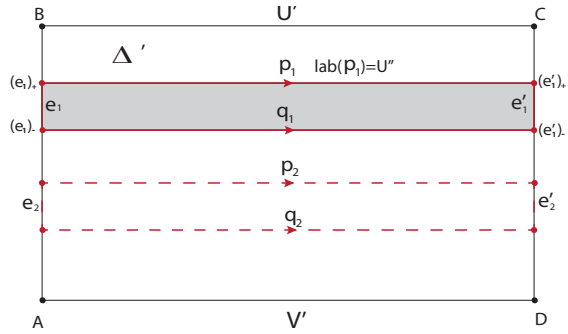}} 
						\caption{The cyclically slender $(U, V)$-conjugacy diagram $\Delta$. Below it is shown that the second $t_i$-band, joining $e_2$ to $e_2'$, actually, cannot exist.} 
						\label{fig:  Conjugacy-diagram_Problem-7_5}
					\end{figure}
					Let us denote by $\Delta'$ the $(U', U'')$-conjugacy subdiagram of $\Delta$ which is bounded between $(e_1)_+$, $B$, $C$ and $(e_1)'_+$. See Figure \ref{fig:  Conjugacy-diagram_Problem-7_5}.
				Since $\Delta'$ does not contain $R_i$-cells and $t_i$-bands, we conclude that it is a $(U', U'')$-conjugacy diagram of rank $j$ where $0 \leq j < i$, hence $U'$ is conjugate to $U''$ in $G_j$. On the other hand, since $U'' \in X^* $ or $Y^*$ and since $i$ was chosen to be minimal, by inductive argument on $i$ we conclude that either $U''$ is a cyclic shift of $U'$ (hence also of $U$) or $U'$ is a cyclic shift of $lab(q_1)$ (this means that $(U, lab(p_1))$ is a $\Lambda$-pair).
				
				Now we are going to show that besides the considered $t_i$-band, $\Delta$ does not contain any other $t_i$-band. Assume that this is not true. Then there exist edges $e_2$ and $e_2'$ on $AB$ and $DC$, respectively, such that they have a label from $\{t_i ^{\pm 1}\}$,  and between $e_2$ and $e_1$ (also between $e_2'$ and $e_1'$, respectively) there is no other edge with label from $\{t_i ^{\pm 1}\}$.  See Figure \ref{fig:  Conjugacy-diagram_Problem-7_5}. Then it must be that $e_2$ is connected to $e_2'$ by a $t_i$-band. Correspondingly, define $p_2$ and $q_2$ as we defined $p_1$ and $q_1$.  Then repeating the above stated arguments we get that $\big(lab(q_1), lab(q_1)\big)$ is a $\Lambda$-pair. The last observation implies that either $lab(p_1)\equiv lab(p_2)$ or $lab(p_1)\equiv lab(2_2)$, which is impossible by Lemma \ref{lemma_about_opposite_points_in_slender_conjugacy_diagrams}. Thus Lemma \ref{lemma-*} is proved.

				\end{proof}
				The next lemma is a stronger version of Lemma \ref{lemma-*}.
\begin{lemma}
	\label{lemma_important_auxiliary}
	Let $U \in (X_0 \cup Y_0 \cup Z_0)^*$ and $V \in X_0^* \cup Y_0^*$. Then $U \sim_{conj} V$ in $G_{\mathcal{L}}$ if and only if $(U', V)$ is a $\Lambda$-pair, where $U'$ is any $(\lambda_i, c_i, \epsilon_i, 1-122\lambda_i\mu_i)$-cyclic reduction of $U$ for $i= \mathcal{I}(\|U\|)$.
\end{lemma}
\begin{proof}
Suppose that $U\sim_{conj} V$ in $G_{\mathcal{L}}$. Let us fix a $(\lambda_i, c_i, \epsilon_i, 1-122\lambda_i\mu_i)$-cyclic reduction $U'$ of $U$. Then, clearly $U'\sim_{conj} V$ in $G_{\mathcal{L}}$. Let us separately consider two cases: Case 1 -- $U'$ is $G$-conjugate to $V$ in $G_{\mathcal{L}}$, and Case 2 -- $U'$ is $H$-conjugate to $V$ in $G_{\mathcal{L}}$.\\
~\\
\textit{Case 1.} ($U'$ is $G$-conjugate to $V$ in $G_{\mathcal{L}}$). In case $U\sim_{conj} V$ in $G_0$ the statement of the lemma is trivial. Now suppose that $U \not\sim_{conj} V$ in $G_0$. Then, by Lemma \ref{lemma-about-structure-of-CP-in-barG}, there exists an index $j\in \mathbb{N}$ such that $j\leq i$ and $U \sim_{conj} V$ in $G_j$ but $U \not\sim_{conj} V$ in $H_j$. Therefore, by Lemma \ref{lemma_about_slender_conjugacy_diagrams}, $U'$ and $V$ contain $(\epsilon_j, \kappa_1)$- and $(\epsilon_j, \kappa_2)$-arcs, respectively, such that $\kappa_1+\kappa_2 \geq 1-122\lambda_j\mu_j$. Also, since  $V \in X_0^* \cup Y_0^*$, by property (b$_j$), $\kappa_2 \leq \rho_j$. Therefore, $U'$ contains a $(\epsilon_j, 1-122\lambda_j\mu_j)$-subword, which contradicts to the fact that $U'$ is cyclically $(\lambda_j, c_j, \epsilon_j, 1-122\lambda_j\mu_j)$-reduced. 
So we are done with Case 1.\\
~\\
\textit{Case 2.} ($U'$ is $H$-conjugate to $V$ in $G_{\mathcal{L}}$). Then for some $k \in \mathbb{N}$, $U' \sim_{conj} V$ in $H_k$ and $U' \not\sim_{conj} V$ in $G_{k-1}$, and consequently, there exists a slender $(U', V)$-conjugacy diagram $\Delta$ over $H_k=\langle X\cup \{t_1, \ldots, t_k\} \rangle$ which contains at least one $t_k$-band. Note that, since $V$ does not contain a letter from $\{t_k^{\pm 1}\}$, $\Delta $ must contain only horizontal $t_k$-bands. Without loss of generality assume that $V \in X_0^* \cup Y_0^*$ is chosen so that $\Delta$ contains minimal number of horizontal $t$-bands, $t\in \{t_1, t_2, \ldots \}$. Therefore, by a standard inductive argument and by Lemma \ref{lemma-*}, we get that $U'$ and $V$ form a $\Lambda$-pair with a label of a side of any (horizontal) $t_k$-band from $\Delta$, hence, since $\Lambda$-pair relation is an equivalence relation, we get that  $(U', V)$ is a $\Lambda$-pair.
\end{proof}

\begin{lemma}
\label{Lemmmmma}
		Let $U, V \in (X_0 \cup Y_0 \cup Z_0)^*$. Suppose that $U', V' \in  (X\cup \{t_1, \ldots, t_j\})^*$ are any $(\lambda_i, c_i, \epsilon_i, 1-122\lambda_i\mu_i)$-cyclic-reductions of $U$ and $V$, respectively, where $i=\mathcal{I}(\|U\|+\|V\|)$. Then  $U$ is $H$-conjugate to $V$ in $G_{\mathcal{L}}$ if and only if $(U', V')$ is a $\Lambda$-pair and $U'$ is not a cyclic shift of $V'$.
	\end{lemma}
\begin{proof}
	First of all, if $(U, V)$ is  a $\Lambda$-pair, then, clearly, $U \sim_{conj} V$ in $G_{\mathcal{L}}$. Moreover, by Lemma \ref{lemma-*}, if $U$ is not a cyclic shift of $V$, then $U$ is not $G$-conjugate to $V$, hence $U$ is $H$-conjugate to $V$ in $G_{\mathcal{L}}$.
	
	Now let us assume that $U$ is $H$-conjugate to $V$ in $G_{\mathcal{L}}$. Then there exists an index $j\in \mathbb{N}$ such that  $U$ is conjugate to $V$ in $H_j$, but $U$ is not conjugate to $V$ in $G_{j-1}$. 
	This means that there exists a slender $(U, V)$-conjugacy diagram $\Delta$ over $H_j$ which contains a horizontal $t_j$-band. Therefore, $U$ and $V$ are conjugate to conjugate words from $X_0^* \cup Y_0^*$. Hence, by Lemma \ref{lemma_important_auxiliary} and by the fact that $\Lambda$-pair relation is an equivalence relation, we get that $(U', V')$ is a $\Lambda$-pair. 
\end{proof}
~\\
The next lemma is an obvious corollary from the structure of the words of $\Lambda(\mathcal{A}^*)$ and definition of $\Lambda$-pairs.
\begin{lemma}
\label{11111111111111}
	The decision problem which for any pair of words  $U, V \in (X_0 \cup Y_0 \cup Z_0)^*$ asks whether   or not $(U, V)$ is a $\Lambda$-pair can be strongly reduced to the membership problem for $\mathcal{L} \subseteq \mathcal{A}^*$  in $\bigO(\|U\|+\|V\|)$ time.
\end{lemma}

The combination of Lemma \ref{lemma-*} and Lemma \ref{Lemmmmma} implies the following proposition.
\begin{proposition}
\label{22222222222222}
	Suppose that $U, V \in (X_0 \cup Y_0 \cup Z_0)^*$.	Then, $U \sim_{conj} V$ if and only if exactly one of the following holds.
	\begin{enumerate}
		\item $U$ is $G$-conjugate to $V$ in $G_{\mathcal{L}}$;
		\item $(U', V')$ is a  $\Lambda$-pair and $U'$ is not a cyclic shift of $V'$, where $U', V' \in  (X_0 \cup Y_0 \cup Z_0 \cup \{ t_i \})^*$ are $(\lambda_i, c_i, \epsilon_i, 1-122\lambda_i\mu_i)$-cyclic reductions of $U$ and $V$, respectively, for $i=\mathcal{I}(\|U\|+\|V\|)$. 
	\end{enumerate}
\end{proposition}
\begin{proof}
	This proposition directly follows from Lemma \ref{lemma-*}  and Lemma \ref{Lemmmmma}.
\end{proof}

\begin{proposition}
	\label{main-proposition-1}
	The membership problem for $\mathcal{L} \subseteq \mathcal{A}^*$  can be strongly reduced to the conjugacy problem in $G_{\mathcal{L}}$ in linear time; and the $H$-conjugacy problem in $G_{\mathcal{L}}$ can be strongly reduced to the membership problem for $\mathcal{L} \subseteq \mathcal{A}^*$ in almost linear time.
\end{proposition}
\begin{proof}
	Indeed, it follows from the definition of $\Lambda$-pairs that for any $\omega \in \mathcal{A}^*$, $\omega \in \mathcal{L}$ if and only if the pair of words $\Lambda(\omega)$ is a $\Lambda$-pair. Therefore, since $\Lambda(\omega) \in Y_0^* \times Z_0^*$,  by Lemma \ref{lemma-*}, $\omega \in \mathcal{L}$ if and only if $\Lambda(\omega)$ is a pair of words conjugate in $G_{\mathcal{L}}$.
	
	Now let us show the opposite side. For that let us consider a pair of words $(U, V) \in X^* \times X^*$. Then one can find $(\lambda_i, c_i,\epsilon_i, 1-122\lambda_i\mu_i)$-cyclic-reductions $U'$ and $V'$ of  $U$ and $V$, respectively, in almost linear time, where $i=\mathcal{I}(\|U\|+\|V\|)$ (see Remark \ref{an-important-remark}). Therefore, since by Lemma \ref{Lemmmmma} $U'$ is $H$-conjugate to $V'$ in $G_{\mathcal{L}}$ if and only if $(U', V')$ is a $\Lambda$-pair, by Lemma \ref{11111111111111}, the $H$-conjugacy problem in $G_{\mathcal{L}}$ can be strongly reduced to the membership problem for $\mathcal{L} \subseteq \mathcal{A}^*$ in almost linear time.
\end{proof}




\subsubsection{Geometry of slender $G$-conjugacy diagrams and time complexity of the $G$-conjugacy problem in $G_{\mathcal{L}}$}

	
\begin{lemma}
\label{lemma-structure-of-g-diagrams-G-T}
	Let for some $i \in \mathbb{N}$, $U, V \in (X_0 \cup Y_0 \cup Z_0 \cup \{t_1, \ldots, t_i\})^*$ be $(\lambda_i, c_i, \epsilon_i, 1-122\lambda_i\mu_i)$-cyclic-reduced words in $\Gamma\big(H_i, X_0 \cup Y_0 \cup Z_0 \cup \{t_1, \ldots, t_i\} \big)$ and suppose $U \sim_{conj} V$ in $G_i$ but $U \not\sim_{conj} V$ in $H_i$. Then there exists a $(U, V)$-conjugacy diagram $\Delta$ over 
	\begin{align*}
		G_i=\langle X_0 \cup Y_0 \cup Z_0 \cup \{t_1, \ldots, t_i\} \mid  \mathcal{R}_j, t_j^{-1}u_j t_j v_j^{-1}, 1\leq j \leq i \rangle
	\end{align*}
	such that $\partial\Delta=ABCD$, $lab(AD) \equiv U$, $lab(BC) \equiv V$, $lab(AB) \equiv lab(DC)$ and for every cell $\Pi$ in $\Delta$, $\partial\Pi \cap AD, ~\partial\Pi \cap BC \neq \emptyset$. Moreover, if $\Pi$ is an $\mathcal{R}_j$-cell for some $1\leq j \leq i$, then $\|\Pi \cap AD \|, \|\Pi \cap BC \| \geq \rho_j \|\Pi\|$. Also, if $\Pi$ is a cell with label of the form $t_j^{-1}u_j t_j v_j^{-1}$, then $u_j$ is contained either in $lab(\Pi \cap AD)$ or in $lab(Pi \cap BC)$, and the same is true about $v_j$
\end{lemma}
\begin{proof}
	Let $\Delta_0$ be a reduced cyclically slender $(U, V)$-conjugacy diagram over $G_i$. Let $\partial\Delta_0=A_0B_0C_0D_0$ be such that $lab([A_0, D_0])=U'$ and $lab([B_0, C_0])=V'$ for some cyclic shifts $U'$ and $V'$ of $U$ and $V$, respectively. Then, by Lemma \ref{lemma_about_slender_conjugacy_diagrams}, there exists an $\mathcal{R}_i$-cell $\Pi$ connected by non-empty $\epsilon_i$-conjugacy subdiagrams $\Gamma$ and $\Gamma'$ to $A_0D_0$ and $B_0C_0$ such that $(\Pi, \Gamma, A_0D_0)+ (\Pi, \Gamma', B_0C_0)\geq 1-121\lambda_i\mu_i$. Without loss of generality assume that $\Gamma$ and $\Gamma'$ are truncated. Now, since $U'$ and $V'$ are $(\lambda_i, c_i, \epsilon_i, 1-122\lambda_i\mu_i)$-reduced, we get $(\Pi, \Gamma, A_0D_0), (\Pi, \Gamma', B_0C_0) < 1-122\lambda_i\mu_i$. Therefore,
	\begin{align*}
		(\Pi, \Gamma, A_0D_0), (\Pi, \Gamma', B_0C_0) > \mu_i.
	\end{align*}
	Hence, by property (b$'_i$), we get that $Area(\Gamma)=Area(\Gamma')=0$.
	\begin{align*}
		\|W\|, \|W'\| \geq \mu \|R_i'\|.
	\end{align*}
	
	{Now the proof of Lemma \ref{lemma-structure-of-g-diagrams-G-T}
follows after applying some standard inductive arguments.
}\end{proof}

Visually, Lemma \ref{lemma-structure-of-g-diagrams-G-T}
 tells us that if for some $i \in \mathbb{N}$, $U, V \in (X_0 \cup Y_0 \cup Z_0 \cup \{t_1, \ldots, t_i\})^*$ are $(\lambda_i, c_i, \epsilon_i, 1-122\lambda_i\mu_i)$-cyclic-reduced words in $\Gamma\big(H_i, X_0 \cup Y_0 \cup Z_0 \cup \{t_1, \ldots, t_i\} \big)$ and  $U \sim_{conj} V$ in $G_i$ but $U \not\sim_{conj} V$ in $H_i$, then there exists a $(U, V)$-conjugacy diagram $\Delta$ which looks like in Figure \ref{fig:  G-conjugacy-diagram-for-G-T}, where by $\Pi_1$, \ldots $\Pi_k$ we denoted the cells of $\Delta$.

 \begin{figure}[H]
			\centering
			\includegraphics[clip, trim=0.8cm 17.9cm 0cm 5.9cm, width=1\textwidth]{{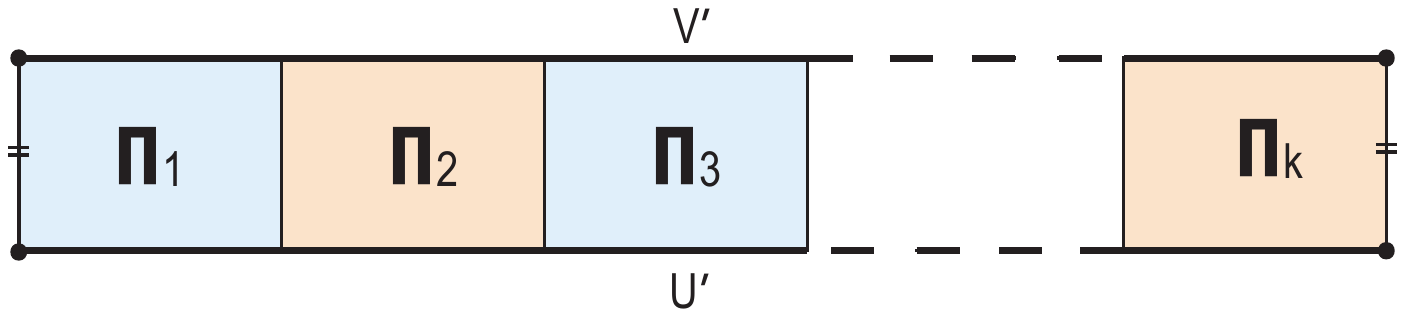}} 
			\caption{} 
			\label{fig:  G-conjugacy-diagram-for-G-T}
		\end{figure}
 
 
 \begin{lemma}
 \label{lemma-G-conj-in-WP-CP}
 	The $G$-conjugacy problem in $G_{\mathcal{L}}$ is decidable in almost linear time.
 \end{lemma}
 \begin{proof}
 	Indeed, in order to check for a given pair of words $(U, V) \in X^* \times X^*$ whether or not $U$ is $G$-conjugate to $V$ in $G_{\mathcal{L}}$, one can first compute $(\lambda_i, c_i, \epsilon_i, 1-122\lambda_i\mu_i)$-cyclic-reductions $U'$ and $V'$ of $U$ and $V$, respectively, in almost linear time (see Remark \ref{an-important-remark}), then check does there exist a $(U', V')$-conjugacy diagram satisfying the properties described in Lemma \ref{lemma-structure-of-g-diagrams-G-T}. Notice that the last checking can be done in $\bigO(\|U\|+\|V\|)$ time. Therefore, the whole checking  can be done in almost linear time.
 \end{proof}
 \subsubsection{Main properties of $G_{\mathcal{L}}$ }
 Combining Proposition \ref{main-proposition-1} with Lemma \ref{lemma-G-conj-in-WP-CP} one immediately gets the statements $(II.i)$ and $(II.ii)$ of Theorem \ref{theorem_about_connecton_of_word_and_conjugacy_problems}.
 \begin{corollary}
 	The membership problem for $\mathcal{L} \subseteq \mathcal{A}^*$  can be strongly reduced to the conjugacy problem in $G_{\mathcal{L}}$ in linear time; and the conjugacy problem in $G_{\mathcal{L}}$ can be strongly reduced to the membership problem for $\mathcal{L} \subseteq \mathcal{A}^*$ in almost linear time.
 \end{corollary}
 
 Another corollary is the following.
 \begin{corollary}
  The individual conjugacy problems in $G_{\mathcal{L}}$ are decidable in almost linear time.
 \end{corollary}
 \begin{proof}
 	Let us fix an element $g \in G_{\mathcal{L}}$ and let $U \in X^*$ be a reduced word representing $g$. The key observation is that since there are only finitely many words $W$ such that $(U, W)$ form a $\Lambda$-pair, it can be checked in a fixed time whether or not $(U, V')$ form a  $\Lambda$-pair. Therefore, without loss of generality we can assume that $(U, V')$ do not form a  $\Lambda$-pair. Hence, from Lemmas \ref{lemma-*}  and \ref{Lemmmmma} it follows that $U \sim_{conj} V'$ in $G_{\mathcal{L}}$ if and only if $U$ is $G$-conjugate to $V'$ in $G_{\mathcal{L}}$, which can be checked in almost linear time according to Lemma \ref{lemma-G-conj-in-WP-CP}.
 \end{proof}

\section{ Proof of Theorem \ref{theorem-schupp-miasnikov-question}}
\label{section-MS}
In this section we are going to construct a group $\tilde{G}$ which satisfies the properties of Theorem \ref{theorem-schupp-miasnikov-question}. $\tilde{G}$ will be constructed as a direct limit of a chain
 \begin{align}
	    \label{main seq of M-S}
		G_0 \stackrel{\beta_0}\hookrightarrow H_1 \stackrel{\gamma_1}\twoheadrightarrow G_1 \stackrel{\beta_1}\hookrightarrow H_2 \stackrel{\gamma_2}\twoheadrightarrow \ldots.
	\end{align}
 of non-elementary torsion-free hyperbolic groups of the form \eqref{main seq of gps***} according to the scheme described in Section \ref{section-about-general-scheme}. 

	More specifically, $G_0 \stackrel{def}= F(X)$, where $X=\{x_1, x_2\}$.
	
	Let $\mathcal{N} = \{n_1, n_2, \ldots \} \subset \mathbb{N}$ be a recursively enumerable but not recursive subset of positive integers. 
	Let us enumerate elements of $G_0$ according to their lexicographical order as $G_0=\{1=u_0, u_1, u_2, \ldots \}$ and denote $\mathcal{U}=\{ u_1, u_2, \ldots \}$. Let $\mathcal{V}=\sqcup_{i=1}^{\infty} \mathcal{U} =\{v_1, v_2, \ldots \}$ be a disjoint union of copies of $\mathcal{U}$ with recursive enumeration such that before the next copy of a given element $u \in \mathcal{U}$ appears in $\mathcal{V}$, all  the elements preceding $u \in \mathcal{U}$ already appeared in $\mathcal{V}$ at least once. 
	
	Denote by $\mathcal{P}=\{p_1, p_2, \ldots \}$ the set of prime numbers indexed in their natural order. 
	
	
	\subsubsection{Definition of $H_{i+1}$ for $\tilde{G}$}
	Suppose that for $i\geq 0$, $G_i$ is already constructed and it  satisfies the following properties, which we call condition $A$: 
	\begin{enumerate}
		\item[A1.] $G_i$ is a non-elementary torsion-free $\delta_i$-hyperbolic group for $\delta_i \in \mathbb{N}$,
		\item[A2.] words  of the form $x_1^{n}x_2$, $n\in \mathbb{Z}$, are not proper powers in $G_i$. 
	\end{enumerate}
   Below we show how to construct $G_{i+1}$ from $G_i$ which, in particular, preserves properties A1 and A2.\\
   
	Suppose that the set $\{\tilde{v}_{j_1}, \tilde{v}_{j_2}, \ldots, \tilde{v}_{j_i} \}$ of words from $X^*$ is such that its elements are not proper powers in $G_i$, i.e. $ E(v_{j_k})=\langle \tilde{v}_{j_k} \rangle$ in $G_i$ for $1\leq k \leq i$. Then define $v_{j_{i+1}}$ as the element from $\{ v_{j_i+1}, v_{j_i+2}, \ldots \}$ of minimal index such that $v_{j_{i+1}}$ does not represent the trivial element in $G_i$. Now define $\tilde{v}_{j_{i+1}}$ as follows.
	\begin{itemize}
	    \item If $v_{j_{i+1}}$ commensurates with any element from $ \{ \tilde{v}_{j_1},  \ldots, \tilde{v}_{j_i} \}$ in $G_i$, then define $\tilde{v}_{j_{i+1}}= \tilde{v}_{j_k} $, where $1\leq k \leq i$ is the smallest index such that $v_{j_{i+1}}$  commensurates with $\tilde{v}_{j_k}$ in $G_i$;
		\item Otherwise, if $v_{j_{i+1}}$ commensurates in $G_{i}$ with an element of the form $x_1^nx_2$, then define $\tilde{v}_{j_{i+1}}=x_1^{n_0}x_2$, where $|n_0|$ is the smallest such that $v_{j_{i+1}}$ commensurates in $G_{i}$ with   $x_1^{n_0}x_2$; 
		\item Otherwise, if $v_{j_{i+1}}$ is not a proper power in $G_{i}$, then define $\tilde{v}_{j_{i+1}}=v_{j_{i+1}}$; 
		\item Otherwise, if $v_{j_{i+1}}$ is a proper power in $G_{i}$, then define $\tilde{v}_{j_{i+1}}$ to be cyclically geodesic word  in $\Gamma(\tilde{G}, X)$ such that $E({v}_{j_{i+1}})= \langle \tilde{v}_{j_{i+1}} \rangle$ in $G_{i}$.
	\end{itemize}
	
	Define $q_1=p_1$ and suppose that the set $\{q_1, \ldots, q_i \}$ of prime numbers is already defined. Then, define $q_{i+1} \in \mathbb{N}$ as follows.
	\begin{itemize}
		\item If $\tilde{v}_{j_{i+1}} \equiv \tilde{v}_{j_{k}}$ for some $1\leq k \leq i$, then define $q_{i+1} =q_k$;
		\item Otherwise, define $q_{i+1}=p_{i+1}$.
	\end{itemize}	
	~\\

Define $\tilde{n}_{i+1}= n_{s+1} \in \mathcal{N}$, where $s=\# \{ 1\leq k\leq i \mid \tilde{v}_{j_k}\equiv \tilde{v}_{j_{i+1}} \}$.
Now define $\tilde{w}_{i+1}=x_1^{q_{i+1}^{\tilde{n}_{i+1}}}x_2$. Now define $H_{i+1}$ as follows. 
If $\tilde{w}_{i+1}$ commensurates with $\tilde{v}_{j_{i+1}}$ in $G_i$, then $H_{i+1}=G_{i}$, otherwise $H_{i+1}$ is an HNN-extension of $G_i$. More precisely,
	\begin{align}
	\label{def-H}
		H_{i+1}=\langle G_i, t_{i+1} \mid t_{i+1}^{-1}\tilde{v}_{j_{i+1}} t_{i+1} =  \tilde{w}_{i+1} \rangle.
	\end{align}
~\\	
\begin{lemma}
	\label{lem-existence}
	 $H_{i+1}$  is non-elementary torsion-free $\delta_{i+1}'$-hyperbolic group for some $\delta_{i+1}' \in \mathbb{N}$.	
\end{lemma}
\begin{proof}

 Hyperbolicity of $H_{i+1}$ follows directly from the above mentioned assumption A2 when combined with Theorem \ref{th hnn extension}.
 
The fact that $H_{i+1}$ is non-elementary and torsion-free follows from the basic properties of HNN-extensions (see, for example, \cite{lyndon schupp}).

\end{proof}
\begin{lemma}
	\label{lemm-ms}
	 The words $\{\tilde{v}_{j_1}, \ldots, \tilde{v}_{j_{i+1}} \}$ and the words of the form $x_1^mx_2$ are not proper powers in $H_{i+1}$ (assuming that this statement holds in $G_i$).
\end{lemma}
\begin{proof}
	Directly follows from Lemma \ref{lemma_about_proper_powers_in_hnn}.
\end{proof}

\begin{lemma}
\label{lem-ma}
	$t_{i+1} \notin E(x_1) \cup E(x_2)$, $x_1 \notin E(x_2)$ and $x_2 \notin E(x_1)$ in $H_{i+1}$.
\end{lemma}	
\begin{proof}
	Indeed, it follows from the basic properties of HNN-extensions and from the assumption that $G_i$ is a non-elementary group.
\end{proof}
	
	\subsubsection{Definition of $G_{i+1}$ for $\tilde{G}$}
	Suppose that $H_{i+1}$ is already constructed and it is a non-elementary torsion free $\delta_{i+1}'$-hyperbolic group for  $\delta_{i+1}'\in \mathbb{N}$ such that the map $i+1 \mapsto \delta'_{i+1}$ is computable. Define $G_{i+1}$ as follows: If $H_{i+1}=G_i$, then $G_{i+1}=G_i$, otherwise
	\begin{align}
	\label{def-g-i+1}
		G_{i+1}=H_{i+1} / \ll \mathcal{R} \big(\{t_{i+1} \}, x_1, x_2,  \delta'_{i+1}, \lambda_{i+1}, c_{i+1}, \epsilon_{i+1},  \mu_{i+1}, \rho_{i+1} \big)\gg.
	\end{align}
	Denote $ \mathcal{R} \big(\{t_{i+1} \}, x_1, x_2,  \delta'_{i+1}, \lambda_{i+1}, c_{i+1}, \epsilon_{i+1},  \mu_{i+1}, \rho_{i+1} \big)$ by $\mathcal{R}_{i+1}$.
	~\\
	~\\
	\begin{lemma}
		\label{observation}
		For sparse enough standard parameters $\lambda_{i+1}, c_{i+1}, \epsilon_{i+1},  \mu_{i+1}, \rho_{i+1}$,  no word of the form $(x_1^mx_2)^{\pm 1}$, $m\in \mathbb{Z}$, has a $(\epsilon_{i+1}, \lambda_{i+1}\mu_{i+1})$-subword with respect to the quotient $G_{i+1}=H_{i+1} / \ll \mathcal{R}_{i+1} \gg$, and  $(\lambda_{i+1}, c_{i+1})$-quasi-geodesic words in $\Gamma(G_{i+1}, X\cup \{t_{i+1} \})$ do not contain $(\epsilon_{i+1}, 1-122\lambda_{i+1}\mu_{i+1})$-subwords.

	\end{lemma}
	\begin{proof}
		Indeed, the first statement follows from the fact that the words $\mathcal{R}_{i+1}$ satisfy the small cancellation condition $C'(\lambda_{i+1}, c_{i+1}, \epsilon_{i+1},  \mu_{i+1}, \rho_{i+1})$ (see conditions (2.1) and (2.2) in the definition of $C'(\lambda, c, \epsilon,  \mu, \rho)$ condition).  The second statement follows from Lemma \ref{lemma on Greendlinger's cell}. 
	\end{proof}
	
	\begin{lemma}
		For sparse enough standard parameters $\lambda_{i+1}, c_{i+1}, \epsilon_{i+1},  \mu_{i+1}, \rho_{i+1}$, the group $G_{i+1}$ is torsion-free non-elementary $\delta_{i+1}$-hyperbolic for some  $\delta_{i+1} \in \mathbb{N}$.
	\end{lemma}
	\begin{proof}
		Follows directly from Lemmas \ref{lem-ma} and \ref{lem 7.2}.
	\end{proof}
	
	\begin{lemma}
	   \label{lem-q-g-x1x2}
		For sparse enough standard parameters $\lambda_{i+1}, c_{i+1}$, words of the form $x_1^mx_2$ are $(\lambda_{i+1}, c_{i+1})$-quasi-geodesic in $\Gamma(H_{i+1}, X \cup \{t_{i+1}\})$.
	\end{lemma}
	\begin{proof}
		The statement of the lemma follows immediately from Lemma \ref{lemma_aux_q-g_cyclically_minimal}.
	\end{proof}
	
	\begin{lemma}
	\label{proper power in G_k}
		For sparse enough standard parameters $\lambda_{i+1}$, $c_{i+1}, \epsilon_{i+1}$,  $\mu_{i+1}$, $\rho_{i+1}$, the words $\{\tilde{v}_{j_1}, \ldots, \tilde{v}_{j_{i+1}} \}$ and the words of the form $x_1^mx_2$ are not proper powers in $G_{i+1}$, assuming that these statements hold in $H_{i+1}$.
	
	\end{lemma}
	\begin{proof}
		The part about the words $\{\tilde{v}_{j_1}, \ldots, \tilde{v}_{j_{i+1}} \}$ immediately follows from Lemma \ref{lemma_about_proper_powers_in_quotients}, because, since the standard parameters  are assumed to be sparse enough, in particular, we can assume that $\rho_{i+1}$ is sufficiently larger than $\max\{ \| \tilde{v}_{j_1} \|, \ldots,  \| \tilde{v}_{j_i} \| \}$ and then apply Lemma \ref{lemma_about_proper_powers_in_quotients}.
		
		Now, by contradiction, assume that for some fixed $m \in \mathbb{N}$, the word $x_1^mx_2$ is a proper power in $G_{i+1}$. Then, there exists $k>1$ and $u \in (X \cup \{t_{i+1} \} )^*$ such that $u$ is cyclically minimal in $\Gamma(G_{i+1}, X \cup \{t_{i+1} \} )$ and 
		$$ x_1^m x_2  \sim_{conj} u^k$$ in $G_{i+1}$. By Lemma \ref{lemma_aux_q-g_cyclically_minimal}, for sparse enough standard parameters $\lambda_{i+1}$, $c_{i+1}, \epsilon_{i+1}$,  $\mu_{i+1}$, $\rho_{i+1}$, the word $u^k$ is $(\lambda_{i+1}, c_{i+1})$-quasi-geodesic in $\Gamma(H_{i+1}, X \cup \{t_{i+1} \} )$. Therefore, since by Lemma \ref{observation}, $x_1^mx_2$ does not contain a $(\epsilon_{i+1}, \mu_{i+1}\lambda_{i+1})$-subwords with respect to the quotient $G_{i+1}=H_{i+1}/\ll \mathcal{R}_{i+1} \gg$,  by Lemma \ref{lemma_about_slender_conjugacy_diagrams}, we get that $u^k$ must contain a $(\epsilon_{i+1}, 1-122\lambda_{i+1}\mu_{i+1})$-subword with respect to $G_{i+1}=H_{i+1}/\ll \mathcal{R}_{i+1} \gg$, which is impossible by Lemma \ref{lemma_about_proper_powers_in_quotients}. A contradiction.

	\end{proof}

	\subsection{Properties}
	  
	  Define 
	  \begin{align*}
	  	\tilde{\mathcal{V}}=\{\tilde{v}_{j_1}, \tilde{v}_{j_2}, \ldots \}
	  \end{align*}
	  and for all $i \in \mathbb{N}$, define
	  \begin{align*}
	  	\tilde{\mathcal{V}}_i=\{\tilde{v}_{j_k} \in \tilde{\mathcal{V}} \mid \tilde{v}_{j_k}\equiv \tilde{v}_{j_i}\}
	  \end{align*}
	 and 
	  \begin{align*}
	 	\tilde{\mathcal{N}}_i=\{n \in \mathbb{N}  \mid \tilde{v}_{j_i}\sim_{conj} x_1^{q_i^n}x_2 \text{~in~} \tilde{G} \}.
	  \end{align*}
~\\
~\\
\begin{lemma}
	\label{lemma-SM-2}
	For all $i, m \in \mathbb{N}$, words of the form $\tilde{v}_{j_i}$ and $x_1^{m} x_2$ are not proper powers in $\tilde{G}$.
\end{lemma}
\begin{proof}
Follows immediately from Lemmas \ref{lemm-ms} and \ref{proper power in G_k}.
\end{proof}

	\begin{lemma}
		\label{lemma-commensurability-x1x2}
		Words of the form $(x_1^{m_1}x_2)^{\tau_1}$ and $(x_1^{m_2}x_2)^{\tau_2}$, where $m_1, m_2 \in \mathbb{Z}$, $\tau_1, \tau_2 \in \{\pm 1\}$, are $G$-conjugate in $\tilde{G}$ if and only if $m_1=m_2$ and $\tau_1=\tau_2$. 
	\end{lemma}
	\begin{proof}
	If $(x_1^{m_1}x_2)^{\tau_1}$ and $(x_1^{m_2}x_2)^{\tau_2}$ are conjugate in $G_0$, then clearly $m_1=m_2$ and $\tau_1=\tau_2$. 
	Now suppose that $(x_1^{m_1}x_2)^{\tau_1}$ and $(x_1^{m_2}x_2)^{\tau_2}$ are $G$-conjugate in $\tilde{G}$, but $(x_1^{m_1}x_2)^{\tau_1} \not\sim_{conj} (x_1^{m_2}x_2)^{\tau_2}$ in $G_0$. Then there exists $s \in \mathbb{N}$ such that $(x_1^{m_1}x_2)^{\tau_1} \sim_{conj} (x_1^{m_2}x_2)^{\tau_2}$ in $G_s$ but $(x_1^{m_1}x_2)^{\tau_1} \not\sim_{conj} (x_1^{m_2}x_2)^{\tau_2}$ in $H_s$.
	Since by Lemma \ref{lem-q-g-x1x2}, the words $(x_1^{m_1}x_2)^{\tau_1}$ and $(x_1^{m_2}x_2)^{\tau_2}$ are cyclically $(\lambda_s, c_s)$-quasi-geodesic in $\Gamma(H_s, X \cup \{t_s\})$, and since by Lemma \ref{observation}, $(x_1^{m_1}x_2)^{\tau_1}$ and $(x_1^{m_2}x_2)^{\tau_2}$ do not contain $(\epsilon_s, \lambda_s\mu_s)$-subwords with respect to the quotient $G_s=H_s / \ll \mathcal{R}_s \gg$, by Lemma \ref{lemma_about_slender_conjugacy_diagrams}, we get a contradiction, because Lemma  \ref{lemma_about_slender_conjugacy_diagrams} tells us that in case $(x_1^{m_1}x_2)^{\tau_1} \sim_{conj} (x_1^{m_2}x_2)^{\tau_2}$ in $G_s$ but $(x_1^{m_1}x_2)^{\tau_1} \not\sim_{conj} (x_1^{m_2}x_2)^{\tau_2}$ in $H_s$, at least one of the words $(x_1^{m_1}x_2)^{\tau_1}$ and $(x_1^{m_2}x_2)^{\tau_2}$ contains a $(1-121\lambda_s\mu_s)/2$-subword with respect to the quotient $G_s=H_s / \ll \mathcal{R}_s \gg$, which contradicts to the assertion of Lemma \ref{observation}.

	\end{proof}
	
	\begin{lemma}
	\label{lemma-SM-0}
	For all $i, k \in \mathbb{N}$, $\tau \in \{\pm 1\}$, $\tilde{v}_{j_i}$  is $G$-conjugate with $(\tilde{v}_{j_k})^{\tau}$ in $\tilde{G}$ if and only if $\tilde{v}_{j_i} \equiv 	\tilde{v}_{j_k}$ and $\tau=1$.
	\end{lemma}
\begin{proof}
    If $\tilde{v}_{j_i} \sim_{conj} \tilde{v}_{j_k}^{\tau}$ in $G_0$, then clearly $\tilde{v}_{j_i} \equiv 	\tilde{v}_{j_k}$ and ${\tau}=1$. The inverse is true as well. Now assume that $\tilde{v}_{j_i}$  is $G$-conjugate with $\tilde{v}_{j_k}^{\tau}$ in $\tilde{G}$, but $\tilde{v}_{j_i} \not\sim_{conj} \tilde{v}_{j_k}^{\tau}$ in $G_0$. Then there exists $s \in \mathbb{N}$ such that $\tilde{v}_{j_i} \sim_{conj} \tilde{v}_{j_k}^{\tau}$ in $G_s$ but $\tilde{v}_{j_i} \not\sim_{conj} \tilde{v}_{j_k}^{\tau}$ in $H_s$.  
    
    Without loss of generality assume that $i < k$. Then, since by definition of $\tilde{v}_{j_k}$, $\tilde{v}_{j_k}^{\tau}$ is not conjugate in $G_{k-1}$ with any element from
    \begin{align*}
    	\{ \tilde{v}_{j_l} \mid 1\leq l <k, ~\tilde{v}_{j_k}^{\tau} \not \equiv \tilde{v}_{j_l} \}, 
    \end{align*}
    we get that $s \geq k$.
     However, by Lemma \ref{lemma_about_slender_conjugacy_diagrams}, if the standard parameters $\lambda_s, c_s, \epsilon_s, \mu_s, \rho_s$ are sparse enough, in particular, if $\rho_s$ is much larger than $\|\tilde{v}_{j_i}\|$ and $\|\tilde{v}_{j_k}\|$, then it cannot happen that $\tilde{v}_{j_i} \sim_{conj} \tilde{v}_{j_k}$ in $G_s$ but $\tilde{v}_{j_i} \not\sim_{conj} \tilde{v}_{j_k}$ in $H_s$.

\end{proof}

	\begin{lemma}
	\label{lemma-33}
		If for some $i\in \mathbb{N}$, $n \in \mathbb{Z}$, $\tau \in \{\pm 1\}$, $\tilde{v}_{j_i}$ is $G$-conjugate to $(x_1^n x_2)^{\tau}$ in $\tilde{G}$, then $\tilde{v}_{j_i} \equiv (x_1^n x_2)^{\tau}$.
	\end{lemma}
	\begin{proof}
		If $\tilde{v}_{j_i} \sim_{conj} (x_1^n x_2)^{\tau}$ in $G_0$, then clearly $\tilde{v}_{j_i} \equiv (x_1^n x_2)^{\tau}$.
		Now assume that  $\tilde{v}_{j_i}$  is $G$-conjugate with $(x_1^n x_2)^{\tau}$ in $\tilde{G}$, but $\tilde{v}_{j_i} \not\sim_{conj} (x_1^n x_2)^{\tau}$ in $G_0$. Then there exists $s \in \mathbb{N}$ such that $\tilde{v}_{j_i} \sim_{conj} (x_1^n x_2)^{\tau}$ in $G_s$ but $\tilde{v}_{j_i} \not\sim_{conj} (x_1^n x_2)^{\tau}$ in $H_s$.  
		
		If $s<i$, then by the definition of $\tilde{v}_{j_i}$, the fact that $\tilde{v}_{j_i} \sim_{conj} (x_1^n x_2)^{\tau}$ in $G_s$ implies that $\tilde{v}_{j_i} \equiv x_1^{n'} x_2$ for some $n' \in \mathbb{Z}$. Therefore, by Lemma \ref{lemma-commensurability-x1x2}, $x_1^{n'} x_2 \equiv (x_1^{n} x_2)^{\tau}$, which implies that $\tilde{v}_{j_i} \sim_{conj} (x_1^{n} x_2)^{\tau}$ in $G_0$. A contradiction.
		
		If $s\geq i$, then since by Lemma \ref{observation}, the word $(x_1^n x_2)^{\tau}$ does not contain a $(\epsilon_s, \lambda_s \mu_s)$-subword with respect to the quotient $G_s=H_s / \ll \mathcal{R}_s \gg$, by Lemma \ref{lemma_about_slender_conjugacy_diagrams}, the word $\tilde{v}_{j_i}$ must contain $(\epsilon_s, 1-122\lambda_s \mu_s)$-subwords with respect to the quotient $G_s=H_s / \ll \mathcal{R}_s \gg$, which is impossible granted that the standard parameters $\lambda_s, c_s, \epsilon_s, \mu_s, \rho_s$ are sparse enough, in particular, if $\rho_s$ is much larger than $\|\tilde{v}_{j_i}\|$.
	\end{proof}

	\begin{lemma}
		\label{lem-central}
		
		Let 
		\begin{align*}
			w_1, w_2 \in \big\{ \tilde{v}_{j_i}, (x_1^n x_2)^{\tau} \mid i \in \mathbb{N}, n \in \mathbb{Z}, \tau \in \{\pm 1\} \big\}
		\end{align*}
		and $w_1 \not \equiv w_2$ such that $w_1 \sim_{conj} w_2$ in $\tilde{G}$. 
		Then, for the group
		\begin{align*}
			H'_s = \langle X, t_1, t_2, \ldots, t_s \mid t_1^{-1} \tilde{v}_{j_1} t_1 =  x_1^{q_1^{\tilde{n}_1}}x_2, \ldots, t_s^{-1} \tilde{v}_{j_s} t_s= x_1^{q_s^{\tilde{n}_s}}  x_2\rangle,
		\end{align*}
		there exists $T \in \{ t_1, t_2, \ldots, t_s\}^*$ such that 
		\begin{align*}
			T^{-1} w_1 T = w_2 \text{~in~} H_s',
		\end{align*}
		where $s$ is such that $w_1 \sim_{conj} w_2$ in $H_s$, but $w_1 \not\sim_{conj} w_2$ in $G_{s-1}$
	\end{lemma}
	\begin{proof}
		Suppose that $w_1 \not\equiv w_2$, then by Lemmas \ref{lemma-commensurability-x1x2}, \ref{lemma-SM-0} and \ref{lemma-33}, the fact that $w_1 \sim_{conj} w_2$ in $\tilde{G}$, implies that $w_1$ is $H$-conjugate to $w_2$ in $\tilde{G}$. Therefore, there exists $s\in \mathbb{N}$ such that $w_1 \sim_{conj} w_2$ in $H_s$, but $w_1 \not\sim_{conj} w_2$ in $G_{s-1}$. Therefore, there exists $h\in H_s$ such that 
		\begin{align*}
			h^{-1} w_1 h = w_2 \text{~in~} H_s
		\end{align*}
		such that $\theta(h)$ is minimal and $\theta(h)>0$, where $\theta$ is defined in Subsection \ref{section-HNN}.
		
		We will prove by induction on $(s, \theta(h)$, where we define $(s_1, \theta(h_1)) <(s_2, \theta(h_2))$ if either $s_1<s_2$ or $s_1=s_2$ and $\theta(h_1) < \theta(h_2)$.
		
		If $s=1$, then note that $H'_s$ coincides with $H_s$, and  the statement of the lemma follows from the combination of Collins' Lemma (see Lemma \ref{lemma_about_conjugacy_in_HNN_extensions}) with the fact that $G_0$ is a free group.
		
		Now assume that $s>1$ and for all smaller pairs $(s', \theta(h'))$ the statement is true. Note that, by Collins' Lemma, there exist $m\in \mathbb{Z}$, $h_1, h_2 \in H_s$ such that $h_1 t_s h_2=h$ and, in particular, $\theta(h_1), \theta(h_2) < \theta(h)$, and either
		\begin{align*}
			h_1^{-1} w_1 h_1 =_{H_s} (\tilde{v}_{j_s})^m \text{~and~} h_2^{-1} (x_1^{q_s^{\tilde{n}_s}}x_2)^m h_2 =_{H_s} w_2
		\end{align*}
		or
		\begin{align*}
			h_1^{-1} w_1 h_1 =_{H_s}(x_1^{q_s^{\tilde{n}_s}}x_2)^m \text{~and~} h_2^{-1} (\tilde{v}_{j_s})^m h_2 =_{H_s} w_2.
		\end{align*}
	Since by Lemma \ref{lemma-SM-2}, the words $w_1$ and $w_2$ are not proper powers, we get that $m \in \{\pm 1\}$. Also, not that since by Lemmas \ref{lemma-commensurability-x1x2}, \ref{lemma-SM-0} and \ref{lemma-33}, $w_1, (\tilde{v}_{j_s})^m$ and $w_1, (x_1^{q_s^{\tilde{n}_s}}x_2)^m $  are not $G$-conjugate in $\tilde{G}$, in case $\theta(h_1)=0$, we get that either $w_1$ is conjugate to $(\tilde{v}_{j_s})^m$ in $H_{s-1}$ or $w_1$ is conjugate to $(x_1^{q_s^{\tilde{n}_s}}x_2)^m$ in $H_{s-1}$. Analogous statement is true for the pairs $(w_2, (\tilde{v}_{j_s})^m)$ and $(w_2, (x_1^{q_s^{\tilde{n}_s}}x_2)^m )$ if $\theta(h_2)=0$. Therefore, the statement of the lemma follows from inductive hypothesis. 
	\end{proof}
	~\\
	\begin{lemma}
	\label{lem-*!}
		Let $i\in \mathbb{N}$. Then for all but finitely many $m\in \mathbb{N}$, if the word $x_1^{q_i^{m}}x_2$ is conjugate with $(\tilde{v}_{j_i})^{\tau}$, $\tau \in \{\pm 1\}$, in $\tilde{G}$, then $ m \in \mathcal{N}$ and $\tau=1$.
	\end{lemma}
	\begin{proof}
		First of all, by Lemma \ref{lemma-commensurability-x1x2}, if $x_1^{q_i^{m}}x_2$ is conjugate with $(\tilde{v}_{j_i})^{\tau}$ in $\tilde{G}$, then $x_1^{q_i^{m}}x_2$ is $H$-conjugate with $(\tilde{v}_{j_i})^{\tau}$ in $\tilde{G}$. Therefore, there exists $s\in \mathbb{N}$ such that $x_1^{q_i^{m}}x_2 \sim_{conj} (\tilde{v}_{j_i})^{\tau}$ in $H_s$, but $x_1^{q_i^{m}}x_2 \not\sim_{conj} (\tilde{v}_{j_i})^{\tau}$ in $G_{s-1}$.
Then, by Lemma \ref{lem-central}, for the group
		\begin{align*}
			H'_s = \langle X, t_1, t_2, \ldots, t_s \mid t_1^{-1} \tilde{v}_{j_1} t_1 =  x_1^{q_1^{\tilde{n}_1}}x_2, \ldots, t_s^{-1} \tilde{v}_{j_s} t_s= x_1^{q_s^{\tilde{n}_s}}x_2 \rangle,
		\end{align*}
		there exists $T \in \{ t_1, t_2, \ldots, t_s\}^*$ such that $T$ is of minimal length for which
		\begin{align*}
			T^{-1} (\tilde{v}_{j_i})^{\tau} T = x_1^{q_i^{m}}x_2 \text{~in~} H_s'.
		\end{align*}
Now, from the last identity, to see that $\tau=1$ is a simple exercise.\\
~\\
\textit{Claim 1.} $T$ does not contain subwords of the form $t_{s_1}t_{s_2}^{-1}$, where $1\leq s_1, s_2 \leq s$.
\begin{proof}[Proof of Claim 1]
By contradiction suppose that $T=T_1 t_{s_1}t_{s_2}^{-1} T_2$. Then, by Britton's Lemma (see Lemma \ref{lem-britton}), the identities $$T^{-1} \tilde{v}_{j_i} T (x_1^{q_i^{m}}x_2)^{-1}=_{H_s'}1$$ and $$ \tilde{v}_{j_i} T (x_1^{q_i^{m}}x_2)^{-1}T^{-1}=_{H_s'}1$$ imply that
$$(T_1t_{s_1})^{-1} \tilde{v}_{j_i} T_1t_{s_1} \in \langle x_1^{q_{s_1}^{\tilde{n}_{s_1}}}x_2 \rangle$$ and $$t_{s_2}^{-1} T_2 x_1^{q_{i}^{m}}x_2 (t_{s_2}^{-1} T_2)^{-1} \in \langle x_1^{q_{s_2}^{\tilde{n}_{s_2}}}x_2 \rangle $$ in $H'_s$, and on the other hand
$$(T_1t_{s_1})^{-1} \tilde{v}_{j_i} T_1t_{s_1}=_{H'_s} t_{s_2}^{-1} T_2 x_1^{q_{i}^{m}}x_2 (t_{s_2}^{-1} T_2)^{-1}.$$
Therefore, we get $s_1=s_2$ and hence $T=_{H'_s}T_1T_2$, which contradicts  the assumption that $T$ was chosen to be of minimal length.
\end{proof}
~\\
\textit{Claim 2.} If $\|T\|\geq 2$, then $T$ is of the form $T_1 t_{s_0}^{-1}$, for some $1\leq s_0 \leq s$.
\begin{proof}[Proof of Claim 2]
Indeed, if $T$ was of the form $T_1 t_{s_0}$ for some $1\leq s_0 \leq s$, then by Britton's Lemma, the identity $ \tilde{v}_{j_i} T (x_1^{q_i^{m}}x_2)^{-1}T^{-1}=_{H_s'}1$ would imply that $t_{s_0} x_1^{q_i^{m}}x_2 t_{s_0}^{-1} \in \langle \tilde{v}_{j_i} \rangle$ in $H'_s$, which implies that $t_{s_0} x_1^{q_i^{m}}x_2 t_{s_0}^{-1} =  \tilde{v}_{j_i} $. However, the last identity contradicts the assumption that $\|T\|\geq 1$ and $T$ was chosen of minimal length.
\end{proof}
Note that if $\|T\|=1$, then the identity $T^{-1} \tilde{v}_{j_i} T (x_1^{q_i^{m}}x_2)^{-1}=_{H_s'}1$ can hold only for finitely many values of $m \notin \mathcal{N}$. Hence without loss of generality, let us assume that $\|T\|>2$. Then, by Claims 1 and 2, $T$ is of the form $T=t_{s_1}^{-1} \ldots t_{s_k}^{-1}$, where $k\geq 2$ and $1\leq s_1, \ldots, s_k \leq s$. 

Now, note that by Britton's Lemma, for some $n\in \mathbb{N}$, $t_{s_1} v_{j_n} t_{s_1}^{-1}= \tilde{v}_{j_i}$, and hence $\tilde{v}_{j_i}=x_1^{q_{s_1}^{\tilde{n}_{s_1}}}x_2$, which implies that $s_1$ is defined uniquely. The same was $s_2$, \ldots, $s_k$ are defined uniquely. Therefore, if for some $m_1 \neq m$, we have $$(T')^{-1} \tilde{v}_{j_i} T' (x_1^{q_i^{m_1}}x_2)^{-1}=_{H_s'}1$$ and $\|T'\|\geq 1$ also $\|T'\|$ is minimal, then either $T'$ is a prefix of $T$ or $T$ is a prefix of $T'$ and $T'$ is of the same form as $T$. However, an application of Britton's Lemma show that this cannot happen. Indeed, if without loss of generality we assume that $\|T'\|> \|T\|$, then $T'= T t_{s_{k+1}}^{-1} \ldots t_{s_{k+l}}^{-1}$. Then, since $T^{-1} \tilde{v}_{j_i} T =_{H'_s}  x_1^{q_i^{m_1}}x_2$, we would have
$$ t_{s_{k+1}}T^{-1} \tilde{v}_{j_i} Tt_{s_{k+1}}^{-1} = t_{s_{k+1}}x_1^{q_i^{m_1}}x_2 t_{s_{k+1}}^{-1} \in \langle \tilde{v}_{j_{s_{k+1}}} \rangle $$
and also $q_{i}= q_{s_{k+1}}$. However, from the definition of the elements $q_1, q_2, \ldots, $ and from the last identities, we get $\tilde{v}_{j_{s_{k+1}}} =\tilde{v}_{j_{i}}$ and $t_{s_{k+1}}x_1^{q_i^{m-1}}x_2 t_{s_{k+1}}^{-1} =_{H'_s} \tilde{v}_{j_{i}}$. The last identity contradicts the assumption that $\|T'\|\geq 2$ and $T'$ was chosen to be of minimal length.

Thus the lemma is proved.


 		
	\end{proof}
\begin{lemma}
\label{lem-**!!}
	Let $i\in \mathbb{N}$ and let $i_0$ be the smallest index such that $\tilde{v}_{i_0}=_{\tilde{G}} \tilde{v}_{i}$. Then, the set $\tilde{\mathcal{V}}_{i_0}$ is infinite and also the set $\mathcal{N}\bigtriangleup \tilde{\mathcal{N}}_{i_0} = (\mathcal{N} \setminus \tilde{\mathcal{N}}_{i_0}) \cup ( \tilde{\mathcal{N}}_{i_0} \setminus \mathcal{N})$ is finite.
\end{lemma}	
\begin{proof}
	The first statement follows from the definition of $\tilde{\mathcal{V}}_{i_0}$ and the elements \\$\{\tilde{v}_{j_1}, \tilde{v}_{j_2}, \ldots \}$.
	
	As for the second statement, first of all, note that Lemma \ref{lem-*!} implies that $ \tilde{\mathcal{N}}_{i_0} \setminus \mathcal{N}$ is finite.
	
	 Also, since the set $\mathcal{V}_{i_0}$ is infinite, by the definition of the set $\{\tilde{n}_1, \tilde{n}_2 \ldots \} \subseteq \mathcal{N}$ we get  $\{\tilde{n}_1, \tilde{n}_2 \ldots \} = \mathcal{N}$. Therefore, $\mathcal{N} \setminus \tilde{\mathcal{N}}_{i_0} = \emptyset$. Thus $\mathcal{N}\bigtriangleup \tilde{\mathcal{N}}_{i_0}$ is finite.
	 
	 \end{proof}
	

\begin{lemma}
		\label{conjugacy-x_1x_2}
		Let $i\in \mathbb{N}$, $m\in \mathbb{Z} \setminus\{0\}$. Let $i_0$ be the smallest index such that $\tilde{v}_{j_{i_0}}$ is conjugate to $\tilde{v}_{j_i}$ in $\tilde{G}$. Then for  all but finitely many positive integers $n$, $(\tilde{v}_{j_{i_0}})^m$ is conjugate to $\big(x_1^{q_{i_0}^{n}}x_2\big)^{\tau m}$ in $\tilde{G}$, where $\tau \in \{\pm 1\}$, if and only if $n \in \mathcal{N}$ and $\tau=1$.
	\end{lemma}
	\begin{proof}
		Indeed, suppose that $(\tilde{v}_{j_{i_0}})^m$ is conjugate with $(\tilde{v}_{j_{i_0}})^{\tau m}$ in $\tilde{G}$. Then, there exists $s\in \mathbb{N}$, such that $(\tilde{v}_{j_{i_0}})^m$ is conjugate with $\big(x_1^{q_{i_0}^{n}}x_2\big)^{\tau m}$ in $G_s$. Therefore, $E\big((\tilde{v}_{j_{i_0}})^m\big)$ is conjugate with $E\big(\big(x_1^{q_{i_0}^{n}}x_2 \big)^m \big)$ in $G_s$. But since $G_s$ is a torsion-free hyperbolic group and  by Lemma \ref{lemma-SM-2}, $\tilde{v}_{j_{i_0}}$ and $x_1^{q_i^{n}}x_2$ are not proper powers, we get that $\langle \tilde{v}_{j_{i_0}} \rangle $ is conjugate with $\langle x_1^{q_{i_0}^{n}}x_2  \rangle$ in $G_s$. Consequently, $\tilde{v}_{j_{i_0}} $ is conjugate with  $ \big(x_1^{q_{i_0}^{n}}x_2 \big)^{\tau}$ in $\tilde{G}$. 
		 Therefore, by Lemma \ref{lem-*!}, for all but finitely many $n$, we get $n \in\mathcal{N}$.
		
		The inverse statement follows immediately from Lemma \ref{lem-**!!}.
	\end{proof}
\begin{lemma}
\label{lemma-MS-*}
	For any word $u \in X^*$ representing a non-trivial element of $\tilde{G}$, there exists an element   $\tilde{v}_{j_i}\in\tilde{\mathcal{V}}$ and $m\in \mathbb{Z}$ such that $u \sim_{conj} \tilde{v}_{j_i}^m$ in $\tilde{G}$.
	\end{lemma}
\begin{proof}
Indeed, by the definition of the words $\{ \tilde{v}_{j_1}, \tilde{v}_{j_2}, \ldots \}$, for each $u \in X^*$, there exists $i \in \mathbb{N}$ and $v \in X^*$ such that $u$ commensurates with $\tilde{v}_{j_i}$ in $G_i$. In other words, $v^{-1}uv \in E(\tilde{v}_{j_i})$ in $G_i$. But since $G_{i}$ is a torsion-free hyperbolic group and  by Lemma \ref{lemma-SM-2}, $\tilde{v}_{j_i}$ is not a proper power, we get that $v^{-1}uv \in E(\tilde{v}_{j_i})$ is equivalent to $v^{-1}uv \in \langle \tilde{v}_{j_i} \rangle$ in $G_i$. Therefore, for some $m \in \mathbb{Z}$, $u \sim_{conj} \tilde{v}_{j_i}^m$ in $\tilde{G}$.


\end{proof}
\begin{lemma}
	\label{lemma-main-MS-1}
	The word problem in $\tilde{G}$ is decidable in almost linear time, however, for each $g \in \tilde{G}\setminus \{1\}$, the individual conjugacy problem $ICP(g)$ is undecidable.
\end{lemma}
\begin{proof}
	The decidability of the word problem in almost linear time follows from Lemma \ref{lemma-word-problem-scheme}.
	
	Now suppose that $g\neq 1$ in $\tilde{G}$. Then, by Lemma \ref{lemma-MS-*}, there exist   $m\in \mathbb{Z}$ and minimal index $i$ such that $\tilde{v}_{j_i}^m$ is conjugate to $g$ in $\tilde{G}$. Therefore, by Lemma \ref{conjugacy-x_1x_2}, for all but finitely many $n \in \mathbb{N}$, the question of whether or not $(x_1^{q_i^{n}}x_2)^m$ is  conjugate to $g$ is equivalent to the question of whether or not $n \in \mathcal{N}$. Therefore, since $\mathcal{N}$ is not recursive, we get that the decision problem which for each input $n\in \mathbb{N}$  asks whether or not  $(\tilde{v}_{j_i}^nx)^m$ is  conjugate to $g$ in $\tilde{G}$ is undecidable. In particular, this implies that $ICP(g)$ is undecidable.
\end{proof}


\section{Appendix}
\label{appendix}
\subsection{Proof of Lemma \ref{lem 2.1}}
	
	Let $U, V, T_1, T_2, L, \lambda, c, m, n$ be defined as in the statement of  Lemma \ref{lem 2.1}.
	
	Let us assume that
		\begin{align}
	\label{equatiion}
	 L \leq 	\frac{\|U\|}{12\lambda}m.	
	\end{align}

	Then in the Cayley graph $\Gamma(G, X)$ there exists a rectangle $ABCD$ such that $lab(AB)=T_1$, $lab(BC)= U^m$, $lab(CD)=T_2$ and $lab(AD)=V^n$. Since the sides $BC$ and $AD$ are $(\lambda, c)$-quasi-geodesic and $\|T_1\|, \|T_2\| \leq L$, by Corollary \ref{corollary on hausdorff distance between quasi-geodesics}, we get that the Hausdorff distance between $BC$ and $AD$ is bounded from above by $L+2R_{\lambda, c}+2\delta$. Moreover, by Corollary \ref{another corollary about hausdorff distance},
 for any point $o \in BC$ such that its distance from $B$ and $C$ is more than $L+R_{\lambda, c}+2\delta$, we have  $dist(o, AD)\leq 2R_{\lambda, c}+2\delta$. 
	Let us fix the points $B', C' \in BC$ such that 
	\begin{align}
	\label{equation000}
	lab(BB')=lab(C'C)=U^{2\lfloor (\lambda(L+R_{\lambda, c}+2\delta))+c)/\|U\| \rfloor+1}.
	\end{align}
	Note that then $d(B, B'), d(C, C') > L+ R_{\lambda, c}+2\delta$ and 
	\begin{equation}
	\begin{aligned}
\label{equation007}
		&2\big(2\lfloor (\lambda(L+R_{\lambda, c}+2\delta))+c)/\|U\| \rfloor+1 \big)\\ 
		\leq &\frac{4\lambda L}{\|U\|}+\frac{4\lambda R_{\lambda, c}+8\delta+2c}{\|U\|}+2 < \frac{4\lambda L}{\|U\|}+|X|^{2R_{\lambda, c}+2\delta+\|V\|}\\
		\text{by \eqref{equatiion}~,} \leq &\frac{m}{3}+|X|^{2R_{\lambda, c}+2\delta+\|V\|} \leq \frac{2m}{3}.
	\end{aligned}
	\end{equation}

	Following Olshanskii, \cite{Olsh G-groups}, we call a point on $CD$ a phase vertex, , say $O$, if $lab(BO)$ is a power of $U$. Correspondingly, we call a point on $AD$, say $O'$,  a phase vertex, if $lab(AO')$ is a power of $V$. Since $B'C'$ is contained in the $(2R_{\lambda, c}+2\delta)$-neighborhood of $AD$,  for each phase vertex $O \in B'C'$ there exists a phase vertex $O'\in AD$ such that $d(O, O')\leq 2R_{\lambda, c}+2\delta+ \|V\|$. This follows from Lemma \ref{another corollary about hausdorff distance} and from the simple observation that the set of phase vertices on $AD$ is a $\|V\|$-net.
	
	 By \eqref{equation000} and \eqref{equation007} we get that the number of phase vertices on $B'C'$ is greater than $|X|^{2R_{\lambda, c}+2\delta+\|V\|}$ (recall that $X$ is a symmetric set). 
	Therefore, by the pigeonhole principle, there exist at least two phase vertices $O_1, O_2 \in BC$ and two phase vertices $O_1', O_2' \in AD$ such that $d(O_1, O_1'), d(O_2, O_2') \leq 2R_{\lambda, c}+2\delta+ \|V\|$ and $lab(O_1 O_1')\equiv lab(O_2 O_2')$, where by  $lab(O_1 O_1')$ and $ lab(O_2 O_2')$ we mean the labels of some geodesic paths joining $O_1$ to $O_1'$ and $O_2$ to $O_2'$, respectively. 
	
	Denote $Q = lab(O_1 O_1')= lab(O_2 O_2')$. Then we have that for some integers $m_0$ and $n_0$, $Q^{-1}U^{m_0}Q =_G V^{n_0}$. On the other hand, $T_1=_G U^{m_1} Q V^{n_1}$, where the integers $m_1$, $n_1$ are such that $lab(BO_1)=U^{m_1}$ and $lab(O_1'A)=V^{n_1}$. But this means that $T_1U^{m_0}T_1^{-1}=_G V^{n_0}$. Therefore, since every element of a hyperbolic group is contained in a unique maximal elementary subgroup (see \cite{Olsh G-groups}), $T_1UT_1^{-1}$ and $V$ are contained in the same subgroup $E(V)$. The same way $T_2UT_2^{-1}\in E(V)$.
	
	In case $U=_G V$, by the properties described in the beginning of Section \ref{section elementary subgroups}, the fact that $T_1UT_1^{-1} \in E(V) (=E(U))$ implies that $T_1 \in E(V) (=E(T_2))$. The same way  $T_2 \in E(V) (=E(U))$. Also, since $V^{n_0}$ is a label of a subpath of $DA$, as it follows from the above described, we get that the sign of $n_0$ coincides with the sign of $n$. Therefore, $T_1, T_2 \in E^{+}(U)$ for $n>0$ and $T_1, T_2 \in E^-(U)$ for $n\leq 0$.
\begin{proof}[]
\end{proof}
\subsection{Proof of Lemma \ref{lemma_about_slender_conjugacy_diagrams}}
\label{subsection-proof-of-lemma-slender-c-d}
		Since $\Delta$ is minimal and contains an $\mathcal{R}$-cell, by Lemma \ref{lem 6.6}, it must contain an essential $\mathcal{R}$-cell. Let us consider an essential $\mathcal{R}$-cell $\Pi$ in $\Delta$, connected to $AB$, $BC$, $CD$ and $DA$ by contiguity subdiagrams $\Gamma_1$, $\Gamma_2$, $\Gamma_3$ and $\Gamma_4$, respectively. Then,
				in general, our diagram $\Delta$ looks like in Figure \ref{fig:  conjugacy diagram 1}, with a possibility that some of the contiguity subdiagrams $\Gamma_1$, $\Gamma_2$, $\Gamma_3$ and $\Gamma_4$, in fact, are empty (i.e. do not exist).	
		\begin{figure}[H]
			\centering
			\includegraphics[clip, trim=2cm 12.20cm 0cm 4.25cm, width=.6\textwidth]{{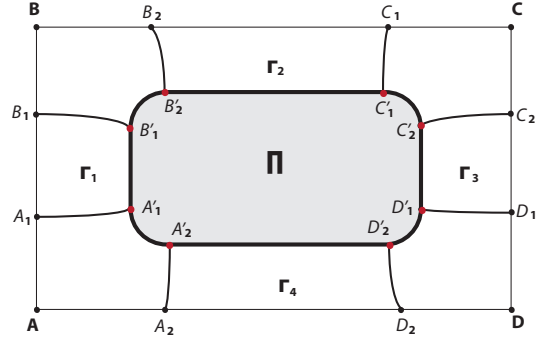}} 
			\caption{$lab(AB)=lab(DC)$ are geodesic words and $lab(BC)$, $lab(AD)$ are cyclic shifts of $U$ and $V$, respectively.} 
			\label{fig:  conjugacy diagram 1}
		\end{figure}	
		~\\
		\textit{\underline{Proof of part (1) of Lemma \ref{lemma_about_slender_conjugacy_diagrams}.}}\\
		
		First of all, by contradiction assume that at least one of $\Gamma_2$ and $\Gamma_4$ is empty. First we will consider the case when just one of them is empty and then, separately, the case when both of them are empty.\\
		~\\
			\textit{\underline{Case 1.1.}} (Exactly one of $\Gamma_2$ and $\Gamma_4$ is empty).\\
			For this case, without loss of generality assume that $\Gamma_4$ is empty. Then our conjugacy diagram $\Delta$ would look like in Figure \ref{fig: conjugacy diagram 2}.
			\begin{figure}[H]
				\centering
				\includegraphics[clip, trim=2cm 12.35cm 0cm 4.25cm, width=.6\textwidth]{{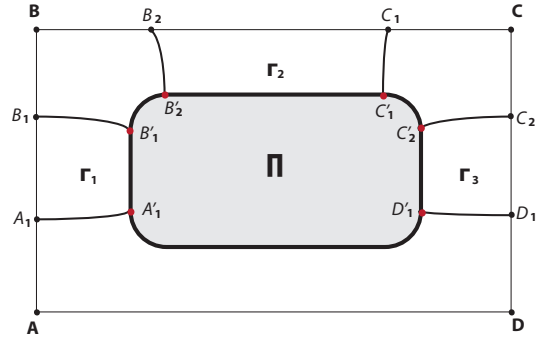}} 
				\caption{$\Gamma_4$ is empty.} 
				\label{fig:  conjugacy diagram 2}
			\end{figure}
			Since $lab(BC)=U$ is a cyclically $(\lambda, c, \epsilon, 1-121\lambda\mu)$-reduced word, we get that 
			\begin{align}
			\label{athens}
				(\Pi, \Gamma_2, BC)<1-121\lambda\mu.
			\end{align}
			 Therefore, since $\Pi$ is an essential cell, meaning that $\sum_{i=1}^4(\Pi, \Gamma_i, \partial\Delta) > 1-23 \mu$, it must be that 
			\begin{align}
			\label{inequality-dddddd}	
			(\Pi, \Gamma_1, AB)+(\Pi, \Gamma_3, CD) >(121\lambda-23)\mu>98\lambda \mu.
			\end{align}

			 In particular, at least one of $\Gamma_1$ and $\Gamma_3$ is non-empty. In fact, we claim that neither one of $\Gamma_1$ and $\Gamma_3$ is empty.\\
			~\\
			\textit{Claim.} Neither one of $\Gamma_1$ and $\Gamma_3$ is empty.
		\begin{proof}[Proof of the claim]
			 First of all, without loss of generality assume that $\Gamma_1$ is non-empty.
			
			Now since $\Delta$ is a slender $(U, V)$-conjugacy diagram, it must be that $d(A, B) \leq d(A, C_1)$. For the next chain of inequalities, in case $\Gamma_3$ is empty, we will simply assume $d(D_1', C_2')=0$. Thus we have 
			\begin{equation}
			\begin{aligned}
		 \label{ineq *}
			d(A, B) = &d(A, A_1)+d(A_1, B) \leq d(A, C_1)\\
			\leq & d(A, A_1)+d(A_1, C_1)\\
			\leq & d(A, A_1)+d(A_1, A_1')\\
			&+d(A_1', D_1')+d(D_1', C_2')+d(C_2', C_1')+d(C_1', C_1),\\
			\text{consequently, since~} d(A_1', D_1')&+d(C_2', C_1')\leq 23\mu\|\Pi\|, \text{~we have}\\
			d(A, B)\leq & d(A, A_1)+\epsilon+d(D_1', C_2')+ 23\mu \|\Pi\|+\epsilon.
			\end{aligned}
			\end{equation}
			Therefore, $d(A_1, B_1) \leq d(A_1, B) \leq d(D_1', C_2')+ 23\mu \|\Pi\|+2\epsilon$.
			Combining this with the inequality $\big\|[A_1', B_1']\big\|\leq \lambda (d(A_1, B_1)+2\epsilon)+c$, we get
			\begin{align*}
			\frac{\big\|[A_1', B_1']\big\|-c}{\lambda}-2\epsilon \leq d(A_1, B_1) \leq d(D_1', C_2')+ 23\mu \|\Pi\|+2\epsilon.
			\end{align*}
			
			Now, in case $\Gamma_3$ is empty, i.e. if $d(D_1', C_2')=0$, we also have
			\begin{equation}
			\begin{aligned}
			\label{ineq **}
				\big\|[A_1', B_1']\big\| &\geq \|\Pi\|-23\mu\|\Pi\| - \big\|[B_2', C_1']\big\|\\
				& > (1-23\lambda\mu)\|\Pi\| - \big\|[B_2', C_1']\big\|\\
				& > (1-23\lambda\mu)\|\Pi\| - (1-121\lambda\mu)\|\Pi\| \text{,~by \eqref{athens}}\\
				& = 98\lambda\mu \|\Pi\|\\
				&> \lambda(4\epsilon+23\mu\|\Pi\|)+c, \text{~by LPP}.
			\end{aligned}
			\end{equation}
		
		From (\ref{ineq **}) it follows that $d(A_1, B_1) \geq  23\mu \|\Pi\|+2\epsilon$. Therefore, $d(A, B) \geq d(A, A_1)+d(A_1, B_1) \geq d(A, A_1)+\epsilon+ 23\mu \|\Pi\|+\epsilon$, but this contradicts (\ref{ineq *}). Therefore, in order not to have contradictions, $\Gamma_3$ have to be non-empty.
	\end{proof}

			
			Note that 
			\begin{align*}
				d(A_1, D_1) &\leq d(A_1, A_1')+d(A_1',  D_1')+d(D_1', D_1)\\
				                 & \leq 2\epsilon + 23\mu\|\Pi\|.
			\end{align*}			
			Therefore, since $d(A, B)=d(D, C)$ and since by the property of cyclically slenderness, $d(A, B)\leq d(A, C)$, $d(D, C) \leq d(D, B)$, we get
			\begin{align}
			\label{inequality-aaa}
			\big| d(B, A_1)- d(C, D_1) \big| \leq d(A_1, D_1) \leq 	2\epsilon + 23\mu\|\Pi\|.
			\end{align}
			 Also, since $d(A, B) \leq d(A, B_2)$, we get
			 \begin{align}
			 \label{inequality-bbb}
			 d(B_1, B) \leq d(B_1, B_2) \leq d(B_1, B_1')+d(B_1', B_2')+d(B_2', B_2) \leq 2\epsilon+ 23\mu \|\Pi\|.	
			 \end{align}
			 Analogously, we get
			 \begin{align}
			 \label{inequality-ccc}
			 d(C_2, C) \leq d(C_2, C_1) \leq 2\epsilon + 23\mu \|\Pi\|.	
			 \end{align}
			
			After combining inequalities \eqref{inequality-aaa}, \eqref{inequality-bbb} and \eqref{inequality-ccc}, we get that 
			\begin{align*}
			\left| d(A_1, B_1) - d(D_1, C_2) \right| \leq 2(2\epsilon+23\mu\|\Pi\|) = 4\epsilon + 46\mu\|\Pi\|.
			\end{align*}
			Moreover, since $lab(AB) = lab(DC)$, 
			we get that $lab(A_1B_1)$ and $lab(D_1C_2)$ have a common subword of length at least $\max\big \{ \big\|[A_1, B_1]\big\|,  \big\|[D_1, C_1]\big\|\big\}-(4\epsilon + 46\mu\|\Pi\|)$. We will show that this is impossible. 	
						
			Assume that  it is possible. Then there exist $O_1, O_2 \in [A_1, B_1]$ such that $lab(O_1O_2)$ is also a subword of $lab(D_1C_2)$ and
			\begin{align}
			\label{inequality-eeeee}
				 \big\|[O_1, O_2]\big\|\geq \max \big\{ \big\|[A_1, B_1]\big\|,  \big\|[D_1, C_1]\big\| \big\}-(4\epsilon + 46\mu\|\Pi\|).
			\end{align}			
			 In light of \eqref{inequality-dddddd}, without loss of generality we can assume that $\big\|[A_1', B_1']\big\|\geq 49 \lambda \mu \|\Pi\|$, which, by \eqref{inequality-eeeee}, implies that 
			 \begin{align}
			 \label{inequality-ffff}
			 \big\|[O_1, O_2]\big\| \geq  49 \lambda \mu \|\Pi\|-(4\epsilon + 46\mu\|\Pi\|). 	
			 \end{align}

			 Now note that, by Corollary \ref{corollary on hausdorff distance between quasi-geodesics}, there exist $O_1', O_2' \in [A_1', B_1']$ such that $d(O_1, O_1')$, $d(O_2, O_2') \leq \epsilon+R_{\lambda, c}+2\delta \leq 2\epsilon$. Therefore, by the triangle inequality, we have
			\begin{align*}
				\big\|[O_1', O_2']\big\| &\geq \big\|[O_1, O_2]\big\| - 2(\epsilon+R_{\lambda, c}+2\delta)\\
				 &\geq 49 \lambda \mu \|\Pi\|-(4\epsilon + 46\mu\|\Pi\|)- 2(2\epsilon)\\ \text{~by \eqref{inequality-ffff}}
				 &>2\mu \|\Pi\|, \text{~by LPP}.
			\end{align*}
			The last inequality contradicts Lemma \ref{lemma_about_contiguity_arcs_with_the_same_label}.
		Therefore, we got a contradiction, which means that	
we are done with Case 1.1.\\

\underline{Illustration.} For the sake of clarity of the above arguments, let us consider the following diagram: let us consider a $(U, V)$-conjugacy-diagram $\bar{\Delta}$ which is a copy of $\Delta$ with $\bar{\Delta} = \bar{A}\bar{B}\bar{C}\bar{D}$ and all points and subdiagrams inside have the same notations but with $\bar{bar}$ and let us attach this diagram to $\delta$ along the sides $DC$ and $\bar{A}\bar{B}$. Let us denote the new diagram obtained this way by $\bar{\Delta}'$. See Figure \ref{fig:  Conjugacy-diagram_3}.\\
				\begin{figure}[H]
					\centering
					\includegraphics[clip, trim=0cm 14.1cm .7cm 8.3cm, width=1.1\textwidth]{{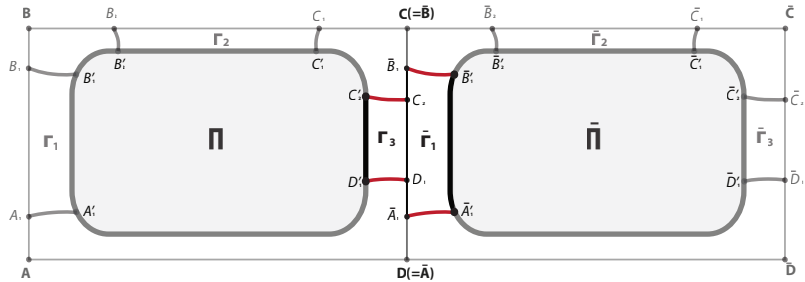}} 
					\caption{$\bar{\Delta}'$: in the figure depicted the case when $O_1=D_1$ and $O_2=C_2$.} 
					\label{fig:  Conjugacy-diagram_3}
				\end{figure}
~\\
			\textit{\underline{Case 1.2.}} (Both $\Gamma_2$ and $\Gamma_4$ are empty).\\
			In this case the $(U, V)$-conjugacy diagram $\Delta$ looks like in Figure \ref{fig:  conjugacy diagram case-2}.
			\begin{figure}[H]
				\centering
				\includegraphics[clip, trim=2cm 12.4cm 0cm 4.25cm, width=.6\textwidth]{{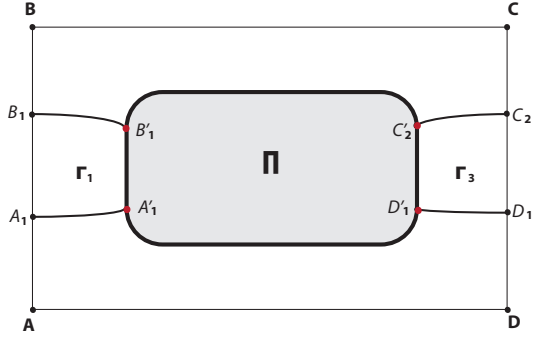}} 
				\caption{$\Gamma_2$ and $\Gamma_4$ are empty.} 
				\label{fig:  conjugacy diagram case-2}
			\end{figure}
			The emptiness of $\Gamma_2$ and $\Gamma_4$ implies the following estimation of the lengths of arcs $[B_1', C_2']$ and $[D_1', A_1']$:  $\big\|[B_1', C_2']\big\|, \big\|[D_1', A_1']\big\| \leq 23 \mu \|\Pi\|$. Therefore, from the cyclically slenderness of $\Delta$, it follows that
			\begin{align*}
				d(A, B) &= d(A, A_1) + d(A_1, B)  \leq d(A, C)\\
				 & \leq d(A, A_1)+ d(A_1, A_1') + d(A_1', D_1') + d(D_1', D_1) + d(D_1, C)\\
				 & \leq 23 \mu \|\Pi\| + 2\epsilon + d(D_1, C).
			\end{align*}
			Therefore, we get that $d(A_1, B) - d(D_1, C) = d(D, D_1)- d(A, A_1) \leq 23 \mu \|\Pi\|$. And from the symmetric arguments, we obtain $\left| d(A, A_1)-d(D, D_1) \right| \leq 23 \mu\|\Pi\|$. Analogously, $\left| d(B, B_1)-d(C, C_2) \right| \leq 23 \mu \|\Pi\|$. The rest is just a repetition of arguments of Case 1.
			
			Thus the conclusion from Case 1 and Case 2 is that, in fact, $\Gamma_2$ and $\Gamma_4$ are non-empty.\\
				At this point we already showed that $\Gamma_1$ and $\Gamma_4$ must be non-empty, i.e. we are done with the first part of the lemma. Thus the part (1) of the lemma is proved.\\
		
		Now we are in a position to show the parts (2) and (3).\\
		~\\
		\textit{\underline{Proof of parts (2) and (3) of Lemma \ref{lemma_about_slender_conjugacy_diagrams}.}}\\
		~\\
		First of all, note that since $\Pi$ is an essential cell, i.e.  $\sum_{i=1}^4(\Pi, \Gamma_i, \partial\Delta) > 1-23 \mu$, part (2) immediately follows from part (3). Therefore, it is enought= to prove the statement of part (3).
		
		To that end, let us first consider the case when at least one of $\Gamma_1$ and $\Gamma_3$ is empty. If both of $\Gamma_1$ and $\Gamma_3$ are empty, then there is nothing to prove for part $(3)$, and part $(2)$ is also true in that case, because $\Pi$ is an essential cell. Therefore, let us separately consider two cases: when exactly one of $\Gamma_1$ and $\Gamma_3$ is empty and when both of them are non-empty.\\
		~\\
		\textit{\underline{Case 2.1.}} (Exactly one of $\Gamma_1$ and $\Gamma_3$ is empty).\\
		\begin{figure}[H]
				\centering
				\includegraphics[clip, trim=2.cm 9cm 1cm 7.3cm, width=.5\textwidth]{{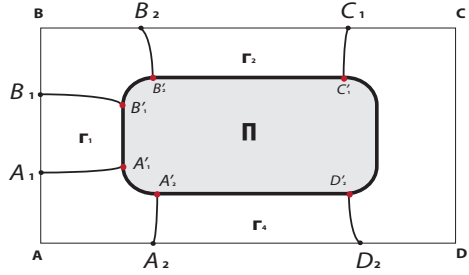}} 
				\caption{$\Gamma_3$ is empty, but $\Gamma_1$ is not.} 
				\label{fig:  conjugacy diagram case-2-1}
			\end{figure}
		
		 For this case, without loos of generality let us assume that $\Gamma_1$ is non-empty and $\Gamma_3$ is empty. See Figure \ref{fig:  conjugacy diagram case-2-1}. Then, since,  by cyclic slenderness property, we have
		\begin{align*}
		d(D, C) \leq d(D_2, C_1) \leq d(D_2, D_2')+d(D_2', C_1')+d(C_1', C_1)\leq 2\epsilon+23\mu \|\Pi\|
		\end{align*}
		and $d(A_1, B_1) \leq d(A, B) = d(D, C)$, we get that $d(A_1, B_1) \leq 2\epsilon+23\mu \|\Pi\|$. But also, since $lab[A'_1, B'_1]$ is a $(\lambda, c)$-quasi-geodesic word in $\Gamma(H, X)$, we have that
		\begin{align*}
			\big\|[A'_1, B'_1]\big\| \leq & \lambda d(A_1', B_1') +c \leq \lambda(d(A_1, B_1) + 2\epsilon) + c\\
			 \leq & \lambda(23\mu \|\Pi\| + 2\epsilon) + c < 29 \lambda \mu \|\Pi\| \text{~by LPP}.
		\end{align*}
		 Thus we are done in the case when at least one of $\Gamma_1$ and $\Gamma_3$ is empty, i.e with Case 2.1.\\
		 ~\\
		 \textit{\underline{Case 2.2.}} (Both $\Gamma_1$ and $\Gamma_3$ are non-empty).\\
		 Since we already showed that $\Gamma_2$ and $\Gamma_4$ are non-empty, this case is equivalent of saying that all $\Gamma_i$, $i=1,2,3,4$, are non-empty, that is the case depicted in Figure \ref{fig:  conjugacy diagram 1}.

	For this case, by contradiction, assume that $\max\{(\Pi, \Gamma_1, AB), (\Pi, \Gamma_3, CD)\} > 49 \lambda \mu$.
	
	 Now, since $\Delta$ is cyclically slender, we get that $d(B, A) \leq d(B, A_2)$. Therefore,
		\begin{align*}
		d(A_1, A) \leq d(A_1, A_2) \leq d(A_1, A_1')+d(A_1', A_2')+d(A_2', A_2) \leq 2\epsilon +23\mu \|\Pi\|. 
		\end{align*}
		 The same way we get that 	$d(B, B_1), d(C, C_2), d(D, D_1) \leq 2\epsilon +23\mu \|\Pi\|$. Therefore, since $d(A, B)=d(D, C)$, we get that $\left|d(A_1, B_1)-d(D_1, C_2) \right| \leq 2(2\epsilon +23\mu \|\Pi\|)$. Moreover, this observation, combined  with the fact that $lab(AB)=lab(DC)$, implies that $lab([A_1, B_1])$ and $lab([D_1, C_2])$ have a common subword of length at least $\max \{\|[A_1, B_1]\|, \|[D_1, C_2]\| \}-2(2\epsilon +23\mu \|\Pi\|)$.
		  But this is exactly a situation which we discussed while dealing with Case 1.2. Moreover, there we showed that this case is impossible if $\max\{(\Pi, \Gamma_1, AB), (\Pi, \Gamma_3, CD)\} > 49 \lambda \mu$, hence we get a contradiction. This finishes the discussion of Case 2.2.
		  
		  Thus part $(3)$ of the lemma is proved too. 
		  \begin{proof}[]
		  \end{proof}
~\\
~\\
~\\

	\printindex
	
\end{document}